\documentclass[11pt, letterpaper, reqno]{amsart}
\usepackage[utf8]{inputenc} 	
\usepackage{microtype} 			
\usepackage{geometry} 			
\usepackage{amsmath}			
\usepackage{amsthm}		 		
\usepackage{amssymb}	 		
\usepackage{bm}					
\usepackage{mathrsfs}			
\usepackage{xcolor}				
\definecolor{darkred}{rgb}{0.6, 0.1, 0.1}
\definecolor{darkblue}{rgb}{0.2, 0.2, 0.6}
\definecolor{darkgreen}{rgb}{0.2, 0.4,.1}
\definecolor{mellowyellow}{rgb}{1,.8,.2}
\definecolor{bettercyan}{rgb}{0.1, 0.4, 0.7}
\usepackage{tikz}				
\usepackage{tikz-3dplot}
\usepackage{pgfplots}
\usepackage[all]{xy}
\pgfplotsset{compat=1.13}
\usetikzlibrary{decorations.markings,decorations.pathmorphing,shapes}
\usepackage{subcaption}
\usepackage{floatrow}
\usepackage{enumitem} 			
\usepackage{array}
\usepackage{comment}
\usepackage{lmodern}			
\usepackage[T1]{fontenc}		
\usepackage{cases}
\usepackage[colorlinks=true, citecolor=bettercyan, linkcolor=darkred, urlcolor=magenta]{hyperref}
\usepackage{cleveref}
\usepackage[backend=bibtex, style=alphabetic, date=year, backref=true, firstinits=true, isbn=false]{biblatex}
\makeatletter
\renewbibmacro{in:}{}
\DeclareFieldFormat{pages}{#1}
\renewcommand*{\bibnamedash}{%
	\leavevmode\raise +0.6ex\hbox to 5.5ex{\hrulefill}.\space\space}

\InitializeBibliographyStyle{\global\undef\bbx@lasthash}

\newbibmacro*{bbx:savehash}{\savefield{fullhash}{\bbx@lasthash}}

\renewbibmacro*{author}{%
	\ifboolexpr{
		test \ifuseauthor
		and
		not test {\ifnameundef{author}}
	}
	{%
		\iffieldequals{fullhash}{\bbx@lasthash}
		{\bibnamedash\addcomma\space}
		{\printnames{author}}%
		\usebibmacro{bbx:savehash}%
		\iffieldundef{authortype}
		{}
		{%
			\setunit{\addcomma\space}%
			\usebibmacro{authorstrg}%
		}%
	}
	{\global\undef\bbx@lasthash}%
}
\makeatother
\numberwithin{equation}{section}

\newtheorem{propositionx}{Proposition}[section]
\newenvironment{proposition}
{\pushQED{\qed}\propositionx}
{\popQED\endpropositionx}

\newtheorem{theoremx}[propositionx]{Theorem}
\newenvironment{theorem}
{\pushQED{\qed}\theoremx}
{\popQED\endtheoremx}

\newtheorem{lemmax}[propositionx]{Lemma}
\newenvironment{lemma}
{\pushQED{\qed}\lemmax}
{\popQED\endlemmax}

\theoremstyle{definition}

\newtheorem{examplex}[propositionx]{Example}
\newenvironment{example}
{\pushQED{\qed}\examplex}
{\popQED\endexamplex}

\theoremstyle{remark}
\newtheorem{remarkx}[propositionx]{Remark}
\newenvironment{remark}
{\pushQED{\qed}\remarkx}
{\popQED\endremarkx}

\newtheorem*{warning}{\textsc{\textbf{Warning}}}

\makeatletter
\@namedef{subjclassname@2020}{\textup{2020} Mathematics Subject Classification}
\makeatother

\newcommand{\VVV}{\scalebox{.9}[1]{$V$}\hspace{-0.32em}W}

\newcommand{\dd}{\,\mathrm{d}}

\newcommand{\calctwo}{\natural\mathrm{2res}} 
\newcommand{\calc}{\natural\mathrm{res}} 
\newcommand{\calczero}{\natural} 
\newcommand{\calcshort}{\natural} 

\newcommand{\cl}{\mathrm{cl}}

\newcommand{\NPcalctwo}{P_1}

\newcommand{\la}{\ensuremath{\langle}}
\newcommand{\ra}{\ensuremath{\rangle}}

\newcommand{\bbC}{\mathbb{C}}

\newcommand{\bbM}{\mathbb{M}}
\newcommand{\bbN}{\mathbb{N}}

\newcommand{\bbR}{\mathbb{R}}
\newcommand{\bbS}{\mathbb{S}}

\newcommand{\calA}{\mathcal{A}}
\newcommand{\calB}{\mathcal{B}}
\newcommand{\calC}{\mathcal{C}}

\newcommand{\calE}{\mathcal{E}}
\newcommand{\calF}{\mathcal{F}}

\newcommand{\calO}{\mathcal{O}}

\newcommand{\calR}{\mathcal{R}}
\newcommand{\calS}{\mathcal{S}}

\newcommand{\calV}{\mathcal{V}}

\newcommand{\calX}{\mathcal{X}}
\newcommand{\calY}{\mathcal{Y}}

\newcommand{\scrF}{\mathscr{F}}

\newcommand{\mk}{\ensuremath{\mathfrak}}

\newcommand{\frakD}{\mathfrak{D}}

\newcommand{\fraka}{\mathfrak{a}}

\newcommand{\frakm}{\mathfrak{m}}

\newcommand{\bfA}{\mathbf{A}}
\newcommand{\bfB}{\mathbf{B}}

\newcommand{\bfj}{\mathbf{j}}
\newcommand{\bfk}{\mathbf{k}}

\newcommand\WF{\operatorname{WF}}

\newcommand\ang[1]{\langle #1 \rangle}
\newcommand\aang[1]{\langle \! \langle #1 \rangle \! \rangle}
\newcommand\taun{\tau_{\natural}}
\newcommand\xin{\xi_{\natural}}
\newcommand\zetan{\zeta_{\natural}}

\newcommand{\pa}{\partial}
\newcommand{\CI}{C^\infty}

\newcommand{\supp}{\operatorname{supp}}

\renewcommand{\Box}{\square}
\newcommand{\ep}{\epsilon}

\newcommand{\Diff}{\mathrm{Diff}}

\newcommand{\Diffnatinf}{\mathrm{Diff}_{\natural,\infty}}

\newcommand{\RR}{\mathbb{R}}
\newcommand{\Cx}{\mathbb{C}}

\newcommand{\R}{\mathbb{R}}

\newcommand{\cS}{\mathcal S}

\newcommand{\be}[1]{\begin{equation}\label{#1}}
\newcommand{\ee}{\end{equation}}

\newcommand{\wavetime}{\mathfrak{t}}
\newcommand{\natdiff}{{}^\natural d}
\newcommand{\Tnat}{{}^{\natural}T}

\title[Microlocal analysis of the non-relativistic limit of the Klein--Gordon equation]{Microlocal analysis of the non-relativistic limit of the Klein--Gordon equation: Asymptotics}
\author{Andrew Hassell} 
\address{Department of Mathematics, Australian National University, Canberra ACT 0200, AUSTRALIA}
\email{Andrew.Hassell@anu.edu.au}
\author{Qiuye Jia}
\address{Mathematical Sciences Institute, Australian National University, Canberra ACT 2601, AUSTRALIA}
\email{Qiuye.Jia@anu.edu.au}
\author{Ethan Sussman}
\address{Department of Mathematics, Northwestern University, Evanston Illinois, USA}
\email{ethan.sussman@northwestern.edu}
\author{Andr\'as Vasy}
\address{Department of Mathematics, Stanford University, California, USA}

\date{November 10th, 2025.}

\begin{document}

\begin{abstract}
	This is the less technical half of a two-part work in which we introduce a robust microlocal framework for analyzing the non-relativistic limit of relativistic wave equations with time-dependent coefficients, focusing on the Klein--Gordon equation. 
	Two asymptotic regimes in phase space are relevant to the non-relativistic limit: one corresponding to what physicists call ``natural'' units, in which the PDE is approximable by the free Klein--Gordon equation, and a low-frequency regime in which the equation is approximable by the usual Schr\"odinger equation. As shown in the companion paper, combining the analyses in the two regimes gives global estimates which are uniform as the speed of light goes to infinity.
	In this paper, we derive  \emph{asymptotics} from those estimates.
	Our framework differs from those in previous works in that ours is based on spacetime phase-space analysis. 
\end{abstract}
\maketitle
\tableofcontents
\section{Introduction}
\label{sec:intro}

In this article, we study the Klein--Gordon operator with the speed of light, $c>0$, as a large parameter. In the main, we consider variable-coefficient metrics. One of the novelties of our approach is that fairly general asymptotically flat spacetimes, the sort that arises in gravitational physics, can be handled. For example, we prove, in a large amount of generality, that the mass present in the spacetime enters the non-relativistic limit as an effective Newtonian potential (outside of the region where the mass is present).\footnote{We treat gravity as a relativistic effect, so one that is suppressed as $c\to\infty$. See the precise assumptions in \S\ref{subsec:assumptions}. It is a particular $O(1/c^2)$ term in the metric that enters the non-relativistic limit as an effective Newtonian potential.}
However, for expository purposes, we consider in this introduction only the flat (i.e.\ Minkowski) metric, with time-dependent lower-order terms. 
Our operator $P = \{P(c)\}_{c>0}$, under these restrictions, will be a family of differential operators on $\bbR^{1,d}$ with coordinates $(t,x)$ taking the form 

\begin{equation}
P = -\frac{1}{c^2} \Big( \frac{\partial}{\partial t} + i V \Big)^2 - (i \nabla+\bfA )^2 - c^2 + W, 
\label{eq:P_intro}
\end{equation}
where $\nabla = (\partial_{x_1},\dots,\partial_{x_d})$, 
\begin{itemize}
	\item $c>0$ is the speed of light, which is assumed to be large,
	\item $V,W\in C^\infty(\bbR^{1+d};\bbR)$, 
	\item $\bfA=(A_1,\dots,A_d)\in C^\infty(\bbR^{1+d};\bbR^d)$,
	\item The coefficients of the PDE are all $O(\rho)$, where $\rho = 1/\langle r \rangle$. Thus, we do not allow potentials which are longer-range than Coulomb.
    \item The following \textit{technical} condition, which seems to reflect mainly the limitations of our mathematical tools: the coefficients of the PDE depend smoothly on the coordinates $\hat{t}=t/r $, $\rho=1/r$, and $\theta = \rho x$, all the way down to $\rho=0$: 
	\begin{equation}
	    W,V,A_1,\dots,A_d \in \rho C^\infty(\bbR_{\hat{t}} \times [0,\infty)_\rho \times \bbS^{d-1}_\theta ). 
		\label{eq:extendability}
	\end{equation}
    So, the coefficients of the operator $P$ admit full asymptotic expansions in increasing powers of $1/r$, with coefficients in $C^\infty(\bbR_{\hat{t}}\times \bbS^{d-1}_\theta)$.\footnote{\textit{For those readers familiar with the language of compactification}: we can combine \cref{eq:extendability} with the statement that $V,W,\bfA$ are smooth functions of Cartesian coordinates by writing
		\[
        W,V,A_1,\dots,A_d\in C^\infty(\bbR_{t/\langle r\rangle} ; \langle r \rangle^{-1}C^\infty(\overline{\bbR^{d}_x})).
        \]
        Here, $\overline{\bbR^d}= \bbR^d\sqcup \infty\bbS^{d-1}$ is the radial compactification $\overline{\bbR^d}\hookleftarrow \bbR^d$ of $\bbR^d$.
        A slightly stronger requirement is 
        \[ 
        W,V,A_1,\dots,A_d\in (1+t^2+r^2)^{-1/2} C^\infty(\bbM),
        \]
        where $\bbM=\overline{\bbR^{1+d}}$ is the radial compactification \emph{of spacetime}. This latter requirement is stronger because it specifies how the coefficients behave near the ``north/south poles'' $\partial \bbM\cap \operatorname{cl}_\bbM\{r=0\}$ of the spacetime. But this behavior is unimportant here, because our main concern is how solutions behave in bounded intervals of time. Causality guarantees that the behavior of the coefficients of the PDE near the north/south poles of $\bbM$ has no effect on the solution until $|t|\gg 1$. We will discuss this point more in \S\ref{sec:true_intro}.
    }
\end{itemize}
Thus, our operator $P$ constitutes what is usually called classical, possibly long-range (but not very long-range), perturbation of the free Klein--Gordon operator $P_0=\square - c^2 = -c^{-2} \partial_t^2 -\triangle - c^2$.\footnote{We use $\triangle$ to denote the positive semi-definite Euclidean Laplacian, $\triangle = -(\partial_{x_1}^2+\dots+\partial_{x_d}^2)$). See \S\ref{sec:notation} for an index of this and other notation.}
   
We will study the operator $P$ as a family $P=\{P(c)\}_{c>0}$ indexed by the parameter $c$, keeping everything else fixed, and we would like to understand the asymptotics of solutions of the Klein--Gordon equation $Pu=0$ in the ``non-relativistic limit,'' $c \to \infty$. 

See \Cref{rem:comparison} for a discussion of the physical interpretation of \cref{eq:P_intro}.

In this section, we aim to give a gentle introduction to the non-relativistic limit for the operator in \cref{eq:P_intro} and to some of the intricacies that arise in the analysis thereof. 
The rest of the manuscript is self-contained, which means that this section can technically be skipped. 
In particular, readers familiar with the non-relativistic limit or microlocal methods
may wish to skip directly to  \S\ref{sec:true_intro} (after reading the literature review
\S\ref{subsec:literature_only}), which functions as a second, more technical introduction, aimed at an audience familiar with microlocal analysis. It is there that our main results appear in full generality, together with an outline (see \S\ref{subsec:outline}) of the remainder of the manuscript. In comparison, the present section is more expository. Unlike the rest of the paper, it is intended to be readable by an audience not familiar with microlocal analysis.

One of the standard pieces of lore of the non-relativistic limit, going back \cite{KraghI, KraghII}\cite{Dirac} to the early years of quantum mechanics, is that (sufficiently slowly oscillating) solutions  of relativistic wave equations are, as $c\to\infty$, well-approximated by solutions of the Schr\"odinger equation. 
More precisely, we would like to approximate 
\begin{equation}
	u \approx e^{-i c^2 t} v_- + e^{ic^2 t} v_+, 
\end{equation}
where $v_\pm$ are solutions of some specific Schr\"odinger equations.
The most elementary of our main theorems, when applied to the operators at hand, \cref{eq:P_intro}, provides a version of this:
\begin{theorem}
	Let $\varphi,\psi \in \calS(\bbR^{d})$ \footnote{Here, $\calS(\bbR^d)$ is the space of Schwartz functions on $\bbR^d$.}. 
	Suppose that $u \in C^\infty(\bbR^{1,d})$ is the solution of the Cauchy problem 
	\begin{equation}
	\begin{cases}
	Pu=0, \\
	u|_{t=0} = \varphi, \\ 
	\partial_t u|_{t=0} = c^2\psi.
	\end{cases}
	\tag{IVP}
	\label{eq:IVP}
	\end{equation}
	For each sign $\pm$, let $v_\pm \in C^\infty(\bbR^{1,d})$ denote the solution of the Schr\"odinger initial-value problem 
	\begin{equation}
	\begin{cases}
	(\pm i   \partial_t +  2^{-1}(i\nabla+\bfA)^2 + V_{\mathrm{eff}} ) v_\pm = 0, \\ 
	v_\pm|_{t=0} = \varphi_\pm,
	\end{cases}
	\label{eq:Schrodinger_initial}
	\end{equation}
	where $V_{\mathrm{eff}}$ is given by $V_{\mathrm{eff}}=\mp V-2^{-1} W$ and the initial data is $\varphi_\pm = ( \varphi \mp i \psi)/2$.
	Then, letting $v=\exp(-ic^2 t) v_- + \exp(ic^2 t) v_+$,
	we have, for any $\varepsilon,T>0$ and $p\in [2,\infty]$, the estimate 
	\begin{equation}
	\lVert u - v \rVert_{\langle r \rangle^{3/2+\varepsilon} L^p( [-T,T]_t\times \bbR^d_{x}) } = O \Big( \frac{1}{c^2} \Big)
	\label{eq:misc_a06}
	\end{equation}
	as $c\to\infty$.
	Here, the constant on the right-hand side can depend on $\varepsilon,T$, and on the coefficients of the operator $P$, as well as on the initial data $\varphi,\psi$. So, $u-v\to 0$ as $c\to\infty$, uniformly in compact subsets of $\bbR_t\times \bbR^d_x$. 
    In fact, we also have control on derivatives: for any constant-coefficient differential operator  $L$ formed from the derivatives $c^{-2} \partial_t, \partial_{x_j}$, 
	\begin{equation}
	\lVert L(u - v )\rVert_{\langle r \rangle^{3/2+\varepsilon} L^p( [-T,T]_t\times \bbR^d_{x}) } = O \Big( \frac{1}{c^2} \Big)
    \label{eq:derivatives_control}
	\end{equation}
	as $c\to\infty$. 
	\label{thm:simplest}
\end{theorem}

See \Cref{thm:Cauchy} for the full version of this theorem -- allowing variable metrics and containing a long-time result -- and see \Cref{thm:inhomog} for our main microlocal theorem, proven in \cite{NRL_I}. Let us emphasize that it is the latter theorem on which the novelty of our work hinges, but it appears to us that even \Cref{thm:simplest} is novel in the setting with time-dependent coefficients.

\begin{warning}
    It should be emphasized from the outset that the previous theorem does not provide a complete understanding of the non-relativistic limit. We will see that when the initial data is not smooth, then the non-relativistic limit receives high-frequency corrections. These may be $O(c^{-2})$ in $L^\infty$ norm, in which case they do not spoil the specific theorem statement above (though they will not be $O(c^{-\infty})$), but if the initial data is irregular enough, the corrections will be $\Omega(c^{-2})$.  
    In the geometric language of \cite{NRL_I}, these are located ``at the $\natural$-face $\natural\mathrm{f}$.'' See \S\ref{subsec:natural}, \S\ref{sec:inhomogeneous} for further discussion of this important conceptual point. 
\end{warning}

The statement on derivatives in \cref{eq:derivatives_control} allows us to rigorously verify the expected non-relativistic limit of the \emph{four current} $j=(\rho,\bfj)$ of $u$, defined\footnote{Physicists sometimes use different normalization conventions on the four-current in the relativistic and non-relativistic settings. Our convention, insofar as this can be called a convention, is such that the former agrees with the latter in the non-relativistic limit. 
Consequently, our convention differs from some in the physics literature by a factor of $2$.} by 
\begin{align}
	\begin{split} 
	\rho &= c^{-2}\Im (u^*\partial_t u) = (2ic^2)^{-1}(u^* \partial_t - u \partial_t u^*), \\
	\bfj &= \Im(u^* \nabla u) = (2i)^{-1}(u^* \nabla u - u\nabla u^*).
	\end{split} 
\end{align}
Indeed, the components of $j$ converge at an $O(1/c^2)$ rate, uniformly in compact subsets of spacetime, to the components of the four current of $v$, which is computed to be 
\begin{equation}
(|v_+|^2-|v_-|^2, \Im (v_+^* \nabla v_+)+\Im (v_-^* \nabla v_-) +  \underbrace{\Im (e^{-2ic^2t} v_+^* \nabla v_- + e^{2ic^2t}v_-^* \nabla v_+)}_{\text{Zitterbewegung term}})
\end{equation}
modulo $O(1/c^2)$ terms. That is,
\begin{equation}
\rho \to |v_+|^2-|v_-|^2
\end{equation}
and 
\begin{align}
\bfj - \Im ( e^{-2ic^2 t}v_+^* \nabla v_- + e^{2ic^2t}v_-^* \nabla v_+)&\to \Im (v_+^* \nabla v_+)+\Im (v_-^* \nabla v_-),
\end{align}
uniformly in compact subsets of spacetime. 
The quantity $|v_+|^2-|v_-|^2$ is the non-relativistic formula for charge density, and $\Im (v_+^* \nabla v_+)+\Im (v_-^* \nabla v_-)$ is the non-relativistic formula for current density. The necessity of subtracting off the highly oscillatory Zitterbewegung term 
from $\bfj$ is usually seen as surprising. Zitterbewegung is an interference effect between the particle and anti-particle parts of the wavefunction $u$.

\begin{example}\label{ex:FEM}
	We depict in Figs.\ \ref{fig:FEM}, \ref{fig:FEM2}, \ref{fig:FEM3}, \ref{fig:FEM4} 
	a numerically computed example of the non-relativistic limit in the setting where the coefficients of the PDE are time-dependent. Specifically, $W,\bfA=0$, and
	\begin{equation}
		V(t,x) = \frac{8}{1+(x-1+t)^2}.
        \label{eq:V_ex}
	\end{equation}
	For initial data, we took $\varphi(x)=e^{-x^2}$ to be Gaussian and $\psi=0$.  
	The factor of $8$ in the potential was chosen so as to make the potential's effects strong without having a negative impact on numerical precision.

	As can be seen in \Cref{fig:FEM}, the solution $u$ is highly oscillatory but well-approximated by the ansatz $v$, which depends on $c$ only through the oscillatory prefactors multiplying $v_\pm$. The error $u-v$ is decaying as $c\to\infty$, uniformly in compact subsets of spacetime. The numerics show that it is decaying at the expected $O(1/c^2)$ rate.
	
	This example shows clearly how the electric potential causes the ``particle'' and ``anti-particle'' parts of the wavefunction, which agree at $t=0$ as a consequence of the choice of the initial data $\varphi,\psi$, to separate as $|t|$ increases; the $v_+$ part is attracted to the potential, appearing to stick to its center $\{x=1-t\}$, while the $v_-$ part is strongly repelled, as seen in \Cref{fig:FEM2}.

    The non-relativistic limits of $\rho,\bfj$ are depicted in \Cref{fig:FEM3}, \Cref{fig:FEM4}, respectively.
\end{example}

\begin{figure}[t]
	\begin{subfigure}{\textwidth}
		\centering
		\includegraphics[width=.99\textwidth]{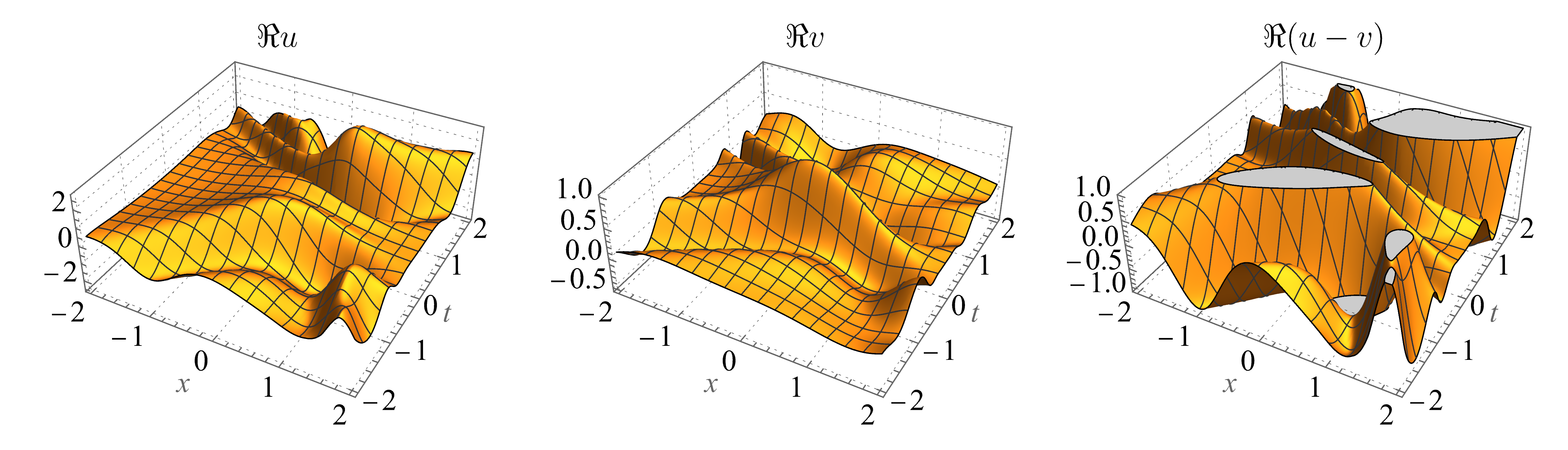}
		\caption{$c=1$. Note that $u\not\approx v$.}\vspace{1em}
	\end{subfigure}
	\begin{subfigure}{\textwidth}
		\centering
		\includegraphics[width=.99\textwidth]{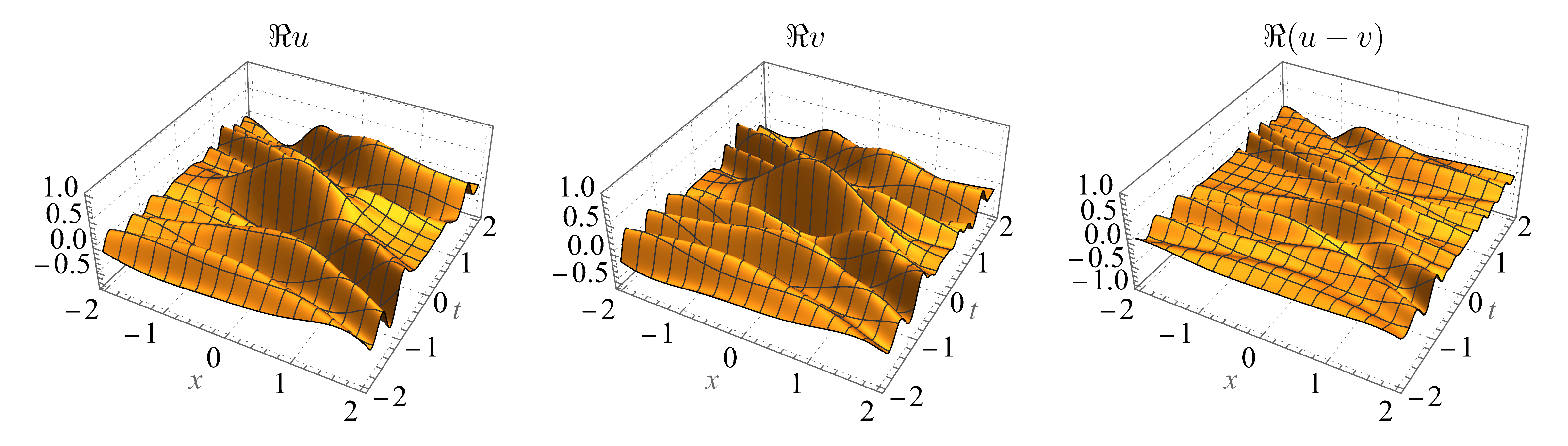}
		\caption{$c=3$. Now $u\approx v$.}\vspace{1em}
	\end{subfigure}
	\begin{subfigure}{\textwidth}
		\centering
		\includegraphics[width=.99\textwidth]{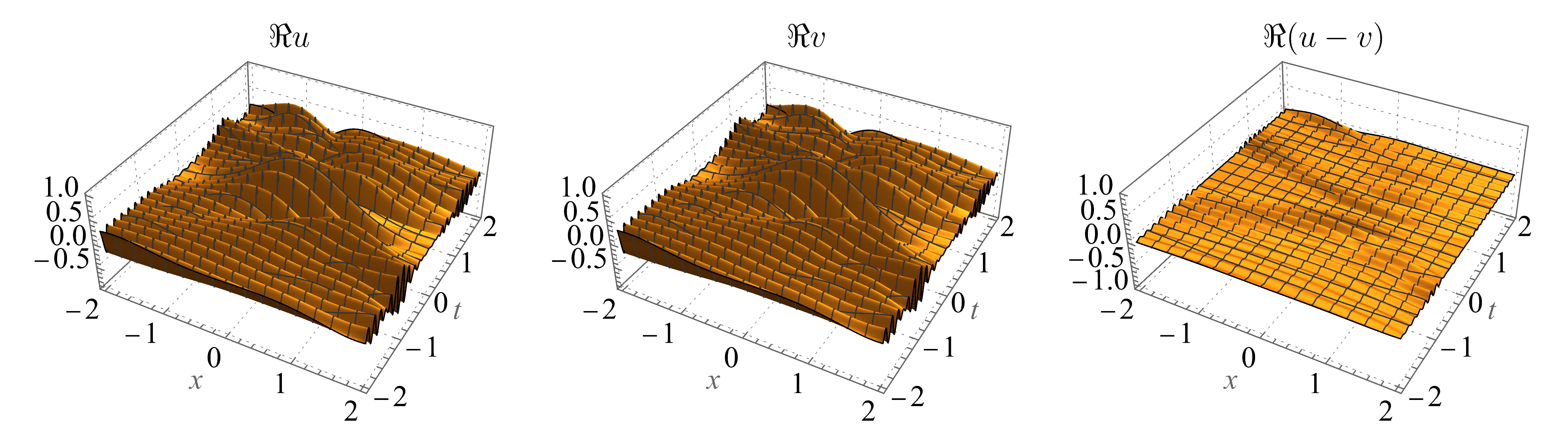}
		\caption{$c=6$. Now $u,v$ are almost indistinguishable. Notice the oscillations $\sim e^{\pm ic^2 t}$.}
	\end{subfigure}
	\caption{\Cref{ex:FEM}. The real parts of the functions $u,v,u-v$ are shown for three different values of $c$. Evidently, $u\approx v$ if $c\gg 1$. See \Cref{fig:FEMIm} for the imaginary parts.}
	\label{fig:FEM}
\end{figure}

\begin{figure}[t]
	\includegraphics[width=.65\textwidth]{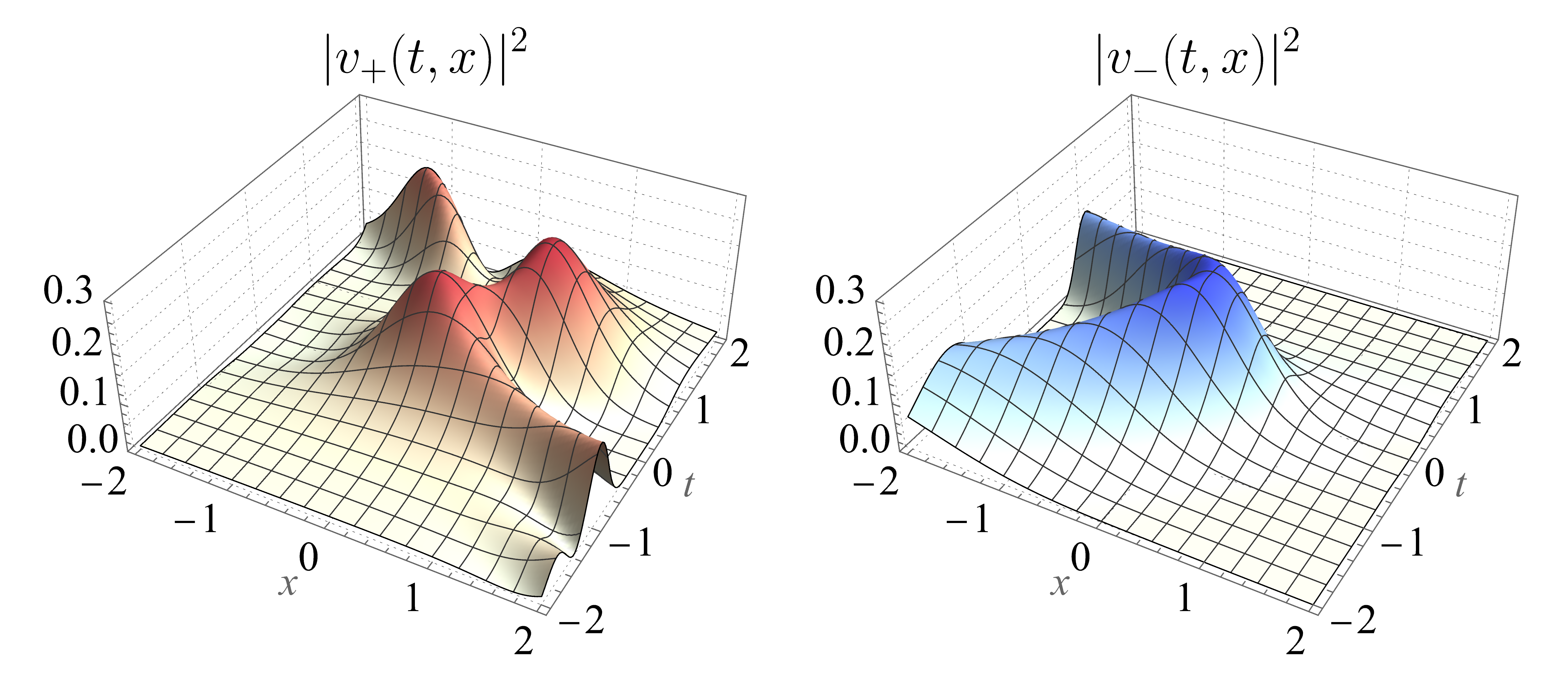} 
	\caption[Continuation of \Cref{ex:FEM}.]{Continuation of \Cref{ex:FEM}, showing the $L^2$-density of the Schr\"odinger solutions $v_\pm$. 
    The solutions are concentrated near the paths $\gamma_\pm$ followed by a classical particle moving in the force-field $F=\mp \nabla V$ generated by the potential $V$ defined by \cref{eq:V_ex}; see \Cref{fig:paths}.
    }
	\label{fig:FEM2}
\end{figure}

\begin{remark}[Optimality or lack thereof]
	We have chosen \Cref{thm:simplest} to present here because of its simplicity, but it is sub-optimal in several ways:
	\begin{itemize}
		\item (\textit{Regarding allowed data}.) Although we stated the theorem only for Schwartz initial data, the actual proof, which can be found in \S\ref{subsec:Cauchy}, requires only a finite amount of regularity and a finite amount of decay.
		\item (\textit{Regarding $L^p$-weight}.) The $\langle r \rangle^{3/2+\varepsilon}$ weight in \cref{eq:misc_a06} is likely not sharp. The loss enters in a few places, including lossy Sobolev embeddings.
		\item (\textit{Subleading terms.}) The $O(1/c^2)$ estimate in \cref{eq:misc_a06} is sharp in the sense that, even if we allow $\varphi_\pm$ to depend on $c$, then, typically, 
		\begin{equation}
			\lVert u - v \rVert_{\langle r \rangle^{3/2+\varepsilon} L^\infty( [-T,T]_t\times \bbR^d_{x}) } = \Omega\Big( \frac{1}{c^2} \Big) . 
			\label{eq:misc_08}
		\end{equation}
		In fact, $u$ will admit a full asymptotic expansion in powers of $c^{-2}$. The approximation $u\approx v$ is just the leading term in that expansion. 
		The next term typically has size $O(c^{-2})$, hence \cref{eq:misc_08}.

		We will not discuss subleading terms much in this paper. 
		For the free Klein--Gordon equation (i.e.\ the constant-coefficient case) we will present the computation of the first few subleading terms in \S\ref{sec:free}. See specifically \S\ref{subsub:freeerror}.
		For the variable-coefficient case, the tools we develop later do provide a means to compute the subleading terms.  We leave the details to the interested reader. 
		An interesting and related question, which we do not address and do not have the tools to address, involves analyticity in $1/c^2$. 
		\item (\textit{Uniformity in} $T$.) In the theorem, nothing is assumed about the $t\to\infty$ behavior of the coefficients of the PDE, so nothing is stated about the $t\to\infty$ behavior of $u$; the constant in the estimate is allowed to depend on the parameter $T$.
		Later, we will show that if the coefficients are well-behaved on the radial compactification 
		\begin{equation} 
			\overline{\bbR^{1+d}} = \bbR^{1+d}\sqcup \infty \bbS^{d}
			\label{eq:radial_comp_of_spacetime}
		\end{equation} 
		of \emph{spacetime} (which forces $V,\bfA,W$ to die away as $t\to\infty$), then the estimate can be improved to
		\begin{equation}
			\lVert u - v \rVert_{(1+r^2+t^2)^{3/4+\varepsilon} L^p( \bbR_t\times \bbR^d_{x}) } = O \Big( \frac{1}{c^2} \Big).
			\label{eq:misc_a07}
		\end{equation}
		In fact, the way we will prove the theorem above is by modifying the coefficients of the PDE in $\{|t|\geq 2 T\}$ so that they become well-behaved functions on the radial compactification of spacetime and then applying spacetime phase-space/microlocal methods; \cref{eq:misc_a06} will then be deduced from \cref{eq:misc_a07}.
		\item (\textit{Form of time-dependence}.) Above, we required \cref{eq:extendability} instead of the simpler 
		\begin{equation}
			W,V,A_1,\dots,A_d\in C^\infty(\bbR_{t} ; \langle r \rangle^{-1}C^\infty(\overline{\bbR^{d}_x})).
			\label{eq:extendability_2}
		\end{equation}
		We have no reason to believe that the theorem would fail to hold given only \cref{eq:extendability_2}, but, at the moment, our spacetime microlocal tools do not allow us to prove this.
		However, given any  bounded, open, nonempty $\Omega\subset \bbR^d_x$,
		any function of form $C^\infty(\bbR_t; C^\infty_{\mathrm{c}}(\Omega))$ satisfies \cref{eq:extendability}, so we allow arbitrary time-dependence in compact regions of space. For most physical applications of the non-relativistic limit, this suffices. 
		
	An example of the sort of time-dependent coefficient that we exclude would be $\chi(t) f(r)$ for $\chi\in C_{\mathrm{c}}^\infty(\bbR)$ and $f\notin \calS(\bbR^d)$, where the time-dependence appears in the $r\to\infty$ expansion. However, $\chi(t/\langle r \rangle ) f(r)$ would be fine.

		The geometric interpretation of 
		 \cref{eq:extendability} is that 
		 \begin{equation} 
		 	C^\infty(\bbR_{t/\langle r\rangle} ; \langle r \rangle^{-1}C^\infty(\overline{\bbR^{d}_x})) = \langle r \rangle^{-1} C^\infty(\overline{\bbR^{1+d}} \backslash (\mathrm{S}\sqcup \mathrm{N}))
		 \end{equation} 
		 consists of smooth, vanishing-at-the-boundary functions on the radial compactification of spacetime without the north/south poles $\mathrm{N},\mathrm{S} \in \infty \bbS^{d}$, where the time-axis $\{r=0\}$ hits the boundary of the compactification. Thus, the difference between what we are requiring in this section and what we stated above suffices to get uniformity in $T$ is well-behavedness at $\mathrm{N},\mathrm{S}$. 
	\end{itemize}
	The microlocal estimates in this paper are sharp (up to $\varepsilon$ losses) vis-\`a-vis the $L^2$-based Sobolev spaces they are stated in terms of. Insofar as those are considered our main results, our main results are essentially sharp. Moreover, as far as the microlocal analysis is concerned, \cref{eq:extendability} is more natural a requirement on the time-dependence of our coefficients than \cref{eq:extendability_2}, though we would like to be able to handle the latter sort in the future.
\end{remark}

\begin{figure}[t]
	\includegraphics[scale=.4]{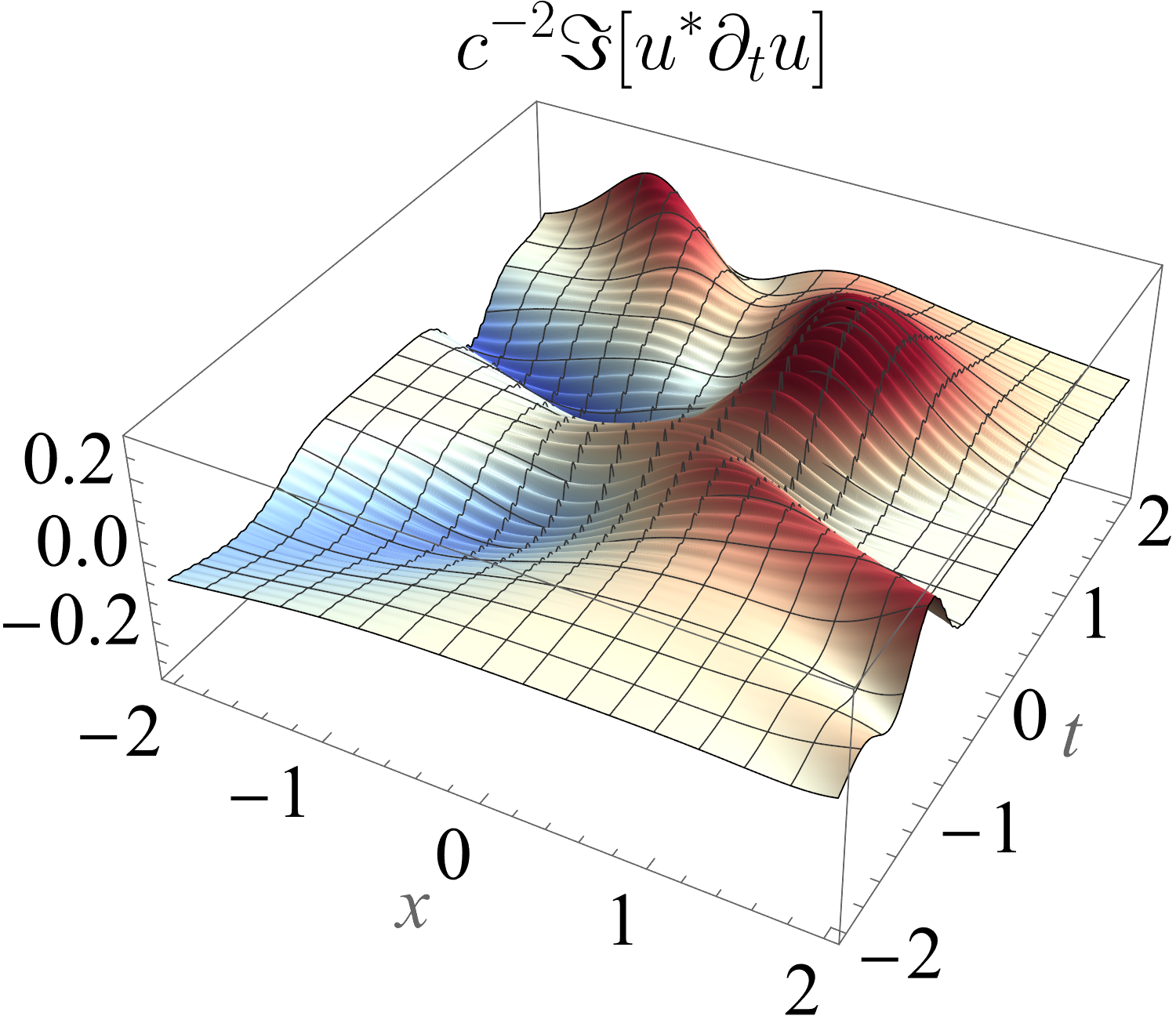}\quad 
	\includegraphics[scale=.4]{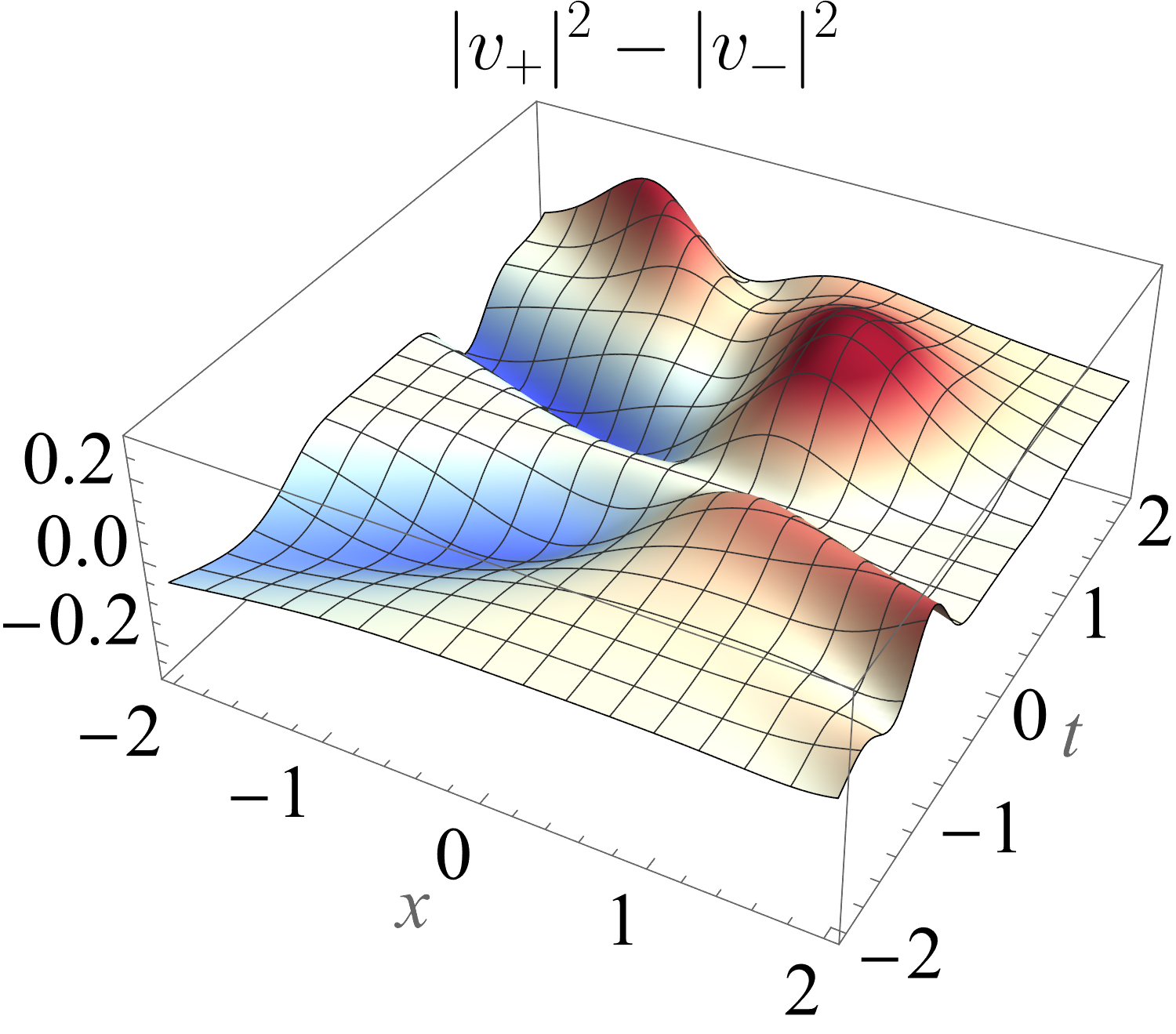}\quad
	\includegraphics[scale=.4]{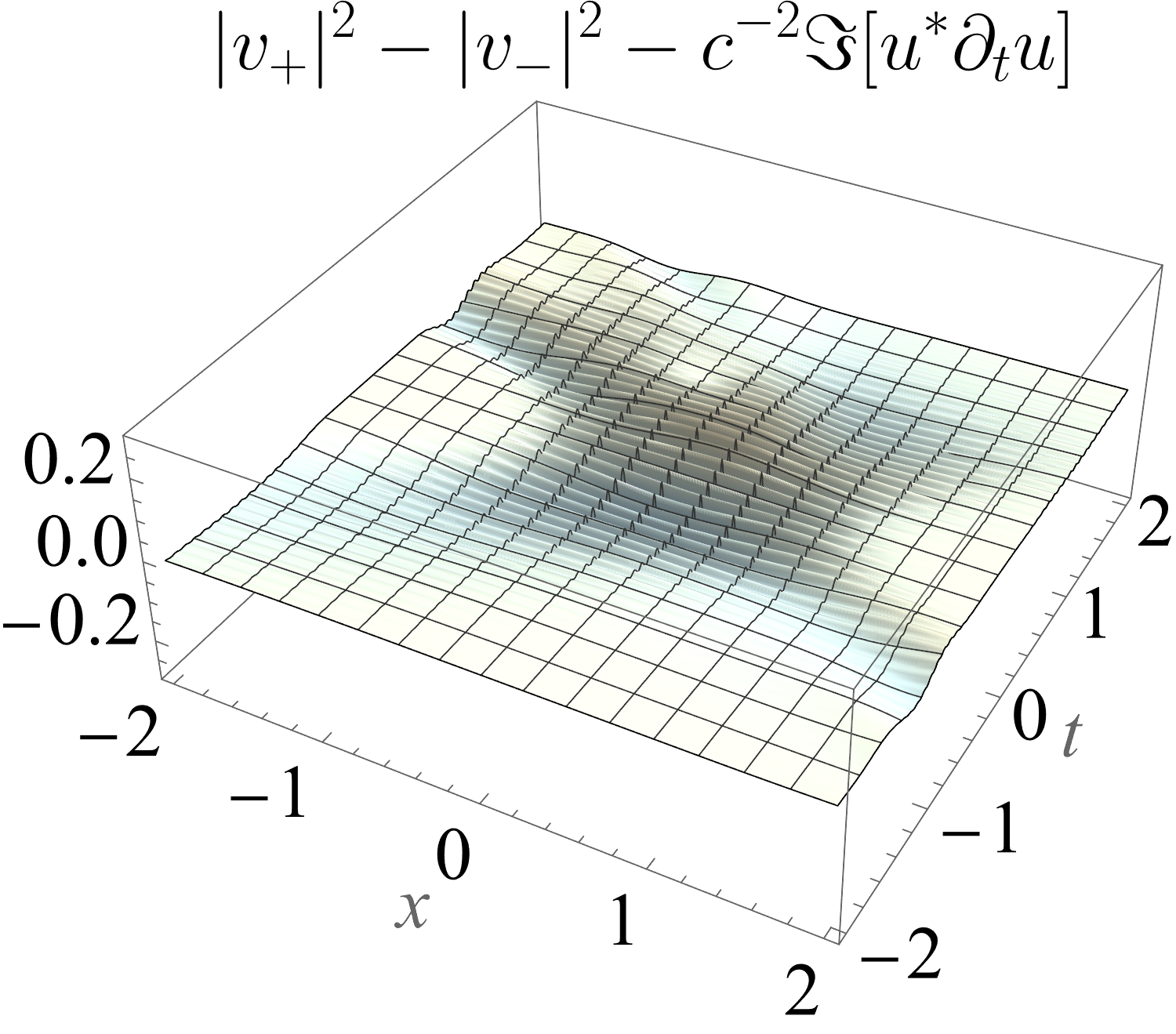}
	\caption{Continuation of \Cref{ex:FEM}, showing (\textit{left}) a plot of the charge density $\rho= c^{-2}\Im(u^*\partial_t u)$ of $u$, for $c=6$, and (\textit{middle}) the expected non-relativistic limit, $|v_+|^2-|v_-|^2$. This is just the difference of the two plots in \Cref{fig:FEM2}. The difference (\textit{right}) between $\rho$ and $|v_+|^2-|v_-|^2$ is highly oscillatory, but it is $O(1/c^2)$ as $c\to\infty$.} \label{fig:FEM3}
\end{figure}
\subsection{Literature review}
\label{subsec:literature_only}

Of course, we are not the first to rigorously study the non-relativistic limit of the linear Klein--Gordon equation with variable coefficients. Starting in the 1970s, mathematical physicists turned their attention to the problem, and a number of results followed. 
We are aware of the following works:
\begin{itemize}
	\item 
	Veseli{\'c} \cite{veselic1970spectral, veselic1971perturbation, veselic1983nonrelativistic} considered the electrostatic case, in which $\bfA,W=0$ and $V$ is independent of time. Under these severe restrictions, he showed that the frequencies and profiles of standing wave solutions of the Klein--Gordon/Dirac equations converge to those of the associated Schr\"odinger/Pauli equations, respectively. 
	Analyticity of the perturbation near $c=\infty$ was shown, first in $c^{-1}$ \cite{veselic1971perturbation}, and then in $c^{-2}$ \cite{veselic1983nonrelativistic}.  
	\item Another line of work, \cite{Kis}\cite{Smoller}\cite{Zlamal}\cite{OMalley}\cite{Friedman}\cite{SchoeneI, SchoeneII} treats the Cauchy problem for the Klein--Gordon and Dirac equations with time-independent coefficients (i.e.\ the \emph{static} case) using the framework of singular perturbation theory. 
	We highlight the work \cite{SchoeneII} of Schoene in particular, which contains the proof of the static case of \Cref{thm:simplest} (except the result is stated as an $L^2$-based bound rather than an $L^\infty$-bound). 
	\item  In \cite{CirincioneThesis,CC}, Cirincione and Chernoff apply the pseudoresolvent methods of Veseli{\'c} to the Cauchy problem via a Trotter--Kato theorem. In the process, they develop an abstract framework that is sufficiently general to allow the spatial slices of spacetime to be any complete Riemannian manifold. 
\end{itemize}

For other works regarding the non-relativistic limit of the (linear) Dirac equation that do not study the Klein--Gordon equation, see Titchmarsh \cite{Tit1, Tit2}, Hunziker \cite{Hunziker}, Thaler et.\ al.\ \cite{ThallerI, ThallerII, ThallerIII}, Grigore--Nenciu--Purice \cite{Grigore}, Kutzelnigg \cite{Kutz}, White \cite{White}, Ito--Yamada \cite{Yamada, ItoYamada, ItoYamada2}, and Bechouche--Mauser--Poupaud \cite{BechoucheDirac}. This list is likely incomplete.
We mention also the work of Ichinose \cite{Ichinose1, Ichinose2}, which concerns the ``imaginary time'' Klein--Gordon equation in the presence of an electromagnetic field and uses rigorous path-integral (Feynman-Kac) techniques.

In all of the works cited above, the coefficients of the PDE are not allowed to be time-dependent, with the sole exception of \cite{BechoucheDirac}, and even this only deals with the Dirac equation, not the Klein--Gordon equation. 
Even Cirincione--Chernoff, in \cite{CirincioneThesis,CC}, only consider ultra-static spacetimes. (Recall that an \emph{ultra-static} spacetime is one isometric to a product $\bbR_t\times X$ for some Riemannian manifold $X$ \cite{Fulling, FullingII}\cite{Verch}. Thus, the metric is independent of time and is lacking cross terms such as $\dd t \dd x$, and the coefficient of $\dd t^2$ is constant. Besides the flat case, astrophysical spacetimes are not ultra-static.)
It therefore appears to us that the existing literature does not adequately treat the non-relativistic limit for Klein--Gordon equations with time-dependent coefficients. This is especially true for non- ultra-static spacetimes. Certainly we are aware of no systematic attack. 

One reason is that the case of time-dependent coefficients lacks a canonical separation of the solution of the Cauchy problem into two components, each of which limits to a Schr\"odinger solution as $c\to\infty$. See \S\ref{subsec:splitting_diff} for a further discussion of this point. 
In \cite{BechoucheDirac}, Bechouche--Mauser--Poupaud settle on using certain pseudodifferential operators to split their spinor-valued solution into electron/positron components; see \cite[Eq.\ 1.21]{BechoucheDirac}. This seems natural. However, Bechouche--et al.\ use \emph{spatial} pseudodifferential operators, whereas we use \emph{spacetime} pseudodifferential operators. Spacetime pseudodifferential operators seem necessary when working with scalars, because one has more plane waves (two independent solutions to the free Klein--Gordon equation for each spatial frequency $\xi\in \mathbb{R}^d$) than components of the solution $u$ (namely one, since $u$ is a scalar). 
In contrast, different plane wave solutions of the free Dirac equation with the same spatial frequency can be distinguished by looking at the ratio of different components --- Dirac spinors in $(3+1)$D have four components, and, for each $\xi \in \bbR^3$, we have four independent plane waves solutions to the free Dirac equation with the same spatial frequency $\xi$.

Although we are postponing the discussion of non-flat spacetimes until \S\ref{sec:true_intro}, let us mention some motivation here. 
A natural heuristic
for the non-relativistic limit on astrophysical spacetimes, one that we are surely not the first to consider, is that
\begin{itemize}
	\item[($\star$)]  the mass $M\geq 0$ of the spacetime (i.e.\ the mass of whatever objects are ``generating'' the gravitational field) shows up in the limiting Schr\"odinger operators in the form of a Newtonian potential 
	$2M /r$. 
\end{itemize}
This is one manner in which the effects of Newtonian gravity are recovered in the non-relativistic limit.  If the list of citations above is complete, the heuristic ($\star$) does not appear to have been put on rigorous footing yet, despite its long history in the physics literature.\footnote{This issue is separate from the weak-gravity limit of the Einstein field equations.}
Later, we will state \Cref{thm:Cauchy} in sufficient generality to include astrophysical spacetimes --- see \Cref{ex:Kerr}. 

In summary, the mathematics literature is lacking a robust framework for handling the non-relativistic limit of the Klein--Gordon equation outside of the ultra-static case, in enough generality to deliver what physicists require of it. The goal of this manuscript is to provide such a framework. It is described at length in \S\ref{sec:true_intro}. We offer \Cref{thm:simplest} (and \Cref{thm:Cauchy}) as evidence of its sufficiency, as far as establishing the basic facts of the non-relativistic limit is concerned. However, some results are still out of reach --- see \S\ref{subsec:todo}.

For the non-relativistic limit of nonlinear equations see e.g.\ Nakanishi, \cite{nakanishi2002nonrelativistic}, Machihara, \cite{machihara2002nonrelativistic, machihara2002nonrelativistic, machihara2003NLDirac}, Masmoudi et.\ al.\ \cite{masmoudi2002nonlinear, masmoudi2003MKG, masmoudi2005KGZ, masmoudi2008energy-Zakhorav, masmoudi2010KGZ-singular}, Bechouche et.\ al.\ \cite{BechoucheMaxwell}, Pasquali \cite{Pasquali}, Bao--Lu \cite{bao-lu-zhang2023convergence}, and Candy--Herr \cite{candy2023massless}.
In \cite{BechoucheMaxwell}, Bechouche, Mauser, and Selberg study the initial value problem for the Klein--Gordon--Maxwell system, and they establish convergence in the non-relativistic limit. In their work, no background fields are present, so the only force felt by the Klein--Gordon equation is that due to the dynamical electromagnetic fields. However, since the electromagnetic fields are dynamic, they are likely time-dependent. It is therefore natural to expect that the tools in \cite{BechoucheMaxwell} (or perhaps \cite{BechoucheDirac}) would be sufficient to prove \Cref{thm:simplest}.

\begin{figure}[t!]
	\includegraphics[scale=.4]{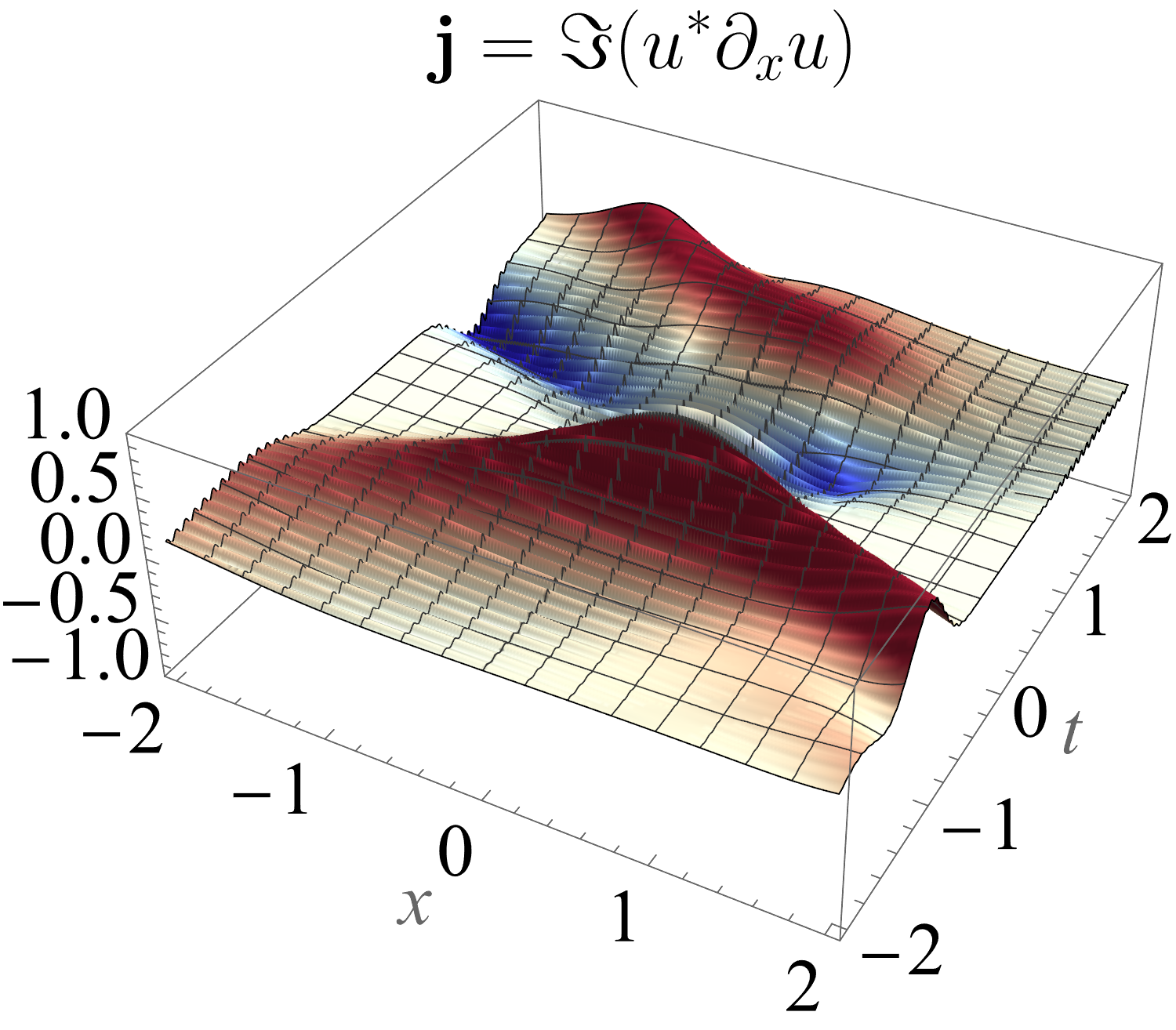}
	\includegraphics[scale=.4]{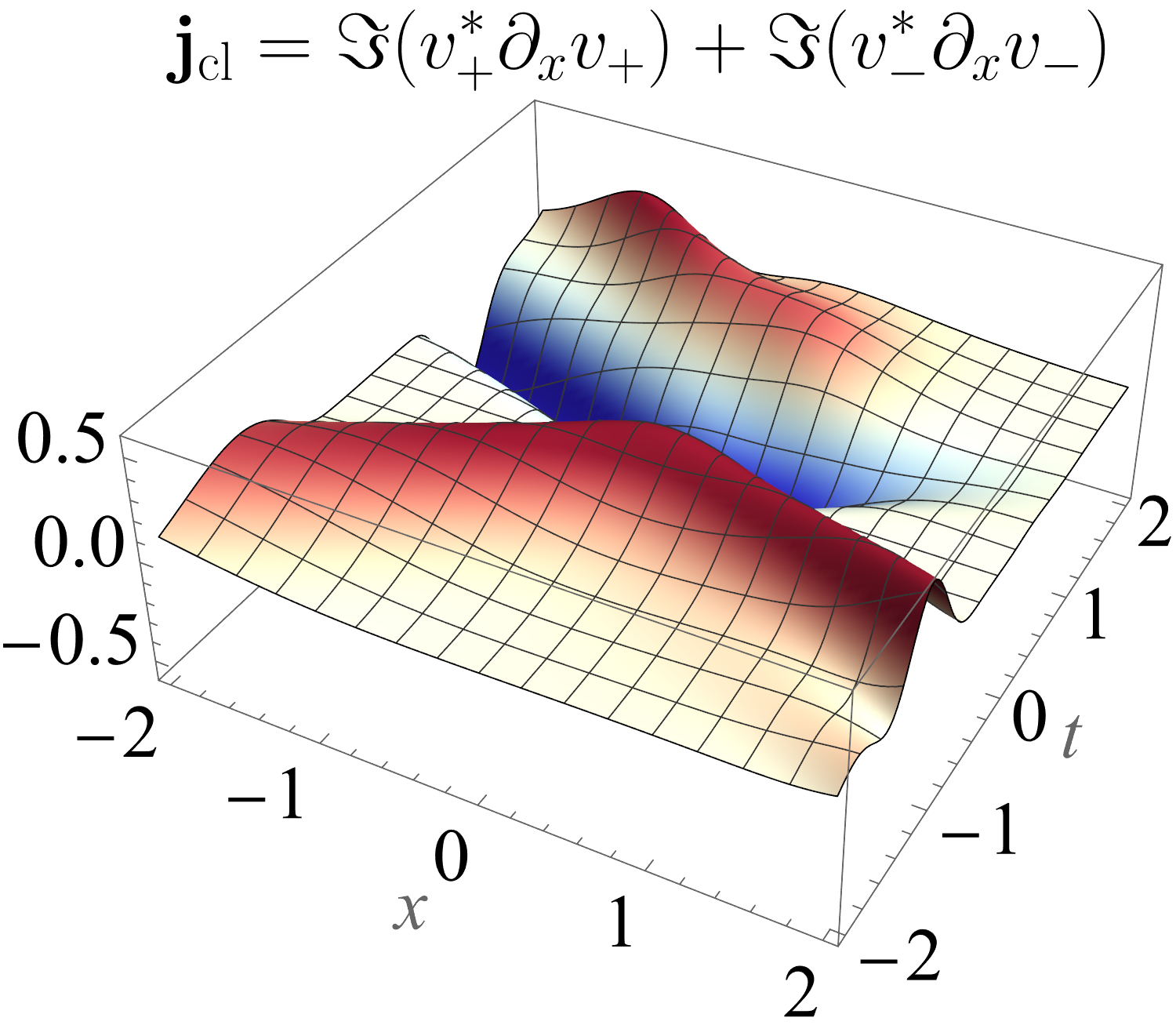}
	\includegraphics[scale=.4]{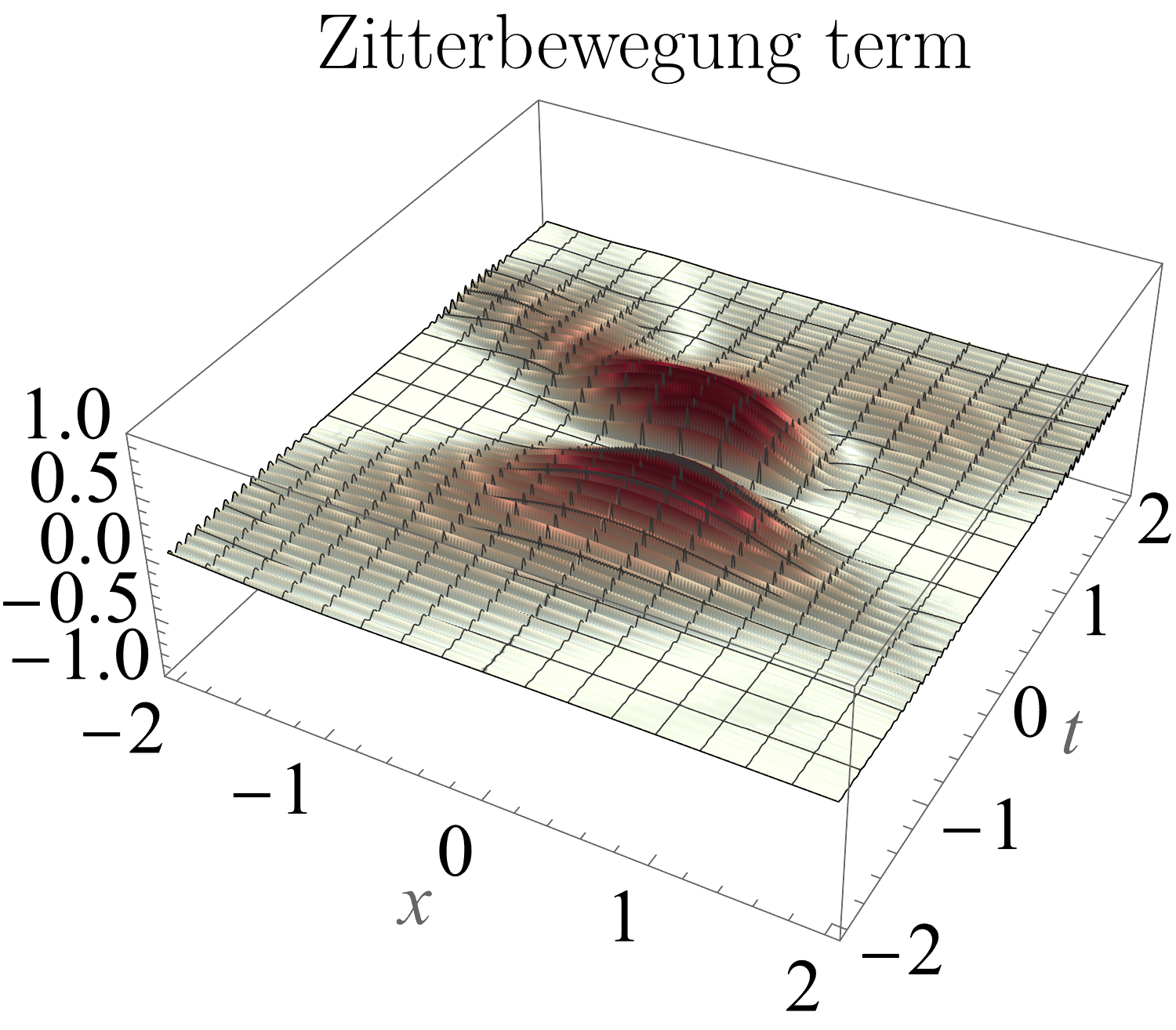}
	\includegraphics[scale=.4]{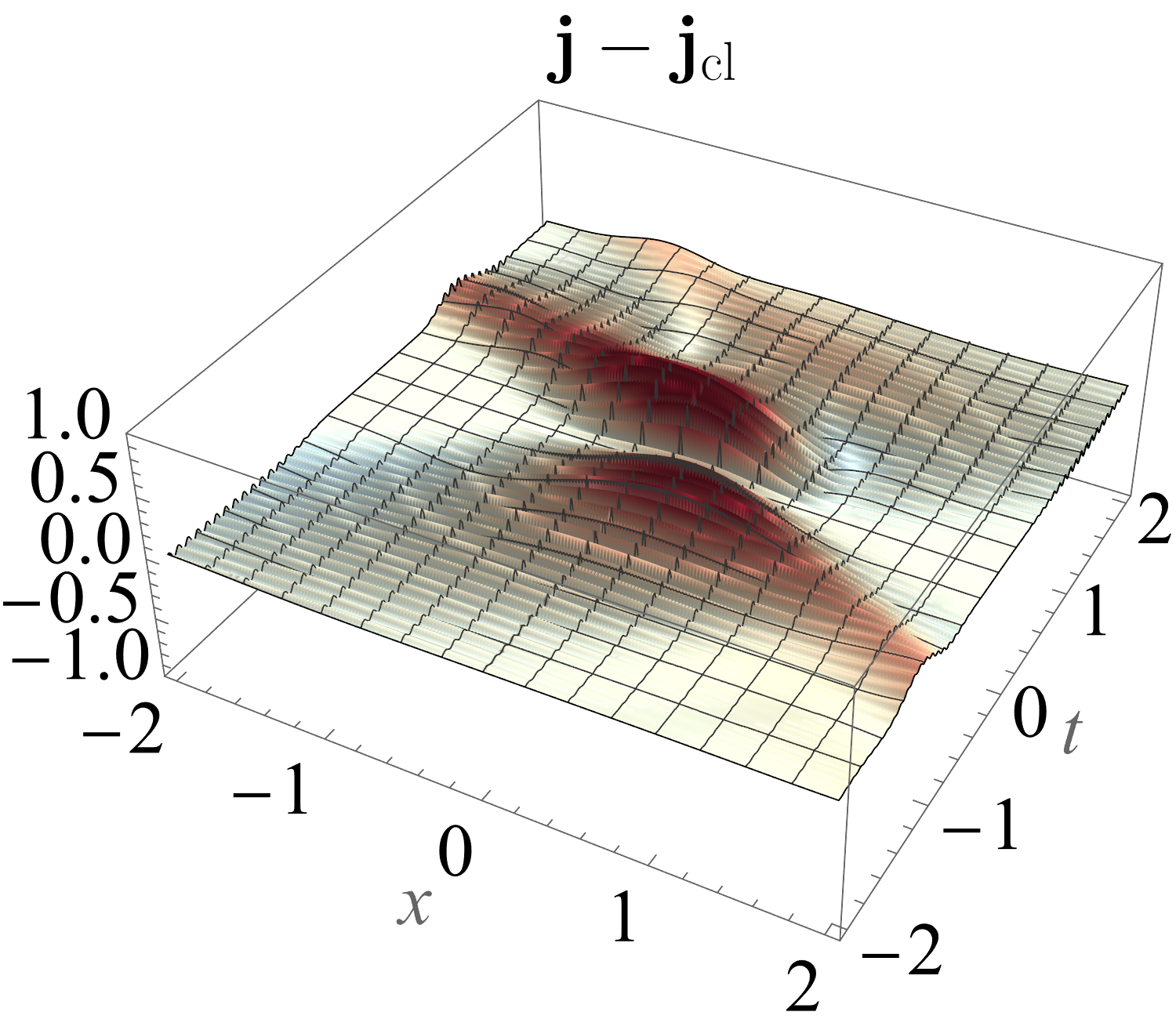}
	\includegraphics[scale=.4]{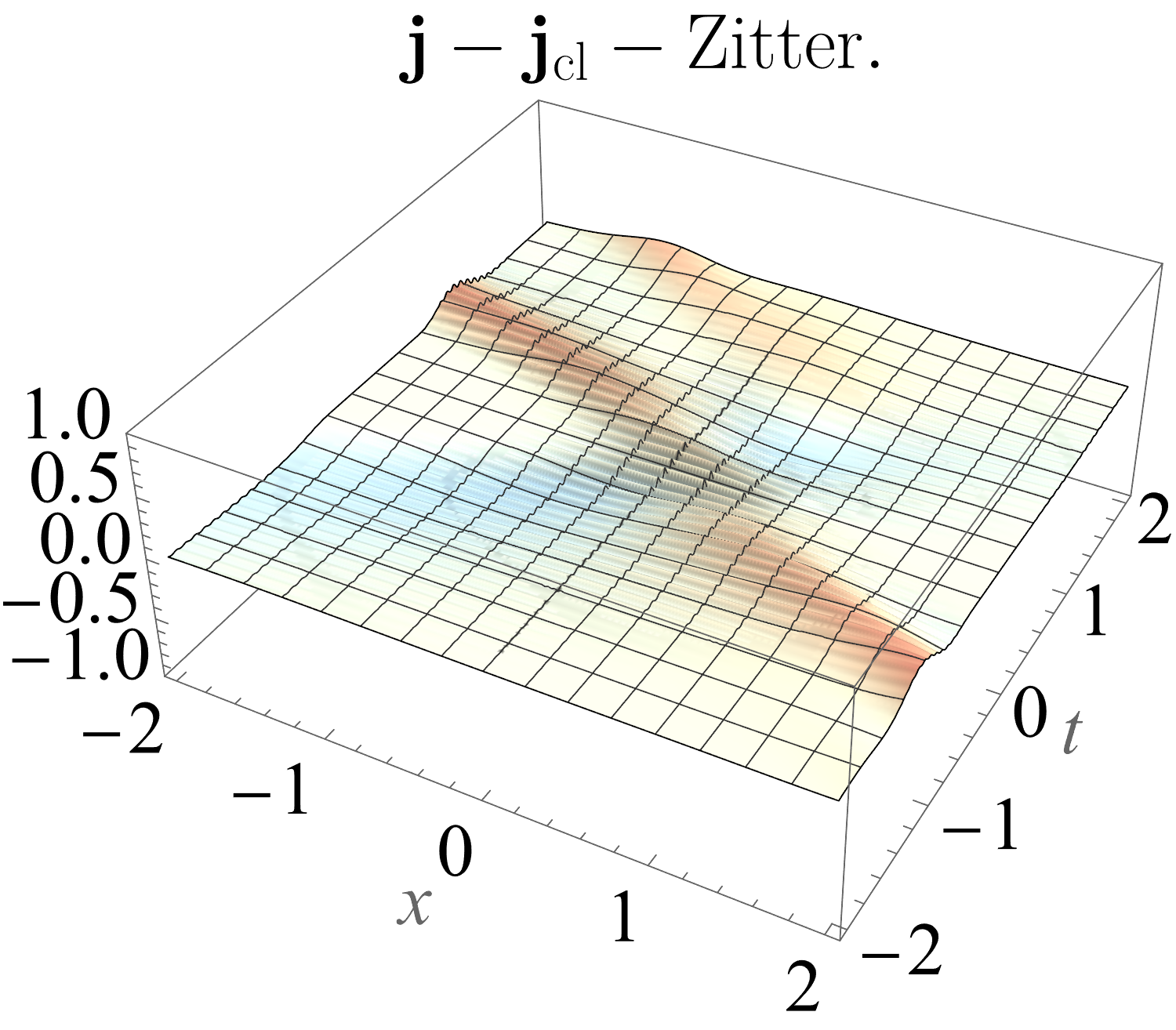}
	\caption{Continuation of \Cref{ex:FEM}, now showing the non-relativistic limit of the current density (still at $c=6$). In this example, $d=1$, so we can view $\bfj$ as a scalar. It is not true that $\bfj$ agrees with the non-relativistic prediction $\mathbf{j}_{\mathrm{cl}}=\Im(v_+^*\partial_x v_+) + \Im(v_-^*\partial_x v_-)$ to leading order. Instead, the highly oscillatory, $\Omega(1)$ Zitterbewegung term needs to be subtracted from $\bfj$ before it converges in the $c\to\infty$ limit to $\mathbf{j}_{\mathrm{cl}}$.} \label{fig:FEM4}
\end{figure}
\begin{remark}[Comparison with the physics literature]
	\label{rem:comparison}
	In the physics literature, the Klein--Gordon equation for a charged particle in the presence of an electric potential $V$ and magnetic potential $\bfA$ and coupled to an external scalar field $W$ reads $Pu=0$ for 
	\begin{equation}
		P = - \frac{1}{c^2} \Big( \hbar \frac{\partial}{\partial t} - ieV \Big)^2 - (i \hbar \nabla - e \bfA)^2  - m^2 c^2 +W  
		\label{eq:P_intro_physics}
	\end{equation}
	\cite[Eq. 4.1]{CC}, where $\hbar \approx 6.6\times 10^{-34} \mathrm{J}\cdot \mathrm{s}$ is Planck's constant, $m$ is the mass of the particle whose wavefunction is under consideration, and $e$ denotes its charge.
	In the non-relativistic limit, the parameters $\hbar,m,e$ are all to be fixed. By rescaling $P$ and the spacetime coordinates, we may set $\hbar,m=1$, without loss of generality.  
	Physicists sometimes use the phrase ``non-relativistic limit'' to refer to the $m\to\infty$ limit. This is not equivalent to the $c\to\infty$ limit, unless one also scales $t$. (The $m\to\infty$ limit is more like a traditional semiclassical limit, in which the semiclassical parameter is $h=1/m$ instead of Planck's constant.) We will only use the phrase ``non-relativistic'' when referring to $c\to\infty$. 
	
	In addition, by rescaling $V,\bfA$, we may assume that $e= \mp 1$ without loss of generality. Then, \cref{eq:P_intro_physics} becomes \cref{eq:P_intro}.
	A confusing point is that most references in the physics literature, see e.g.\ \cite[Eq. 7]{Gordon}\cite[Eq.\ 42.10]{Schiff}\cite[Eq. 2.15]{Villars}\cite[Eq. 1.1.2]{WeinbergQTF1}\cite[\S1.9]{Greiner}, have $e\bfA/c$ in place of $e\bfA$ in \cref{eq:P_intro_physics}. This is because their definition of the vector potential differs the usual definition in magnetostatics by a factor of $c$; what they call ``$\bfA$'' we would call ``$c \bfA$.'' The factors in \cref{eq:P_intro_physics} are correct (as can be checked by dimensional analysis) for the usual definition of the vector potential $\bfA$, whose SI units are Tesla-meters or kilogram-meters-per-coulombs-per-second, $\mathrm{kg}\,\mathrm{m}\, \mathrm{C}^{-1} \mathrm{s}^{-1}$. 
	In contrast, the units of $\bfA_{\mathrm{cgs}}=c\bfA$ are Tesla-seconds, which is just volts. Then, the four-potential $A=(V,\bfA_{\mathrm{cgs}})$ has units of volts, which physicists find appealing. 
	The distinction between $\bfA_{\mathrm{cgs}}$ and $\bfA$, while unimportant if $c$ is fixed, matters when taking the $c\to\infty$ limit. The mathematical question is how the $\partial_x$ terms in \cref{eq:P_intro} are to scale with $c$; we have taken them to be $O(1)$, but fixing $\bfA_{\mathrm{cgs}}$ means taking $\bfA=O(1/c)$, in which case the non-relativistic limit is also a \emph{weak magnetic field limit}.
	Thus, it is unclear whether works in the mathematical literature which fix $\bfA_{\mathrm{cgs}}$ (see \cite[Eq.\ 3]{SchoeneII}\cite[Eq.\ 4]{BechoucheMaxwell}) are studying the most natural limit. It seems to us that fixing the ordinary vector potential $\bfA$ (as Cirincione--Chernoff do in \cite{CC}) is most natural when studying the Klein--Gordon equation in the presence of a macroscopic magnetic field.
	It may make sense to fix $\bfA_{\mathrm{cgs}}$ when the magnetic field is generated by a few moving electric charges. There are no magnetic monopoles in nature, so all microscopic magnetic fields are of this sort. That moving electric charges generate magnetic fields can be interpreted as a relativistic effect, since only an electric field is present in the rest frame of the particle. It is then natural to expect that magnetic fields be suppressed  in the non-relativistic limit. For example, the Biot--Savart law \cite[Chp.\ 5]{OldHickory}, which describes the magnetic field generated by a constant current moving through a wire, gives a magnetic potential whose dependence on $c$ is through the magnetic constant $\mu_0$. The magnetic constant is related to the vacuum permittivity $\varepsilon_0$ by $\varepsilon_0\mu_0 = 1/c^2$. So, if we fix $\varepsilon_0$, then $\mu_0 = O(1/c^2)$. Then, $\bfA = O(1/c^2)$ is weak in the $c\to\infty$ limit, but it is even weaker than the $O(1/c)$ that results from fixing $\bfA_{\mathrm{cgs}}$.
	For these reasons, we have, in this section, chosen to fix $\bfA$ instead of $\bfA_{\mathrm{cgs}}$. However, our general theorem, \Cref{thm:Cauchy}, is formulated broadly enough to include both cases. 
\end{remark}

\subsection{Natural units}
\label{subsec:natural}
A key fact is that the non-relativistic limit depends on the scale on which the initial data, forcing, etc.\ varies; the Klein--Gordon equation is only well-approximated by the Schr\"odinger equation when the various pieces of data involved vary
on scales which are larger than what physicists call \emph{natural}\footnote{Natural units are ``natural'' in that they are adapted to the scale on which solutions of the PDE oscillate. They do not reflect the scale on which the coefficients of the PDE vary.} \emph{units}, 
\begin{equation}
	t_\natural = mc^2 t/\hbar, \quad x_\natural = mc x / \hbar,
	\label{eq:nat}
\end{equation}
where $\hbar$ is Planck's constant and $m>0$ is the particle mass. (Elsewhere in the paper, we are setting $\hbar,m=1$, which, as discussed above, we can do without loss of generality. Here, it is clarifying to have explicit factors. Note that, if we assign $m,c,\hbar,t,x$ SI units in the usual way, then $t_\natural,x_\natural$ are \textit{dimensionless}.)

The easiest way to see this is to consider the free Klein--Gordon operator $P_0$ (which we write with factors of $\hbar,m$ in the usual places, as in \cref{eq:P_intro_physics}). 
In terms of the coordinates $t_\natural,x_\natural$, 
\begin{equation}
	P_0(c) = -\frac{\hbar^2}{c^2} \frac{\partial^2}{\partial t^2} - \hbar^2 \triangle -  m^2 c^2  = -m^2c^2 \Big(  \frac{\partial^2}{\partial t_\natural^2} + \triangle_\natural + 1 \Big), 
	\label{eq:P_nat}
\end{equation}
where $\triangle_\natural$ is the Laplacian in $x_\natural$. So, up to a factor of $c^2$ (which is irrelevant when solving the homogeneous problem $P_0u=0$, and only relevant to the scaling of the inhomogeneity otherwise), the operator $P_0(c)$ is, in natural units, just $P_0(1)$. At first glance, this should be confusing: how can the Klein--Gordon equation be well-approximated in the non-relativistic limit by the Schr\"odinger equation if the $c>0$ case of the free Klein--Gordon equation can be reduced to the $c=1$ case? This would seem to suggest that the $c\to\infty$ limit should be completely uninteresting.

To make this more precise, suppose that, instead of considering the Cauchy problem, \cref{eq:IVP}, in which the initial data is \emph{fixed} as $c\to\infty$, we consider the closely related problem 
\begin{equation}
\begin{cases}
Pu=0 \\
u(0,x)\, = \varphi(x_\natural) \\ 
\partial_{t} u|_{t=0} = c^2\psi(x_\natural),
\end{cases}
\tag{IVP$\natural$}
\label{eq:IVP_2}
\end{equation}
in which the initial data is dilated naturally. (Here, $P$ is as in \cref{eq:P_intro}.) 
Then, $u=u(-;c)$ is just given by $u(t,x;c) = u(t_\natural,x_\natural;1)$. That is, the solution for each $c>0$ is given by dilating (in the $t$-direction more than the $x$-direction) the solution for $c=1$. So, $u(t,x;c)$ should not be well-approximated in the $c\to\infty$ limit by any solution of the free Schr\"odinger equation. This appears in tension with the usual heuristic.

The same paradox applies, with slight modifications, even if $V,\bfA,W\neq 0$. Indeed, we can still write 
\begin{multline}
P(c) = -\frac{1}{c^2} \Big(\hbar^2\frac{\partial}{\partial t}+i V\Big)^2 - (i\hbar\nabla+\bfA)^2 - m^2c^2 +W 
\\ = -m^2c^2 \Big[  \Big(\frac{\partial}{\partial t_\natural} + \frac{iV}{mc^2}\Big)^2 + \Big(i \nabla_\natural + \frac{\bfA}{mc} \Big)^2+1 +  \frac{W}{m^2c^2}\Big].
\end{multline}
So, the solution of \cref{eq:IVP_2} is given by 
\begin{equation} 
	u(t,x;c,V,\bfA,W) = u(t_\natural,x_\natural;1,m^{-1}c^{-2} V,m^{-1} c^{-1} \bfA,m^{-2}c^{-2} W),
\end{equation} 
where $u(-;c,V,\bfA,W)$ is the solution of \cref{eq:IVP}. If $c\gg 1$, then $c^{-2} V,c^{-1} \bfA,c^{-2} W \approx 0$, so we should have\footnote{This heuristic fails when considering the original initial-value problem, \cref{eq:IVP}. It has to, because otherwise \Cref{thm:simplest} would be false and the non-relativistic limit would not involve $V,\bfA,W$ at all. In the present setting, that of solving \cref{eq:IVP_2} instead of \cref{eq:IVP}, the heuristic \emph{can} be rigorously justified, using a simple energy estimate. This is what \S\ref{sec:naturalCauchy} does.}
\begin{equation}
 u(t_\natural,x_\natural;1,m^{-1}c^{-2} V,m^{-1} c^{-1} \bfA,m^{-2}c^{-2} W) \approx u_0(t_\natural,x_\natural;1),
\end{equation} 
where $u_0(-;c)= u(-;c,0,0,0)$ is the solution in the free  case. So, rather than being approximated by a solution of the Schr\"odinger equation, $u$ is approximated by a solution of the \emph{free} Klein--Gordon equation, written in natural units. Apparently, it is the latter PDE relevant to studying the $c\to\infty$ limit of the problem \cref{eq:IVP_2}, not the Schr\"odinger equation. So, we have two different model problems governing the two different initial-value problems \cref{eq:IVP}, \cref{eq:IVP_2} in the non-relativistic limit. If our heuristic is that the Schr\"odinger equation governs the non-relativistic limit, then our heuristic requires modification.

The appropriate modification is suggested by the observation that the initial data in \cref{eq:IVP_2}, 
$\varphi(x_\natural)= \varphi(mcx/\hbar)$ and $\psi(x_\natural) = \psi(mcx/\hbar)$, is varying on the scale of natural units, whereas, in the original Cauchy problem, the initial data is varying on the original scale. The former scale is much smaller than the latter if $c\gg 1$. Therefore, the following dichotomy appears to hold:
\begin{itemize}
	\item on the laboratory scale, that of the original variables $t,x$, the Klein--Gordon equation is well-approximated by the Schr\"odinger equation, with an electric/magnetic potential,
	\item on the small (i.e.\ natural) scale, namely that of natural units, $t_\natural,x_\natural$, the Klein--Gordon equation is well-approximated by the free Klein--Gordon equation, even if a potential is present.
\end{itemize}
If our initial data is approximately constant on neighborhoods of size $\hbar/(mc)$, then we should expect the non-relativistic limit to be governed by the Schr\"odinger equation. Otherwise, if the data is concentrated on that scale, then we should expect the free Klein--Gordon equation. 
One of the purposes of the microlocal tools we introduce later is to separate these two scales, so that we may analyze the non-relativistic limit in each without the troubles caused by the other.

In order to keep the discussion as elementary as possible, we have chosen to present this introduction using Schwartz initial data, but it is often of interest to study the PDE for data having only  a \emph{finite} amount of regularity. We should then expect the analyses on \emph{both} the laboratory scale and the natural scale to contribute to the analysis of the non-relativistic limit. 
For example, initial data having a jump discontinuity can be thought of as varying over arbitrarily small scales, unlike the data in \cref{eq:IVP}, but not concentrated on some small neighborhood like the data in \cref{eq:IVP_2}. In this case, it should be possible to decompose the solution to the Cauchy problem into a regular component obeying the usual Schr\"odinger-approximated non-relativistic limit (with a full asymptotic expansion in powers of $c$) and another, rapidly oscillating part, which is approximated by a solution of the free Klein--Gordon equation (but suppressed by some power of $c$). Our investigations into the Cauchy problem do not go into this depth, but our microlocal results in \cite{NRL_I} do investigate the non-relativistic limit for irregular solutions of the PDE. They describe how energy is propagated according to the free Klein--Gordon equation on the natural scale and according to the Schr\"odinger equation on the laboratory scale. For irregular solutions of the PDE, energy is present on both scales, and so both modes of propagation are important to understand.

See \S\ref{subsec:main} for our discussions on the inhomogeneous problem with irregular forcing.

Are the laboratory and natural scales the only scales relevant? The answer is yes, according to the analysis that comprises the rest of our manuscript.
At no point is it necessary to consider a third scale. No intermediate scale poses a problem. Neither does any scale smaller than natural units,  nor any scale larger than the original. (This paragraph is made precise in \S\ref{subsec:description}.)
In conclusion, there are only two PDE governing the non-relativistic limit: the Schr\"odinger equation, and the free Klein--Gordon equation, and we know the scales on which each is relevant.

\begin{example}
	If $m$ is the mass of an electron, then 
	\begin{equation}
		x_\natural \approx (2.59 \times 10^{12} \,\mathrm{m}^{-1}) x
		\label{eq:misc_018}
	\end{equation}
	in SI units. For applications to atomic physics (e.g.\ the hydrogen atom), the wavefunctions of interest are varying on atomic scales. A measure for the size of a hydrogen atom is the Bohr radius $a_0=4\pi \varepsilon_0 \hbar^2 / (e^2 m) \approx 5.29\times 10^{-11} \,\mathrm{m}$, where $\varepsilon_0$ is the vacuum permitivity constant and $e$ is the charge of an electron. Substituting this into \cref{eq:misc_018}, we get that, in natural units, the hydrogen atom is 
	\begin{equation}
		x_\natural \approx 137.036
	\end{equation}
	units ``long.'' The fact that $x_\natural \gg 1$ means that, e.g.\ if one were to solve the Klein--Gordon equation (in the positive energy space) with initial data given by the ground state of the non-relativistic hydrogen atom, then the solution can be expected to be close to that of the Schr\"odinger equation with the same initial data. 
	
	The quantity $\alpha=1/x_\natural\approx 1/137$ is better known as the ``fine structure constant,'' because it governs the size of the corrections to atomic spectra due to relativistic effects. The fact that $\alpha\approx 100$ is why calculations of atomic spectra using the non-relativistic Schr\"odinger calculation are accurate to about one part in one-hundred. 
\end{example}

\begin{remark}
	The discussion in this section shows that the constants in the estimate \cref{eq:misc_a06} depend crucially on the initial data. The smaller the length scale on which our initial data vary, the worse the constant in the estimate is. In other words, our results can be phrased in terms of the convergence of certain solution operators in a suitable strong operator topology, not a uniform operator topology.
\end{remark}

Our focus in this paper is on the laboratory Cauchy problem, not the natural Cauchy problem, but we do study the latter in \S\ref{sec:naturalCauchy}; see \Cref{thm:naturalCauchy}.

\subsection{The ultra-static setting}
\label{subsec:time_independent}

In the constant-coefficient case, one can use the spacetime Fourier transform to analyze the non-relativistic limit --- see \S\ref{sec:free}.
This does not work in the variable-coefficient case. In this subsection, we discuss the case where $V=0$ and $\bfA,W$ are independent of time. Then, $P$ becomes 
\begin{equation}
P = -\frac{1}{c^2}\frac{\partial^2}{\partial t^2} - D^2 - c^2 + W 
= - \frac{1}{c^2} \frac{\partial^2}{\partial t^2} - (i\nabla + \bfA(x))^2 - c^2 + W(x),
\label{eq:P_ultrastatic}
\end{equation}
where $D= i \nabla + \bfA(x)$. 
We call this the \emph{ultra-static} case (``static'' because the coefficients are time-independent, and ``ultra'' because of the absence of an term linear in $\partial_t$). It is very classical; it is essentially a special case of \cite[\S4]{CC}, but it was surely considered folklore earlier.\footnote{In \cite{CC}, Cirincione--Chernoff have $W=0$, so \cref{eq:P_ultrastatic} is not quite a special case of their setup, but this is an inessential point.}

Our discussion here will not generalize to the setting with time-dependent coefficients. Since our later methods will include the ultra-static case as a special case, we do not include full proofs here, or even sketches.
Our purpose in discussing this special case is three-fold: 
\begin{enumerate}
	\item it provides intuition as to how the Schr\"odinger equation arises in the non-relativistic limit, 
	\item in \S\ref{subsec:time_dependent}, we will reference the half-Klein--Gordon operators (defined below) in explaining why the time-dependent case is more difficult, and
	\item it will serve as a first, brief, illustration of the key idea that the scale-dependence of the non-relativistic limit (discussed in \S\ref{subsec:natural}) can be dealt with using phase space methods, of which microlocal analysis is one sort.
\end{enumerate}
In this subsection, the use of microlocal/pseudodifferential tools will be kept at a minimum. However, one amenable feature of the ultra-static case is that the relevant microlocal tools are standard. So, it would not be difficult to make the discussion below rigorous.
This is unlike the case with time-dependent coefficients, for which the required pseudodifferential calculi did not exist before we developed them in \cite{NRL_I}. (See \S\ref{sec:calculus} for a summary of these calculi.) 

\subsubsection{Reduction to the half-Klein--Gordon equation}

The operator $P$ defined by \cref{eq:P_ultrastatic} is formally a difference of squares, leading to the factorization 
\begin{equation} 
P= -c^{-2} D_+D_- = -c^{-2} D_- D_+,
\label{eq:KG_factorization}
\end{equation} 
where $D_\pm$ are the commuting (pseudodifferential) operators  
\begin{equation}
D_\pm = \pm i\frac{\partial}{\partial t} +  c^2 \sqrt{1+\frac{D^2 }{c^2}  - \frac{W}{c^2} }.
\label{eq:Dpm_def}
\end{equation}
These are the two \emph{half}- Klein--Gordon operators.
(The square root can be made sense of using the functional calculus for unbounded self-adjoint operators, at least as long as $c$ is sufficiently large so as to ensure that the operator under the square root has non-negative spectrum.)

Let $\ker D_\pm, \ker P\subset C^\infty(\bbR_t; \mathcal{S}(\bbR^d_x))$ denote the kernels of $D_\pm, P$ in $C^\infty(\bbR_t; \mathcal{S}(\bbR^d_x))$.\footnote{One can talk more broadly about the kernels of $D_\pm,P$ in $\calS'(\bbR^{1,d})$, which we don't pursue here. This way, we can make use of more basic functional analytic tools.}
Evidently, \cref{eq:KG_factorization} implies that $\ker D_\pm \subseteq \ker P$.
In fact, $\ker P$ is a direct sum of $\ker D_\pm$: 
\begin{equation} 
\ker P = \ker D_+ \oplus \ker  D_-,
\label{eq:misc_027}
\end{equation}
as can be proven using the temporal Fourier transform.\footnote{For this, it may be necessary to take $c$ to be a little bit larger, so that $c^2 + D^2+W$ has strictly positive spectrum.}
So, every solution $u$ of $Pu=0$ in $C^\infty(\bbR_t; \mathcal{S}(\bbR^d))$ can be decomposed into a sum $u=u_-+u_+$ of elements $u_\pm$ of $\ker D_\pm$. 
So, if we want to understand the $c\to\infty$ limit, it suffices to analyze each of the half-Klein--Gordon operators $D_\pm$ individually.

The initial values of $u_\pm$ can be solved for in terms of the initial data of $u$. Indeed, letting
\begin{equation}
    \varphi_\pm(x)=u_\pm(0,x),\quad \varphi(x) = u(0,x), \quad \psi(x) = c^{-2} \partial_t u(0,x), 
\end{equation}
the relation 
\begin{equation}
\varphi_\pm = \frac{1}{2}\Big( \varphi -i \Big(1+\frac{D^2}{c^2}-\frac{W}{c^2} \Big)^{1/2} \psi \Big)
\label{eq:misc_026}
\end{equation}
holds. So, $\varphi_\pm$ is a smooth function of $1/c^2$ valued in Schwartz functions: $\varphi_\pm \in C^\infty([0,\infty)_{1/c^2};\calS(\bbR^d_x) )$.

\subsubsection{Non-relativistic limit for the half-Klein--Gordon equation}

By the Stone--von Neumann theorem, 
\begin{equation}
	\ker D_\pm  = \{ e^{\pm i ct  \sqrt{c^2 + D^2 - W} } f : f\in \mathcal{S}(\bbR^d_x) \} .
\end{equation}
If $c\gg 1$, then we should be able to approximate the exponent, when applied to amenable functions, by formally expanding the square-root to first order:
\begin{equation}
	c^2\sqrt{1+\frac{1}{c^2} ( D^2-W) } = c^2 + \frac{D^2}{2} - \frac{W}{2} + O\Big( \frac{\langle \triangle \rangle^2 }{c^2} \Big).
	\label{eq:misc_031}
\end{equation}
It is reasonable to ignore the error term in \cref{eq:misc_031}, since it is suppressed by a factor of $c^{-2}$.
So, we should have 
\begin{align}
\begin{split} 
	\ker D_\pm  &= \{ e^{\pm i ct  \sqrt{c^2+D^2-W} } f : f\in \mathcal{S}(\bbR^d_x) \} \\
    &\approx  \{ e^{\pm i t (c^2+(D^2-W)/2)} f : f\in \mathcal{S}(\bbR^d_x) \} = \ker \Big\{ \underbrace{\pm i \frac{\partial}{\partial t} + c^2 + \frac{D^2}{2} - \frac{W}{2}}_{N(D_\pm)} \Big\}.
    \end{split}
	\label{eq:misc_033}
\end{align}
The operator $N (D_\pm)$ on the right-hand side is exactly the operator appearing in the Schr\"odinger equation, except with an additional additive term $ c^2$. The effect of this  constant is to multiply solutions by a factor of $\exp(\pm i t  c^2)$, which physically corresponds to the rest energy. So, 
\begin{equation}
	\ker \Big\{ \pm i \frac{\partial}{\partial t} + c^2 + \frac{D^2}{2} - \frac{W}{2} \Big\} =    e^{\pm i t c^2}\ker \Big\{ \pm i \frac{\partial}{\partial t} + \frac{D^2}{2} - \frac{W}{2} \Big\}.
	\label{eq:modulation in t}
\end{equation}

Let us briefly comment on the use of the notation ``$N$,'' which stands for ``normal.'' There is a tradition in geometric microlocal analysis of using the term \emph{normal operator} to refer to the (usually exactly solvable) model problems that arise when studying PDE in some asymptotic limit, in this case $c\to\infty$. One says that $N(D_\pm)$ is the normal operator of $D_\pm$ in this regime. This is terminology we will use later (for different operators).

\subsubsection{The key approximation}
\label{subsec:misc_key}
When can we expect the error in \cref{eq:misc_031}, 
\begin{equation}
	E = c^2\sqrt{1+\frac{1}{c^2} ( D^2-W) } - c^2 - \frac{D^2}{2} + \frac{W}{2}, 
	\label{eq:E}
\end{equation}
to be negligible? Note that $E$ is a pseudodifferential operator, which means that we can analyze it in phase space, 
\begin{equation} 
	(0,\infty)_c \times T^* \bbR^d = (0,\infty)_c \times \bbR^d_x\times \bbR^d_\xi.
\end{equation} 
Replacing $D$ by the frequency coordinate $\xi$, the approximation $c^{-2} \langle \triangle \rangle \ll 1$ needed to justify \cref{eq:misc_033} is only plausible in the region of phase space where the frequency coordinates $\xi_1,\ldots,\xi_d$ all satisfy $|\xi_j| \ll c$. 
When physicists speak of the ``nonrelativistic limit,'' this is what they typically mean, rather than $c\to\infty$ per se.

However, unless  the Fourier transform of the initial data is compactly supported, a full analysis of the $c\to\infty$ limit requires analyzing the region of phase space where where $\lVert \xi \rVert \not \ll c$. 
In fact, we already saw this in \S\ref{subsec:natural}, using position-space language. The large frequency regime in which $\lVert \xi\rVert\sim c$ corresponds to the natural unit $x_\natural = cx$. Indeed, if $\varphi \in \calS(\bbR^d)$ is fixed, then $\varphi(x_\natural)$, which varies on the scale of natural units, has its Fourier transform spread out on scales $\lVert \xi \rVert \sim c$. In line with the heuristic discussed in the previous subsection, we should expect the analysis of the Klein--Gordon equation on these scales to be similar to the analysis of the \emph{free} Klein--Gordon equation, for which one has as good a theory as could be desired.

Assuming that the large-frequency regime $\lVert \xi \rVert \not\ll c$ (which does not just include the natural frequency regime $\lVert \xi \rVert \sim c$, but also $\lVert \xi \rVert \gg c$) is under even a modicum of control, then, as long as the Fourier transform of our initial data is concentrated on $\lVert \xi \rVert \ll c$, then we should expect the contribution of $E$ to the propagator 
\begin{align}
	e^{\pm i ct  \sqrt {c^2 + D^2-W} } &= e^{\pm i t(c^2 + 2^{-1}(D^2-W) + E) },
	\intertext{when applied to amenable initial data,
		to be negligible: }
	e^{\pm i ct  \sqrt {c^2 + D^2-W} } &\approx  \underbrace{e^{\pm i t(c^2 + 2^{-1}(D^2-W)) }}_{\text{Schr{\"o}dinger propagator}}.
\end{align}
Thus, \cref{eq:misc_033} is justified.
This is essentially the argument in \cite{CC}, which they made rigorous by using the Trotter--Kato formula.

\subsection{Challenges in the non- ultra-static or time-dependent settings}
\label{subsec:time_dependent}

Having discussed the ultra-static case in \S\ref{subsec:time_independent}, let us discuss what makes the non- ultra-static case, and especially the time-dependent case, more difficult.

\subsubsection{The lack of half-Klein--Gordon operators}

According to \Cref{thm:simplest}, the solution $u$ of the Cauchy problem \cref{eq:IVP} is approximated by a specific linear combination of solutions of the two possible Schr\"odinger equations. No other linear combination would work.
The reason there are ``two'' possible Schr\"odinger equations is the choice of sign of the $\pm i \partial_t$ term in \cref{eq:Schrodinger_initial}. Usually, both possibilities contribute to the non-relativistic limit.

In \S\ref{subsec:time_independent}, we saw how, in the ultra-static case, the factorization of the Klein--Gordon operator $P$ into the product $P=-D_+D_-$ of half-Klein--Gordon operators $D_\pm$ can be used to decompose $u$ into the portions approximated by each of the two possible Schr\"odinger equations. Specifically, 
\begin{equation}
	u= u_+ + u_- 
\end{equation} 
where $u_\pm$ solved the half-Klein--Gordon equation $D_\pm u_\pm = 0$, and, in the $c\to\infty$ limit, is approximated by $\exp(\pm i  c^2 t)$ times a solution $v_\pm$ of the Schr\"odinger equation in which the time derivative is term is $\pm i \partial_t$.

If $W$ is time-\emph{dependent}, then the factorizations $P=-D_+D_-$, $P=-D_-D_+$ no longer hold, because $D_+D_-$, $D_-D_+$ contain some terms in which a time derivative falls on the potential. (A similar issue arises if $V\neq 0$ in \cref{eq:P_intro}.) One could hope that such a term is ``lower order,'' in a suitable sense, so that $P=-D_+D_-,-D_-D_+$ still holds to leading order in $1/c^2$. We will not pursue this further, as, when we consider variable-coefficient metrics in the rest of the manuscript, it is not even clear how $D_\pm$ should be defined. So, the half-Klein--Gordon operators $D_\pm$ are of little use below.

\subsubsection{Partitioning into positive/negative energy solutions}
\label{subsec:splitting_diff}
Even without the half-Klein--Gordon operators, we could try using the temporal Fourier transform to split $u$ into positive/energy components. 
The idea would be that, if $u=u_++u_-$ for $u_\pm \approx \exp(\pm i c^2 t) v_\pm$ for $v_\pm$ independent of $c$, then we can extract $u_\pm$ from $u$ by looking at those components with large positive/negative temporal frequency. However, this does not work when the coefficients of the PDE are time-dependent, for the usual reason that the Fourier transform is of limited utility for variable-coefficient PDE.

Naturally, the solution will be to use 
spacetime microlocal methods. 
Thus, unlike the purely spatial microlocal methods useful in the time-independent case, our pseudodifferential calculi  will be constructed using the spacetime Fourier transform in an appropriate way. There exist pseudodifferential calculi for this purpose -- for example, the Parenti--Shubin (sc-) calculus $\Psi_{\mathrm{sc}}(\bbR^{1,d})$ on spacetime. Unfortunately, this calculus 
is ill-suited for analyzing the non-relativistic limit because, if $c\gg 1$, then $u_\pm\approx \exp(\pm i c^2 t)v_\pm$ is oscillating very fast in $t$, and these oscillations are getting faster as $c\to\infty$, whereas $\Psi_{\mathrm{sc}}$ is designed to detect, distinguish, and analyze oscillations with \emph{finite} frequency.

Instead, we will define in \S\ref{subsec:calczero} another pseudodifferential calculus, $\Psi_{\calczero}\subset C^\infty((0,\infty)_c;\Psi_{\mathrm{sc}})$, the \emph{natural} ($\natural$-) calculus, which can detect, distinguish, and analyze oscillations like $\exp(\pm i c^2 t)$. The simplest way of describing this pseudodifferential calculus is that it is the sc-calculus in natural units. Indeed, using the natural unit $t_\natural = c^2 t$, 
\begin{equation}
e^{\pm i  c^2 t} = e^{\pm i  t_\natural}
\end{equation}
has $t_\natural$-frequency independent of $c$. So, $\Psi_{\calczero}$ can be used to distinguish and analyze $u_\pm$. We can find $Q\in \Psi_{\calczero}$ such that $u_+= Qu$ and $u_- = (1-Q)u$, perhaps up to terms which are suppressed by a factor of $1/c$.  So, while $u_+$ is not exactly positive-energy, it is modulo lower order terms, where in this context ``lower order'' means suppressed as $c\to\infty$; the analogous statement applies to $u_-$. 

The calculus $\Psi_{\calczero}$ is, by itself, not sufficient to analyze the non-relativistic limit. We will need to define more sophisticated calculi $\Psi_{\calc},\Psi_{\calctwo}$ for this purpose (see \S\ref{sec:calculus}). But, $\Psi_{\calczero}$ is sufficient for splitting $u$ into its components $u_\pm$. 

\subsubsection{The lack of local energy estimates}

In \S\ref{subsec:time_independent}, the Stone--von Neumann theorem was used in order to establish a  a priori control on solutions of the half- Klein--Gordon equations. In the time-dependent setting, this is not available.

For individual $c>0$, the usual way to proceed is to use energy estimates. 
Unfortunately, standard, local energy estimates are ill-suited for the non-relativistic limit for reasons having to do with infinite speed of propagation. The issue is that solutions of the Schr\"odinger equation have supports/singular supports/wavefront sets which spread infinitely quickly. For each individual $c>0$, the solution of the Klein--Gordon equation \cref{eq:IVP} propagates at finite speed, but as $c\to\infty$, the speed grows.
If $c\gg 1$, then a solution  will quickly probe the large-$r$ structure of the potential (or, more generally, the large-scale structure of the spacetime under consideration). It will therefore not be possible to prove local energy estimates. Any estimate will, by necessity, depend on the global-in-space structure of the PDE.

It turns out that microlocal propagation estimates are general enough to control the Schr\"odinger equation; this was done by Gell-Redman, Gomes, and Hassell in \cite{Parabolicsc}, using a pseudodifferential calculus $\Psi_{\mathrm{par}}$ (the parabolic sc-calculus) in place of the ordinary pseudodifferential calculi used by H\"ormander. The global-in-space nature of the problem is not an issue but is reflected in the manner in which singularity/regularity is propagated.

As we will see, our new pseudodifferential calculi $\Psi_{\calc},\Psi_{\calctwo}$ allow one to interpolate between the usual energy estimates for the Klein--Gordon equation (which can be thought of as propagation estimates in $\Psi_{\mathrm{sc}}$) and the propagation estimates of Gell-Redman--Gomes--Hassell. This is one core part of our treatment of the non-relativistic limit.

\section{Main theorem and summary of methods}
\label{sec:true_intro}

Based on the discussion in \S\ref{sec:intro}, it is necessary to study at least two regimes in the spacetime phase space. Using $\tau$ to denote the frequency coordinate dual to $t$ and $\xi_j$ to denote the frequency coordinate dual to $x_j$, we have:
\begin{enumerate}[label=(\Roman*)]
	\item A regime in which the Klein--Gordon operator, conjugated by $e^{\pm i c^2t}$, is well-approximated by the time-dependent Schr\"odinger operator. In this regime, $|\xi|\ll c$.  
	In this paper, this regime appears associated with names such as ``pf,'' ``$\mathrm{pf}_\pm$,'' the `p' standing for ``parabolic'' (and `f' for ``face'').
	\label{it:I}
	\item The regime $|\xi|/c\geq \varepsilon > 0$, in which one should remove the $c$-dependence in the main part of the operator by working in natural units and in which the operator is well-approximated by the \emph{free} Klein--Gordon operator (times $c^{-2}$), as in \S\ref{subsec:natural}. In this paper, this regime appears associated with the word ``natural,'' which we abbreviate with the symbol $\natural$; for instance, `$\natural\mathrm{f}$' stands for ``natural face.''
	\label{it:II}
\end{enumerate}
One of the upshots of this manuscript, and of the companion paper \cite{NRL_I}, is that these are the \emph{only} two asymptotic regimes that need to be distinguished to understand the nonrelativistic limit, at least as far as basic solvability theory is concerned. Intermediate scales are encoded, in a precise way, as the corner between the main regimes.

\subsection{Microlocal methods}
\label{subsec:description}
In order to do microlocal analysis in these two asymptotic frequency regimes, we develop in \cite{NRL_I} pseudodifferential calculi $\Psi_{\calczero},\Psi_{\calc},\Psi_{\calctwo}$ (the members of which are one-parameter families of pseudodifferential operators on spacetime) suited to the task. 
The calculi are defined via quantizing symbols certain symbols on $T^* \bbR^{1+d}\times (0,\infty)_c$, namely those that are well-behaved with respect to particular compactifications\footnote{This is technically not a compactification because the $c\to 0$ direction is non-compact, but this is not our regime of interest.  Whenever we speak of ``compactifications,'' we will always mean compactifications of the $c\to\infty$ portion of phase space.}
\begin{equation}\label{eq:Planck,res intro}
	{}^{\calczero}\overline{T}^* \bbM,{}^{\calc}\overline{T}^* \bbM,	{}^{\calctwo}\overline{T}^* \bbM \hookleftarrow T^* \bbR^{1,d}\times (0,\infty)_c, \quad \bbM=\overline{\mathbb{R}^{1,d}}.
\end{equation} 
One says, in short, that such symbols are symbols on ${}^{\calc}\overline{T}^* \bbM,	{}^{\calctwo}\overline{T}^* \bbM$. (Note that what we call ``phase space'' includes the parameter $c$ as a degree-of-freedom, so it is actually $2d+3$-dimensional.)
Each of the $\calc$-, $\calctwo$- phase spaces is obtained by blowing up a certain submanifold of 
\begin{equation}
	{}^{\calczero}\overline{T}^* \bbM\hookleftarrow T^* \bbR^{1,d} \times (0,\infty)_c,
\end{equation}
the $\natural$-phase space, 
which can be described as a product
\begin{equation}\label{eq:Planck intro}
	\overline{\bbR^{1,d}_{t,x}} \times \overline{\bbR^{1,d}_{\tau_{\natural},\xi_{\natural} }}  \times (0,\infty]_c,
\end{equation} 
where $\tau_\natural = c^{-2} \tau$ and $\xi_{\natural} = c^{-1} \xi$ are the ``natural frequencies'' (dual to the natural position-space coordinates $t_\natural,x_\natural$ discussed in \S\ref{subsec:natural}).

Quantizing symbols on the $\calczero$-phase space yields the pseudodifferential calculus $\Psi_{\calczero}$, which we discuss in \S\ref{subsec:calczero}. Typical differential operators in it are polynomials in 
\begin{equation}
	\partial_{t_\natural} = c^{-2} \partial_t,\quad \partial_{x_{\natural,j}} = c^{-1} \partial_{x_j} 
\end{equation}
with coefficients that are symbolic functions of $(t,x)$. The calculus of $\calczero$-pseudodifferential operators is very similar to the semiclassical calculus with $h=1/c$, with the difference that each time derivative is paired with $h^2=c^{-2}$ instead of $h=c^{-1}$. The `natural face' $\natural\mathrm{f}=\{c = \infty\}$ corresponds to analysis on the scale of natural units. In total, the $\calczero$-phase space has three boundary hypersurfaces. In addition to $\natural\mathrm{f}$, we have fiber infinity, $\mathrm{df} = \{ |(\tau_{\natural},\xi_{\natural})| = \infty \}$, and base infinity, $\mathrm{bf}=\{|(t,x)|=\infty\}$.

The slightly more complicated phase space  $\smash{{}^{\calc}\overline{T}^* \bbM}$ (defined precisely in \cite{NRL_I}; see \S\ref{subsec:calc} for a sketch)  is the simplest phase space that allows us to study both $c\to\infty$ frequency regimes, and such that each slice of constant $c$ is just the phase space 
\begin{equation} 
	{}^{\mathrm{sc}}\overline{T}^* \bbM = \overline{\bbR^{1,d}_{t,x}}\times \overline{\bbR^{1,d}_{\tau,\xi}}
	\label{eq:misc_041}
\end{equation} 
of the scattering, or Parenti--Shubin (sc-) calculus $\Psi_{\mathrm{sc}}$ usually used to analyze the Klein--Gordon equation, such as by Vasy in \cite{VasyGrenoble} and Baskin--Doll--Gell-Redman in \cite{BDGR}. 
For the reader's convenience, we review the sc-calculus in \S\ref{subsec:sc}. 

At $c=\infty$, the $\calc$-phase space ${}^{\calc}\overline{T}^* \bbM$ has \emph{two} boundary hypersurfaces, $\natural\mathrm{f}$ and the `parabolic face' pf. As the name implies, $\natural \mathrm{f}$ is what is left from the face with the same name in the $\calczero$-phase space and thus corresponds to the regime \ref{it:II}. The other one, pf, admits the original frequency variables $\tau,\xi$ as smooth coordinates on its interior. Thus, this corresponds to \cref{it:I}.  
Both asymptotic regimes are present in our calculus, each corresponding to some boundary hypersurface of our compactified phase space. In addition, the way they fit together is encoded precisely as well, as the corner $\mathrm{pf}\cap\natural\mathrm{f}$. 

Concretely, $\smash{{}^{\calc}\overline{T}^* \bbM}$ will be defined by performing a ``parabolic'' blowup of the portion of the zero section of the $\calczero$-phase space over $\{c=\infty\}$. 
This process is sometimes known as \emph{second microlocalization}; $\Psi_{\calc}$ is the pseudodifferential calculus that results from second microlocalizing $\Psi_{\calczero}$ at the portion of the zero section over $\{c=\infty\}$. An alternative take is that $\Psi_{\calc}$ arises from second- microlocalizing the calculus $\Psi_{\mathrm{par,I}}$ (see \S\ref{subsec:parI}) of smooth one-parameter families of pseudodifferential operators valued in a pseudodifferential calculus known as $\Psi_{\mathrm{par}}$, about which we will say more below; see \S\ref{subsec:par} for a summary of its properties.

In total, the $\calc$-phase space has four boundary hypersurfaces. In addition to the three boundary hypersurfaces of the $\calczero$-phase space, we have the front face $\mathrm{pf}$ of the desingularizing blowup.  
The boundary hypersurface $\mathrm{pf}$ of the $\calc$-phase space can be identified with the phase space 
\begin{equation}
	{}^{\mathrm{par}}\overline{T}^* \bbM \hookleftarrow T^* \bbR^{1,d} 
\end{equation}
used by Gell-Redman--Gomes--Hassell to study the Schr\"odinger equation in \cite{Parabolicsc}. This is the phase space associated with $\Psi_{\mathrm{par}}$. One way of thinking about ${}^{\calc}\overline{T}^* \bbM$ is that it interpolates, in a natural way, between
\begin{itemize}
	\item ${}^{\mathrm{sc}}\overline{T}^* \bbM$ for $c<\infty$ and
	\item ${}^{\mathrm{par}}\overline{T}^* \bbM$ at $c=\infty$.
\end{itemize}
Similarly, the $\calc$-calculus $\Psi_{\calc}$ provides a framework to interpolate between the Klein--Gordon analysis carried out using the sc-calculus $\Psi_{\mathrm{sc}}$ and the Schr\"odinger analysis carried out using $\Psi_{\mathrm{par}}$.

Recall the functions $u_\pm$ referred to in \Cref{thm:simplest}. The $\calc$-calculus is really only sufficient to analyze each of these two halves of the Cauchy problem's solution $u$. (In the ultra-static case, these were solutions of the half-Klein--Gordon equations.)
It is in order to treat the full Klein-Gordon equation that we define the calculus $\Psi_{\calctwo}$. 
The corresponding phase space, 
\begin{equation} 
	{}^{\calctwo}\overline{T}^* \bbM\hookleftarrow T^* \bbR^{1,d}\times (0,\infty)_c,
\end{equation} 
is defined very similarly to  ${}^{\calc}\overline{T}^* \bbM$, but instead of second microlocalizing at the zero section, we second microlocalize at each of the points in $\natural \mathrm{f}$ corresponding to the oscillations $\exp(\pm i c^2 t)$ present in solutions of  the PDE. Recalling that multiplication by an oscillatory function corresponds to a translation in frequency space, if we call the front faces of the blowup 
\begin{equation}
	\mathrm{pf}_\pm \subset {}^{\calctwo}\overline{T}^* \bbM
\end{equation}
(there are two because we are second microlocalizing at two different points),
a neighborhood of $\mathrm{pf}_\pm$ in ${}^{\calctwo}\overline{T}^* \bbM$ is canonically identifiable with a neighborhood of $\mathrm{pf}$ in ${}^{\calc}\overline{T}^* \bbM$.

There is a corresponding $\calctwo$-pseudodifferential calculus with $5$ orders, one for each boundary hypersurface of phase space:
\begin{equation}
	\Psi_{\calctwo} = \bigcup_{m,\ell,q_-,q_+\in \bbR}\bigcup_{\mathsf{s}}  \Psi_{\calctwo}^{m,\mathsf{s},\ell;q_+,q_-}.
\end{equation}
We use $m$ for the differential order, $\mathsf{s}$ for the spacetime order, which will be a variable order, $\ell$ for the order at $\natural\mathrm{f}$, and orders $q_\pm$ at $\mathrm{pf}_\pm$. 
Corresponding to this pseudodifferential calculus is a scale of ($L^2$-based) Sobolev spaces 
\begin{equation} 
	H_{\calctwo}^{m,\mathsf{s},\ell;q_+,q_-} =  
	\{ H_{\calctwo}^{m,\mathsf{s},\ell;q_+,q_-}(c)\}_{c>0}\subset \calS'(\bbR^{1+d}).
\end{equation} 
For each individual $c$, 
\begin{equation} 
	H_{\calctwo}^{m,\mathsf{s},\ell;q_+,q_-}(c) = H_{\mathrm{sc}}^{m,\mathsf{s}}
\end{equation} 
at the level of sets, but the key point is that we have a $c$-dependent norm which captures the regularity and decay as measured using $\Psi_{\calctwo}$.

As is standard in microlocal analysis, the main consideration is the microlocal regularity of solutions; here, this takes place in the $\calctwo$-phase space. Since microlocal regularity is straightforward in the elliptic region of this phase space, we concentrate on the characteristic set, i.e.\ the complement of the elliptic region. Here, propagation of regularity holds along integral curves of the (suitably rescaled) Hamiltonian flow, 
which has a smooth extension to the boundary faces of the $\calctwo$-phase space. This is the microlocal analogue of a hyperbolic energy estimate.
A key role is played by the sources/sinks of the flow within the characteristic set, called
\emph{radial sets}. Each of the two components of the characteristic set have two radial sets, one corresponding to ``incoming'' waves and another to ``outgoing'' waves.

\subsection{Main results}
Our main result in \cite{NRL_I} is the uniform invertibility, in the non-relativistic limit $c \to \infty$, of the Klein--Gordon operator $P$ between suitable function spaces based on the $\calctwo$-Sobolev spaces just introduced. To state this introduce, for orders $m,\ell, q \in \bbR$ and variable order $\mathsf{s}$, the spaces 
\begin{equation}\begin{aligned}
\mathcal{X}^{m,\mathsf{s},\ell,q} &= \big\{ u \in H_{\calctwo}^{m,\mathsf{s},\ell;q,q} : Pu \in H_{\calctwo}^{m-1,\mathsf{s}+1,\ell-1;q,q} \big\}, \\
\mathcal{Y}^{m,\mathsf{s},\ell,q} &= H_{\calctwo}^{m,\mathsf{s},\ell;q,q}.
\end{aligned}\end{equation}
We prove:
\begin{theorem}[\cite{NRL_I}]
	\label{thm:inhomog}
	Let $P$ be as in \S\ref{subsec:assumptions} (see \cref{eq:P_def}).
	Let $\mathsf{s} \in C^\infty({}^{\calctwo}\overline{T}^* \bbM)$
	denote a variable order 
	satisfying the following conditions: 
	\begin{itemize}
		\item $\mathsf{s}$ is monotonic under the $\calctwo$-Hamiltonian vector field $\mathsf{H}_p$ near each component of the characteristic set, and 
		\item (threshold condition) on each component of the characteristic set, $\mathsf{s}>-1/2$ on one radial set and $\mathsf{s}<-1/2$ on the other.
	\end{itemize}
	(It is allowed that $\mathsf{s}$ be decreasing on one component of the characteristic set and increasing on the other.)	
	 We impose two more technical conditions: 
	\begin{itemize}
		\item $\mathsf{s}$ is constant near each radial set,
		\item $|\mathsf{H}_p \mathsf{s} |^{1/2}$ is smooth near the characteristic set.
	\end{itemize}
	Then, $P$ is an invertible operator 
	\begin{equation}
		P : \mathcal{X}^{m,\mathsf{s},\ell} \to \mathcal{Y}^{m-1,\mathsf{s}+1,\ell-1}
	\end{equation}
	for $c$ large enough.
	Moreover, the invertibility of $P$ is uniform in the sense that the norm of the inverse is uniformly bounded as $c \to \infty$; equivalently, 
	there exist $c_0,C>0$ such that the estimate
	\begin{equation}
		\lVert   u \rVert_{H_{\calctwo}^{m,\mathsf{s},\ell;0,0}} \leq C \lVert  Pu \rVert_{H_{\calctwo}^{m-1,\mathsf{s}+1,\ell-1;0,0}}
	\end{equation} 
	holds for any $c>c_0$ and 
	$u \in H_{\mathrm{sc}}^{m,\mathsf{s}(c)}$ such that $Pu \in H_{\mathrm{sc}}^{m-1,\mathsf{s}(c)+1}$.
\end{theorem}
\begin{remark}
    In the theorem above, the pair $(0,0)$ of $\mathrm{pf}_\pm$ orders can be replaced by $(q,q)$ for any $q\in \bbR$. Indeed, apply the theorem above to $h^{-q} u$ (for shifted $\ell$). Thus, 
    \begin{equation}
		\lVert   u \rVert_{H_{\calctwo}^{m,\mathsf{s},\ell;q,q}} \leq C \lVert  Pu \rVert_{H_{\calctwo}^{m-1,\mathsf{s}+1,\ell-1;q,q}}
	\end{equation} 
    holds for $\mathsf{s}$ as above (and some other $C>0$).
\end{remark}

While this result is precise, it has the disadvantage that it is stated in terms of unfamiliar and somewhat intricate function spaces. 
The next theorem shows that our analysis implies results that can be easily stated in terms of simple and familiar function spaces. As in \S\ref{sec:intro}, we consider the Cauchy problem, for simplicity with Schwartz initial data:
\begin{equation}
\begin{cases}
Pu=0, \\ 
u|_{t=0} = \varphi \in \calS(\bbR^d), \\ 
u_t|_{t=0} = c^{2} \psi \in \calS(\bbR^d).
\end{cases}
\label{eq:Cauchy-intro}
\end{equation}
Here $P$ is a variable-coefficient Klein--Gordon operator, which is asymptotic, at spacetime infinity, to the standard constant-coefficient Klein--Gordon operator; see \S\ref{subsec:assumptions} for details. (If the reader would prefer, they make take $P$ to be as in \S\ref{sec:intro}, \textit{with} loss of generality. They may also skip ahead to \S\ref{subsec:examples} for examples.)
The well-posedness of the Cauchy problem tells us that there exists a unique classical solution $u$ to this problem. 

We want to approximate $u$ by a linear combination $v=e^{-ic^2 t} v_- + e^{ic^2 t} v_+$, where $v_\pm$ are solutions of the Schr\"odinger initial value problem
\begin{equation}
\begin{cases}
N(P_\pm) v_\pm = 0 \\
v_\pm(0,x) = \varphi_\pm 
\end{cases}
\label{eq:Schrodinger_Cauchy_intro}
\end{equation}
for some suitable $\varphi_\pm\in  \calS(\bbR^d)$. Here, the $N(P_\pm)$ are the normal operators of $P$ at the two parabolic faces $\mathrm{pf}_\pm$. In the case of the standard Klein--Gordon operator these take (up to normalization) the form $\pm 2i \partial_t + \triangle$; see Prop.\ \ref{prop:normal} for the general case, which provides a formula for $N(P_\pm)$ in terms of the coefficients of our operator $P$. 

We choose the $\varphi_\pm$ so that the Cauchy data of $v$ matches the desired Cauchy data of $u$ as in \eqref{eq:Cauchy-intro}. This requires 
\begin{equation} 
	\varphi_\pm = (\varphi\mp i \psi)/2.
\end{equation} 
Then we prove: 
\begin{theorem}\label{thm:Cauchy}
	As $c\to\infty$, the difference $u-v$ between the solution $u$ of the Cauchy problem, \cref{eq:Cauchy-intro}, and the ansatz $v = \exp(-ic^2 t) v_- + \exp(ic^2 t) v_+$ defined above converges to $0$ uniformly in compact subsets of spacetime. Indeed, for any $\varepsilon>0$, $p\in [2,\infty]$,
	\begin{equation} 
	\lVert u-v \rVert_{(1+r^2+t^2)^{3/4+\varepsilon}L^p(\bbR^{1,d}) } =O(1/c)
	\label{eq:misc_ab7}
	\end{equation} 
	as $c\to\infty$. Moreover, the same estimate applies to $L(u-v)$, for any (constant coefficient) combination of the derivatives $c^{-2} \partial_t, \partial_{x_j}$. 

    Under some assumptions discussed in \Cref{prop:advanced/retarded_asymp_main}, we can improve $O(1/c)$ to $O(1/c^2)$. For instance, this applies if the functions $\alpha,w,h,\beta,B,W$ discussed in \S\ref{subsec:assumptions} are smooth functions of $1/c^2$, all the way down to $c=\infty$.
\end{theorem}
\Cref{thm:simplest} is a corollary of \Cref{thm:Cauchy}. 
The argument is that, given the setup of \Cref{thm:simplest}, it is possible, owing to \cref{eq:extendability}, to modify the PDE's coefficients outside of $\{|t|<T\}$ to ones which satisfy the requirements in \S\ref{subsec:assumptions}; simply multiply each of $V,A_j,W$ by $\chi_T(t/\langle r \rangle)$, where $\chi \in C_{\mathrm{c}}^\infty(\bbR;\bbR)$ is identically $1$ on some neighborhood of $[-T,+T]$ (see \Cref{fig:modification_zone}). The requirement \cref{eq:extendability} is what gives us \cref{eq:decay-perturbations-generic} in \S\ref{subsec:assumptions}, and the fact that $V,A_j,W$ were all assumed to be real gives us the other requirements \cref{eq:decay-skew-generic} and \cref{eq:decay-skew-generic-2}. Moreover, the coefficients $\beta,B,W$ (which are related to $V,A,W$ in \cref{ex:EM}) are independent of $c$, and $\alpha,w,h$ (which are related to the metric perturbation) are all zero in this example, so the hypothesis of the last clause of \Cref{thm:Cauchy} holds. 
Modifying the PDE outside of $\{|t|<T\}$ does not modify the solutions of the Cauchy problem \emph{within} $\{|t|<T\}$, so the desired estimate, \cref{eq:misc_a06}, follows from \cref{eq:misc_ab7}. 

\begin{figure}
	\begin{tikzpicture}
		\filldraw[fill=lightgray!20] (0,0) circle (2);
		\begin{scope}
			\clip (0,0) circle (2);
			\fill[lightgray!50] (2,0) to[out=200, in=-20] (-2,0) to[out=20,in=160] (2,0);
			\fill[darkred!50] (0,.7) to[out=45, in=180] (2,1.5) to[out=90, in=0] (0,2.2) to[out=180, in=90] (-2,1.5) to[out=0,in=135] cycle;
			\fill[darkred!50] (0,-.7) to[out=-45, in=-180] (2,-1.5) to[out=-90, in=0] (0,-2.2) to[out=-180, in=-90] (-2,-1.5) to[out=0,in=-135] cycle;
			\node[darkred] () at (0,1.5) {$U$};
			\node () at (0,0) {$\{|t|<T\}$};
		\end{scope}
		\draw (0,0) circle (2);
		\node () at (1.7,1.7) {$\bbM$};
	\end{tikzpicture}
	\caption{The region $U=\operatorname{supp}(1-\chi_T(t/\langle r \rangle))$ in $\bbM$ where we are modifying the coefficients $V,A,W$ of the PDE in \Cref{thm:simplest} in order to deduce that theorem from \Cref{thm:Cauchy}. This does not intersect the region $\{|t|<T\}$ (dark gray).}
	\label{fig:modification_zone}
\end{figure}
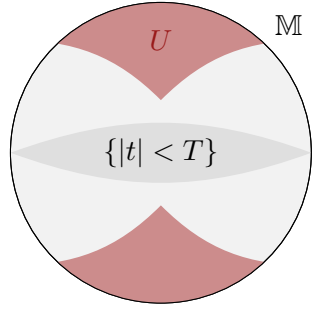

\subsection{Analytic and geometric assumptions}
\label{subsec:assumptions}
Let us now repeat the setup of \cite{NRL_I}, noting any additional hypotheses assumed here but not in our previous paper.

For each $c>0$, let $\eta(c) = -c^2 \dd t^2 + \dd x^2$ denote the exact Minkowski metric on $\bbR^{1,d}$ with speed of light $c$. (We will usually suppress the $c$-dependence in the notation.) Let $g=\{g(c)\}_{c>1}$ denote a family of Lorentzian metrics on $\bbR^{1,d}$ of the form 
\begin{equation} 
	g(c) = \eta(c) +  \alpha \dd t^2 + \sum_{j=1}^d \frac{w_j}{c}\mathrm{d} t\, \mathrm{d} x_j +\frac{1}{c^2} \sum_{j,k=1}^d h_{j,k} \dd x_j \dd x_k
	\label{eq:metric_form}
\end{equation}
for 
\begin{align} 
\begin{split}
	h_{j,k}=h_{k,j},\alpha,w_j     &\in C^\infty([0,\infty)_{1/c};  S^{-1}(\bbR^{1,d};\bbR)) \\ 
	&= C^\infty([0,\infty)_{1/c}; (1+t^2+r^2)^{-1/2} S^0(\bbR^{1,d};\bbR)) ,
	\label{eq:misc_055}
\end{split}
\end{align} 
where $j,k\in \{1,\dots,d\}$ and $S^0(\bbR^{1,d};\bbR)$ is the set of real-valued symbols on $\bbR^{1,d}$. So, $(1-\alpha/c^2)^{1/2}$ is the \emph{lapse function} of the time-function $t$, and $w_j/c^2$ is the spacetime \emph{wind}. Note that each term in $g-\eta$ is $O(1/c^2)$, in the sense that if we rewrite the metric in terms of $\mathrm{d}(ct)$ and $\mathrm{d} x_j$, 
\begin{equation}
    g(c) - \eta(c) = \frac{1}{c^2} \Big( \alpha \dd (ct)^2   +\sum_{j=1}^d w_j \mathrm{d} (ct)\, \mathrm{d} x_j +  \sum_{j,k=1}^d h_{j,k} \dd x_j \dd x_k\Big)
\end{equation}
then the coefficients are all $O(1/c^2)$. Thus, $g$ is a small perturbation of $\eta$ when $c\gg 1$. As stated in the introduction, we consider gravity a relativistic effect which is to be suppressed, by two factors of $1/c$ it turns out, in the $c\to\infty$ limit.

We also recall a few other dynamical assumptions stated in \cite{NRL_I}.
We assume that the metric $g$ is non-trapping for each $c>0$. For applications to the Cauchy problem, we will also assume that $g$ is globally hyperbolic with the Minkowski coordinate $t$ a time-function for $g$ and $\{t=0\}$ a Cauchy surface for the spacetime. Consequently, the Cauchy problem with initial data on the spacelike hypersurface $\{t=0\}$ is globally well-posed. Under the stated assumptions on $g$, these conditions are automatically satisfied as long as $c$ is sufficiently large. Since we are concerned with the $c\to\infty$ limit, the assumptions in this paragraph do not limit applicability.

The d'Alembertian $\square_{g}$ is defined using the sign convention
\begin{equation}
	\square_g = \frac{1}{\sqrt{|g|}} \sum_{i,j=0}^d \frac{\partial}{\partial z^i} \Big( \sqrt{|g|} g^{ij} \frac{\partial}{\partial z^j} \Big) =  \sum_{i,j=0}^d \Big( g^{ij} \frac{\partial^2}{\partial z^i \partial z^j} + \frac{\partial g^{ij}}{\partial z^i} \frac{\partial}{\partial z^j} + \frac{g^{ij}}{2|g|} \frac{\partial |g|}{\partial z^i} \frac{\partial}{\partial z^j} \Big),
	\label{eq:4hk6nv}
\end{equation}
where $z^0 = t$ and $\smash{z^j=x_j}$, and where $g^{ij}$ are the entries of the inverse of the matrix $\smash{\{g_{i,j}\}_{i,j=0}^d}$. 
We think of $\square_g$ as a one-parameter family of differential operators on $\bbR^{1,d}$, parametrized by $c$ (suppressing this $c$-dependence in the notation). 

It must be remembered that our Minkowski d'Alembertian $\square$ is also a one-parameter family of operators, depending on $c$.

The $O(c^{-4})$ term in the coefficient of $\partial_t^2$ in $\square_g - \square$ (this is the leading $\partial_t^2$ term of the difference $\square_g-\square$, two factors of $1/c$ suppressed compared to the leading term $c^{-2}\partial_t^2$ in $\square_g,\square$) will end up being important.  We call this $\aleph$; the precise definition is 
\begin{multline} 
\square_g -\square - c^{-4} \aleph \partial_t^2 \in \operatorname{span}_{C^\infty([0,1)_{1/c} ; S^{-1}(\bbR^{1,d};\bbR)) } \Big\{ \frac{1}{c^5} \frac{\partial^2}{\partial t^2}, \frac{1}{c^3} \frac{\partial^2}{\partial t \partial x_j} , \frac{1}{c^2} \frac{\partial}{\partial x_j x_k}\Big\} 
\\ + \operatorname{span}_{C^\infty([0,1)_{1/c} ; S^{-2}(\bbR^{1,d};\bbR)) } \Big\{ \frac{1}{c^3}\frac{\partial}{\partial t},\frac{1}{c^2} \frac{\partial}{\partial x_j} \Big\}.
\label{eq:aleph_def}
\end{multline} 
For black hole spacetimes, the large-$r$ behavior of $\aleph$  determines the mass $\frakm$ of the black hole, via  
\begin{equation} 
	\aleph = \Big(1+O\Big(\frac{1}{r}\Big)\Big) \frac{\frakm}{r}, \quad r\to\infty,\quad t\text{ fixed}.
\end{equation} 
See \Cref{ex:Kerr}.
More generally, $\aleph$ measures the mass enclosed in large ``spheres'' in spacetime \cite{SchoenYauI, SchoenYauII}\cite{Witten}. It is only this feature of the metric that will end up contributing to the non-relativistic limit.

We now specify the class of differential operators treated in this paper. 
We define the class of $m$-th order classical symbols to be
\begin{equation} \label{eq:classical-symbol-base}
S^{m}_{\cl}(\bbR^{1,d};\bbR)=(1+t^2+r^2)^{m/2} C^\infty(\overline{\bbR^{1,d}};\bbR).
\end{equation}
Let $P = \{P(c)\}_{c>1}$ denote a family of differential operators on $\bbR^{1,d}$ of the form 
\begin{equation}
	P = \square_{g} -  c^2 + \frac{i\beta}{c^2} \frac{\partial}{\partial t} + \sum_{j=1}^d iB_j \frac{\partial}{\partial x_j} + W,
	\label{eq:P_def}
\end{equation}
where  
\begin{equation} 
	\beta ,B_j,W\in C^\infty([0,1)_{1/c}; S_{\cl}^{-1}(\bbR^{1,d})).
	\label{eq:decay-perturbations-generic}
\end{equation}  
Note that, by assumption, these functions (as well as the coefficients of $\square_g$) can be restricted to $c=\infty$, i.e. there is a well-defined limit, in $S_{\cl}^{-1}(\bbR^{1,d})$, as $c \to \infty$. 

A requirement that we made in \cite{NRL_I} (for reasons of simplicity) is that 
\begin{equation}
	\Im \beta , \Im B_j, \Im W \in C^\infty([0,1)_{1/c}; S_{\cl}^{-2}(\bbR^{1,d})),
	\label{eq:decay-skew-generic}
\end{equation}
and
\begin{equation}
	 \Im B_j|_{c=\infty} = 0.
	\label{eq:decay-skew-generic-2}
\end{equation}
Then, $P$ is $L^2(\bbR^{1,d},g)$-symmetric to leading order in $c$.

In lieu of the conjugated half-Klein--Gordon $D_\pm$ operators discussed in \S\ref{subsec:time_independent}, which are difficult to define  in the time-dependent setting, we will find it more convenient to work with the operators that result from conjugating $P$ by the oscillations seen in solutions of the PDE. These ``conjugated Klein--Gordon operators'' are defined by
\begin{equation}\label{eq:P conjugated}
	P_\pm(c) = e^{\mp i c^2 t } P e^{\pm i c^2   t},
\end{equation}
meaning the maps $\calS'\ni u\mapsto  e^{\mp i c^2 t} P ( e^{\pm i c^2 t} u)$. The explicit computation of $P_\pm$ can be found in \cite{NRL_I}, but the formula is not important here. 
\begin{remark}[The conjugated perspective]
	From the microlocal point of view, the point of conjugating by the exponentials is that, in frequency space, this translates the frequencies of interest to the zero section of the appropriate cotangent bundle, and ``second microlocalization'' at the zero section is a particularly simple and concrete instance of second microlocal analysis, which is usually considered somewhat esoteric, even by experts. This ``conjugated'' perspective has been used in \cite{VasyLA, VasyN0L}\cite{SussmanACL} to analyze the Schr\"odinger--Helmholtz equation.
\end{remark}

We define the notion of ($\calc$-)normal operator in \cite{NRL_I}. Roughly, this is the leading part of the considered operator at the parabolic face pf in our compactified phase space (corresponding to regime (I) at the beginning of this section). Since the normal operator $N(P_\pm)$ of $P_\pm$ appears prominently in \Cref{thm:Cauchy} --- it is the particular Schr\"odinger operator governing the non-relativistic limit in that theorem --- we record here:
\begin{proposition}[\cite{NRL_I}] \label{prop:normal}
	The normal operators of $P_\pm$ are given by 
	\begin{equation}\label{eq:Ppm normal}
		N(P_\pm) = \mp 2 i \partial_t - \triangle \pm  \beta|_{c=\infty} + \sum_{j=1}^d i B_j|_{c=\infty}\partial_{x_j} + W|_{c=\infty} - \aleph,  
	\end{equation}	
	where $\beta,B,W,\aleph$ are as above.
\end{proposition}
Notice that the term $\aleph$, whose physical interpretation we described above, appears in this normal operator. This is the mechanism via which gravity contributes an effective Newtonian potential to the non-relativistic limit.

\begin{remark}
    Above, we required that the coefficients of the PDE be smooth functions of $1/c$. However, the examples arising in physics (see the examples in \S\ref{subsec:examples}) often involve coefficients that are \emph{even} in $1/c$, which means that they are smooth functions of $1/c^2$. Thus, no $O(1/c)$, $O(1/c^3)$, $O(1/c^5)$ terms, and so on, are present in their $c\to\infty$ expansion. We do not state any estimates that depend on the absence of $O(1/c^k)$ terms for $k\geq 2$, but it is worth noting that, when the coefficients of the PDE have no $O(1/c)$ terms, then we can often improve $O(1/c)$ estimates to $O(1/c^2)$; see e.g.\ \Cref{thm:Cauchy}. This is plausible, since no $O(1/c)$ terms should arise in the proofs of such estimates (assuming that the forcing, initial data, etc.\ also have no $O(1/c)$ terms).

    To be precise, suppose that $\calO$ is a quantity (a coefficient in the PDE, piece of initial data, etc.) and that our basic assumptions include
    \begin{equation}
    \mathcal{O} \in C^\infty([0,1)_{1/c};\bullet),
\end{equation}
with $\bullet$ being a Fr\'echet or Banach space depending on the setting.
Then, 
\begin{equation} \label{eq:def-general-quadratic-convergence}
    \mathcal{O} - \mathcal{O}|_{1/c=0} \in c^{-2}C^\infty([0,1)_{1/c};\bullet)
\end{equation}
makes precise the requirement that $\calO$ have no $O(1/c)$ term in its $c\to\infty$ expansion. 

For example, for the PDE coefficients in \cref{eq:decay-perturbations-generic}, the condition \cref{eq:def-general-quadratic-convergence} says
\begin{equation} 
	\beta - \beta|_{1/c=0},\; B_j-B_j|_{1/c=0},\; W-W|_{1/c=0}\in c^{-2}C^\infty([0,1)_{1/c}; S_{\cl}^{-1}(\bbR^{1,d})).
	\label{eq:quadratic-converge-1}
\end{equation}
If these hold, we will just say that $P$ itself satisfies \cref{eq:def-general-quadratic-convergence}.

We do \emph{not} need an additional assumption stating that $g$ or the coefficients in $\square_g$ satisfy \cref{eq:def-general-quadratic-convergence},  because this is built into the assumptions on $g$ that we have already imposed. Indeed, in \cref{eq:metric_form}, each perturbation term is $O(1/c^2)$ relative to the corresponding term in the Minkowski metric. 
Correspondingly, 
\begin{equation} \label{eq:Boxg-Box-membership}
    \Box_g - \Box \in \mathrm{Diff}_{\calctwo}^{2,-1,0;-2,-2};
\end{equation}
this is what \cite[Prop.\ 3.1]{NRL_I} says. 
So, compared to $\square,\square_g\in \operatorname{Diff}_{\calctwo}^{2,0,2;0,0}$, the difference $\square_g-\square$ is already automatically suppressed by \emph{two} orders at each of $\natural\mathrm{f},\mathrm{pf}_\pm$, not just one. 
\end{remark}

\begin{remark}
The condition \cref{eq:def-general-quadratic-convergence} is usually insensitive to the function space that $\bullet$ is a placeholder for. For example, together with \cref{eq:decay-skew-generic} and \cref{eq:decay-skew-generic-2}, it implies
\begin{equation}
	\Im \beta - \Im \beta|_{1/c=0}, \; \Im B_j, \; \Im W - \Im W|_{1/c=0} \in c^{-2}C^\infty([0,1)_{1/c}; S_{\cl}^{-2}(\bbR^{1,d})).
	\label{eq:quadratic-converge-2}
\end{equation}
Consequently, it is not necessary to specify $\bullet$ precisely. 
\end{remark}

An example where the condition \cref{eq:def-general-quadratic-convergence} does not hold is discussed as part of \Cref{rem:comparison}; fixing the magnetic field \emph{in cgs units} results in a magnetic field which is $O(1/c)$ as $c\to\infty$.

\subsection{Examples}\label{subsec:examples}
We now present two examples, which serve to clarify the notation and demonstrate that the assumptions imposed above are reasonable and do not exclude the traditional cases of interest. 

\begin{example}[Scalar electrodynamics in Minkowski]
	Consider the setup of \S\ref{sec:intro}, 
	\begin{equation}
		-\frac{1}{c^2}\Big(  \frac{\partial}{\partial t} + iV \Big)^2u  - (i \nabla  + \bfA)^2 u  +  c^2 u + \VVV u = 0.
	\end{equation}
	As discussed in \Cref{rem:comparison}, this is the PDE proffered by physics texts for describing the time-evolution of a classical scalar field (with mass $m=1$ and charge $e=1$) coupled to a background scalar field $\VVV$ and an electromagnetic four-potential $A = (c^{-1} V,\bfA)\in C^\infty(\bbR^{1,3};\bbR^4)$.
	Here, $\bfA \in C^\infty(\bbR^{1,3};\bbR^3)$ is the vector (a.k.a.\ magnetic) potential, measured in units of momentum per unit charge, and $V\in C^\infty(\bbR^{1,3};\bbR)$ is the scalar (a.k.a.\ electric) potential, measured in volts.

	Expanding this out, 
	\begin{equation}
		-\frac{1}{c^2} \frac{\partial^2 u}{\partial t^2} - \triangle u - 2i\Big( \frac{V}{c^2} \frac{\partial}{\partial t} +  \bfA\cdot \nabla\Big) u - \Big( - \frac{V^2}{c^2} +  \bfA^2 +\frac{i\dot{V}}{c^2} + i \nabla\cdot \bfA - \VVV\Big) u - c^2 u = 0,
		\label{eq:misc_479}
	\end{equation}
	where $\dot{V}=\partial_t V$. 
	So, $Pu=0$ for $P$ the operator defined in \cref{eq:P_def} for $g$ the Minkowski metric, $\beta=2 V$, $B_j = 2 A_j$ for $j=1,\dots,d$, and 
 	$W= \VVV+i \nabla \cdot \bfA + \bfA^2 + c^{-2}(i\dot{V}-V^2)$. 
	All of the terms in $W$ involving the electromagnetic potential $V$ are suppressed by at least one power of $1/c$ and therefore do not contribute to $N(P_\pm)$; the contribution to $N(P_\pm)$ involving $V$ comes from $\beta$. 
	\Cref{prop:normal} tells us that, in this case, 
	\begin{align}
		\begin{split} 
			N(P_\pm) & = \mp 2 i \frac{\partial}{\partial t} -\triangle + \Big( \pm \beta + \sum_{j=1}^d i B_j \partial_j + W \Big)\Big|_{c=\infty} \\
			&=\mp 2 i \frac{\partial}{\partial t} -\triangle \pm 2 V - 2i \bfA\cdot \nabla - i\nabla\cdot \bfA - \bfA^2 +\VVV.
		\end{split} 
	\end{align}
	That is, $2^{-1} N(P_\pm)  = \mp i \partial_t - 2^{-1}(i\nabla + \bfA)^2  \pm V +2^{-1}\VVV$
	is the conventionally normalized Schr\"odinger operator describing a non-relativistic particle with unit charge and mass $m=1$ in the presence of a background potential $V$ and vector potential $\bfA$.  Note that, if $V\geq 0$, then the term $V$ it contributes to $N(P_+)$ is purely repulsive and the term $-$ contributes to $N(P_-)$ is purely attractive. The physical interpretation of this fact is that if a scalar particle has charge $e$, its anti-particle has charge $-e$, just as the electron and positron have opposite charges. 
	
	A comment on the range of the four-potential is in order.
	Our assumption on the decay of $\beta,B_j,W$ as $(t,r)\to\infty$ requires $V,A_j \in (1+t^2+r^2)^{-1/2} L^\infty$.
	So, this allows Coulomb potentials which fade away as $t\to\infty$. Longer range potentials, such as asymptotically constant potentials or the linear potentials generated by uniform electric fields are excluded. These have a serious effect on the asymptotic theory of the PDE. 
	\label{ex:EM}
\end{example}

\begin{example}[Astrophysical Kerr-like spacetimes] \label{ex:Kerr}
	First, we recall the definition of the exact Kerr exterior. For simplicity, we restrict attention in this example to $c\geq 1$. The Kerr exterior is a 4-dimensional Lorentzian manifold $(\bbR_t \times \bbS^2 \times \bbR_{r> r_+}, g_{\frakm,\fraka} )$ with an explicit metric $g_{\frakm,\fraka}$. This has parameters $\frakm>0$ (the mass of the black hole) and $\fraka\in (-\frakm,\frakm)$ (the angular momentum).
	\begin{itemize}
		\item First, recall the Boyer--Lindquist coordinate system, in which one parametrizes the sphere $\bbS^2$ in $\bbR_t \times \bbS^2 \times \bbR_{r> r_+}$ using the polar angle $\theta \in (0,\pi)$ and the azimuthal angle $\phi \in \bbR$. (Thus, the black hole's axis of rotation $\{\theta=0,\pi\}$ is omitted from this coordinate chart.) In Boyer--Lindquist coordinates, 
		\begin{equation}
			g_{\mk{m},\mk{a}} = -\Big(1-\frac{r_{\mathrm{s}}r}{\Sigma}\Big)c^2 \dd t^2 + \frac{\Sigma}{\mathfrak{D}} \dd r^2 +\Sigma \dd \theta^2 
			+ \Big(r^2+\mk{a}^2+\frac{r_{\mathrm{s}} r \mk{a}^2}{\Sigma} \sin^2\theta\Big)\sin^2\theta \dd \phi ^2
			- \frac{2r_{\mathrm{s}}r\mk{a}\sin^2\theta}{\Sigma}c\dd t \dd \phi,
			\label{eq: kerr metric}
		\end{equation}
		where $r_{\mathrm{s}} = 2 G \frakm/c^2$ is the Schwarzschild radius (with $G$ the gravitational constant), $\Sigma = r^2 + \fraka^2 \cos^2 \theta$, and $\mathfrak{D} = r^2 - r_{\mathrm{s}} r + \fraka^2$.
	The outer radius $r_+$ is the largest zero of $\mathfrak{D}(r)$, if $\frakD(r)$ has a zero in $r>0$, or $0$ otherwise. 
		
		Since we are ultimately concerned with the $c\to\infty$ limit, let us point out that the parameter $\fraka$ is related to the ``angular momentum'' $J$ of the black hole by $\fraka = J/\frakm c$. So, $\fraka$ has units of length. In the non-relativistic limit, it is most natural to fix the physically significant parameters $J$ and $\frakm$ and therefore let $\fraka\to 0^+$. However, we allow $J \in C^\infty([0,\infty)_{1/c})$ to be a any smooth function of $1/c$. The case $J=0$ is already interesting.
        If $J$ is fixed and $c$ is small enough, then $\frakD$ is non-vanishing; the metric is the exterior of an ``over extremal'' black hole, with a naked singularity. Physicists do not like the naked singularity, but as a model for spacetime outside of some astrophysical body, there is no issue with the over extremal case. We are going to modify the metric near the origin anyways.

        Newton's constant $G$ should stay fixed.
		\item The Kerr--Schild coordinate system is useful because it covers a neighborhood of the axis of rotation of the black hole and presents the metric as the Minkowski metric plus a perturbation. In this coordinate system, the spacetime is a subset of $\bbR_t\times \bbR^3_{x,y,z}$, and the metric has the form 
		\begin{equation}
			g_{\frakm,\fraka} = \eta + f k_\mu k_\nu \dd x^\mu \dd x^\nu  
			\label{eq:misc_190}
		\end{equation}
		for $f= 2 G \frakm r^3_0  c^{-2}  (r^4_0 + \fraka^2 z^2)^{-1}$, $k=(1,\bfk)$, 
		\begin{equation}
			{\bfk} = \Big(\frac{r_0 x + \fraka y}{r_0^2 + \fraka^2}, \frac{r_0 y - \fraka x}{r^2_0 + \fraka^2}, \frac{z}{r_0 } \Big),
			\label{eq:KerrSchild_k}
		\end{equation}
		where $x^0=ct$ and $r_0$ is the solution to $(x^2+y^2) (r^2_0+\fraka^2)^{-1} + z^2/r_0^2 = 1$. 
	\end{itemize}
	
	The Kerr exterior does not fit directly into our framework. 
	One reason is the presence of the horizon, $\{r=r_+\}$, which is an additional boundary on the spacetime with a large effect on the long-time asymptotics of solutions of the Klein--Gordon equation. We do not want to study the $r\to r_+$ limit. The natural course of action is to consider a stationary metric on $\bbR^{1,3}$ isometric to Kerr, in Kerr--Schild coordinates, outside of some spatially compact set. Such metrics are of astrophysical interest because they describe the gravitational field generated by astrophysical bodies lacking the sufficient density to form a black hole. Indeed, the gravitational field generated by the Earth is an example of such a spacetime. Since we are interested in the $c\to\infty$ limit, we must also specify how the metric is behaving in this limit.
	
	Let us define an \emph{astrophysical Kerr-like spacetime} $(\bbR^{1,3},\tilde{g}_{\frakm,\fraka})$ to be a Lorentzian spacetime of the form 
	\begin{equation}
		\tilde{g}_{\frakm,\fraka} = \eta + \tilde{f} \tilde{k}_\mu \tilde{k}_\nu \dd x^\mu \dd x^\nu + c^{-2} \delta g 
	\end{equation}
	for 
	\begin{itemize}
		\item $\tilde{f},\tilde{k}$ of the form above except defined with $\tilde{r}_0$ in place of $r_0$ for 
		\begin{equation} 
			\tilde{r}_0 \in C^\infty( [0,\delta)_\fraka; C^\infty(\bbR^3_{x,y,z} ; (1,\infty) ) )
		\end{equation}
		which agrees with $r_0$ outside of a large Euclidean spatial ball,
		\item  $\delta g$ a spatially compactly supported symmetric two-tensor depending smoothly on $1/c$, all the way down to $1/c=0$. That is, for some compact $K \subset \bbR^3$, 
		\begin{equation}
			\delta g \in C^\infty( [0,1)_{1/c}; C^\infty( \bbR^{1,3} ; \operatorname{Sym}^2 T^* \bbR^{1,3} )  ) , \quad \supp \delta g \subset \bbR_t \times K,
		\end{equation}
		\item and we require that $\eta + \varepsilon(\tilde{f} \tilde{k}_\mu \tilde{k}_\nu \dd x^\mu \dd x^\nu + c^{-2} \delta g) $ is a Lorentzian metric for all $\varepsilon \in [0,1]$.
	\end{itemize}
	Thus, any astrophysical Kerr-like spacetime is isometric to the Kerr exterior outside of a spatially compact region, with the isometry provided by endowing the Kerr exterior with Kerr--Schild coordinates. The gravitational field configurations generated by sufficiently low-density astrophysical objects should satisfy the assumptions above.

	Even after the concession that $\tilde{g}_{\frakm,\fraka}$ be modified near the horizon, it does \emph{not} fall into the framework of the rest of the paper, because it is stationary. Rather, we want our metrics to be asymptotically Minkowski both as $r\to\infty$ (which is true for Kerr) and as $t\to\infty$.
	While this is not true for Kerr, we can modify $\tilde{g}_{\frakm,\fraka}$ so as to make it true. The resulting spacetimes will be rather unphysical. However, as long as the modification is done only in the region $|t|>T$ for some large $T$, we do not change the solution of the Klein--Gordon Cauchy problem for $|t|<T$. Consequently, any uniform-in-spacetime bound for solutions of the Cauchy problem on the modified spacetime yields a uniform-in-$[-T,+T]_t\times \bbR^3_x$ bound on solutions on the original astrophysical spacetime $(\bbR^{1,3},\tilde{g}_{\frakm,\fraka})$. This is the same argument used to deduce \Cref{thm:simplest} from \Cref{thm:Cauchy}.
	
	So, let 
	\begin{equation}
		\hat{g}_{\frakm,\fraka} = \eta + \varepsilon(t,x) (\tilde{f} \tilde{k}_\mu \tilde{k}_\nu \dd x^\mu \dd x^\nu + c^{-2} \delta g) 
	\end{equation}
	for $\varepsilon \in C^\infty(\bbM;[0,1])$ such that $\varepsilon = 1$ identically if $|t|<T$, for some $T>0$, $\varepsilon$ is vanishing if $|t| > \Upsilon_0 + \Upsilon_1 |r|$, where $\Upsilon_0>T,\Upsilon_1>0$. This is a Lorentzian metric. 
	
	Let us now check that its coefficients have the required form. Because $\varepsilon$ is vanishing near the north/south pole of $\bbM$, it suffices to prove that the Kerr metric $g_{\frakm,\fraka}$ has the required form on $\bbM$ in a neighborhood of the closure of the temporal axis chosen sufficiently large so as to include the event horizon for all $c\geq 1$. Since we are restricting attention away from the temporal axis, we can use $1/\langle r \rangle$ as a boundary-defining-function. The function $f$ in the Kerr--Schild formula  \cref{eq:misc_190} is given by
	\begin{equation}
		f = \frac{2G \frakm r_0^3}{c^2(r_0^4 + \fraka^2 z^2)} = \frac{2G \frakm r_0^3}{c^2(r_0^4 + J^2  z^2 / \frakm^2 c^2)}.
	\end{equation}
	From the discussion above, this lies in $\langle r \rangle^{-1} C^\infty((c_0,\infty]_c ; C^\infty(\overline{\bbR^3_{x,y,z}} \cap \{x^2+y^2+z^2 > R \}) ) $ for $c_0,R$ sufficiently large. On the other hand, the vector $\bfk$ defined in \cref{eq:KerrSchild_k} satisfies 
	\begin{equation}
		{\bfk} \in C^\infty((c_0,\infty]_c ; C^\infty(\overline{\bbR^3_{x,y,z}} \cap \{x^2+y^2+z^2 > R \} ;\bbR^3 ) )
	\end{equation}
	for $c_0,R$ sufficiently large.  Therefore, the Kerr--Schild metric has the required form. 
	
	These spacetimes are similar to the notion of radiating Schwarzschild (a.k.a.\ Vaidya) spacetimes, except here $\frakm,J$ are essentially functions of $t/r$ rather than $t-r$. The latter is more natural from the point-of-view of gravitational physics, but let us repeat that the spacetime $(\bbR^{1,3},\hat{g}_{\frakm,\fraka})$ is playing a purely technical role. There is therefore no requirement that $\hat{g}_{\frakm,\fraka}$ have a sensible physical significance. See  \cite[\S9.5]{griffiths2009exact} for details on the Vaidya spacetimes. 
\end{example}

\subsection{Open problems}
\label{subsec:todo}

\begin{itemize}
	\item \textit{When the coefficients of the PDE are analytic in $1/c$, can it also be shown that}
    \begin{equation*} 
        u = e^{-ic^2 t} u_-+e^{ic^2 t} u_+
    \end{equation*}
  \textit{for $u_\pm$ analytic in $1/c$} (say, pointwise in spacetime)? This would be the analogue for the Cauchy problem of Veseli{\'c}'s spectral-theoretic results in \cite{veselic1971perturbation, veselic1983nonrelativistic}. One could hope to apply the estimates in \S\ref{sec:Cauchy} to the terms in the $c\to\infty$ expansion in order to establish bounds sufficiently good to guarantee that the power series in $1/c$ has a nonzero radius of convergence. We do not know if this is possible.
	\item 
	It is known that solutions of the initial-value problems for the Klein--Gordon and Schr\"odinger equations (with, say, Schwartz initial data), admit full asymptotic expansions at infinity. We would like to understand \emph{joint} asymptotics as $c\to\infty$ \emph{and} $1+r^2+t^2 \to \infty$, whereas in this work we only analyzed $c\to\infty$ asymptotics in compact subsets of spacetime. For example, it would be of interest to understand the non-relativistic limit of the scattering matrix for the Klein--Gordon equation. We cannot say anything about this using our existing tools.

	Joint asymptotics are expected to be somewhat subtle. For example, for individual $c>0$, it was proven in \cite{desc} that solutions of the Cauchy problem admit full asymptotic expansions at spacetime infinity, but this required a blowup of null infinity and microlocal analysis on the blown up space. It was essentially proven in \cite{gell2023scattering} that solutions of the Schr\"odinger initial-value problem admit full asymptotic expansions at spacetime infinity, but this also required a blowup of the equator of the radial compactification in spacetime (however, only in the last step; no corresponding microlocal analysis was required). So, corresponding blowups are likely required to understand the non-relativistic limit.

    To make matters worse, we have seen in \cite[Pf.\ of Prop.\ 5.11]{NRL_I} indications that a blowup of the whole corner of $[0,\infty)_{1/c}\times \bbM$ is required. It is unclear how this meshes with the blowups described in the previous paragraph.
	
	\item We have assumed that the coefficients of our PDE are well-behaved on the radial compactification of the spacetime away from the north/south poles. This seems unnecessarily restrictive as far as the Cauchy problem is concerned. 
    In particular, we cannot handle general coefficients of type
    \begin{equation*} 
        C^\infty(\bbR_t;C^\infty(\overline{\mathbb{R}^d}_x))
    \end{equation*}
    (with the suitable amount of decay), as these are singular at the equator of the boundary $\partial \bbM$ of the radial compactification of the spacetime. 
    To handle these using our methods, it would be necessary to blow up the equator and analyze the PDE using a variant of $\Psi_{\calczero},\Psi_{\calc},\Psi_{\calctwo}$ adapted to that blowup. This has not even been done for the Schr\"odinger equation; however, if one restricts attention to bounded times (this being the same thing as working near points in the deep interior of the front face of the aforementioned blowup),   there is extensive literature, and microlocal tools have been developed by Wunsch \cite{Wunsch}, following earlier work of Craig, Kappeler, and Strauss \cite{Craig}.
	
	\item As already mentioned in \S\ref{subsec:literature_only}, there have been many works investigating the non-relativistic limit for nonlinear Klein-Gordon equations with e.g.\ power type nonlinearlities $N[u]$, like $|u|^{p-1}u$ or $u^p$. Unfortunately, the $\calctwo$-Sobolev spaces here do not suffice for the analysis of the nonlinear PDE; for example, we cannot prove estimates on the $L^\infty$-norm of $u$ with a sharp spatial weight, and the sharp $L^2$-estimates we do prove are not enough to control products like $u^p$.
    
    Future work may overcome this by employing function spaces which measure, in addition to $\calctwo$-Sobolev regularity, a suitable sort of so-called ``module regularity,'' as in \cite{gell2023scattering}. Hopefully, estimates for the linear equation, of the sort proven here except phrased using these module-regularity spaces, would suffice to analyze the nonlinear equation (with small data/forcing) as a perturbation of the linear equation. This may be as difficult as the problem, described above, of understanding joint asymptotics.
\end{itemize}

\subsection{Outline for remainder of manuscript}
\label{subsec:outline}

The sections in the rest of the manuscript are as follows:
\begin{itemize}
	\item In \S\ref{sec:calculus}, we will cover the pseudodifferential calculi out of which $\Psi_{\calc},\Psi_{\calctwo}$ are built, namely $\Psi_{\mathrm{sc}},\Psi_{\mathrm{par}},\Psi_{\calczero}$, in addition to $\Psi_{\calc},\Psi_{\calctwo}$. 
	\item In \S\ref{sec:inhomogeneous}, we apply the main estimates from \cite{NRL_I} to the production of $c\to\infty$ asymptotics for the four basic inhomogeneous problems, the advanced/retarded and Feynman/anti-Feynman problems. 
	\item In \S\ref{sec:Cauchy}, we study the Cauchy problem in the non-relativistic limit. It is here that we prove the main theorems of this paper. 
\end{itemize}

In addition, we have a handful of short appendices:
\begin{itemize}
	\item \S\ref{sec:notation} is an index of notation. 
	\item \S\ref{sec:free} contains a discussion of the non-relativistic limit of the free Klein--Gordon equation using the spacetime Fourier transform. This appendix is expository, but hopefully helpful nonetheless. In some sense, our microlocal methods are built to ``microlocalize'' what one can do for the constant-coefficient PDE using the spacetime Fourier transform, so the free case is useful to keep in mind. 
	\item \S\ref{sec:relations} contains a few elementary lemmas relating various Sobolev norms used in this paper. 
    \item In \S\ref{sec:naturalCauchy}, we study the ``natural'' Cauchy problem, in which the initial data is dilated with $c$. Here, we prove that the non-relativistic limit is governed by the free Klein--Gordon equation, in concordance with the discussion in \S\ref{subsec:natural}. 
\end{itemize}

\section{
\texorpdfstring{Summary of pseudodifferential calculi: $\Psi_{\mathrm{sc}},\Psi_{\mathrm{par}},\Psi_{\calczero}$, $\Psi_{\calc}$, $\Psi_{\calctwo}$}{Summary of pseudodifferential calculi}  }
\label{sec:calculus}

The goal of this section is to discuss the various pseudodifferential calculi which form the building blocks of our analysis. This section is expository, as the relevant calculi are all defined and developed elsewhere. In particular, we refer to \cite{NRL_I} for the development of the three new calculi 
$\Psi_{\calczero},\Psi_{\calc},\Psi_{\calctwo}$. The telegraphic presentation below is the minimum required to make our exposition here self-contained.

We begin by reviewing two existing calculi, $\Psi_{\mathrm{sc}},\Psi_{\mathrm{par}}$, in \S\ref{subsec:sc} and \S\ref{subsec:par}, respectively. 
The ``natural-res'' calculus $\Psi_{\calc}$ introduced in \S\ref{subsec:calc} can be considered as interpolating between them. More precisely, an element $A\in \Psi_{\calc}$ is a one-parameter family $A=\{A(h)\}_{h>0}$ of $A(h)\in \Psi_{\mathrm{sc}}$ that degenerates in a controlled way to an element of $\Psi_{\mathrm{par}}$ as one approaches $\mathrm{pf}$. 

The calculus of smooth families of elements of $\Psi_{\mathrm{par}}$ is covered briefly in \S\ref{subsec:parI}. This is called $\Psi_{\mathrm{par,I}}$. 

The natural-res calculus $\Psi_{\calc}$ is built by ``second-microlocalizing'' $\Psi_{\calczero}$. In \S\ref{subsec:calczero}, we introduce this calculus and record a few lemmas which will be required later. 
As the reader will see, $\Psi_{\calczero}$ is a variant of the semiclassical calculus $\Psi_\hbar$ \cite{Zworski}; it is even more closely related to the semiclassical foliation calculus of \cite{Vasy:Semiclassical-X-ray}.  The main difference is that in the natural calculus the vector fields that are considered $O(1)$ as $h\to 0$ are $h^2 \partial_t,h \partial_{x_j}$, whereas in the ordinary semiclassical calculus they are $h\partial_t,h\partial_{x_j}$. (Recall that $h=1/c$.) The extra factor of $h$ on the time derivatives is important for the intended application to the non-relativistic limit, and $\Psi_{\calczero}\neq \Psi_\hbar$, but the two calculi are structurally similar.\footnote{In \cite{Vasy:Semiclassical-X-ray} a similar scaling is used, but $t$ is replaced by a boundary defining function of the radial compactification.}  Symbolic constructions -- in particular, the microlocalized elliptic parametrix construction and the commutant constructions at the heart of propagation/radial point estimates -- work similarly in both. Thus, our exposition in \S\ref{subsec:calczero} will be brief. 

Unlike the calculi $\Psi_{\calc},\Psi_{\calctwo}$, the calculi $\Psi_{\mathrm{sc}},\Psi_{\mathrm{par}},\Psi_{\calczero}$ are ``fully symbolic''. This means that, for each calculus $\Psi_\bullet$, there exists 
a principal symbol map $\sigma_\bullet$ that assigns to elements of $\Psi_\bullet$ certain (equivalence classes of) smooth functions on the phase space $T^* \bbR^{1,d}$ or $T^* \bbR^{1,d}\times (1,\infty)_c$, such that operators in these calculi are captured by their principal symbols modulo ``small'' errors. 
Here ``small'' means -- as applied to operators of order 0 in every index, for ease of explanation -- operators of negative order in every index. In the differential index, this means the ``small'' error is a smoothing operator; while in the spacetime index, it means decay-inducing. Such operators are compact on $L^2$. In addition, for $\Psi_{\calczero}$, a negative semiclassical index means that the $L^2$-operator norm tends to zero as $h \to 0$. 

In \S\ref{subsec:calc}, we discuss $\Psi_{\calc}$, and in \S\ref{subsec:calctwo}, we discuss $\Psi_{\calctwo}$. To reiterate, more detailed presentations of the theory of these calculi can be found in \cite{NRL_I}.

\begin{remark}[Equality in the sense of compactifications]
	In order to avoid needing to keep track of natural diffeomorphisms between different compactifications of phase space, 
	we employ the notion of equality in the \emph{sense of compactifications}. That is, we consider two manifolds-with-corners (mwc) equal if we have a canonical identification of their interiors extending to a diffeomorphism, including the boundary and corners. 
	When working with two different partial compactifications $X\hookrightarrow X_1,X_2$
	we will assume without loss of generality that 
	the mwc-theoretic interiors of $X_1,X_2$ are literally equal.
	So, we write $X_1=X_2$ if there exists a diffeomorphism $\phi$ such that the diagram 
	\begin{equation} 
		\xymatrix{ 
			X_1 \ar[r]^\phi & X_2 \\ 
			X \ar@{^{(}->}[ur] \ar@{^{(}->}[u]
		}
	\end{equation} 
	commutes, in which case $\phi$ is necessarily unique. 
	
	For example, we will view $\overline{\bbR^d}\backslash \{0\}$ and $[0,\infty)_{1/r}\times \bbS^{d-1}_\theta$ as literally equal. Likewise, $(0,\infty]_r= [0,\infty)_{1/r}$.
\end{remark}

\subsection{
\texorpdfstring{Scattering calculus $\Psi_{\mathrm{sc}}$}{Scattering calculus}} (Cf.\ \cite{MelroseSC}\cite{VasyGrenoble})
\label{subsec:sc}
We denote the radial compactification of Minkowski space $\bbR^{1,d}$ by $\bbM$, and the radial compactification of the dual space $(\bbR^{1,d})^*$ by $\bbM^*$. The scattering, or Parenti--Shubin, calculus $\Psi_{\mathrm{sc}}$ on Minkowski space $\bbR^{1,d}$ takes place on the compactified phase space  \footnote{This is compactifying the base and the fibres respectively, hence has different smooth structure to the radial compactification of $T^* \bbR^{1,d}$ as a whole. }
\begin{equation} 
	{}^{\mathrm{sc}}\overline{T}^*\bbM \overset{\mathrm{def}}{=} \bbM \times \bbM^*.
\end{equation} 
Coordinates on the interior of the phase space are $z = (t,x)$ and $\zeta = (\tau, \xi)$, where $t \in \bbR$, $x \in \bbR^d$ and $(\tau, \xi)$ are the dual coordinates. 

For $m,s \in \bbR$, the class of ``scattering symbols'' $S_{\mathrm{sc}}^{m,s}(\bbM)$ with differential order $m$ and spacetime order $s$,  
is defined to be the set of smooth functions $a(z,\zeta)$ on $\bbR^{1,d}_z\times (\bbR^{1,d})^*_\zeta$ such that there exist constants $C_{\alpha \beta}$ such that
\begin{equation} 
	\label{eq: sc symbol}
	|\partial_z^\alpha\partial_{\zeta}^\beta a(z,\zeta)| \leq C_{\alpha\beta} \la z \ra^{ s -|\alpha|} \la \zeta \ra^{m-|\beta|}
\end{equation}
for all multi-indices $\alpha,\beta\in \bbN^{1+d}$.
The topology on $S_{\mathrm{sc}}^{m,s}(\bbM)$ is the Fr\'echet topology induced by the countably many seminorms determined by \eqref{eq: sc symbol}.
When $s$ is replaced by $\mathsf{s} \in C^\infty({}^{\mathrm{sc}}\overline{T}^*\bbM)$, we have a definition similar to \cref{eq: sc symbol} but with extra, minor logarithmic losses in $\la z \ra$ introduced by differentiating $\mathsf{s}$ in the exponent; we denote this symbol class by $S_{\mathrm{sc}}^{m,\mathsf{s}}(\bbM)$.

Then the space of scattering pseudodifferential operators of order $(m,\mathsf{s})$, which is denoted by $\Psi_{\mathrm{sc}}^{m,\mathsf{s}}(\bbM)$, consists of operators acting by the standard (left) quantization 
\begin{equation}
	\operatorname{Op}(a) f(z)  = \frac{1}{(2\pi)^{1+d}} \int_{\bbR^{1+d}} \int_{\bbR^{1+d}}   e^{i \zeta\cdot (z-z')} a(z,\zeta) f(z') \dd^{1+d} z' \dd^{1+d} \zeta,
	\label{eq:quant_sc}
\end{equation}
with $a \in S_{\mathrm{sc}}^{m,\mathsf{s}}(\bbM)$. 

The map between (left-reduced) symbols $a \in S_{\mathrm{sc}}^{m,\mathsf{s}}(\bbM)$ and operators $A = \operatorname{Op}(a) \in \Psi_{\mathrm{sc}}^{m,\mathsf{s}}(\bbM)$ is bijective; we write $a = \sigma(A)$ to denote the full (left) symbol of $A$. We can thus define the principal symbol $\sigma_{\mathrm{sc}}^{m,\mathsf{s}}(A)$ of $A$ to be the equivalence class 
\begin{equation}
	\sigma_{\mathrm{sc}}^{m,\mathsf{s}}(A) = a \bmod \cap_{\delta > 0}S_{\mathrm{sc}}^{m-1,\mathsf{s}-1+\delta}(\bbM)
\end{equation}
of 
$a = \sigma(A)$ in $S_{\mathrm{sc}}^{m,\mathsf{s}}(\bbM) / \cap_{\delta > 0} S_{\mathrm{sc}}^{m-1,\mathsf{s}-1+\delta}(\bbM)$. Here, the intersection over $\delta > 0$ is required for variable orders due to the incurred logarithmic loss relative to  \eqref{eq: sc symbol}; we can take $\delta = 0$ for constant orders. 

We will often work with a smaller class of operators, which are quantizations of classical symbols. Classical symbols of order $(m,\mathsf{s})$, denoted by
\begin{equation} 
	S_{\mathrm{sc, cl}}^{m,\mathsf{s}}(\bbM) = \ang{\zeta}^m \ang{z}^\mathsf{s} C^\infty({}^{\mathrm{sc}}\overline{T}^*\bbM) \subset S_{\mathrm{sc}}^{m,\mathsf{s}}(\bbM),
\end{equation} 
are those of the form $\ang{\zeta}^m \ang{z}^\mathsf{s}$ times a smooth function on the compactified phase space $\bbM\times \bbM^*$; it can readily be checked that all such functions satisfy the required estimates. 
If $a$ is classical, then $A=\operatorname{Op}(a)$ is called classical. 
The principal symbol of a classical operator $A$ of order $(m,\mathsf{s})$ can then be identified with the restriction of 
\begin{equation}
	\ang{\zeta}^{-m} \ang{z}^{-\mathsf{s}} a_0, \qquad a_0 \in \sigma_{\mathrm{sc}}^{m,\mathsf{s}}(A)
\end{equation}
to the boundary of ${}^{\mathrm{sc}}\overline{T}^*\bbM$, which exists by assumption. 

Similarly, given a classical operator $P$  with \emph{constant} orders $(m,s)$,   its Hamiltonian vector field $H_p$, multiplied by $\ang{\zeta}^{-m+1} \ang{z}^{-s+1}$, is a smooth vector field 
\begin{equation} 
	{}^{\mathrm{sc}}\mathsf{H}_p = \ang{\zeta}^{-m+1} \ang{z}^{-s+1} H_p
\end{equation} 
on ${}^{\mathrm{sc}}\overline{T}^*\bbM$ which is tangent to the boundary. It therefore defines a flow on each boundary hypersurface of phase space. As observed by Melrose \cite{MelroseSC}, microlocal propagation for solutions of $Pu = f$ takes place on the whole boundary of the compactified phase space, with respect to this rescaled Hamilton vector field ${}^{\mathrm{sc}}\mathsf{H}_p$. 
This includes not just fiber-infinity over $\bbR^{1,d}$, the usual stage for microlocal propagation (as in H\"ormander's basic theory), but also the fiber (the \emph{whole} fiber) over points at spacetime infinity. It is important to understand that \emph{all} frequencies, not just high frequencies, are involved at spacetime infinity. This is a consequence of the stronger estimates \eqref{eq: sc symbol} obeyed by spacetime derivatives of symbols in the scattering calculus, where each spacetime derivative produces additional decay.

The usual properties of pseudodifferential calculus are valid:
\begin{itemize}
	\item The composition of $A \in \Psi_{\mathrm{sc}}^{m,\mathsf{s}}$ with $B \in \Psi_{\mathrm{sc}}^{m',\mathsf{s}'}$ is an operator $C \in \Psi_{\mathrm{sc}}^{m+m',\mathsf{s}+\mathsf{s}'}$. Moreover, the symbol of $C$ satisfies an asymptotic expansion 
	\begin{equation}
		c(z,\zeta) \sim  \sum_{\alpha\in \bbN^D} \frac{i^{|\alpha|}}{ \alpha! } \big(D_\zeta^\alpha a(z,\zeta) \big) D_{z}^{\alpha} b(z,\zeta)  
		\label{eq:moyal_explicit_sc}
	\end{equation}
	(the \emph{Moyal} expansion)
	in which subsequent terms decrease in order both at fiber-infinity and at spacetime infinity. In this way, it is an asymptotic expansion that holds over the whole of the boundary of compactified phase space -- not just at fiber-infinity. 
	
	\item The Moyal expansion, \cref{eq:moyal_explicit_sc}, directly implies that the principal symbol of $AB$ is the product of the principal symbols of $A$ and $B$, and the principal symbol of the commutator $i[A, B]$ is the Poisson bracket of the principal symbols of $A$ and $B$. 
	\item An operator $A \in \Psi_{\mathrm{sc}}^{m,\mathsf{s}}$ is said to be (totally) elliptic if its symbol $a = \sigma_{\mathrm{sc}}^{m,\mathsf{s}}(A)$ satisfies 
	\begin{equation} 
		|a| \geq C^{-1} \ang{\zeta}^m \ang{z}^\mathsf{s}
	\end{equation} 
	whenever either $\ang{\zeta} \geq C$ or $\ang{z}\geq C$, for sufficiently large $C$. Then if $A$ is elliptic, there is an elliptic parametrix $B \in \Psi_{\mathrm{sc}}^{-m,-\mathsf{s}}$, i.e. an operator $B$ such that $AB - 1 \in \Psi_{\mathrm{sc}}^{-\infty,-\infty}$. In particular an elliptic operator is Fredholm (as a consequence of the compact embedding in \cref{eq:sc op compact} below).
    
	\item We define standard weighted Sobolev spaces on $\bbR^{1,d}$, for constant $l, r \in \bbR$, by 
	\begin{equation}
		H^{l, r}(\bbR^{1,d}) \overset{\mathrm{def}}{=} \ang{z}^{-r} \ang{D}^{-l} L^2(\bbR^{1,d}).
	\end{equation}
	We also define Sobolev spaces with variable spacetime order as follows: given orders $(l, \mathsf{r})$, with $l$ constant,  we choose a constant $R < \inf \mathsf{r}$ and choose an elliptic $A \in  \Psi_{\mathrm{sc}}^{l,\mathsf{r}}$. We then define 
	\begin{equation} 
		H^{l, \mathsf{r}}(\bbR^{1,d}) \overset{\mathrm{def}}{=} \{ u \in H^{l, R}(\bbR^{1,d}) : Au \in L^2 \},
	\end{equation} 
	with squared norm $\| u \|^2_{H^{l, R}} + \| Au \|_{L^2}^2$. 
	It is straightforward to check that a different choice of $R$ or $A$ leads to an equivalent norm. 
	
	Then operators in $\Psi_{\mathrm{sc}}^{m,\mathsf{s}}$ act boundedly between such spaces: if $A \in \Psi_{\mathrm{sc}}^{m,\mathsf{s}}$ then 
	\begin{equation}\label{eq:sc op bounded}
		A : H^{l,\mathsf{r}}(\bbR^{1,d}) \to H^{l-m,\mathsf{r}-\mathsf{s}}(\bbR^{1,d}) 
	\end{equation}
	boundedly.
	Moreover, if $\mathsf{r}' < \mathsf{r}-\mathsf{s}$ and $l' < l-m$ then 
	\begin{equation}\label{eq:sc op compact}
		A : H^{l,\mathsf{r}}(\bbR^{1,d}) \to H^{l',\mathsf{r}'}(\bbR^{1,d}) 
	\end{equation}
	is compact.
	
\item The $\mathrm{sc}$-operator wavefront set \begin{equation} 
	\operatorname{WF}'_{\mathrm{sc}}(A) \subset \partial ({}^{\mathrm{sc}} \overline{T}^*\bbM)
\end{equation} 
of $A=\operatorname{Op}(a)$ is defined to be the essential support of $a$, that is, the closed subset of $\partial ({}^{\mathrm{sc}} \overline{T}^*\bbM)$ whose complement is 
\begin{multline}
\{ p \in \partial ({}^{\mathrm{sc}} \overline{T}^*\bbM) : \exists \text{ neighbourhood $U$ of $p$ in } {}^{\mathrm{sc}} \overline{T}^*\bbM \\
\text{ such that } \forall \, N \in \mathbb{N}, \  \exists \, C_N \text{ with }  |\la z \ra^{N} \la \zeta \ra^N a| \leq C_N \text{ in } U \}.
\end{multline}
Equivalently, the essential support is the complement of the set of points near which $a$ is Schwartz.
\item 
The $(m,\mathsf{s})$-order $\mathrm{sc}$-wavefront set of a tempered distribution $u$ is the closed set $\operatorname{WF}_{\mathrm{sc}}^{m,\mathsf{s}}(u) \subset \partial ({}^{\mathrm{sc}} \overline{T}^*\bbM)$ defined by
\begin{equation}
	p \notin \operatorname{WF}_{\mathrm{sc}}^{m,\mathsf{s}}(u) \iff \exists A \in \Psi_{\mathrm{sc}}^{m,\mathsf{s}} \text{ that is elliptic at } p \text{ such that } Au \in L^2(\bbM).
\end{equation}
Then we set $\operatorname{WF}_{\mathrm{sc}}(u)$ to be the closure of the union of these sets over $(m, s) \in \bbN \times \bbN$:

\begin{equation}
    \operatorname{WF}_{\mathrm{sc}}(u) = \overline{\bigcup_{m,s \in \bbN} \operatorname{WF}_{\mathrm{sc}}^{m,s}(u)}.
\end{equation}
     
\end{itemize}

\subsection{\texorpdfstring{Parabolic calculus $\Psi_{\mathrm{par}}$}{Parabolic calculus}}\label{subsec:par}
The parabolic calculus on Minkowski space $\bbR^{1,d}$, denoted $\Psi_{\mathrm{par}}$, is a variant of the scattering calculus  in which we treat time and space on unequal footing. More precisely, the differential operator $\partial_t$ is \emph{second} order in the calculus, while, as usual, spatial derivatives $\partial_{x_j}$ are first order. This ensures that the time-derivative in the Schr\"odinger operator $\pm i \partial_t + \triangle$ is principal. This is required for the operator to be of real principal type (away from radial sets), so that it can be analyzed using microlocal propagation estimates. Anisotropic\footnote{Note that this usage of the term `anisotropic' has little to do with the sense in which Sobolev spaces with variable decay order are anisotropic.} calculi, more general than the particular parabolic anisotropy we consider here, were introduced on local patches of $\bbR^{1,d}$ by Lascar \cite{Lascar}. The global compactification of phase space considered here was introduced by Gell-Redman--Gomes--Hassell in \cite{Parabolicsc}.

The (compactified) parabolic phase space is defined by
\begin{equation} 
	{}^{\mathrm{par}}\overline{T}^* \bbM = \bbM \times \overline{(\bbR^{1,d}_{\tau,\xi  })}_{\mathrm{par}},
\end{equation} 
where $\smash{\overline{(\bbR^{1,d}_{\tau,\xi  })}_{\mathrm{par}}}$
is the parabolic compactification of $\bbR^{1,d}$. 
So, 
\begin{equation} 
	\overline{(\bbR^{1,d}_{\tau,\xi  })}_{\mathrm{par}} \cong \underbrace{\overline{\bbR^{1,d}_{\tau,\xi}}}_{\text{radial compactification}}
\end{equation} 
on the topological level --- they are both just $(1+d)$-dimensional balls --- but the canonical identification of their interiors does not extend smoothly, or even continuously, to the boundary. Indeed, the boundary of the mwc on the right can be identified with the (linear) rays through the origin, while the boundary of the mwc on the left, with the parabolic compactification, can be identified with the ``parabolic'' rays   
\begin{equation} 
	\{ (\tau_0 s^2, v s) :s\in \bbR^+  \} ,\quad \tau_0\in \bbR,v\in \bbR^d, \quad \lVert v \rVert \in \{0,1\}
\end{equation} 
where $\tau_0,v$ are not both zero.
An explicit atlas is:
\begin{equation} 
	\overline{(\bbR^{1,d}_{\tau,\xi  })}_{\mathrm{par}} =  \bbR^{1,d}_{\tau,\xi }  \cup \Big( \bigcup_\pm  (\overline{(\bbR^{1,d}_{\tau,\xi  })}_{\mathrm{par}} \cap \{\pm \tau>0\}) \Big) \cup \Big( \bigcup_{\pm, k\in \{1,\ldots,d\}} (\overline{(\bbR^{1,d}_{\tau,\xi  })}_{\mathrm{par}} \cap \{\pm \xi_k>0\}) \Big)\;\,
	\label{eq:misc_005}
\end{equation} 
where, in the sense of equality of compactifications,  
\begin{align}
	&\overline{(\bbR^{1,d}_{\tau,\xi  })}_{\mathrm{par}} \cap \{\pm\tau>0\} = [0,\infty)_{1/(\pm \tau)^{1/2}} \times \bbR^d_{\xi/(\pm\tau)^{1/2}}, \label{eq:misc_006} \\
	&\overline{(\bbR^{1,d}_{\tau,\xi  })}_{\mathrm{par}} \cap \{\pm \xi_k>0\} =  [0,\infty)_{\pm 1/\xi_k} \times \bbR_{\tau/\xi_k^2}\times  \bbR^{d-1}_{\hat{\xi}_k}, \label{eq:misc_007}
\end{align}
where $\hat{\xi}_k$ is the $(d-1)$-tuple $\{\xi_j/\xi_k\}_{j\neq k}$.

\begin{figure}[h!]
	\begin{tikzpicture}[scale=.85]
	\filldraw[fill=lightgray!20] (0,0) circle (2.5);
	\draw[dashed, lightgray] (0,1) ellipse (64pt and 3pt);
	\draw[dashed, lightgray] (0,2) ellipse (43pt and 3pt);
	\draw[dashed, lightgray] (0,-1) ellipse (64pt and 3pt);
	\draw[dashed, lightgray] (0,-2) ellipse (43pt and 3pt);
	\draw[darkred, ->] (-2.4,0) to[out=90, in=250] (-2.28,.7) node[left] {$\tau/\xi^2_1$\;};
	\draw[darkred, ->] (-2.4,0) -- (-1.7,0) node[below] {$|\xi_1|^{-1}$ };
	\node () at (0,0) {$\overline{(\bbR^{1,2}_{\tau,\xi  })}_{\mathrm{par}}$};
	\draw[darkred, ->] (0,2.4) -- (0,1.7) node[left] {$1/\tau^{1/2}$};
	\draw[darkred, ->] (0,2.4) to[out=0, in=160] (.8,2.25) node[above right] {$\xi_1/\tau^{1/2}$};
	\end{tikzpicture}
	\caption{ The parabolic compactification of $\bbR^{1,2}_{\tau,\xi}$. The longitudinal coordinate is not shown.}
\end{figure}
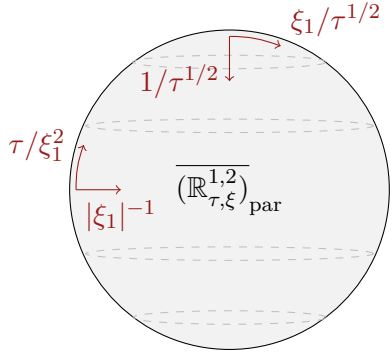

\begin{figure}[t]
	\begin{tikzpicture}[scale=.75]
		\draw[dashed] (-3,-3) rectangle (3,3.25);
		\node () at (-2.5,2.75) {(a)};
		\filldraw[fill=lightgray!20]
		(0,0) circle (2.5);
		\begin{scope}
			\clip (0,0) circle (2.52);
			\fill[fill=white] (2.5,0) circle (.75);
			\fill[fill=white] (-2.5,0) circle (.75);
			\draw[dashed, darkred] (0,0) -- (-2.5,0);
			\draw[dashed, darkred] (0,0) to[out=145, in=45] (-2.5,0);
			\draw[dashed, darkred] (0,0) to[out=-145, in=-45] (-2.5,0);
			\draw[dashed, darkred] (0,0) to[out=0, in=225] (1.76,1.76);
			\draw[dashed, darkred] (0,0) to[out=0, in=245] (1,2.3);
			\draw[dashed, darkred] (0,0) to[out=0, in=200] (2.15,1.1);
		\end{scope}
		\begin{scope}
			\clip (0,0) circle (2.5);
			\filldraw[fill=white] (2.5,0) circle (.75);
			\filldraw[fill=white] (-2.5,0) circle (.75);
		\end{scope}
		\begin{scope} 
			\clip (-3,-3) rectangle (3,3.25);
		\fill[fill=white] (2.5,0) circle (.74);
		\fill[fill=white] (-2.5,0) circle (.74);
		\end{scope} 
		\node () at (-2.35, 0) {$\mathrm{radf}$};
		\node () at (-2.1, -2.3) {$\mathrm{parf}_-$};
		\node () at (2.15, 2.35) {$\mathrm{parf}_+$};
		\node () at (0, -1) {$X$};
	\end{tikzpicture}
	\begin{tikzpicture}[scale=.75]
		\draw[dashed] (-3,-3) rectangle (5,3.25);
		\node () at (-2.5,2.75) {(b)};
		\begin{scope}[scale=.5, shift={(-2.5,0)}]
			\filldraw[fill=lightgray!20]
			(0,0) circle (2.5);
			\begin{scope}
				\clip (0,0) circle (2.55);
				\fill[fill=white] (2.5,0) circle (.75);
				\fill[fill=white] (-2.5,0) circle (.75);
				\draw[dashed, darkred] (0,0) -- (-2.5,0);
				\draw[dashed, darkred] (0,0) to[out=145, in=45] (-2.5,0);
				\draw[dashed, darkred] (0,0) to[out=-145, in=-45] (-2.5,0);
				\draw[dashed, darkred] (0,0) to[out=0, in=225] (1.76,1.76);
				\draw[dashed, darkred] (0,0) to[out=0, in=245] (1,2.3);
				\draw[dashed, darkred] (0,0) to[out=0, in=200] (2.15,1.1);
			\end{scope}
			\begin{scope}
				\clip (0,0) circle (2.5);
				\filldraw[fill=white] (2.5,0) circle (.75);
				\filldraw[fill=white] (-2.5,0) circle (.75);
			\end{scope}
			\begin{scope} 
			\clip (-3,-3) rectangle (3,3.25);
			\fill[fill=white] (2.5,0) circle (.74);
			\fill[fill=white] (-2.5,0) circle (.74);
			\end{scope} 
			\node () at (.5,-1) {$X$};
		\end{scope}
		\begin{scope}[scale=.5, shift={(5.75,3)}]
			\filldraw[fill=lightgray!20]
			(0,0) circle (2.5);
			\node () at (-3.5,2) {$\overline{\bbR^{1,d}_{\tau,\xi}}$};
			\begin{scope}
				\clip (0,0) circle (2.5);
				\draw[dashed, darkred] (0,0) -- (-2.5,0);
				\draw[dashed, darkred] (0,0) -- (-2.5,1);
				\draw[dashed, darkred] (0,0) -- (-2.5,-1);
				\draw[dashed, darkred] (0,0) to[out=0, in=-90] (0,2.5);
				\draw[dashed, darkred] (0,0) to[out=0, in=-45] (.9,1.5) to[out=135,in=-60] (0,2.5);
				\draw[dashed, darkred] (0,0) to[out=0, in=-45] (1.6,1.5) to[out=135,in=-20] (0,2.5);
			\end{scope} 
			\fill[black] (0,2.5) circle (4pt);
			\fill[black] (0,-2.5) circle (4pt);
		\end{scope}
		\begin{scope}[scale=.5, shift={(5.75,-3)}]
			\filldraw[fill=lightgray!20]
			(0,0) circle (2.5);
			\node () at (-4,-1.5) {$\overline{(\bbR^{1,d}_{\tau,\xi})}_{\mathrm{par}}$};
			\fill[black] (2.5,0) circle (4pt);
			\fill[black] (-2.5,0) circle (4pt);
			\clip (0,0) circle (2.5);
			\draw[dashed, darkred] (0,0) -- (-2.5,0);
			\draw[dashed, darkred] (0,0) to[in=45, out=135] (-2.5,0);
			\draw[dashed, darkred] (0,0) to[out=225,in=-45] (-2.5,0);
			\draw[dashed, darkred] (0,0) to[out=0, in=225] (1.76,1.76);
			\draw[dashed, darkred] (0,0) to[out=0, in=245] (1,2.3);
			\draw[dashed, darkred] (0,0) to[out=0, in=200] (2.15,1.1);
		\end{scope}
		\draw[->] (0,1) -- (1.25,1.5);
		\draw[->] (0,-1) -- (1.25,-1.5);
	\end{tikzpicture}
	\caption{(a) The blown up mwc $X$, with some linear rays (hitting radf) and some parabolic rays (hitting $\mathrm{parf}_+$) depicted. (b) The recovery of the radial and parabolic compactifications from blowing down either radf or $\mathrm{parf}_\pm$. }
    \label{fig:double_blowup}
\end{figure}
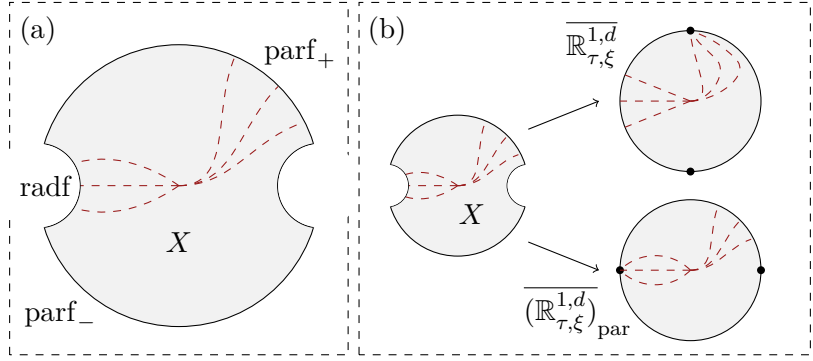

We can observe that $\ang{z}^{-1}$ is a boundary defining function for spacetime infinity, and $\aang{\zeta}^{-1}$ is a boundary defining function for fiber-infinity, where 
\begin{equation} 
	\aang{\zeta} = (1 + \tau^2 + |\xi|^4)^{1/4}, \quad \zeta = (\tau, \xi).
\end{equation} 

One way to comprehend the relation between the radial and parabolic compactifications is to begin with a compactification that ``refines'' both: 
\begin{equation}
	X=[\overline{(\bbR^{1,d}_{\tau,\xi  })}_{\mathrm{par}}  ;\partial \overline{(\bbR^{1,d}_{\tau,\xi  })}_{\mathrm{par}} \cap \mathrm{cl}_{\overline{(\bbR^{1,d}_{\tau,\xi  })}_{\mathrm{par}}}\{\tau=0\}]. 
\end{equation}
That is, take the parabolic compactification and then blow up the equator, resulting in a new boundary hypersurface, $\mathrm{radf}$, separating the original boundary hypersurface into two components $\mathrm{parf}_\pm$. 
Alternatively, the same space can also be constructed from the radial compactification by blowing up the north/south pole parabolically --- see \Cref{fig:double_blowup}. 
This indicates that the radial compactification can be reconstructed from $X$ by \emph{blowing down} the two faces $\mathrm{parf}_\pm$; so, $\mathrm{radf}^\circ$ is identifiable with the boundary of the radial compactification minus the north/south poles. On the other hand, the parabolic compactification is recovered by blowing back down $\mathrm{radf}$.

For $m,s\in \bbR$, the space of parabolic symbols with differential order $m$ and spacetime order $s$, denoted  $S_{\mathrm{par}}^{m,s}(\bbM)$, is the set of smooth functions $a(z, \zeta)$ on $\bbR^{1,d}_z\times (\bbR^{1,d})^*_\zeta$ satisfying 
\begin{equation} \label{eq: par symbol}
	|\partial_z^\alpha \partial_\tau^k \partial_{\xi}^\beta a(z,\zeta)| \leq C_{\alpha k \beta} \ang{z}^{s-|\alpha| } \aang{\zeta}^{m-2k-|\beta|}. 
\end{equation}
The best constants in \cref{eq: par symbol} define a sequence of seminorms that give a Fr\'echet space structure to $S_{\mathrm{par}}^{m,\mathsf{s}}(\bbM)$. 

When $s$ is replaced by $\mathsf{s}\in C^\infty({}^{\mathrm{par}}\overline{T}^* \bbM)$\footnote{Of course, this is not the same as being a smooth function on ${}^{\mathrm{sc}}\overline{T}^* \bbM$.}, the symbol class $S_{\mathrm{par}}^{m,\mathsf{s}}(\bbM)$ is defined in almost the same way except for the usual logarithmic loss in $\la z \ra$.

Then, the space of parabolic pseudodifferential operators of order $(m,\mathsf{s})$, which is denoted by $\Psi_{\mathrm{par}}^{m,\mathsf{s}}(\bbM)$, consists of operators acting by the standard (left) quantization 
\begin{equation}
	\operatorname{Op}(a) f(z)  = \frac{1}{(2\pi)^{1+d}} \int_{\bbR^{1+d}} \int_{\bbR^{1+d}}   e^{i \zeta\cdot (z-z')} a(z,\zeta) f(z') \dd^{1+d} z' \dd^{1+d} \zeta,
	\label{eq:quant, par}
\end{equation}
where $a \in S_{\mathrm{par}}^{m,\mathsf{s}}(\bbM)$. 
Note that this is the \emph{same} formula as \cref{eq:quant_sc}.

The map between (left-reduced) symbols $a \in  S_{\mathrm{par}}^{m,\mathsf{s}}(\bbM)$ and operators $A = \operatorname{Op}(a) \in \Psi_{\mathrm{par}}^{m,\mathsf{s}}(\bbM)$ is bijective; we write $a = \sigma(A)$. We can thus define the principal symbol of $A$ to be the equivalence class 
\begin{equation}
	\sigma_{\mathrm{par}}^{m,\mathsf{s}}(A) = a \bmod \cap_{\delta > 0} S_{\mathrm{par}}^{m-1,\mathsf{s}-1+\delta}(\bbM)
\end{equation}
of $a = \sigma(A)$ in $S_{\mathrm{par}}^{m,\mathsf{s}}(\bbM) / \cap_{\delta > 0} S_{\mathrm{par}}^{m-1,\mathsf{s}-1+\delta}(\bbM)$.

The parabolic calculus works similarly to the scattering calculus. Like the scattering calculus, microlocal propagation takes place on the whole of the boundary of the compactified phase space. (Indeed, away from fiber-infinity, the parabolic calculus is identical to the scattering calculus.) 

We have analogous properties to those listed in the previous subsection:
\begin{itemize}
	\item The composition of $A \in \Psi_{\mathrm{par}}^{m,\mathsf{s}}$ with $B \in \Psi_{\mathrm{par}}^{m',\mathsf{s}'}$ is an operator $C \in \Psi_{\mathrm{par}}^{m+m',\mathsf{s}+\mathsf{s}'}$. Moreover, the symbol of $C$ satisfies an asymptotic expansion 
	\begin{equation}
		c(z,\zeta) \sim  \sum_{\alpha\in \bbN^D} \frac{i^{|\alpha|}}{ \alpha! } \big(D_\zeta^\alpha a(z,\zeta) \big) D_{z}^{\alpha} b(z,\zeta)  
		\label{eq:moyal_explicit_par_0}
	\end{equation}
	in which subsequent terms decrease in order both at fiber-infinity and at spacetime infinity. In this way, it is an asymptotic expansion that holds over the whole of the boundary of compactified phase space. 
	
	\item The principal symbol of $AB$ is the product of the principal symbols of $A$ and $B$, and the principal symbol of the commutator $i[A, B]$ is the Poisson bracket of the principal symbols of $A$ and $B$.

	We remark that at fiber-infinity -- but not at spacetime infinity -- the term $\partial_t a \partial_\tau b - \partial_\tau a \partial_t b$ in the Moyal expansion \cref{eq:moyal_explicit_par_0} can be (but does not need to be) discarded from the principal symbol of $i[A, B]$ as it is lower order.
	
	\item Define  weighted parabolic Sobolev spaces on $\bbR^{1,d}$ with constant orders by 
	\begin{equation}
		H^{l, r}_{\mathrm{par}}(\bbR^{1,d}) \overset{\mathrm{def}}{=} \ang{z}^{-r} \aang{D}^{-l} L^2(\bbR^{1,d}). 
	\end{equation}
	Here $\aang{D}$ is the Fourier multiplier with symbol $\aang{\zeta} = (1 + \tau^2 + |\xi|^4)^{1/4}$. 
	The corresponding norm is just defined by
    \begin{equation} \label{eq:def-par-norm-const-order}
    \lVert u \rVert_{H^{l, r}_{\mathrm{par}}(\bbR^{1,d})}
    = \lVert   \ang{z}^{r} \aang{D}^{l} u \rVert_{L^2(\bbR^{1,d})}.
    \end{equation}
	
	The weighted parabolic Sobolev space with variable spacetime order $(l, \mathsf{r})$ is defined in a similar way as above: we choose a constant $R < \inf \mathsf{r}$ and choose an elliptic $A \in  \Psi_{\mathrm{par}}^{l,\mathsf{r}}$. We then define 
	\begin{equation} 
		H^{l, \mathsf{r}}_{\mathrm{par}}(\bbR^{1,d}) \overset{\mathrm{def}}{=} \{ u \in H^{l, R}_{\mathrm{par}}(\bbR^{1,d}) : Au \in L^2 \},
	\end{equation} 
	with squared norm
\begin{equation} \label{eq:par-squared-norm}
    \| u \|^2_{H_{\mathrm{par}}^{l, R}} + \| Au \|_{L^2}^2.
\end{equation}

	Then operators in $\Psi_{\mathrm{par}}^{m,\mathsf{s}}$ act boundedly between such spaces: if $A \in \Psi_{\mathrm{par}}^{m,\mathsf{s}}$ then 
	\begin{equation}\label{eq:par op bounded}
		A : H^{l,\mathsf{r}}_{\mathrm{par}}(\bbR^{1,d}) \to H^{l-m,\mathsf{r}-\mathsf{s}}_{\mathrm{par}}(\bbR^{1,d}) 
	\end{equation}
	is bounded.
	Moreover, if $l' < l-m$ and $\mathsf{r}' < \mathsf{r}-\mathsf{s}$ then 
	\begin{equation}\label{eq:par op compact}
		A : H^{l,\mathsf{r}}_{\mathrm{par}}(\bbR^{1,d}) \to H^{l',\mathsf{r}'}_{\mathrm{par}}(\bbR^{1,d}) 
	\end{equation}
	is compact.
	
	\item An operator $A \in \Psi_{\mathrm{par}}^{m,\mathsf{s}}$ is said to be (totally) elliptic if its symbol $a = \sigma_{\mathrm{par}}^{m,\mathsf{s}}(A)$ satisfies 
	\begin{equation} 
		|a| \geq C^{-1} \aang{\zeta}^m \ang{z}^\mathsf{s}
	\end{equation} 
	whenever either $\aang{\zeta} \geq C$ or $\ang{z}\geq C$, for sufficiently large $C$. Then if $A$ is elliptic, there is an elliptic parametrix $B \in \Psi_{\mathrm{par}}^{-m,-\mathsf{s}}$, i.e. an operator $B$ such that $AB - 1 \in \Psi_{\mathrm{par}}^{-\infty,-\infty}$. In particular an elliptic operator is Fredholm mapping as in \eqref{eq:par op bounded} (as a consequence of \eqref{eq:par op compact}). 

    \item The $\mathrm{par}$-operator wavefront set \begin{equation} 
	\operatorname{WF}'_{\mathrm{par}}(A) \subset \partial ({}^{\mathrm{par}} \overline{T}^*\bbM)
\end{equation} 
of $A=\operatorname{Op}(a)$ is defined to be the essential support of $a\in S_{\mathrm{par}}^{m,\mathsf{s}}$, that is, the closed subset of $\partial ({}^{\mathrm{sc}} \overline{T}^*\bbM)$ whose complement is 
\begin{multline}
\{ p \in \partial ({}^{\mathrm{par}} \overline{T}^*\bbM) : \exists \text{ neighbourhood $U$ of $p$ in } {}^{\mathrm{par}} \overline{T}^*\bbM \\
\text{ such that } \forall \, N \in \mathbb{N}, \  \exists \, C_N \text{ with }  |\la z \ra^{N} \aang{\zeta}^N a| \leq C_N \text{ in } U \}.
\end{multline}

\item 
The $(m,\mathsf{s})$-order $\mathrm{par}$-wavefront set of a tempered distribution $u$ is the closed set $\operatorname{WF}_{\mathrm{par}}^{m,\mathsf{s}}(u) \subset \partial ({}^{\mathrm{par}} \overline{T}^*\bbM)$ defined by
\begin{equation}
	p \notin \operatorname{WF}_{\mathrm{par}}^{m,\mathsf{s}}(u) \iff \exists A \in \Psi_{\mathrm{par}}^{m,\mathsf{s}} \text{ that is elliptic at } p \text{ such that } Au \in L^2(\bbM).
\end{equation}
Then we set $\operatorname{WF}_{\mathrm{par}}(u)$ to be 
\begin{equation}
    \operatorname{WF}_{\mathrm{par}}(u) = \overline{\bigcup_{m,s \in \bbN} \operatorname{WF}_{\mathrm{par}}^{m,s}(u)}.
\end{equation}

\end{itemize}

\subsection{\texorpdfstring{Parametrized parabolic calculus  $\Psi_{\mathrm{par, I}}$}{Parametrized parabolic calculus}}
\label{subsec:parI}

The phase space that will be used in our analysis involves a smooth parameter $h \in [0, 1)$. The phase space is thus 
\begin{equation} 
	{}^{\mathrm{par,I}}\overline{T}^* \bbM = {}^{\mathrm{par}}\overline{T}^* \bbM \times [0,1)_h.
\end{equation}  
The `I' superscript in ${}^{\mathrm{par,I}}\overline{T}^* \bbM$ stands for the interval $[0,1)_h$. 

For $m \in \bbR, \mathsf{s} \in C^\infty({}^{\mathrm{par,I}}\overline{T}^* \bbM)$, denoting the restriction of $\mathsf{s}$ at fixed $h$ by $\mathsf{s}(h)$.
The symbol class $S_{\mathrm{par,I}}^{m,\mathsf{s},0}(\bbM)$ is the space of $C^\infty$ functions of $h \in [0, 1)$ with values in $S_{\mathrm{par}}^{m,\mathsf{s}}(\bbM)$.\footnote{When working with the par,I-calculus, the variable order $\mathsf{s}$ will always be independent of $h$.} More generally, we define 
\begin{equation} 
	S_{\mathrm{par,I}}^{m,\mathsf{s},q}(\bbM) = h^{-q} S_{\mathrm{par,I}}^{m,\mathsf{s},0}(\bbM).
\end{equation}  
When $q=0$, the corresponding operators are just smooth functions of $h$ valued in $\Psi_{\mathrm{par}}^{m,\mathsf{s}}(\bbM)$.

The principal symbol of $A = \operatorname{Op}(a)$, $a \in S_{\mathrm{par,I}}^{m,\mathsf{s},q}(\bbM)$, is its image 
\begin{equation}
	\sigma_{\mathrm{par,I}}^{m,\mathsf{s},q}(A) = a \bmod  \cap_{\delta > 0} S_{\mathrm{par,I}}^{m-1,\mathsf{s}-1+\delta,q}(\bbM)
\end{equation}
in $S_{\mathrm{par,I}}^{m,\mathsf{s},q}(\bbM) / \cap_{\delta > 0} S_{\mathrm{par,I}}^{m-1,\mathsf{s}-1+\delta,q}(\bbM)$.
We also define the normal operator of $A$, $N(A) \in \Psi_{\mathrm{par}}^{m,\mathsf{s}}$,  to be the operator $\lim_{h \to 0} h^q A(h)$, which exists by the definition of the symbol class --- it is the quantization of 
\begin{equation} 
	\lim_{h \to 0^+} h^q a(z, \zeta, h) \in S_{\mathrm{par}}^{m,\mathsf{s}}.
\end{equation}

Notice that this calculus is \emph{not} ``fully symbolic'' in the sense of the discussion at the beginning of this section: the restriction to $h=0$ is an operator, not a function, and while the map sending pseudodifferential operators to their normal operators is a homomorphism of algebras, it takes values in a noncommutative algebra.

Similar to other calculi above:
\begin{itemize}
	\item The composition of $A \in \Psi_{\mathrm{par,I}}^{m,\mathsf{s},q}$ with $B \in \Psi_{\mathrm{par,I}}^{m',\mathsf{s}',q'}$ is an operator $C \in \Psi_{\mathrm{par,I}}^{m+m',\mathsf{s}+\mathsf{s}',q+q'}$. Moreover, the symbol of $C$ satisfies an asymptotic expansion 
	\begin{equation}
		c(z,\zeta,h) \sim  \sum_{\alpha\in \bbN^D} \frac{i^{|\alpha|}}{ \alpha! } \big(D_\zeta^\alpha a(z,\zeta,h) \big) D_{z}^{\alpha} b(z,\zeta,h)  
		\label{eq:moyal_explicit_parI}
	\end{equation}
	in which subsequent terms decrease in order both at fiber-infinity and at spacetime infinity, but \emph{not} at $h=0$.

	\item The principal symbol of $AB$ is the product of the principal symbols of $A$ and $B$, and the principal symbol of the commutator $i[A, B]$ is the Poisson bracket of the principal symbols of $A$ and $B$. 
	
	\item The normal operator is multiplicative: $N(A \circ B) = N(A) \circ N(B)$.

	\item An operator $A \in \Psi_{\mathrm{par,I}}^{m,\mathsf{s},q}$ is said to be elliptic if its symbol $a = \sigma_{\mathrm{par,I}}^{m,\mathsf{s},q}(A)$ satisfies 
	\begin{equation} 
		|a| \geq C^{-1} \aang{\zeta}^m \ang{z}^{\mathsf{s}} h^{-q}
	\end{equation} 
	whenever either $\ang{\zeta} \geq C$ or $\ang{z}\geq C$, for sufficiently large $C$ and all $h$, and in addition if $N(A)$ is invertible as a map from $L^2$ to $H^{-m, -\mathsf{s}}_{\mathrm{par}}(\bbR^{1,d})$. Then if $A$ is elliptic, there is an inverse $B \in \Psi_{\mathrm{par, I}}^{-m,-\mathsf{s}, -q}(\bbR^{1,d})$ for sufficiently small $h$.
\end{itemize}

We also have a version $\tilde{S}_{\mathrm{par,I}}^{m,\mathsf{s},\ell}$ of $S_{\mathrm{par,I}}^{m,\mathsf{s},\ell}$ in which symbols are only conormal at $h=0$.\footnote{This notation is different from that used in \cite{NRL_I}; there, $S_{\mathrm{par,I}}$ was used to refer to the symbol classes with merely conormal behavior. The $\mathrm{par,I}$-calculus plays a minor role here, so this should not cause confusion.}

\subsection{\texorpdfstring{The resolved parametrized parabolic calculus $\Psi_{\mathrm{par,I,res}}$}{The resolved parametrized parabolic calculus}}
\label{subsec:par-I-res}
In the next two subsections, we discuss two calculi that can be ``glued together'' (both at the level of the phase space, and at the operator level) to form the calculus discussed in Section~\ref{subsec:calc}, the ``resolved natural calculus,'' where most of the analysis takes place. This calculus, in turn, forms the building block for the ``twice-resolved natural calculus'' discussed in Section~\ref{subsec:calctwo}, whose phase space consists of two copies of the phase space of the resolved natural calculus, ``glued'' together after each copy is translated in frequency. 

We begin by discussing the resolved parametrized parabolic calculus $\Psi_{\mathrm{par,I,res}}$.
The resolved parametrized phase space ${}^{\mathrm{par,I}}\overline{T}^*\bbM={}^{\mathrm{par,I,res}}\overline{T}^* \bbM$ is obtained from ${}^{\mathrm{par}}\overline{T}^* \bbM \times [0,1]_h$ by blowing up the intersection of $\{ h = 0 \}$ and the face at fiber-infinity, that is, $\aang{\zeta}^{-1} = 0$:
\begin{equation}
{}^{\mathrm{par,I,res}}\overline{T}^* \bbM = \Big[ {}^{\mathrm{par}}\overline{T}^* \bbM \times [0,1]_h; \{ h = 0, \aang{\zeta}^{-1} = 0 \} \Big]. 
\end{equation}
This ``resolved'' phase space has four boundary hypersurfaces: the differential face $\mathrm{df_1}$ at fiber-infinity (the lift of fiber-infinity from ${}^{\mathrm{par,I}}\overline{T}^* \bbM$); the spacetime face $\mathrm{bf}$ (the lift of spacetime-infinity from ${}^{\mathrm{par,I}}\overline{T}^* \bbM$); the parabolic face $\mathrm{pf}$ (the lift of $\{ h = 0 \}$ from ${}^{\mathrm{par,I}}\overline{T}^* \bbM$); and the new face created by blowup, which we call the front face or the natural face, and denote $\natural\mathrm{f}$ (or ff). We have boundary defining functions 
\begin{equation}
\label{eq:bdfs-par-I-res}
\begin{aligned}
\rho_{\mathrm{df}_1} &= \frac{\aang{\zeta}^{-1}}{h +  \aang{\zeta}^{-1}}, &\qquad 
\rho_{\mathrm{bf}} &= \ang{z}^{-1}, \\
\rho_{\mathrm{par}} &= \frac{h}{h +  \aang{\zeta}^{-1}} , &\qquad
\rho_{\natural\mathrm{f}} &= h + \aang{\zeta}^{-1}.
\end{aligned}\end{equation}

As usual, we define the symbol classes on ${}^{\mathrm{par,I,res}}\overline{T}^* \bbM$ as consisting of functions obeying the same $L^\infty$ bounds under repeated applications of some particular vector fields. In order to figure out the appropriate vector fields here, it is useful to recast the symbol estimates defining the sc- and par-calculi in geometric language. In these estimates, the relevant vector fields are those tangent to the boundary (b-vector fields) of the scattering, resp. parabolic, compactification of the cotangent bundle. Indeed, in the scattering case, $\ang{z} \partial_{z_i}$ and $\ang{\zeta}\partial_{\zeta_i}$ furnish a basis for the b-vector fields on the compactified sc-cotangent bundle. Similarly, in the parabolic case, the vector fields $\ang{z} \partial_{z_i}$ and $\aang{\zeta}\partial_{\xi_i}$, $\aang{\zeta}^2 \partial_\tau$ furnish a basis for the b-vector fields on the compactified par-cotangent bundle. 
Thus, let  $\calV_{\mathrm{b}}({}^{\mathrm{par,I,res}}\overline{T}^* \bbM)$ be the smooth vector fields on this compactified cotangent bundle tangent to the boundary.

For $m,s,\ell,q \in \bbR$, the symbol classes $\tilde S_{\mathrm{par, I,res}}^{m,\mathsf{s},\ell,q}$ are defined by:
\begin{multline}\label{eq: parIres symbol}
\tilde S_{\mathrm{par, I,res}}^{m,s,\ell,q} = \Big\{ a \in \rho_{\mathrm{df}_1}^{-m} \rho_{\mathrm{bf}}^{-s} \rho_{\natural\mathrm{f}}^{-\ell}   \rho_{\mathrm{par}}^{-q} L^\infty({}^{\mathrm{par, I}}\overline{T}^* \bbM) 
: \big( \calV_{\mathrm{b}}({}^{\mathrm{par, I,res}}\overline{T}^* \bbM) \big)^M a \\ \in \rho_{\mathrm{df}_1}^{-m} \rho_{\mathrm{bf}}^{-s} \rho_{\natural\mathrm{f}}^{-\ell}   \rho_{\mathrm{par}}^{-q} L^\infty({}^{\mathrm{par,I}}\overline{T}^* \bbM) \text{ for all } M \in \mathbb{N} \Big\}.
\end{multline}
This definition is analogous to \cref{eq: par symbol} except that $h$ is taken into consideration as well.

When $s$ is replaced by $\mathsf{s} \in C^\infty({}^{\mathrm{par,I,res}}\overline{T}^* \bbM)$, the symbol class $\tilde{S}_{\mathrm{par, I,res}}^{m,\mathsf{s},\ell,q}$ is defined in almost the same way but allowing the usual logarithmic loss at $\mathrm{bf}$ upon differentiation. See \cite[\S2.2]{NRL_I} for details.

Due to \cite[Lem.\ 2.10]{NRL_I}, which is the analogue of the general principle that ``conormal functions lifts to a conormal one when blowing-up a corner'' in this variable order setting, we have
\begin{equation} \label{eq:par-I-par-I-res-symbol-equal}
    \tilde{S}_{\mathrm{par,I}}^{m,\mathsf{s},q}(\bbM) = \tilde{S}_{\mathrm{par, I,res}}^{m,(\beta_{\mathrm{par,I,res}}^*\mathsf{s}),m+q,q},
\end{equation}
where $\mathsf{s} \in C^\infty({}^{\mathrm{par, I}}\overline{T}^* \bbM)$ and 
\begin{equation}
  \beta_{\mathrm{par,I,res}}: \;  {}^{\mathrm{par, I, res}}\overline{T}^* \bbM \to {}^{\mathrm{par, I}}\overline{T}^* \bbM
\end{equation}
is the blow-down map.

More generally, for $\mathsf{s} \in C^\infty({}^{\mathrm{par, I, res}}\overline{T}^* \bbM)$ we have
\begin{equation}  \label{eq:par-I-par-I-res-symbol-contain}
\tilde{S}_{\mathrm{par,I, res}}^{m, \mathsf{s}, l, q}(\bbM) \subset \tilde{S}_{\mathrm{par,I}}^{M, s, q}(\bbM), \, \text{ for } \,	M \geq \max(m, l-q), \, s \geq \max \mathsf{s}.
 \end{equation}

We will work with a slightly smaller class of $\mathrm{par, I, res}$-symbols that are classical, that is, smooth up to an overall power, at the parabolic face $\mathrm{pf}$. This will allow us to define a normal operator at $\mathrm{pf}$. 
To this end, we define the symbol class $S_{\mathrm{par,I,res}}^{m,\mathsf{s},\ell,q}$ (with no tilde) by 
\begin{equation}
S_{\mathrm{par,I,res}}^{m,\mathsf{s},\ell,q} = \Big\{ a \in \tilde{S}_{\mathrm{par,I,res}}^{m,\mathsf{s},\ell,q}: \Big(\frac{\mathrm{d}}{\mathrm{d}h}\Big)^k a \in \tilde{S}_{\mathrm{par,I,res}}^{m,\mathsf{s} + \delta,\ell+k,q} 
\text{ for all } \delta>0, k \in \mathbb{N} \Big\}, 
\end{equation}
where $\mathrm{d}/\mathrm{d} h$ is the partial derivative with respect to $h$ while fixing the spatial and frequency variables $\tau,\xi$.
This condition enforces smoothness at  $\mathrm{pf}$ but not at the natural face, due to the increase in order by $k$ at the natural face when we differentiate $k$ times  in $h$.

Quantizing, we get $\Psi_{\mathrm{par,I,res}}^{m,\mathsf{s},\ell,q} = \operatorname{Op}( S^{m,\mathsf{s},\ell,q}_{\mathrm{par,I,res}} )$, 
\begin{equation} 
\Psi_{\mathrm{par,I,res}} = \bigcup_{m,s,\ell,q \in \mathbb{Z}} \Psi_{\mathrm{par,I,res}}^{m,\mathsf{s},\ell,q},
\end{equation}
which is a multi-graded algebra. The normal operator $N(A)$, is defined for $A \in \Psi^{m, \mathsf{s}, \ell, 0}_{\mathrm{par,I,res}}$. Roughly, 
\begin{equation}
N(A) = A|_{\mathrm{pf}}.
\end{equation} 
More precisely, this is given by restricting the full (left-reduced) symbol of $A$ to $\mathrm{pf}$, and then quantizing to obtain $N(A) \in \Psi_{\mathrm{par}}^{\ell, \mathsf{s}}$. This is well-defined, since the symbol is by definition smooth at $\mathrm{pf}$ and so can be restricted to this face, and the restriction belongs in the parabolic symbol class of order $(\ell, \mathsf{s})$. 

The principal symbol of $A = \operatorname{Op}(a)$, $a \in S_{\mathrm{par,I,res}}^{m,\mathsf{s},\ell, q}(\bbM)$, is its image 
\begin{equation}
	\sigma_{\mathrm{par,I,res}}^{m,\mathsf{s}, \ell, q}(A) = a \bmod  \cap_{\delta > 0} S_{\mathrm{par,I,res}}^{m-1,\mathsf{s}-1+\delta, \ell-1, q}(\bbM)
\end{equation}
in $S_{\mathrm{par,I, res}}^{m,\mathsf{s}, \ell, q}(\bbM) / \cap_{\delta > 0} S_{\mathrm{par,I}}^{m-1,\mathsf{s}-1+\delta,\ell-1,q}(\bbM)$. This definition reflects the fact that this calculus is ``symbolic'' at the front face (``natural face,'' $\natural\mathrm{f}$) as well as $\mathrm{df}_1,\mathrm{bf}$ (like the $\mathrm{par,I}$-calculus). As before, one could set $\delta = 0$ for constant spacetime orders, but for variable orders, logarithmic losses from differentiating $\mathsf{s}$ force an arbitrarily small $\delta$ loss in the spacetime index.

The properties of the $\mathrm{par, I, res}$-calculus is easily deduced from $\mathrm{par, I}$-calculus and the asymptotic expansion of the symbol of the composition, together with \cref{eq:par-I-par-I-res-symbol-contain}. One has 
\begin{itemize}
	\item The composition of $A \in \Psi_{\mathrm{par,I,res}}^{m,\mathsf{s},\ell,q}$ with $B \in \Psi_{\mathrm{par,I,res}}^{m',\mathsf{s}',\ell',q'}$ is an operator $C \in \Psi_{\mathrm{par,I,res}}^{m+m',\mathsf{s}+\mathsf{s}',\ell+\ell',q+q'}$. Moreover, the symbol of $C$ satisfies an asymptotic expansion 
	\begin{equation}
		c(z,\zeta,h) \sim  \sum_{\alpha\in \bbN^D} \frac{i^{|\alpha|}}{ \alpha! } \big(D_\zeta^\alpha a(z,\zeta,h) \big) D_{z}^{\alpha} b(z,\zeta,h)  
		\label{eq:moyal_explicit_parI_res}
	\end{equation}
	in which subsequent terms decrease in order at fiber-infinity, at spacetime infinity, and at the front face, but \emph{not} at $\mathrm{pf}$.

	\item The principal symbol of $AB$ is the product of the principal symbols of $A$ and $B$, and the principal symbol of the commutator $i[A, B]$ is the Poisson bracket of the principal symbols of $A$ and $B$. 
	
	\item The normal operator is multiplicative in the sense that for $A \in \Psi^{m, \mathsf{s}, \ell, 0}_{\mathrm{par,I,res}}$, $B \in \Psi^{m', \mathsf{s}', \ell', 0}_{\mathrm{par,I,res}}$ we have 
\begin{equation}\label{eq:mult normal op parIres}
N(AB) = N(A) N(B) \in \Psi_{\mathrm{par}}^{\ell + \ell', \mathsf{s} + \mathsf{s}'}.
\end{equation}
See \cite[Prop.\ 2.11]{NRL_I} for the proof of \eqref{eq:mult normal op parIres}.
\end{itemize}

For $a \in S_{\mathrm{par,I,res}}^{m,\mathsf{s},\ell,q}$, we say that it is \emph{elliptic} at $q \in \mathrm{df}_1 \cup \mathrm{bf} \cup \natural\mathrm{f} \subset \partial({}^{\mathrm{par,I,res}}\overline{T}^*\bbM)$ if 
\begin{equation} \label{eq:def-par-I-res-elliptic}
    |\rho_{\mathrm{df}_1}^{m}\rho_{\mathrm{bf}}^{\mathsf{s}} \rho_{\natural\mathrm{f}}^{\ell} \rho_{\mathrm{pf}}^{q}a|>c
\end{equation}
for a constant $c>0$ in a neighborhood of $q$, where $\rho_{\bullet}$ are as in \cref{eq:bdfs-par-I-res}.
If $a$ is elliptic at every point of $\mathrm{df}_1 \cup \mathrm{bf} \cup \natural\mathrm{f}$, then we say it is (totally) elliptic. If $a$ is elliptic and $N(A):L^2\to H_{\mathrm{par}}^{\ell,s}$ is invertible, then we say that  $A = \operatorname{Op}(a)$ is elliptic.

The squared $H_{\calc}^{m,\mathsf{s},\ell,q}$-norm is defined by
\begin{equation} \label{eq:parIres-squared-norm}
 h^{-2q}\| u \|^2_{H_{\mathrm{par}}^{l, s}} + \| Au \|_{L^2}^2,
\end{equation}
where $s < \inf \mathsf{s}$ is a constant, $l<\min(\ell-q,m)$ and $A \in  \Psi_{\mathrm{par,I,res}}^{m,\mathsf{s},\ell,q}$ is a fixed elliptic operator. 
Then, 
\begin{equation} \label{eq:def-Sobolev-parIres}
	H_{\mathrm{par,I,res}}^{m,\mathsf{s},\ell,q}(h) = (H_{\mathrm{par}}^{m,\mathsf{s}(h)}(\bbR^{1,d}),\lVert - \rVert_{H_{\mathrm{par,I,res}}^{m,\mathsf{s},\ell,q}} ),
\end{equation}
is a family of Sobolev spaces equivalent to $H_{\mathrm{par}}^{m,\mathsf{s}(h)}$ for each individual $h>0$ but equipped with an $h$-dependent norm.

Now we define the wavefront set in $\mathrm{par,I,res}$-calculus.
The $\mathrm{par,I,res}$-operator wavefront set \begin{equation} 
	\operatorname{WF}'_{\mathrm{par,I,res}}(A) \subset \mathrm{df}_1 \cup \mathrm{bf} \cup \natural\mathrm{f}
\end{equation} 
of $A=\operatorname{Op}(a)$ is defined to be the essential support of $a\in S_{\mathrm{par,I,res}}^{m,\mathsf{s},\ell,q}$, that is, the closed subset whose complement is 
\begin{multline*}
    \{ p \in \mathrm{df}_1 \cup \mathrm{bf} \cup \natural\mathrm{f} : \, \exists \text{ neighbourhood $U$ of $p$ such that} \\
    \text{ for all $N \in \bbN$ } \exists C_N \text{ such that } |\rho_{\mathrm{df}}^{-N} \rho_{\mathrm{bf}}^{-N}\rho_{\mathrm{zf}}^{-N}a| \leq C_N \text{ in  } U \}. 
\end{multline*}

The $(m,\mathsf{s}, \ell)$-order $\mathrm{par,I,res}$-wavefront set of a family of tempered distributions $u$, denoted $\operatorname{WF}_{\mathrm{par,I,res}}^{m,\mathsf{s},\ell}(u)$, is the closed subset of  $\mathrm{df}_1 \cup \mathrm{bf} \cup \natural\mathrm{f}$ defined by
\begin{multline}
	p \notin \operatorname{WF}_{\mathrm{par,I,res}}^{m,\mathsf{s},\ell}(u) \iff \exists N\in \bbR, A \in \Psi_{\mathrm{par,I,res}}^{m,\mathsf{s},\ell,0} \text{  elliptic at } p \text{ such that } Au \in L^2(\bbM) \\ \text{ with $\| Au \|_{L^2}, \lVert u \rVert_{H_{\mathrm{sc}}^{-N,-N}}$ are uniformly bounded in } h.
\end{multline}
Then we set $\operatorname{WF}_{\mathrm{par,I,res}}(u)$ to be 
\begin{equation}
    \operatorname{WF}_{\mathrm{par,I,res}}(u) = \overline{\bigcup_{m,s,\ell \in \bbN} \operatorname{WF}_{\mathrm{par,I,res}}^{m,s,\ell}(u)}.
\end{equation}

\subsection{\texorpdfstring{The natural calculus $\Psi_{\calczero}$}{The natural calculus}}
\label{subsec:calczero}
Next, we discuss the second ingredient used in the construction of the resolved natural calculus in \S\ref{subsec:calc}, namely the natural calculus $\Psi_{\calczero}$. This is based on taking natural frequency coordinates as discussed in \S\ref{subsec:natural}. That is, we take the phase space as 
\begin{equation}
{}^{\calczero} \overline{T}^* \bbM = \overline{\bbR^{1,d}_{t,x}}\times  \overline{\bbR^{1,d}_{\tau_\natural,\xi_\natural}} \times [0,\infty)_h, 
\end{equation}
where $\tau_\natural = h^2 \tau$ the natural frequency variable dual to $t_\natural$ and $\xi_\natural = h\xi$ is the natural frequency variable dual to $x_\natural$.

Next, we define our $\calczero$-symbol classes and we refer to \cite[\S2]{NRL_I} for more details. 
For $m,s \in \bbR$, the symbol class $S^{m,s,0}_{\calczero}(\bbM)$ is set of conormal functions of $h$ with values in $S^{m,\mathsf{s}}_{\mathrm{sc}}$ symbols \emph{in the $\taun, \xin$ variables}. That is, we have estimates 
\begin{equation} \label{eq:nat_symbol}
	|\partial_z^\alpha\partial_{\zetan}^\beta (h\partial_h)^j a(z,\zetan,h) | \leq C_{\alpha\beta j} \la z \ra^{s-|\alpha| } \la \zetan \ra^{m-|\beta|}, \quad \zetan = (\taun, \xin), \quad z=(t,x), 
\end{equation}
for all $j\in \bbN$, multi-indices $\alpha,\beta\in \bbN^{1+d}$.
This is provided with the Fr\'echet topology induced by the seminorms given by the best constants $C_{\alpha\beta j}$ in the inequality above. 
Noticing that $\partial_{\tau_\natural}=h^{-2}\partial_{\tau},\partial_{\xi_\natural}=h^{-1}\partial_{\xi}$, \cref{eq:nat_symbol} in fact gives $h$-improvement when we differentiate in $\zeta = (\tau,\xi)$, which is what is actually used when we consider asymptotic expansions like \cref{eq:moyal_explicit_calczero}.

For the more general case in which $s$ is allowed to be $\mathsf{s} \in C^\infty({}^{\calczero} \overline{T}^* \bbM )$, we allow the usual logarithmic losses, similar to the calculi previously discussed --- see \cite[\S2.3]{NRL_I} for details.

Then, we define 
\begin{equation}
	S_{\calczero}^{m,\mathsf{s},\ell} = h^{-\ell} S_{\calczero}^{m,\mathsf{s},0}.
\end{equation}

The class of $\calczero$-pseudodifferential operators $\Psi_{\calczero}^{m,\mathsf{s},\ell}$ is the set of all $\operatorname{Op}(a)$ acting as in \cref{eq:quant_sc}  with $a\in S_{\calczero}^{m,\mathsf{s},\ell}$.  
The principal symbol map 
\begin{equation}
\sigma_{\calczero}^{m,\mathsf{s},\ell} : \operatorname{Op}(a) \mapsto a \bmod  \cap_{\delta > 0} S_{\calczero}^{m-1,\mathsf{s}-1+\delta,\ell-1 } 
\end{equation}
is valued in 
\begin{equation} 
S_{\calczero}^{m,\mathsf{s},\ell}/ \cap_{\delta > 0} S_{\calczero}^{m-1,\mathsf{s}-1+\delta,\ell-1} .
\end{equation} 
So, \emph{principal symbols capture operators modulo terms suppressed as $h\to 0^+$}. This is a key property of $\Psi_{\calczero}$ (not shared by the $\mathrm{par,I}$- or $\mathrm{par,I,res}$-calculi). 

We record major properties of $\Psi_{\calczero}$ below and refer readers to \cite[\S2]{NRL_I} for full details. 
\begin{itemize}
	\item The composition of $A \in \Psi_{\calczero}^{m,\mathsf{s},\ell}$ with $B \in \Psi_{\calczero}^{m',\mathsf{s}',\ell'}$ is an operator $C \in \Psi_{\calczero}^{m+m',\mathsf{s}+\mathsf{s}',\ell+\ell'}$. Moreover, the symbol of $C$ satisfies an asymptotic expansion 
	\begin{equation}
	c(z,\zeta,h) \sim  \sum_{\alpha\in \bbN^D} \frac{i^{|\alpha|}}{ \alpha! } \big(D_\zeta^\alpha a(z,\zeta,h) \big) D_{z}^{\alpha} b(z,\zeta,h)  
	\label{eq:moyal_explicit_calczero}
	\end{equation}
	in which subsequent terms decrease in order at fiber-infinity, spacetime infinity, \emph{and} $h=0$.

	\item The principal symbol of $AB$ is the product of the principal symbols of $A$ and $B$, and the principal symbol of the commutator $i[A, B]$ is the Poisson bracket of the principal symbols of $A$ and $B$. 
	
	\item The $\calczero$-Sobolev spaces
\begin{equation} 
		H_{\calczero}^{m,\mathsf{s},\ell}(h) = (H_{\mathrm{sc}}^{m,\mathsf{s}(h)}(\bbR^{1,d}),\lVert - \rVert_{H_{\calczero}^{m,\mathsf{s},\ell} } )
	\end{equation}
        are one-parameter families of Sobolev spaces, equipped with an $h$-dependent norm, the $\lVert - \rVert_{H_{\calczero}^{m,\mathsf{s},\ell} }$-norm (defined using $\Psi_{\calczero}$ in the usual way).
	Then, 
	\begin{equation}
	A(h): H_{\calczero}^{m,\mathsf{s},\ell}(h) \to H_{\calczero}^{m-m',\mathsf{s}-\mathsf{s}',\ell-\ell'}(h) 
	\end{equation}
	is bounded uniformly as $h\to 0^+$, for any $A\in \Psi_{\calczero}^{m',\mathsf{s}',\ell'}$. Our conventions on the indexing are that 
	\begin{equation}
		H_{\calczero}^{0,0,0} = L^2(\bbR^{1+d}_{t,x}), 
	\end{equation}
	not $H_{\calczero}^{0,0,0} = L^2(\bbR^{1+d}_{t_\natural,x_\natural})$.
	\item An operator $A \in \Psi_{\calczero}^{m,\mathsf{s},\ell }$ is said to be (totally) elliptic if its symbol $a = \sigma_{\calczero}^{m,\mathsf{s},\ell}(A)$ satisfies 
	\begin{equation} 
		|a| \geq C^{-1} \ang{\zeta_\natural}^m \ang{z}^{\mathsf{s}} h^{-\ell}
	\end{equation} 
	whenever either $\ang{\zeta_\natural} \geq C$ or $\ang{z}\geq C$, and $h\in (0,1/C)$, for sufficiently large $C$. Then if $A$ is elliptic, there is an inverse $B \in \Psi_{\calczero}^{-m,-\mathsf{s}, -\ell}$ for sufficiently small $h$. 
    \item The $\calczero$-operator wavefront set \begin{equation} 
	\operatorname{WF}'_{\calczero}(A) \subset \partial ({}^{\calczero} \overline{T}^*\bbM)
\end{equation} 
of $A=\operatorname{Op}(a)$ is defined to be the essential support of $a\in S_{\calczero}^{m,\mathsf{s},\ell}$. More concretely, for $p \in \partial ({}^{\calczero} \overline{T}^*\bbM)$, 
\begin{equation}
	p \notin \operatorname{WF}'_{\calczero}(A)  \iff  |\rho_{\mathrm{df}}^{-N} \rho_{\mathrm{bf}}^{-N}\rho_{\mathrm{zf}}^{-N}a| \leq C_N \text{ near } p, \text{ for all } N \in \bbR ,\text{ for some }C_N.
\end{equation}
Note that $\operatorname{WF}'_{\calczero}(A)$ is always a closed subset of $\partial ({}^{\calczero} \overline{T}^*\bbM)$.

\item 
The $(m,\mathsf{s},\ell)$-order $\calczero$-wavefront set of a one-parameter family of tempered distributions $u$ is the closed set $\operatorname{WF}_{\calczero}^{m,\mathsf{s},\ell}(u) \subset \partial ({}^{\calczero} \overline{T}^*\bbM)$ defined by
\begin{multline}
	p \notin \operatorname{WF}_{\calczero}^{m,\mathsf{s},\ell}(u) \iff \exists N\in \bbR,   A \in \Psi_{\calczero}^{m,\mathsf{s},\ell} \text{ that is elliptic at } p \text{ such that } Au \in L^2(\bbM), \text{ with } \\
	\| Au(h) \|_{L^2(\bbM)}, h^N \lVert u \rVert_{H_{\mathrm{sc}}^{-N,-N}}  \text{ uniformly bounded in } h. 
\end{multline}
Then we set $\operatorname{WF}_{\calczero}(u)$ to be 
\begin{equation}
    \operatorname{WF}_{\calczero}(u) = \overline{\bigcup_{m,s,\ell \in \bbN} \operatorname{WF}_{\calczero}^{m,s,\ell}(u)}.
\end{equation}

\end{itemize}

\subsection{\texorpdfstring{The resolved natural calculus, $\Psi_{\calc}$}{The resolved natural calculus}}
\label{subsec:calc}

The calculus $\Psi_{\calc}$ is based upon the compactification 
\begin{equation}
	{}^{\calc}\overline{T}^* \bbM \hookleftarrow T^* \bbR^{1,d} 
\end{equation}
obtained by gluing together parts of the $\mathrm{par,I,res}$-phase space with the natural phase space, as depicted in \Cref{fig:gluing}. Another way to think of it is that it is obtained from the natural phase space by blowing up parabolically the portion of the zero section 
\begin{equation} 
	{}^{\calczero} o^* \bbM = \mathrm{cl}_{{}^{\calczero}\overline{T}^* \bbM} \{\xi=0=\tau\} = \{\xi_{\calcshort} = 0 = \tau_{\calcshort} \}
\end{equation} 
in $\{h=0\}$. That is, we replace ${}^{\calczero} o^* \bbM \cap \{h=0\}$ (over each point in $\bbM$) by a face that is parametrized by
\begin{equation}
\tau = \frac{\tau_{\calczero}}{h^2}, \quad \xi = \frac{\xin}{h};
\end{equation}
we thereby recover the standard frequency coordinates, $\tau,\xi$. 
We reproduce from \cite{NRL_I} a figure, \Cref{fig:gluing}, indicating how the gluing works. This gluing makes sense as the natural face in the $\mathrm{par,I,res}$-phase space and the natural face in the $\calczero$-phase space can be canonically identified, at least away from $\zetan = 0$ and $|\zetan| = \infty$. One can think of ${}^{\calc}\overline{T}^* \bbM$ as the output of gluing the part of ${}^{\mathrm{par,I,res}}\overline{T}^* \bbM$ near $\mathrm{pf}$ and the part of ${}^{\calczero}\overline{T}^* \bbM$ near $\mathrm{df}$.

\begin{figure}
	\includegraphics[scale=.95]{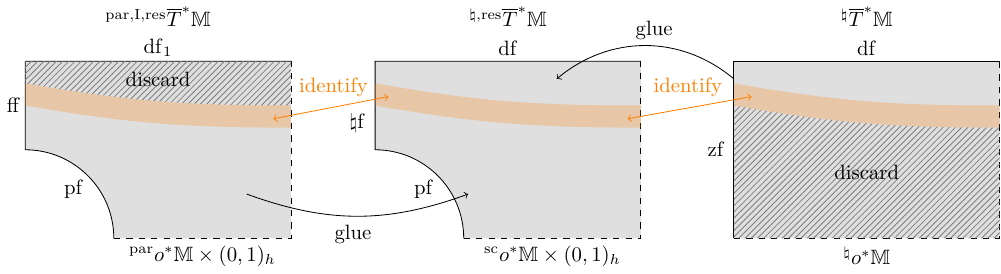}
	\caption{(From \cite{NRL_I}.) Illustration of the gluing procedure used to construct ${}^{\calc} \overline{T}^* \bbM$. A neighborhood of $\mathrm{df}$ in  ${}^{\calczero} \overline{T}^* \bbM$  replaces a neighborhood of $\mathrm{df}_1$ in ${}^{\mathrm{par,I,res}}\overline{T}^* \bbM$. In order to accomplish this, it is important to observe that one may identify neighborhoods of compact subsets of the interior of $\mathrm{ff}$ and $\mathrm{zf}$, e.g.\ the orange set in the figure.   }
    \label{fig:gluing}
\end{figure}

Consequently, ${}^{\calc}\overline{T}^* \bbM$ has four boundary hypersurfaces, fiber infinity $\mathrm{df}$, $\mathrm{bf}$ over $\partial \bbM$, $\natural\mathrm{f}$ (what is left of what we called `$\natural \mathrm{f}$' of the $\calczero$-phase space), and $\mathrm{pf}$. It turns out that 
\begin{equation}
	\mathrm{pf}\cong {}^{\mathrm{par}}\overline{T}^* \bbM 
	\label{eq:pf_cong}
\end{equation}
canonically.

Now we turn to the construction of the $\calc$-symbol class. Let $\mathcal{V}_{\calc}$ be the Lie algebra of smooth vector fields on ${}^{\calc} \overline{T}^* \bbM$ that are tangent to all boundary faces except for $\mathrm{pf}$. This means that we are imposing smoothness at $\mathrm{pf}$ but only conormal regularity at the other faces. 
Then we define 
\begin{equation} 
	S^{0,0,0,0}_{\calc} \subset C^\infty(T^* \bbR^{1,d}\times (0,\infty)_h) 
\end{equation} 
to denote the locally convex topological vector space of smooth functions on $T^* \bbR^{1,d}\times (0,1)_h$ that remain uniformly bounded under iterated applications of vector fields in $\mathcal{V}_{\calc}$.
For $m,s,\ell,q\in \bbR$, 
we set
\begin{equation} 
 S^{m,s,\ell,q}_{\calc} = \rho_{\mathrm{df}}^{-m} 
 \rho_{\mathrm{bf}}^{-s} \rho_{ \calczero\mathrm{f} }^{-\ell} \rho_{\mathrm{pf}}^{-q}  S^{0,0,0,0}_{\calc}.
\end{equation} 
When we want to allow $s$ to be replaced by $\mathsf{s} \in C^\infty({}^{\calc}\overline{T}^* \bbM)$, the usual $\delta$ loss in the spacetime exponent is incurred whenever the symbol is differentiated.  See \cite[\S2.4]{NRL_I} for more details.

The corresponding class of $\calc$-pseudodifferential operators, which is denoted by $\Psi_{\calc}^{m,\mathsf{s},\ell,q}$ consists of 
all $\operatorname{Op}(a)$ acting as in \cref{eq:quant_sc} but with $a\in S_{\calc}^{m,\mathsf{s},\ell,q}$.

Because the $h>0$ slices of ${}^{\calc}\overline{T}^* \bbM$ are canonically diffeomorphic to ${}^{\mathrm{sc}}\overline{T}^* \bbM$, we can think of elements of  $\Psi_{\calc}$ as being families of elements of $\Psi_{\mathrm{sc}}$, depending smoothly on $h>0$ and degenerating in some controlled way as $h\to 0^+$. 
Microlocally, an element in $\Psi_{\calc}$ is in $\Psi_{\mathrm{par,I,res}}$ near $\mathrm{pf}$ while in $\Psi_{\calczero}$ away from $\mathrm{pf}$.

Actually, \cref{eq:pf_cong} tells that elements of $\Psi_{\calc}$  `restrict' at $\mathrm{pf}$ to an element of $\Psi_{\mathrm{par}}$, and we can correspondingly define a normal operator 
\begin{equation}
	N(A) = \operatorname{Op}(h^q a|_{\mathrm{pf}}) \in \Psi_{\mathrm{par}}^{\ell,\mathsf{s}|_{\mathrm{pf}}}
\end{equation}
for $A=\operatorname{Op}(a)$, $a\in S_{\calc}^{m,\mathsf{s},\ell,q}$.

The principal symbol map 
\begin{equation}
	\sigma_{\calc}^{m,\mathsf{s},\ell,q} : \operatorname{Op}(a) \mapsto a \bmod  \cap_{\delta > 0} S_{\calc}^{m-1,\mathsf{s}-1+\delta,\ell-1 ,q} 
\end{equation}
is valued in 
\begin{equation} 
	S_{\calc}^{m,\mathsf{s},\ell,q}/ \cap_{\delta > 0} S_{\calc}^{m-1,\mathsf{s}-1+\delta,\ell-1,q } .
\end{equation} 
So, \emph{principal symbols capture operators modulo terms suppressed at each of $\mathrm{df},\mathrm{bf},\natural\mathrm{f}$}, but not $\mathrm{pf}$.

Now we list properties of $\Psi_{\calc}$ below and refer readers to \cite[Section~2.4]{NRL_I} for details.
\begin{itemize}
	\item The composition of $A \in \Psi_{\calc}^{m,\mathsf{s},\ell,q}$ with $B \in \Psi_{\calc}^{m',\mathsf{s}',\ell',q}$ is an operator $C \in \Psi_{\calc}^{m+m',\mathsf{s}+\mathsf{s}',\ell+\ell',q+q'}$. Moreover, the symbol of $C$ satisfies an asymptotic expansion 
	\begin{equation}
		c(z,\zeta,h) \sim  \sum_{\alpha\in \bbN^D} \frac{i^{|\alpha|}}{ \alpha! } D_\zeta^\alpha a(z,\zeta,h) D_{z}^{\alpha} b(z,\zeta,h),
		\label{eq:moyal_explicit_calc}
	\end{equation}
   in which subsequent terms and the error (that is, the difference between $c(z,\zeta,h)$ and the expansion) decrease in order at $\mathrm{df},\mathrm{bf},\natural\mathrm{f}$, but not at $\mathrm{pf}$.
Thus the asymptotic expansion only captures $c(z,\zeta,h)$ modulo a smooth family times $h^{-q-q'}$ (instead of smoothing and $O(h^{\infty})$-family) of residual operators:
	\begin{equation}
		h^{-q-q'} C^\infty([0,\infty)_h; \Psi_{\mathrm{sc}}^{-\infty,-\infty}(\bbM)).
	\end{equation}
	
	\item As usual, the principal symbol of $AB$ is the product of the principal symbols of $A$ and $B$, and the principal symbol of the commutator $i[A, B]$ is the Poisson bracket of the principal symbols of $A$ and $B$. 
	
	\item The Sobolev space 
	\begin{equation} \label{eq:calc-Sobolev-definition}
	H_{\calc}^{m,\mathsf{s},\ell,q}(h) = (H_{\mathrm{sc}}^{m,\mathsf{s}(h)}(\bbR^{1,d}),\lVert - \rVert_{H_{\calc}^{m,\mathsf{s},\ell,q}} )
    \end{equation}
	is defined in the usual way.
	Then, 
	\begin{equation}
		A(h): H_{\calc}^{m,\mathsf{s},\ell,q}(h) \to H_{\calczero}^{m-m',\mathsf{s}-\mathsf{s}',\ell-\ell',q-q'}(h) 
	\end{equation}
	is bounded uniformly as $h\to 0^+$, for any $A\in \Psi_{\calc}^{m',\mathsf{s}',\ell',q'}$. 
	\item The normal operator is multiplicative: $N(A\circ B)=N(A)\circ N(B)$.
	\item An operator $A \in \Psi_{\calc}^{m,\mathsf{s},\ell ,q}$ is said to be elliptic if its principal symbol $a = \sigma_{\calc}^{m,\mathsf{s},\ell,q}(A)$ satisfies 
	\begin{equation} \label{eq:elliptic-calc}
		|a| \geq C \rho_{\mathrm{df}}^{-m} \rho_{\mathrm{bf}}^{-\mathsf{s}}\rho_{\natural\mathrm{f}}^{-\ell}\rho_{\mathrm{pf}}^{-q} 
	\end{equation}
	in some neighborhood of $\mathrm{df}\cup\mathrm{bf}\cup\natural\mathrm{f}$
	for some $C>0$, and \emph{elliptic} if in addition $N(A)$ is elliptic and invertible. Then if $A$ is elliptic, there is an inverse $B \in \Psi_{\calc}^{-m,-\mathsf{s}, -\ell,-q}$ for sufficiently small $h$. 
\item The $\calc$-operator wavefront set \begin{equation} 
	\operatorname{WF}'_{\calc}(A) \subset \mathrm{df} \cup \mathrm{bf} \cup \natural\mathrm{f}
\end{equation} 
of $A=\operatorname{Op}(a)$ is defined to be the essential support of $a\in S_{\calc}^{m,\mathsf{s},\ell,q}$, that is, the closed subset whose complement is 
\begin{multline*}
    \{ p \in \mathrm{df} \cup \mathrm{bf} \cup \natural\mathrm{f} : \, \exists \text{ neighbourhood $U$ of $p$ such that} \\
    \text{ for all $N \in \bbN$ } \exists C_N \text{ such that } |\rho_{\mathrm{df}}^{-N} \rho_{\mathrm{bf}}^{-N}\rho_{\mathrm{zf}}^{-N}a| \leq C_N \text{ in  } U \}. 
\end{multline*}

\item The $(m,\mathsf{s}, \ell)$-order $\calc$-wavefront set of a family of tempered distributions $u$, denoted $\operatorname{WF}_{\calc}^{m,\mathsf{s},\ell}(u)$, is the closed subset of  $\mathrm{df} \cup \mathrm{bf} \cup \natural\mathrm{f}$ defined by 
\begin{multline}
	p \notin \operatorname{WF}_{\calc}^{m,\mathsf{s},\ell}(u) \iff \exists N\in \bbR, A \in \Psi_{\calc}^{m,\mathsf{s},\ell,0} \text{  elliptic at } p \text{ such that } Au \in L^2(\bbM) \\ \text{ with $\| Au \|_{L^2},\lVert u \rVert_{H_{\mathrm{sc}}^{-N,-N}}$ uniformly bounded in } h.
\end{multline}
Then we set $\operatorname{WF}_{\calc}(u)$ to be 
\begin{equation}
    \operatorname{WF}_{\calc}(u) = \overline{\bigcup_{m,s,\ell \in \bbN} \operatorname{WF}_{\calc}^{m,s,\ell}(u)}.
\end{equation}
\end{itemize}

\subsection{\texorpdfstring{The twice-resolved natural calculus, $\Psi_{\calctwo}$}{The twice-resolved natural calculus}}
\label{subsec:calctwo}
This calculus is based on the compactification ${}^{\calctwo}\overline{T}^* \bbM\hookleftarrow T^* \bbR^{1,d}$
given by blowing up, in the non-standard way that we saw in the discussion of the $\calc$-phase space, \emph{two} points in frequency, namely the points $\{ \taun = \pm 1, \xin = 0, h = 0 \}$ (which are actually submanifolds, as the spatial variables are unrestricted here) instead of the point $\{ \taun = 0, \xin = 0, h=0 \}$.
In other words, we perform the same sort of blowup used to construct ${}^{\calc}\overline{T}^* \bbM$, just this time at the loci $(\tau_\natural,\xi_\natural,h) = (\pm 1,0,0)$ instead of the zero section $(\tau_\natural,\xi_\natural,h) = (0,0,0)$.
The resulting compactification ${}^{\calctwo}\overline{T}^*$ has five boundary hypersurfaces, analogous to those of the $\calc$-phase space, except we now have two front faces $\mathrm{pf}_\pm$. 

Quantizing symbols on this phase space as \cref{eq:quant_sc} yields 
\begin{equation} 
\Psi_{\calctwo}=\bigcup_{m,s,\ell,q_-,q_+\in \bbR}\Psi_{\calctwo}^{m,s,\ell;q_-,q_+}.
\end{equation}
As above, we allow variable spacetime orders $\mathsf{s} \in C^\infty({}^{\calctwo}\overline{T}^* \bbM)$. 
We now have two $q$ orders, $q_\pm$, one at each of $\mathrm{pf}_\pm$.
Every $\calctwo$-operator has a unique left-reduced symbol on ${}^{\calctwo}\overline{T}^* \bbM$. We define ellipticity in terms of this left-reduced symbol, analogously to \cref{eq:elliptic-calc} but with two defining functions and orders associated to $\mathrm{pf}_\pm$ now.

The properties of this calculus are analogous to those of $\Psi_{\calc}$. More precisely, 
\begin{itemize}
	\item The composition of $A \in \Psi_{\calctwo}^{m,\mathsf{s},\ell,q+,q_-}$ with $B \in \Psi_{\calctwo}^{m',\mathsf{s}',\ell',q_+',q_-'}$ is an operator $$C = A \circ B  \in \Psi_{\calctwo}^{m+m',\mathsf{s}+\mathsf{s}',\ell+\ell',q_++q_+',q_- + q_-'}.$$ Moreover, the symbol of $C$ satisfies an asymptotic expansion 
	\begin{equation}
		c(z,\zeta,h) \sim  \sum_{\alpha\in \bbN^D} \frac{i^{|\alpha|}}{ \alpha! } D_\zeta^\alpha a(z,\zeta,h) D_{z}^{\alpha} b(z,\zeta,h),
		\label{eq:moyal_explicit_calctwo}
	\end{equation}
   in which subsequent terms and the error (that is, the difference between $c(z,\zeta,h)$ and the expansion) decrease in order at $\mathrm{df},\mathrm{bf},\natural\mathrm{f}$, but not at $\mathrm{pf}_\pm$.

	\item As usual, the principal symbol of $AB$ is the product of the principal symbols of $A$ and $B$, and the principal symbol of the commutator $i[A, B]$ is the Poisson bracket of the principal symbols of $A$ and $B$. 

    \item There are normal operators $N_\pm$ at $\mathrm{pf}_\pm$, which are multiplicative:
\begin{equation}
    N_\pm(AB) = N_\pm(A) \circ N_\pm(B). 
\end{equation}
The principal symbol of $A \in \Psi_{\calctwo}^{m,\mathsf{s},\ell,q+,q_-}$, together with the two normal operators $N_\pm(A)$, determine $A$ up to an operator in $\Psi_{\calctwo}^{m-1,\mathsf{s}-1+\delta,\ell-1,q+-1,q_--1}$ for every $\delta > 0$. 
\end{itemize}

A useful perspective is that elements of $\Psi_{\calctwo}$ are formed by ``gluing together'' two elements of $\Psi_{\calc}$, each conjugated by the multiplication operator $\exp(\pm i c^2 t) = \exp(\pm i t/h^2)$. At the operator level, this can be implemented using a partition of the identity, $\mathrm{Id} = Q_+ + Q_-$, where 
\begin{equation} 
    Q_\pm \in \Psi_{\calczero}^{0,0,0}
\end{equation}
are such that $Q_\mp$ is microlocally trivial at $\mathrm{pf}_\pm$. Then every $A \in \Psi_{\calctwo}^{m,\mathsf{s},\ell,q+,q_-}$ can be written 
$$
A = e^{it/h^2} A_+ e^{-it/h^2} Q_+ + e^{-it/h^2} A_- e^{it/h^2} Q_-
$$
where $A_\pm \in \Psi_{\calc}^{m, \mathsf{s}_\pm, \ell, q_\pm}$. (There is a technical inconvenience that the variable order $\mathsf{s}$ must be modified to $\mathsf{s}_\pm$ so as to be smooth on the $\natural$-res phase space, but this modification only happens away from the microsupport of $A_\pm$, so this is of little consequence.)

The associated $\calctwo$-Sobolev spaces  are denoted 
\begin{equation} \label{eq:def-calcatwo-Sobolev}
H_{\calctwo}^{m,\mathsf{s},\ell;q_-,q_+}(h) = (H_{\mathrm{sc}}^{m,\mathsf{s}(h)}(\bbR^{1,d}),\lVert - \rVert_{H_{\calctwo}^{m,\mathsf{s},\ell;q_-,q_+}} ).
\end{equation}
This norm can be defined, using any choice of the $Q_\pm$ and $\mathsf{s}_\pm$ as above, by
\begin{equation} \label{eq:definition_2res norm}
\lVert u \rVert_{H_{\calctwo}^{m,\mathsf{s},\ell;q_-,q_+}}
 =   \lVert e^{-it/h^2}Q_+u\rVert_{H_{\calc}^{m,\mathsf{s}_+,\ell,q_+}} +  \lVert e^{it/h^2} Q_-u\rVert_{H_{\calc}^{m,\mathsf{s}_-,\ell,q_-}}  .\end{equation}
The modulation $e^{\mp it/h^2}$ in the first term has the effect of 
translating $\mathrm{pf}_\pm$ to $\mathrm{pf}$, the parabolic face in the $\calc$-phase space.
See \cite[\S2.10]{NRL_I} for more details.

 The $(m,\mathsf{s}, \ell)$-order $\calctwo$-wavefront set of a one-parameter family of tempered distributions $u$, denoted $\operatorname{WF}_{\calctwo}^{m,\mathsf{s},\ell}(u)$, is the closed subset of  $\mathrm{df} \cup \mathrm{bf} \cup \natural\mathrm{f}$ defined by 
\begin{multline}
	p \notin \operatorname{WF}_{\calctwo}^{m,\mathsf{s},\ell}(u) \iff \exists N\in \bbR, A \in \Psi_{\calctwo}^{m,\mathsf{s},\ell,0,0} \text{  elliptic at } p \text{ such that } Au \in L^2(\bbM) \\ \text{ with $\| Au \|_{L^2}, \lVert e^{\mp it/h^2} Q_\pm u \rVert_{H_{\mathrm{sc}}^{-N,-N}}$ uniformly bounded in } h.
\end{multline}
Then we set $\operatorname{WF}_{\calctwo}(u)$ to be 
\begin{equation}
    \operatorname{WF}_{\calctwo}(u) = \overline{\bigcup_{m,s,\ell \in \bbN} \operatorname{WF}_{\calctwo}^{m,s,\ell}(u)}.
    \label{eq:calctwo-WF-def}
\end{equation}

\section{Asymptotics for the inhomogeneous problem}
\label{sec:inhomogeneous}
As already recalled in \S\ref{sec:true_intro} (\Cref{thm:inhomog}), the main results in \cite{NRL_I} are estimates for the solutions of the Klein--Gordon equation's 
four fundamental inhomogeneous problems 
\begin{equation} 
\begin{cases} Pu=F \\
u\in H_{\mathrm{sc}}^{m,\mathsf{s}}
\end{cases} 
\end{equation}
(the advanced/retarded problems and the Feynman/anti-Feynman problems, these being distinguished by the variable decay order $\mathsf{s}$).
The goal of this section is to apply these estimates to produce a leading order asymptotic for $u$ in the $c\to\infty$ limit, for suitable $F$.
First, in \S\ref{subsec:warmup}, we consider forcing $F$ of the form 
\begin{equation}
	F = e^{-ic^2 t} f_- + e^{ic^2 t} f_+ 
	\label{eq:Fform}
\end{equation}
for $f_\pm \in H_{\mathrm{sc}}^{\infty,S}(\bbR^{1,d})$ for $S>3/2$ (independent of $h=1/c$). This warm-up subsection will prepare us to tackle general irregular forcing
\begin{equation}
	F\in H_{\calctwo}^{m-1,\mathsf{s}+1,\ell-1;0,0}
	\label{eq:irreg_forcing}
\end{equation}
in the next subsection, \S\ref{subsec:main} (where $\mathsf{s}$ satisfies the usual threshold conditions, except with $\mathsf{s}>1/2$ on two radial sets, so that $\mathsf{s}-1$ also satisfies the usual threshold conditions). Note that $F$ can be very non-smooth.

In \S\ref{subsec:warmup}, our goal is to find an approximation to our solution $u$ of the form $u\approx v$ for 
\begin{equation}
	v = e^{-ic^2 t} v_- + e^{ic^2 t} v_+
\end{equation}
for $v_\pm$ solving the non-relativistic Schr\"odinger equation, with the sign in the subscript signifying the branch of the Schr\"odinger equation. The natural guess for $v_\pm$ is the solution of the advanced/retarded problem for the Schr\"odinger equation with forcing $f_\pm$. 
Whether $v_-$ solves the \emph{advanced} problem or the \emph{retarded} problem will depend on which of the four possible problems $u$ solves, and likewise for $v_+$. 
If $u$ solves the advanced problem, then the advanced problem is the relevant one for both of $v_\pm$. The retarded problem is analogous. If instead $u$ solves the Feynman/anti-Feynman problem, then one of $v_-,v_+$ will solve the advanced problem and the other will satisfy the retarded problem.

When we turn to \S\ref{subsec:main}, we will need to broaden our approximation to $u\approx v$ for  
\begin{equation}
	v=e^{-ic^2 t} v_- + e^{ic^2 t} v_++u_0 
	\label{eq:misc_139}
\end{equation}
for $v_\pm$ as above, for some choice of inhomogeneity extracted from $f_\pm$, and where $u_0$ (the novel feature of \cref{eq:misc_139}) is a solution of the \emph{free} Klein--Gordon equation with specified inhomogeneity. The reason we need to allow $u_0$  will be that the forcing $F$ (and therefore $u$) may have substantial wavefront set at large frequency. We will need to use the ``model problem'' of $P$ at $\natural\mathrm{f}$, the free Klein--Gordon equation, to solve this away.  This is closely related to the discussion in \S\ref{subsec:natural}, which demonstrates the \emph{necessity} of including a $u_0$ term in \cref{eq:misc_139}. This section demonstrates the sufficiency.

\subsection{Warm-up: smooth forcing}
\label{subsec:warmup}
To repeat the setup indicated above: we consider in this subsection the PDE $Pu=F$ for $F$ of the form \cref{eq:Fform} for $f_\pm \in H_{\mathrm{sc}}^{\infty,S}(\bbR^{1,d})$, where $S\in (3/2,\infty]$. 
Pick one of the advanced/retarded and Feynman/anti-Feynman problems, and let $u$ denote the solution. We can apply \Cref{thm:inhomog} to estimate $u$.
Which of the four problems we are considering is encoded in the allowed variable order $\mathsf{s}$ appearing in that theorem. (Of course, $u$ does not depend on the choice of $\mathsf{s}$, except insofar as it determines which of the four problems we are solving.) There are two radial sets on which $\mathsf{s}$ is required to be below $-1/2$ and two radial sets on which it is required to be above $-1/2$. Let $\calR_{\mathrm{below}}$ denote the union of the two radial sets on which $\mathsf{s}$ is required to be $<-1/2$, and let $\calR_{\mathrm{above}}$ denote the union of the other two. Then, \Cref{thm:inhomog} tells us that 
\begin{equation}
	u \in \bigcap_{\substack{\mathsf{s}|_{\calR_{\mathrm{below}}}<-1/2  \\  \mathsf{s}< S-1 } } H_{\calctwo}^{\infty,\mathsf{s},\infty;0,0}, 
\end{equation}
where the intersection is over all $\mathsf{s} \in C^\infty({}^{\calctwo}\overline{T}^* \bbM )$ allowed in the theorem and satisfying the additional condition $\mathsf{s}< S-1$ (which is required because $f_\pm$ may have only a finite amount of decay). 

Let us now fix such an $\mathsf{s}$ that satisfies the above-threshold condition on $\calR_{\mathrm{above}}$ with one order to spare: 
\begin{equation} 
	\mathsf{s}|_{\calR_{\mathrm{above}}} - 1> -1/2.
	\label{eq:misc_142}
\end{equation}
Because $S>3/2$, this condition is compatible with $\mathsf{s}<S-1$. 
Let $\mathsf{s}_0=\mathsf{s}-1$. Then, \cref{eq:misc_142} tells us that $\mathsf{s}_0$ also satisfies the above-threshold condition $\mathsf{s}_0|_{\calR_{\mathrm{above}}}>-1/2$.
So, (since subtracting $1$ from $\mathsf{s}$ does not interfere with the below-threshold and monotonicity conditions on the order in \Cref{thm:inhomog}) $\mathsf{s}_0$ also satisfies the hypotheses in \Cref{thm:inhomog}. This will be useful later.

Recalling that $\mathrm{pf}_\pm \cong {}^{\mathrm{par}}\overline{T}^* \bbM$ canonically, we can define $\overline{\mathsf{s}}_\pm \in C^\infty({}^{\mathrm{par}}\overline{T}^* \bbM)$ by 
\begin{equation}
	\overline{\mathsf{s}}_\pm = (\mathsf{s}+\varepsilon)|_{\mathrm{pf}_\pm},
\end{equation}
where $\varepsilon>0$ will be taken arbitrarily small.
Then, by \cite[Thm.\ 1.1]{Parabolicsc}\footnote{The theorem in \cite{Parabolicsc} is stated in less generality than we take as our setup here, but it is shown in the appendix of \cite{NRL_I} that the proof goes through in the present generality.}
there exist unique solutions 
\begin{equation} 
	v_\pm \in H^{\infty,\overline{\mathsf{s}}_\pm }_{\mathrm{par}}
\end{equation} 
to $N(P_\pm) v_\pm = f_\pm$ (assuming that $\varepsilon$ is sufficiently small).
Define $v$ by the usual formula $e^{-ic^2 t} v_- + e^{ic^2 t} v_+$.

What we want to prove is that $u\approx v$ as $c\to\infty$. The strategy will be to apply \Cref{thm:inhomog} with the difference $w=u-v$ in place of $u$. In order to do this, we need to know that there exists $m,\ell\in \bbR$ and a variable order $\mathsf{s}_1$ satisfying the required threshold conditions such that
\begin{itemize}
	\item $w$ lies, for each $c>1$ sufficiently large, in the domain $\calX^{m,\mathsf{s}_1,\ell}$,
	\item the function $\tilde{F}$ defined by $Pw=\tilde{F}$ lies in $\calY^{m-1,\mathsf{s}_1+1,\ell-1}$.
\end{itemize}
Then, we get an estimate 
\begin{equation}
	\lVert w \rVert_{\calX^{m,\mathsf{s}_1,\ell}} \lesssim \lVert \tilde{F} \rVert_{\calY^{m-1,\mathsf{s}_1+1,\ell-1}},
	\label{eq:misc_533}
\end{equation}
where the constant is uniform as $c\to\infty$. 
If we can further show that the $\calY^{m-1,\mathsf{s}_1+1,\ell-1}(c)$-norm of $\tilde{F}(c)$ is decaying as $c\to\infty$, then we can conclude that the $\calX^{m,\mathsf{s}_1,\ell}(c)$-norm of $w(c)$ is decaying as $c\to\infty$.  This is the goal. (And, of course, we would like to take $m,\mathsf{s}_1,\ell$ as large as possible in order to get an optimal result.) In terms of 
\begin{equation} 
\calY^{m-1,\mathsf{s}_1+1,\ell,1} = c^{-1} \calY^{m-1,\mathsf{s}_1+1,\ell-1},
\end{equation} 
the goal is to bound $\lVert \tilde{F} \rVert_{\calY^{m-1,\mathsf{s}_1+1,\ell,1}}$ uniformly as $c\to\infty$. Then, we get a bound on $\lVert w \rVert_{\calX^{m,\mathsf{s}_1,\ell+1}}$.

First, $v_\pm \in H_{\mathrm{par}}^{\infty,\overline{\mathsf{s}}_\pm}$ implies, by \Cref{lem:par_comp_better} 
that $v_\pm \in H_{\calc}^{\infty,\mathsf{s},\infty,0}$, so 
\begin{equation}
	v = e^{-ic^2t} v_- + e^{ic^2t} v_+ \in  H_{\calctwo}^{\infty,\mathsf{s},\infty;0,0}.
\end{equation}
Consequently, the difference $w=u-v$ satisfies $w\in  H_{\calctwo}^{\infty,\mathsf{s},\infty;0,0}$.

Let us now compute $\tilde{F}=Pw$. 
\begin{align}
	\begin{split}
		Pw = & Pu - e^{ic^2t}P_+ v_+ - e^{-ic^2t}P_-v_-
		\\ = & Pu - e^{ic^2t}N(P_+) v_+ - e^{-ic^2t}N(P_-)v_- \\ 
		&\qquad\qquad + e^{ic^2t}(N(P_+)-P_+)v_+ + e^{-ic^2t}(N(P_-)-P_-)v_-
		\\ = & F - e^{i c^2t}f_+ - e^{-ic^2t}f_- + e^{ic^2 t}(N(P_+)-P_+)v_+ + e^{-ic^2t}(N(P_-)-P_-)v_-
		\\ = & e^{ic^2 t}(N(P_+)-P_+)v_+ + e^{-ic^2t}(N(P_-)-P_-)v_-.
	\end{split}
\end{align}
So, 
\begin{equation} 
	\tilde{F}= e^{ic^2 t}(N(P_+)-P_+)v_+ + e^{-ic^2t}(N(P_-)-P_-)v_-.
\end{equation} 
By the definition of the normal operator, 
\begin{equation}
	N(P_\pm) - P_\pm \in \Psi_{\calc}^{2,0,2,-1}.
\end{equation}
So, from $v_\pm \in H^{\infty,\overline{\mathsf{s}}_\pm}_{\mathrm{par}}$, \Cref{lem:par_comp_better}, it follows that
\begin{equation} 
	(N(P_\pm)-P_\pm)v_\pm \in H^{\infty,\mathsf{s},\infty,1}_{\calc} = c^{-1} H^{\infty,\mathsf{s},\infty,0}_{\calc},
\end{equation} 
and therefore
\begin{equation}
	e^{ic^2 t}(N(P_+)-P_+)v_+ , e^{-ic^2t}(N(P_-)-P_-)v_-
	\in H^{\infty,\mathsf{s},\infty;1,1}_{\calctwo}.
\end{equation}
Thus, 
\begin{equation} 
	\tilde{F}\in H^{\infty,\mathsf{s},\infty;1,1}_{\calctwo} = H^{\infty,\mathsf{s}_0+1,\infty;1,1}_{\calctwo} = c^{-1}H^{\infty,\mathsf{s}_0+1,\infty;0,0}_{\calctwo}.
\end{equation} 

So, we can apply \cref{eq:misc_533} with $m,\ell$ arbitrarily large and $\mathsf{s}_1=\mathsf{s}_0$; the requisite hypotheses have been checked. We conclude that
\begin{equation}
	\lVert w \rVert_{H_{\calctwo}^{m,\mathsf{s}_0,\ell;0,0}} \lesssim
     c^{-1} \lVert  \tilde{F} \rVert_{H^{m,\mathsf{s}_0+1,\ell+1;1,1}_{\calctwo}} = O\Big(\frac{1}{c} \Big).
\end{equation}

In particular, we have 
\begin{equation}
	\|u-v\|_{H_{\calctwo}^{m,s_0,\ell;0,0}} = O\Big( \frac{1}{c} \Big)
	\label{eq:misc_542} 
\end{equation}
for every $s_0<-3/2$. 

Taking $m,\ell$ large enough and using \Cref{prop:Sobolev_embedding}, we have:
\begin{proposition}  \label{prop: inhomogeneous_warmup}
	Given the setup above, 
	\begin{equation}
		\lVert u - v \Vert_{(1+r^2+t^2)^{3/4+\varepsilon}L^\infty(\bbR^{1,d}) } = O(1/c)
	\end{equation}
	holds in the $c\to\infty$ limit, for any $\varepsilon>0$. 
\end{proposition}
\begin{proof}
	Follows from \cref{eq:misc_542} using the $\natural$-Sobolev embedding lemma, \Cref{prop:Sobolev_embedding}.
\end{proof}

\subsection{Irregular forcing}
\label{subsec:main}
We now turn to $F$ as in \cref{eq:irreg_forcing}, where  $m,\ell\in \bbR$ are given and $\mathsf{s}\in C^\infty({}^{\calctwo}\overline{T}^* \bbM)$ is as in the previous subsection. (In particular, $\mathsf{s}_0=\mathsf{s}-1$ also satisfies the hypotheses of \Cref{thm:inhomog}.)
So, according to \Cref{thm:inhomog}, there exists a unique 
\begin{equation}
	u\in H_{\calctwo}^{m,\mathsf{s},\ell;0,0}
\end{equation}
solving $Pu=F$, and our goal is to approximate $u\approx v$ for some $v$. The strategy is still to estimate $u-v$ using \Cref{thm:inhomog}, but the details now become tedious.

Let $\overline{\mathsf{s}}_\pm = \mathsf{s}|_{\mathrm{pf}_\pm}$. Fix $\varepsilon\in (0,1)$ small enough such that $\mathsf{s}-2\varepsilon$ satisfies the same threshold hypotheses (in particular, \cref{eq:misc_142}) as $\mathsf{s}$, and let $U_\pm$ denote a neighborhood of $\mathrm{pf}_\pm$ on which $\overline{\mathsf{s}}_\pm$ (pulled back via the par,I-blowdown map) is bounded below by  $\mathsf{s}-\varepsilon$.

Fix $O\in \Psi_{\calczero}^{-\infty,0,0}$ satisfying the following:
\begin{itemize}
	\item $O=1$ microlocally near $\mathrm{pf}_-\cup \mathrm{pf}_+$: $
	\operatorname{WF}'_{\calczero}(1-O)\cap (\mathrm{pf}_-\cup \mathrm{pf}_+) = \varnothing$, 
	\item  $\operatorname{WF}'_{\calczero}(O)$ is disjoint from $\mathrm{df}$ and the zero-energy subspace $\mathrm{cl}_{{}^{\calczero}\overline{T}^* \bbM}\{\tau_\natural=0\}$,
	\begin{equation}
		\operatorname{WF}'_{\calczero}(O)\cap (\mathrm{df}\cup \mathrm{cl}_{{}^{\calczero}\overline{T}^* \bbM}\{\tau_\natural=0\})=\varnothing,
	\end{equation}
	\item $\operatorname{WF}'_{\calczero}(O)$ is contained in $U_-\cup U_+$. 
\end{itemize}
Let $O_\pm\in  \Psi_{\calc}^{-\infty,0,0}$ as well, such that
\begin{itemize}
	\item 
	$O_\pm=1$ on the portion of the essential support of $O$ in $\{\pm \tau_\natural >0\}$, 
	\begin{equation}
		\operatorname{WF}'_{\calczero}(1-O_\pm ) \cap ( \operatorname{WF}'_{\calczero}(O) \cap  \mathrm{cl}_{{}^{\calczero}\overline{T}^* \bbM}\{\pm \tau_\natural>0\} )  = \varnothing ,
	\end{equation}
	\item $O_\pm$ is microsupported away from $\mathrm{df}$ and the other portion of the microsupport of $O$: 
	\begin{equation}
		\operatorname{WF}'_{\calczero}(O_\pm) \cap (\mathrm{df}\cup( \operatorname{WF}'_{\calczero}(O) \cap  \mathrm{cl}_{{}^{\calczero}\overline{T}^* \bbM}\{\mp \tau_\natural>0\} )) = \varnothing , 
	\end{equation}
	\item $O_\pm$ is also microsupported within $U_\pm$. 
\end{itemize}

The microsupports of the various microlocal cutoffs $O,O_\pm$ are depicted in \Cref{fig:O_my_O}. 
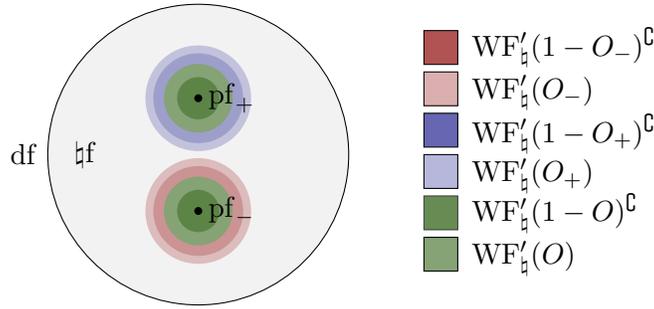
\begin{figure}[h!]
	\begin{tikzpicture}
		\filldraw[fill=lightgray!20] (0,0) circle (2);
		\fill[darkblue, fill opacity=.25] (0,.75) circle (20pt);
		\fill[darkblue, fill opacity=.25] (0,.75) circle (17pt);
		\fill[darkred, fill opacity = .25] (0,-.75) circle (20pt);
		\fill[darkred, fill opacity=.25] (0,-.75) circle (17pt);
		\fill[darkgreen!60] (0,-.75) circle (13pt);
		\fill[darkgreen, fill opacity=.5] (0,-.75) circle (8pt);
		\fill[darkgreen!60] (0,.75) circle (13pt);
		\fill[darkgreen, fill opacity=.5] (0,.75) circle (8pt);
		\fill[black] (0,.75) circle (1.5pt) node[right] {$\mathrm{pf}_+$};
		\fill[black] (0,-.75) circle (1.5pt) node[right] {$\mathrm{pf}_-$};
		\node at (-1.5,0) {$\natural\mathrm{f}$};
		\node[left] at (-2,0) {$\mathrm{df}$};
		\filldraw[fill= darkred, fill opacity=.75] (3,1.2) rectangle (3.45,1.65); 
		\filldraw[fill= darkred, fill opacity=.35] (3,.65) rectangle (3.45,1.1); 
		\filldraw[fill= darkblue, fill opacity=.75] (3,.1) rectangle (3.45,0.55); 
		\filldraw[fill= darkblue, fill opacity=.35] (3,-.45) rectangle (3.45,0); 
		\filldraw[fill=darkgreen, opacity=.75] (3,-1) rectangle (3.45,-.55); 
		\filldraw[fill=darkgreen!60] (3,-1.55) rectangle (3.45,-1.1); 
		\node[right] at (3.5,1.425) {$\operatorname{WF}'_{\calczero}(1-O_-)^\complement$};
		\node[right] at (3.5,-1.375) {$\operatorname{WF}'_{\calczero}(O)$};
		\node[right] at (3.5,-.775) {$\operatorname{WF}'_{\calczero}(1-O)^\complement$};
		\node[right] at (3.5,-.275) {$\operatorname{WF}'_{\calczero}(O_+)$};
		\node[right] at (3.5,.825) {$\operatorname{WF}'_{\calczero}(O_-)$};
		\node[right] at (3.5,.275) {$\operatorname{WF}'_{\calczero}(1-O_+)^\complement$};
	\end{tikzpicture}
	\caption{The fibers of ${}^{\calczero}\overline{T}^* \bbM$, showing the microsupports of $O,O_\pm$ and related operators. The regions are supposed to be  drawn convex; they are only drawn as annular because we are drawing multiple disks on top of each other. }
	\label{fig:O_my_O}
\end{figure}

We now define the building blocks of the ansatz $v$ used to approximate $u$. Define $g_\pm$ by $e^{\pm i c^2 t} g_\pm = O_\pm F$. By construction, 
	\begin{equation}
	g_\pm \in H_{\mathrm{par,I,res}}^{\infty,\overline{\mathsf{s}}_\pm+1-\varepsilon,\ell-1 ,0}.
	\label{eq:g-pm-membership}
	\end{equation}
 So, there exists, for each $c$, a unique (assuming $\varepsilon$ is sufficiently small) $v_{\pm,0}(c) \in H_{\mathrm{par}}^{\infty, \overline{\mathsf{s}}_\pm-\varepsilon}$  solving
    \begin{equation} \label{eq:def-v-pm-0}
        N(P_\pm)v_{\pm,0} = g_\pm.
    \end{equation}
This is appealing to the analysis in \cite[Appendix B]{NRL_I} \emph{for each individual $c$}. The operator, $N(P_\pm)$, does not depend on $c$. However, the forcing $g_\pm$ does, and so it is still nontrivial to understand how $v_{\pm,0}$ behaves as $c\to\infty$. 
    
In order to understand this, we discuss some estimates in $\Psi_{\mathrm{par,I,res}}$, the pseudodifferential calculus discussed in \Cref{subsec:par-I-res}, that we will need for estimating $v_{\pm,0}$ and that are not explicitly spelled out in \cite{NRL_I}. 
All of the estimates are analogous to those in \cite[Prop. 5.2--5.5]{Parabolicsc}, except that our estimates are in terms of $\mathrm{par,I,res}$-norms now. Their proofs go through almost verbatim. Because $N(P_\pm)$ does not depend on $c$, the Hamiltonian flow of $N(P_\pm)$ on the $\mathrm{par,I}$-phase space is independent of $c$. This means that it has a good source-to-sink structure. Blowing up the corner of the phase space (recall this is what the `res' in `par,I,res' refers to) does not change this. 
So, the argument leading up to \cite[Prop.\ 5.5]{Parabolicsc} applies, \emph{mutatis mutandis}. We can almost use the same symbols in the proof (for instance, the same $c$-independent quadratic-defining-function of the radial sets $\calR_{\mathrm{Schr}}$) except that when constructing the commutant at the level of symbols, we now have access to (and need to make use of) a larger class of (microlocal) weights:  powers of the boundary-defining-functions $\rho_{\mathrm{f}}$ for $\mathrm{f}$ a boundary hypersurface of the $\mathrm{par,I,res}$-phase space. This allows us to prove estimates using the $\mathrm{par,I,res}$-Sobolev norms, instead of just the $\mathrm{par}$-Sobolev norms.
So, we arrive at the estimate
\begin{equation} \label{eq:v_pm_0_global-parIres-est}
    \lVert \tilde{v} \rVert_{H_{\mathrm{par,I,res}}^{m,\tilde{\mathsf{s}}_\pm,\ell,q}}  \lesssim \lVert N(P_\pm)\tilde{v}\rVert_{H_{\mathrm{par,I,res}}^{m-1,\tilde{\mathsf{s}}_\pm+1,\ell-1,q}}
    + \lVert \tilde{v} \rVert_{H_{\mathrm{par,I,res}}^{-N,-N,-N,q}} 
\end{equation}
for any variable orders $\tilde{\mathsf{s}}_\pm \in C^\infty({}^{\mathrm{par,I,res}} \overline{T}^* \bbM )$ satisfying the relevant threshold conditions \emph{on the $\mathrm{par,I,res}$-phase space}.\footnote{For example, the pullback to the $\mathrm{par,I,res}$-phase space of any variable order on the $\mathrm{par}$-phase space satisfying the threshold conditions from \cite{Parabolicsc} works. This, namely $\tilde{\mathsf{s}}_\pm$ the pullback of $\overline{\mathsf{s}}_\pm-\varepsilon$, is the case that is actually applied.}
This estimate, and \cref{eq:g-pm-membership}, gives
\begin{equation}
    v_{\pm,0}\in H_{\mathrm{par,I,res}}^{\infty,\overline{\mathsf{s}}_\pm-\varepsilon,\ell,0}.
\end{equation} 
    Define $v_\pm$ by 
	\begin{equation}
	v_\pm = e^{\mp ic^2 t} O  (e^{\pm ic^2 t} v_{\pm,0} ).
    \label{eq:misc_sdf}
	\end{equation}
The cutoff here is just for later technical convenience. 

Next, we discuss what to do with the remaining part of the forcing, $(1-O_--O_+)F$. 
Next, we solve the inhomogeneous \emph{free} Klein--Gordon equation
	\begin{equation}
	P_0 u_{
    0} = (1-O(O_-+O_+)) F - [P,O] (e^{- ic^2 t} v_{-,0} ) - [P,O] (e^{ ic^2 t} v_{+,0} );
    \label{eq:S4_complicated_KG_part}
	\end{equation} 
    this somewhat complicated choice of forcing essentially consists of 
    \begin{itemize}
        \item $(1-O_--O_+)F$, which was to be expected, except we have thrown in an extra cutoff $O$ for later convenience,  
        \item ``error terms'' $[P,O] (e^{- ic^2 t} v_{-,0} ) $, $[P,O] (e^{ ic^2 t} v_{+,0} )$ arising because we decided to define $v_\pm$ from $v_{\pm,0}$ by throwing in an extra cutoff $O$ (see \cref{eq:misc_sdf}) instead of $v_\pm = v_{\pm,0}$.
    \end{itemize}
	The operators $[P,O] \in \Psi_{\calctwo}^{-\infty,-1,1;-\infty,-\infty}$ and $(1-O(O_-+O_+)) \in \Psi_{\calctwo}^{0,0,0;-\infty,-\infty}$
	have essential supports away from $\mathrm{df}$ and $\mathrm{pf}_-\cup \mathrm{pf}_+$, so 
	\begin{equation}
	[P,O] (e^{\pm ic^2 t} v_{\pm ,0} ) \in H_{\calctwo}^{\infty,\mathsf{s}-2\varepsilon+1,\ell-1;\infty,\infty},
	\quad (1-O(O_-+O_+)) F \in H_{\calctwo}^{m-1,\mathsf{s}+1,\ell-1;\infty,\infty}.
	\end{equation}
	Thus, \Cref{thm:inhomog}, applied to the free Klein--Gordon operator $P_0$ in place of $P$, gives 
 \begin{equation} \label{eq:u0-membership}
    u_{0} \in H_{\calctwo}^{m,\mathsf{s}-2\varepsilon,\ell;\infty,\infty} 
 \end{equation}    
satisfying \cref{eq:S4_complicated_KG_part}.

Now let 
\begin{equation} 
v=  e^{-ic^2 t} v_- +  e^{ic^2 t} v_+ + u_0 = O(e^{-i c^2 t} v_{-,0} + e^{ic^2 t} v_{+,0}) + u_0. 
\end{equation} 
We will check that $u\approx v$.

The strategy to show that $u\approx v$ is identical to that employed in the previous subsection: use \Cref{thm:inhomog} to estimate the difference $w=u-v$. So, we will not repeat the exposition here --- we will just begin the argument. The first thing to note is that
\begin{equation}
	u,v \in H_{\calctwo}^{\infty,\mathsf{s}-2\varepsilon,\ell;0,0}.
\end{equation}
So, $w=u-v$ satisfies $w \in  H_{\calctwo}^{\infty,\mathsf{s}-2\varepsilon,\ell;0,0}$ as well.
A direct computation shows:

\begin{align}
\begin{split}
Pw &= Pu-Pv = F- P(O(e^{-i c^2 t} v_{-,0} + e^{ic^2 t} v_{+,0}) + u_0) \\ 
&= F - P_0 u_0 + (P_0-P)u_0 - O P (e^{-i c^2 t} v_{-,0} + e^{ic^2 t} v_{+,0}) \\ &\qquad\qquad\qquad\qquad- [P,O] (e^{-i c^2 t} v_{-,0} + e^{ic^2 t} v_{+,0}) \\
&= F - P_0 u_{0}  - O  ( O_- + O_+ )F  -  [P,O] (e^{-i c^2 t} v_{-,0} + e^{ic^2 t} v_{+,0})  
\\ &\qquad\qquad-O(e^{-i c^2 t} (P_--N(P_-)) v_{-,0} + e^{ic^2 t} (P_+ -N(P_+))v_{+,0}) +(P_0-P)u_{0}\\ 
&= -O(e^{-i c^2 t} (P_--N(P_-)) v_{-,0} + e^{ic^2 t} (P_+ -N(P_+))v_{+,0}) +(P_0-P)u_{0}  ,
\end{split}
\end{align}
where in the second to the third equation we rearranged orders of terms and used $e^{\pm ic^2t} N(P_\pm)v_{\pm,0}=e^{\pm ic^2t} g_\pm = O_\pm F$.

The final terms on the right-hand side are controlled using: 
\begin{equation}
P_\pm - N(P_\pm) \in \Psi_{\mathrm{par,I,res}}^{4,0,2,-1} ,\quad P-P_0 \in \Psi_{\calctwo}^{2,0,1;0,0}, \;  
\end{equation}
and \cref{eq:u0-membership}.

Altogether, we have
\begin{equation}
Pw \in H_{\calctwo}^{m-2, \mathsf{s}-2\varepsilon,\ell-2;1,1 } =  c^{-1} H_{\calctwo}^{m-2, \mathsf{s}-2\varepsilon,\ell-3;0,0 }.
\end{equation}
So, \cref{eq:misc_533} yields that $\lVert u - v \rVert_{H_{\calctwo}^{m-1,\mathsf{s}-1-\varepsilon,\ell-2;0,0 } } = O(c^{-1})$.
More precisely: 
\begin{theorem}
\label{thm:inhomogeneous-NRL2}
	Given the setup above, we have, for each $\varepsilon>0$, the estimate
	\begin{equation}
	\lVert u - v \rVert_{H_{\calctwo}^{m-1,\mathsf{s}-1-\varepsilon,\ell-2;0,0 } } \lesssim c^{-1}\lVert F \rVert_{H_{\calctwo}^{m-1,\mathsf{s}+1,\ell-1;0,0}}
    \label{eq:misc_yu1}
	\end{equation}
	as $c\to\infty$, where the implicit constant is independent of $F$.
\end{theorem}
\begin{proof}
    In the argument above, we only required that $F\in H_{\calctwo}^{m-1,\mathsf{s}+1,\ell-1;0,0}$. So, going through the same argument, quantitatively, we get \cref{eq:misc_yu1}.
\end{proof}

\begin{remark}
    As far as getting an estimate of the form 
    \begin{equation}
	\lVert u - v \rVert_{H_{\calctwo}^{m-1,\mathsf{s}-1-\varepsilon,\ell-K;0,0 } } \lesssim c^{-k}\lVert F \rVert_{H_{\calctwo}^{m-1,\mathsf{s}+1,\ell-1;0,0}}, \quad k,K>0
	\end{equation}
    is concerned, if we do not care what $K>0$ is, then the only factor limiting us to $k\leq 1$ is the order of $P_\pm - N(P_\pm)$ at $\mathrm{pf}$. In particular, if this difference is order $-2$, i.e.\ $O(c^{-2})$, at $\mathrm{pf}$, for whatever reason, then we get an estimate with $k=2$.
\end{remark}

\begin{remark}[An alternative splitting]  
Let $t_{\calczero}=c^2t$ be as in \Cref{subsec:natural} and suppose that the temporal Fourier transform $\calF_{t_{\calczero} \to \tau_{\calczero}} F$ is $C^1$ in $\tau_{\calczero}$. Then, we can split $\calF_{t_{\calczero} \to \tau_{\calczero}} F$ into two parts, 
	\begin{equation}
	\calF_{t_{\calczero}\to\tau_{\calczero}} F = \scrF_1 + \tau_{\calczero} \scrF_2, 
	\end{equation}
	where 
	\begin{align}
	\scrF_1 &= 2^{-1} (\calF_{t_{\calczero} \to \tau_{\calczero}} F(\tau_{\calczero}) + \calF_{t_{\calczero}\to\tau_{\calczero}} F(-\tau_{\calczero}) ) \\
	\scrF_2 &= 2^{-1} \tau_{\calczero}^{-1} (\calF_{t_{\calczero}\to\tau_{\calczero}} F(\tau_{\calczero}) - \calF_{t_{\calczero}\to\tau_{\calczero}} F(-\tau_{\calczero}) )
	\end{align}
	are the $\tau_{\calczero}$-even and $\tau_{\calczero}$-odd parts of $\calF_{t_{\calczero} \to \tau_{\calczero}} F$, respectively.  Now define 
	\begin{equation}
	F_+ = \frac{e^{-ic^2 t}}{2}\mathcal{F}^{-1}_{\tau_{\calczero} \to t_{\calczero}}(\mathscr{F}_1 + \mathscr{F}_2), \quad 
	F_- = \frac{e^{ic^2 t}}{2} \mathcal{F}^{-1}_{\tau_{\calczero} \to t_{\calczero}}(\mathscr{F}_1 - \mathscr{F}_2).
	\end{equation}
	Then, 
	\begin{equation} 
	f  - \frac{1}{2}(h^2D_t+1)(e^{it/h^2}F_+) - \frac{1}{2}(-h^2D_t+1)(e^{-ic^2t}F_-) \in H_{\calctwo}^{m-1,\mathsf{s}+1,\ell-1;1,1}.
	\end{equation}
	which is roughly how we decompose the delta type forcing in the advanced, retarded and Cauchy problems in the following section.
    The point of this decomposition is that $(\pm h^2D_t+1)$ are precisely those operators with symbols vanishing at $\mathrm{pf}_\mp$ to the first order and equals to $1$ to the first order at $\mathrm{pf}_\pm$.
\end{remark}

\section{Asymptotics for the Cauchy problem}\label{sec:Cauchy}
Next, we turn to the analysis of the Cauchy problem. We begin in \S\ref{subsec:advanced/retarded} with the analysis of the advanced/retarded problem with $\delta(t)$-type forcing 
\begin{equation}
F = \delta(t) f(x,c) + c^{-2} \delta'(t) g(x,c) 
\label{eq:delta-forcing}
\end{equation}
for $f(-,c),g(-,c)\in \calS(\bbR^d)$. \Cref{prop:main_advanced/retarded} can be considered our main result there.  Then, in \S\ref{subsec:Cauchy}, we turn to the Cauchy problem proper. Our main result there is a proof of \Cref{thm:Cauchy}.

We discuss the advanced/retarded problem first because our treatment of the Cauchy problem involves a reduction to the advanced/retarded problem. This reduction is done in the usual way: if $u$ satisfies the Cauchy problem, then, letting $\Theta(t)$ denote a Heaviside function, $\Theta u, (1-\Theta)u$ will satisfy the advanced/retarded problems with a $\delta(t)$-type forcing. The reason why this reduction to the advanced/retarded problem is necessary (or at least useful) is that, in order to control $w=u-v$ in the Cauchy problem, it is natural to propagate control forwards and backwards from a neighborhood of the Cauchy surface $\mathrm{cl}_\bbM\{t=0\}$ in $\bbM$, but this requires \emph{already} having control in spacetime cones of the form $\{|t|<\varepsilon r\}$ for some $\varepsilon>0$. At the outset, we do not have such control (unlike in \S\ref{sec:inhomogeneous}, where we did have a priori control). In contrast, the natural way to control $\Theta w ,(1-\Theta)w$ is to propagate control from the far past/future, in one of which $\Theta w,(1-\Theta)w$ vanishes identically, by construction. Indeed, this is the propagation argument in \S\ref{subsec:advanced/retarded}.
There are a few natural concerns arising from the fact that the Schr\"odinger radial sets contain points over the equator of $\partial \bbM$, which intersects the support of $\Theta,1-\Theta$, but the argument goes through.

\begin{remark}[Aforementioned concerns]\hfill
	\begin{itemize} 
		\item The fact that $\calR$ has, at $c=\infty$, points over the equator means that it is not a priori obvious that $\Theta w,(1-\Theta)w$ are controlled as $c\to\infty$ on any component of the radial set $\calR$, even though this holds for each individual $c<\infty$. So, where are we propagating control from? Fortunately, the ``above threshold'' radial point estimates in \cite{NRL_I} just require that, for each \emph{individual} $c<\infty$, the above-threshold condition is satisfied on the speed-$c$ Klein--Gordon radial set. Since $\Theta w,(1-\Theta)w$ automatically vanish near half of the components of the Klein--Gordon radial set for every individual $c<\infty$, this suffices to verify the hypotheses of the above-threshold estimate. We can therefore actually \emph{conclude} that each of $\Theta w,(1-\Theta) w$ is controlled on the relevant components of $\calR$, even though it was not initially obvious.

		\begin{figure}[h!]
			\begin{center}
				\includegraphics{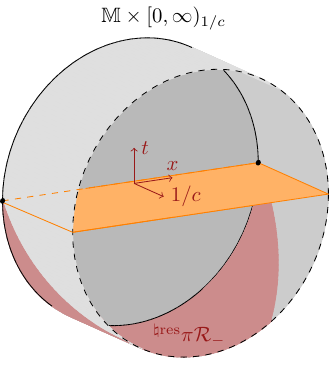}
			\end{center}
			\caption{The projection ${}^{\calc}\pi \calR_-$, in red, of $\calR_-$ onto the base $\bbM\times [0,\infty)_{1/c}$, and $\mathrm{cl}_\bbM\{t=0\}\times [0,\infty)_{1/c}$, in orange. The two sets intersect over the equator of $\partial \bbM$ (black dots). As explained in \Cref{rem:worries}, this intersection might cause some worries, but the argument ends up going through.}
		\end{figure}

		\item Another worry might be whether the wavefront set of the $\delta(t)$-type forcing intersects the characteristic set, or worse, the radial sets. This does actually happen  in the $\calc$-cotangent bundle, showing that the concern is justified, but it does not happen in the $\calctwo$-cotangent bundle. See \Cref{fig:Bfig}. That $\operatorname{WF}_{\calctwo}(\delta)$ intersects $\mathrm{pf}_\pm$ is the reason why, in \cite{NRL_I}, we made sure to prove an estimate with microlocalizers --- they are required to mollify the $\delta$-function. 
	\end{itemize}
	\label{rem:worries}
\end{remark}

So, what we will do to control the Cauchy problem is use the results in \S\ref{subsec:advanced/retarded} to control $\Theta w,(1-\Theta) w$ and then assemble those pieces to control $w$. 

It should be stressed that once the advanced/retarded problem is under control, controlling the Cauchy problem is totally straightforward. The only thing that needs much elaboration is how the Schr\"odinger initial data is related to the Klein--Gordon initial data. This comes from combining the following:
\begin{enumerate}
    \item the relation between the Schr\"odinger initial data and the forcing of the advanced/retarded problem to which it is reduced,
    \item the relation between the Klein--Gordon initial data and the forcing of the advanced/retarded problem to which it is reduced, 
    \item the relation between the Schr\"odinger forcing and Klein--Gordon forcing in the non-relativistic limit of the advanced/retarded problem. (We will discuss this below; see \cref{eq:advanced/retarded_fdecomp}.)
\end{enumerate}
We will have to check that this is essentially equivalent to matching derivatives.

A key point below, in the discussion of the advanced/retarded problem with $\delta$-type forcing, will be to split the forcing $F = \delta(t) f(x) + c^{-2}\delta'(t) g(x)$ into two parts,
\begin{equation}
F= e^{-ic^2 t} F_- + e^{ic^2 t} F_+ 
\label{eq:misc_176}
\end{equation}
such that, if we define $v_\pm$ as the solutions to 
\begin{equation} 
	N(P_\pm) v_\pm = F_\pm
	\label{eq:misc_17a}
\end{equation} 
in some appropriate space, then $u\approx v$ holds for $v$ given by the usual formula $v=e^{-ic^2 t} v_-+e^{ic^2 t} v_+$.
Unless we get $F_\pm$ \emph{exactly} right at $c=\infty$, $u\approx v$ will not hold. So, we must choose $F_\pm$ well. The natural guess is to write 
\begin{equation}
F_\pm = \delta(t)  f_{\pm}(x) + c^{-2} \delta'(t) g_{\pm}(x) 
\label{eq:misc_175}
\end{equation}
for some $f_\pm(x),g_\pm(x)$. Indeed, given $f_\pm,g_\pm$, if we define $F_\pm$ by \cref{eq:misc_175}, then \cref{eq:misc_176} holds if and only if 
\begin{align}
g(x) &= g_{-}(x) + g_{+}(x) \\
f(x) &= f_{-}(x) + f_{+}(x) + i g_{-}(x) - i g_{+}(x).
\label{eq:misc_177}
\end{align}
Of course, this system under-determines $f_{\pm},g_{\pm}$, but suppose we impose the additional constraint that $\exp(\pm 2i c^2 t) F_\pm$ have $\calc$-Sobolev regularity at $\mathrm{pf}$. This is essentially saying that we split $F$ correctly: the two terms in \cref{eq:misc_176} do not secretly harbor a term that oscillates like the other of the two terms.
Concretely, 
\begin{equation}
e^{\pm 2i c^2 t} F_\pm = \delta(t) (f_{\pm}(x) \mp 2i g_{\pm}(x) ) + c^{-2} \delta'(t) g_{\pm}(x).
\label{eq:misc_5k1}
\end{equation}
and what we want is that this is $O(c^{-1})$ \emph{at $\mathrm{pf}$}.
The second term, $c^{-2} \delta'(t) g_{\pm}(x)$ is suppressed by $c^{-2}$ at $\mathrm{pf}$. It should be subleading there, so ignore it. The leading term should therefore be $\delta(t) (f_{\pm}(x) \mp 2i g_{\pm}(x) )$. 
Requiring this to vanish (remember that we are free to specify $f_\pm,g_\pm$) gives two more equations: 
\begin{equation}
f_{\pm}(x) = \pm 2i g_{\pm}(x).
\label{eq:misc_179}
\end{equation}
The system \cref{eq:misc_176}, \cref{eq:misc_179} now determines $f_\pm(t,x)$: 
\begin{equation}
f_\pm = f\pm i g\text{ and }g_\pm = 2^{-1}(g\mp i f).
\label{eq:advanced/retarded_fdecomp}
\end{equation}
So, 
\begin{equation} 
    F_\pm = \delta(t) (f\pm i g) + 2^{-1}c^{-2} \delta'(t) (g\mp i f).
    \label{eq:Fpm}
\end{equation}

Note that $g$, despite being suppressed by $O(c^{-2})$ in the definition of $F$, contributes to the leading term in $F_\pm$. Thus, a $c^{-2}\delta'(t)$ term in $F$ is ``leading order'' in some sense. This might seem to contradict our throwing out of  $c^{-2} \delta'(t) g_{\pm}(x)$ above, in $F_\pm$. The difference is that the relevant phase space to analyze $\smash{e^{\pm 2ic^2 t}F_\pm}$ is the $\calc$-phase space, whereas the relevant phase space to analyze $F$ is the $\calctwo$-phase space. While the $c^{-2} \delta'(t) $ term in $F$ might be lower-order than $\delta(t)$ \emph{at $\mathrm{pf}$}, this is not what matters --- what matters is that it is the same order \emph{at $\mathrm{pf}_\pm$}. This is what multiplying by $\smash{e^{\pm i c^2 t}}$ brings out; that $\delta'(t)$ is secretly $\Omega(c^2)$ relative to $\delta(t)$ at $\mathrm{pf}_\pm$ corresponds to the fact that 
\begin{equation} 
e^{\pm ic^2 t} \delta'(t) = \delta'(t)\mp i c^2 \delta(t) 
\end{equation} 
has a manifest $c^2 \delta(t)$ term in it. Alternatively one can just reason that $\partial_t = O(c^2)$ as a $\natural$-operator, so $\delta'(t)$ should be considered $O(c^2)$  relative to $\delta(t)$ at $\natural\mathrm{f}$, which contains the loci blown up to get $\mathrm{pf}_\pm$.

However, we \emph{will} be able to ignore a $c^{-2} \delta'(t)$ term in $F_\pm$, since $F_\pm$ is the forcing fed to a $c$-\emph{independent} PDE, $N(P_\pm)$. Thus, rather than define $v_\pm$ 
as the solution of the Schr\"odinger equation \cref{eq:misc_17a} 
with forcing $F_\pm$, as defined by \cref{eq:Fpm},  we can instead define $v_\pm$ by the Schr\"odinger equation with simpler forcing
\begin{equation}
N(P_\pm)v_\pm = \delta(t) (f\pm i g)
\end{equation}
--- notice the lack of a  $\delta'(t)$ term ---
and this suffices for constructing an approximation to $u$. This observation will be used in the application to the Cauchy problem, because, whereas the Cauchy problem for Klein--Gordon can be reduced to the advanced/retarded problem with forcing of the form \cref{eq:delta-forcing}, having both a $\delta(t)$ term and a $\delta'(t)$ term, the initial-value problem for Schr\"odinger is reduced to the advanced/retarded problem with simple $\delta(t)$ forcing. Thus, the $v_\pm$ above will be relatable to the solutions of a suitable initial-value problem problem.

We will give a second argument for \cref{eq:advanced/retarded_fdecomp} below.

\subsection{\texorpdfstring{The advanced/retarded problem with $\delta$-type forcing}{The advanced/retarded problem with delta-type forcing}}
\label{subsec:advanced/retarded}

The advanced/retarded problem with $\delta$-type forcing (allowed to depend on $c$) reads
\begin{equation}
\begin{cases}
Pu = \delta(t) f (x,c) +  c^{-2} \delta'(t) g (x,c) \\ 
u(t,x)=0\text{ whenever }\varsigma t<0,
\end{cases}
\label{eq:advanced/retarded}
\end{equation}
where the sign $\varsigma \in \{-,+\}$ toggles between the advanced and retarded problems. 
For forcings, we assume 
\begin{equation}
   f,g\in  C^\infty((1,\infty]_c;\calS(\bbR^d))  \label{eq:advanced/retarded_forcing_generic}
\end{equation}
and some statements below will be strengthened when $f,g$ have no $O(1/c^2)$ in their $c\to\infty$ expansions and \cref{eq:quadratic-converge-1} and \cref{eq:quadratic-converge-2} are satisfied as well.

The well-posedness of the advanced/retarded problems for the Klein--Gordon equation tells us that there exist unique $u=u^\varsigma \in \calS'$ solving \cref{eq:advanced/retarded}.

The forcing in \cref{eq:advanced/retarded} is a sum of distributions of the form 
\begin{equation}
F(t,x,c) = \delta^{(k)}(t) f(x,c) 
\label{eq:misc_403}
\end{equation}
for some $k\in \bbN$, $f\in C^\infty((1,\infty]_c;\calS(\bbR^d))$. Let us study the $\calctwo$-wavefront set $\operatorname{WF}'_{\calctwo}(F)$ defined in \cref{eq:calctwo-WF-def}. 

Unless there are some cancellations between the contributions to $u$ from the different $\delta$-functions in the forcing, we should expect that the $\calctwo$-wavefront sets of the $F$'s have to be present in the $\calctwo$-wavefront sets of $u$. 
What matters for us is that this wavefront set is disjoint from the $\calctwo$-characteristic set of $P$ and therefore does not obstruct propagation estimates. But, $\operatorname{WF}_{\calctwo}(\delta(t))$ \emph{will} intersect the characteristic set. Indeed, $\delta(t)$ will have $\operatorname{WF}_{\calctwo}$ at the $\xi=0$ part of $\mathrm{pf}_\pm$ over the equator of $\partial \bbM$, because $\delta(t)$ is not rapidly decaying as $x\to\infty$, and the $\mathrm{par}$-wavefront set detects this. The next lemma shows that this portion of the wavefront set of $\smash{\delta^{(k)}}$ does not enter $F$, because when we form $F=\smash{\delta^{(k)} f}$, the multiplication by $f$ (which is Schwartz for each $c>0$, with control as $c\to \infty$) kills it.

To be more precise, let $\calB$ denote the portion  of the closure $\operatorname{cl} \{\xi=0\}$ (in the $\calctwo$-phase space) of $\{\xi=0\} \subset T^* \bbR^{1,d} \times \bbR^+_c$ contained in $ \mathrm{df}\cup \natural\mathrm{f}$ and over $ \mathrm{cl}_\bbM\{t=0\}$:
\begin{equation} \label{eq:calB-def}
    \calB \overset{\mathrm{def}}{=} 
    \operatorname{cl} \{\xi=0\} \cap (\mathrm{df} \cup \natural \mathrm{f}) \cap \mathrm{cl}\{t=0\}.
\end{equation}

Then, $\calB$ is disjoint from the $\calctwo$-characteristic set $\Sigma$ of $P$:
\begin{equation} \label{eq:calB-property}
\calB\cap \Sigma =\varnothing.
\end{equation} 
See \Cref{fig:Bfig}. (Because the closure is \emph{in} ${}^{\calctwo}\overline{T}^* \bbM$ and not in the topology induced by the blowdown map ${}^{\calctwo}\overline{T}^* \bbM \to {}^{\calczero}\overline{T}^* \bbM$, $\mathrm{cl}\{\xi=0\}$ does not include all of $\mathrm{pf}_\pm $, as can be seen from \Cref{fig:Bfig}.\footnote{In contrast, $\{\xi_\natural=0\}$ makes prima facie sense as a subset of ${}^{\calctwo}T^* \bbM$ (note that we are excluding $\mathrm{df}$)  and contains all of $\mathrm{pf}_\pm$. But, because $\mathrm{cl}\{\xi=0\}$ is defined instead by starting with $\{\xi=0\}$ \emph{inside the interior} of the $\calctwo$-phase space and \emph{only then} taking a closure, the set $\calB$ ends up only containing a subset of $\mathrm{pf}_\pm$.})
The key point is:

\begin{proposition}
	For $F$ defined by \cref{eq:misc_403}, we have $\operatorname{WF}_{\calctwo}(F) \subseteq \calB$. 
    \label{prop:delta_WF}
\end{proposition}
In other words, if $B\in \Psi_{\calctwo}^{0,0,0;0,0}$ satisfies $\operatorname{WF}^{\prime}_{\calctwo}(B)\cap \calB = \varnothing$, then, for any $m,\ell\in \bbR$ and for any variable $\calctwo$-order $\mathsf{s}$, 
	\begin{equation}
	\lVert B F \rVert_{H_{\calctwo}^{m,\mathsf{s},\ell;0,0} }= O(1) 
	\end{equation}
	as $c\to\infty$.
\begin{proof}
    Because $F$ is supported at $t=0$, we know that $\operatorname{WF}_{\calctwo}(F)$ lies only over $ \mathrm{cl}_\bbM\{t=0\}$.

	In order to see that $\operatorname{WF}_{\calctwo}(F)$ is contained inside the closure $\calC = \mathrm{cl} \{\xi=0\}$ of the set $\{\xi=0\}$ (again, this is the closure in the $\calctwo$-phase space of a subset of the interior), we can use e.g.\ $\triangle \delta^k(t) = 0$, where $\triangle$ is the spatial Laplacian. 
    Since 
    \begin{equation} 
		\partial_{x_j} \in \operatorname{Diff}_{\calc}^{1,0,1,0}
	\end{equation} 
	and does not change when conjugated by $\exp(i\alpha t)$ for $\alpha\in \bbR$, we have $\partial_{x_j} \in \operatorname{Diff}_{\calctwo}^{1,0,1;0,0}$, so
	\begin{equation} 
	\triangle \in \operatorname{Diff}_{\calctwo}^{2,0,2;0,0},
	\end{equation} 
	and its $\calctwo$-characteristic set is precisely $\calC \cap (\mathrm{df}\cup\mathrm{bf}\cup\natural\mathrm{f})$. (Recall that all characteristic sets are contained in $\mathrm{df}\cup\mathrm{bf}\cup\natural\mathrm{f}$.) So, elliptic regularity  in the $\calctwo$-calculus suffices to give that the wavefront set \begin{equation} 
    \operatorname{WF}_{\calctwo}(\delta^{(k)}(t))\subset \calC .
    \end{equation}
    It easily follows that $\operatorname{WF}_{\calctwo}(F)$ is contained in $\calC$. (For example, this follows from \cref{eq:misc_523} and the fact that $\calctwo$-pseudodifferential operators do not spread $\calctwo$-wavefront set.)

    The last thing we need to check is that $\operatorname{WF}_{\calctwo}(F)$ is contained in $\mathrm{df}\cup \natural\mathrm{f}$, thus ruling out wavefront set in $\mathrm{bf} \backslash (\mathrm{df}\cup \natural\mathrm{f})$. This is where we use that multiplying $\delta^k(t)$ by $f$ kills some of the former's $\calctwo$-wavefront set. 
    To prove this, note that we can write
	\begin{equation} 
		F(t,x,c)=\delta^{(k)}(t) f_{\mathrm{ext}}(t,x,c)
	\end{equation} 
	for some\footnote{Any Schwartz $f_{\mathrm{ext}}$ agreeing with $f$ at $t=0$ and whose first $k-1$ temporal derivatives of $f_{\mathrm{ext}}$ identically vanish near $\operatorname{cl}_\bbM\{t=0\}$ works.}   $f_{\mathrm{ext}} \in \calS(\bbR^{1,d})$. (The ``ext'' stands for ``extension.'') The multiplication map $M_{f_{\mathrm{ext}}}$ satisfies 
	\begin{equation}
	M_{f_{\mathrm{ext}}} \in \operatorname{Diff}_{\calctwo}^{0,-\infty,0;0,0}. 
    \label{eq:misc_523}
	\end{equation}
	The full left symbol of $M_{f_{\mathrm{ext}}}$  is just $f_{\mathrm{ext}}$, which, when viewed as a function on ${}^{\calctwo}\overline{T}^* \bbM$ not depending on the frequency variables or on the speed of light, lacks essential support at $\mathrm{bf}$ except where $\mathrm{bf}$ intersects $\natural\mathrm{f} \cup \mathrm{df}$. 
    (No $\mathrm{pf}_\pm$ is required in the previous sentence. Again, this is because essential support is defined to be in $\mathrm{df}\cup\mathrm{bf}\cup\natural\mathrm{f}$.) Thus:
    \begin{equation}
        \operatorname{WF}'_{\calctwo}(M_{f_{\mathrm{ext}}}) \subseteq (\mathrm{df}\cup\natural\mathrm{f})
    \end{equation}     
    So, we can conclude that $\operatorname{WF}_{\calctwo}(F)\subseteq (\mathrm{df}\cup\natural\mathrm{f})$,
	which completes the proof. 
\end{proof}

\begin{figure}
	\begin{center}
		\includegraphics{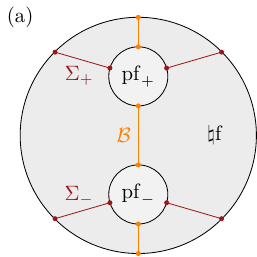}
		\qquad 
		\includegraphics{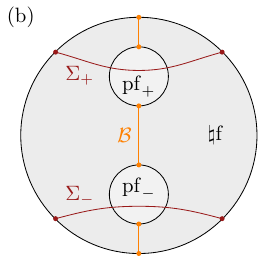}
        \hspace{2em}
        \includegraphics{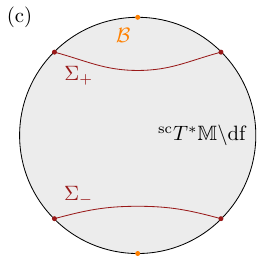}
	\end{center}
	\caption{(a) The set $\calB$ as it sits inside ${}^{\calctwo}\overline{T}^*\bbM$ over a point in $\{t=0\}\subset \bbM^\circ=\bbR^{1,d}$ and at $c=\infty$. 
	(b) The situation over 	$\partial \bbM\cap \operatorname{cl}_\bbM\{t=0\}$ is similar, except over $\partial \bbM $ the characteristic set continues through $\mathrm{pf}_\pm\backslash \natural\mathrm{f}$ rather than stopping at $\mathrm{pf}_\pm \cap \natural\mathrm{f}$. The key point is that $\calB$ does not intersect the characteristic set. (c) The same as (b), but for $h>0$. Note that the only portion of $\calB$ here is at fiber infinity.}
	\label{fig:Bfig}
\end{figure}

We want to exhibit the solution $u=u^\varsigma$ of \cref{eq:advanced/retarded} in the form $u= e^{-ic^2 t} u_- + e^{ic^2 t} u_+$ for $u_\pm\approx v_\pm$, where $v_{\pm}=v_\pm^\varsigma$ solve the advanced/retarded problem for the Schr\"odinger equation:
\begin{equation}
\begin{cases}
N(P_\pm) v_\pm = \delta(t) \phi_\pm(x) \\
v_\pm(t,x) = 0\text{ whenever } \varsigma t<0
\end{cases}
\label{eq:misc_407}
\end{equation}
for some $\phi_\pm=\phi^\varsigma_\pm \in  \calS(\bbR^d)$ which will need to be determined later. The $v_\pm$ so-defined satisfy 
\begin{equation}
v_\pm \in H_{\mathrm{par}}^{-N,\overline{\mathsf{s}}_\pm}
\end{equation}
for variable orders $\overline{\mathsf{s}}_\pm$ as above and $N\in \bbR$ sufficiently large.
Then, let $v= e^{-ic^2 t} v_- + e^{ic^2 t} v_+$. 
If $u_\pm \approx v_\pm$, then we should expect $u\approx v$. Proving this will take a bit of work. In \S\ref{subsub:vpm_props}, we will discuss the properties of $v_\pm$. In \S\ref{subsub:init_data}, we will discuss further the choice of initial data $\phi_\pm$ fed to \cref{eq:misc_407}. Finally, the estimates on $u-v$ are found in \S\ref{subsub:main}.

\begin{figure}
	\includegraphics{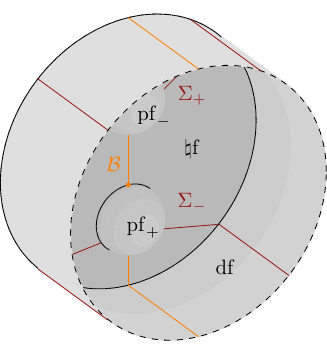}
	\caption{The same sets depicted in \Cref{fig:Bfig}, $\Sigma$ and $\calB$, shown over a point $(t,x)\notin \partial \bbM$ (so this corresponds to \Cref{fig:Bfig}(a)), except we are now including the $h$ coordinate.}
\end{figure}

\subsubsection{ \texorpdfstring{Properties of $v_\pm,v$}{Properties of v+,v-,v }}
\label{subsub:vpm_props}
The $\mathrm{par,I,res}$-wavefront set $\operatorname{WF}_{\mathrm{par,I,res}}(v_\pm)$ of $v_\pm$ receives contributions from three places: 
\begin{enumerate}
    \item In the elliptic region of $N(P_\pm)$ within the $\mathrm{par,I,res}$-phase space, it lies precisely where the $\mathrm{par,I,res}$-wavefront set of $\delta(t)\phi_\pm$ is.  \label{item:v-pm-WF-1}
    \item Inside the $\mathrm{par,I,res}$-characteristic set of $N(P_\pm)$, except for the outgoing radial set (where what ``outgoing'' means depends on whether we are studying the advanced problem or the retarded problem), $\operatorname{WF}_{\mathrm{par,I,res}}(v_\pm)$ lies where the $\mathrm{par,I,res}$-wavefront set of $\delta(t)\phi_\pm$ can flow to along $N(P_\pm)$'s Hamiltonian flow, 
    \item the incoming radial set, at which $v_\pm$ has as much below-threshold regularity as the forcing allows.
\end{enumerate}
We discussed this in \S\ref{sec:inhomogeneous}. The difference between that discussion and this discussion is that the forcing here, $\delta(t)\phi_\pm$, has 
\begin{equation} 
    \operatorname{WF}_{\mathrm{par,I,res}}(\delta(t)\phi_\pm)\neq \varnothing
\end{equation}
(generically). In fact, 
\begin{equation} 
\operatorname{WF}_{\mathrm{par,I,res}}(\delta(t)\phi_\pm)\subseteq \calB_{\mathrm{par}}\overset{\mathrm{def}}{=}\operatorname{cl}_{{}^{\mathrm{par,I,res}}\overline{T}^*\bbM } \big(\{\xi = 0, t=0\} \backslash \mathrm{df}_1 \big)
\end{equation} 
(with equality generically), as the argument used to prove \Cref{prop:delta_WF} shows. Consequently, $\operatorname{WF}_{\mathrm{par,I,res}}(\delta(t)\phi_\pm)$ is completely disjoint from the $\mathrm{par,I,res}$-characteristic set of $N(P_\pm)$. So, within the characteristic set, we have empty wavefront set:
\begin{equation}
\operatorname{WF}_{\mathrm{par,I,res}}^{m,s_0,\ell,0}(v_\pm) \cap \operatorname{char}_{\mathrm{par,I,res}}^{2,0,2,0}(N(P_\pm)) =\varnothing
\end{equation}
for all $s_0<-1/2$ and $m,\ell\in \bbR$. In fact, we have \begin{equation} 
\operatorname{WF}_{\mathrm{par,I,res}}^{m,s_0,\ell,0}(v_\pm) \cap \calB_{\mathrm{par}} = \varnothing.
\label{eq:misc_529}
\end{equation}

Given that $N(P_\pm)$ is the leading part of $P$ at $\mathrm{pf}_\pm$, we really only have a right to expect $u\approx v$ near $\mathrm{pf}_-\cup \mathrm{pf}_+$. If we want a more global approximation, we should presumably write $u\approx Ov$ for $O$ such that:
\begin{align}  \label{eq:O-def-1}
\begin{split}
& O\in \Psi_{\calczero}^{-\infty,0,0}
\text{ has essential support close to } \mathrm{pf}_-\cup \mathrm{pf}_+\text{ and }
\\& 1-O \text{ has essential support away from } \mathrm{pf}_-\cup \mathrm{pf}_+.
\end{split}
\end{align} 
  Fortunately, the difference between $v$ and $Ov$, for $v_\pm$ given by \cref{eq:misc_407} for Schwartz $\phi_\pm$, is unimportant, owing to the following:
\begin{proposition}
	Let $v =v^\varsigma_\pm$ be defined by \cref{eq:misc_407} for $\phi_\pm=\phi^\varsigma_\pm \in \calS(\bbR^d)$, $O$ be as above, and $\chi \in C^\infty(\bbM)$ have support disjoint from  $\operatorname{cl}_\bbM\{t=0\}$, then, 
	\begin{equation}
	\lVert M_\chi (1-O) v \rVert_{H_{\calctwo}^{m,s_0,\ell;0,0} } = O\Big( \frac{1}{c^\infty} \Big)
	\end{equation}
	as $c\to\infty$, 
	for every $m,\ell\in \bbR$ and $s_0<-1/2$. Consequently, via the Sobolev embedding theorem \Cref{prop:Sobolev_embedding}, $M_\chi (1-O)v \to 0$ uniformly in compact subsets of $\bbR^{1,d}$ as $c\to\infty$, faster than any polynomial in $1/c$. 
	\label{prop:ansatz_insensitivity}
\end{proposition}
Intuitively, the reason why this holds is that owing to the smoothness of $v_\pm$ away from $\{t=0\}$, it is suppressed by $O(1/c^\infty)$ at $\natural\mathrm{f}$ (which lies at high frequency!), as long as we are measuring it in a Sobolev space with below-threshold decay. 
\begin{proof}
We have 
	\begin{equation} 
	\lVert M_\chi (1-O) v \rVert_{H_{\calctwo}^{m,s_0,\ell;0,0} } \leq \lVert  (1-O) M_\chi v \rVert_{H_{\calctwo}^{m,s_0,\ell;0,0} }+\lVert [M_\chi, O] v \rVert_{H_{\calctwo}^{m,s_0,\ell;0,0} }.
    \label{eq:misc_532}
	\end{equation}

    \Cref{eq:misc_529} tells us that 
\begin{equation}
\operatorname{WF}_{\calctwo}^{m,s_0,\ell}(c^N e^{\pm i c^2 t}M_{\chi} v_\pm ) = \varnothing 
\end{equation}
near the essential support of $1-O$.\footnote{It is necessary to translate between $\mathrm{par,I,res}$-regularity and $\calctwo$-regularity, but this can be done e.g.\ using \S\ref{sec:relations}. We have an infinite amount of $\mathrm{par,I,res}$-regularity at $\mathrm{df}_1$, so this is straightforward.} This means that $\lVert  (1-O) M_\chi v \rVert_{H_{\calctwo}^{m,s_0,\ell;0,0} } = O(1/c^N)$. 

Likewise, $\operatorname{WF}_{\calctwo}^{m,s_0,\ell}(c^N e^{\pm i c^2 t}v_\pm ) = \varnothing$ near the essential support of $[M_\chi, O]$ (which is disjoint from both where $O=0$ essentially and where $O=1$ essentially, hence disjoint from $\mathrm{pf}_\pm$ and $\mathrm{df}$). This means that 
\begin{equation} 
    \lVert [M_\chi, O] v \rVert_{H_{\calctwo}^{m,s_0,\ell;0,0} } = O(1/c^N).
\end{equation}
\end{proof}

\begin{figure}[t]
	\centering 
	\begin{tikzpicture}[scale=.85]
		\filldraw[fill=lightgray!20] (0,0) circle (2.5);
		\filldraw[fill=lightgray!20] (0,1) circle (.6);
		\filldraw[fill=lightgray!20] (0,-1) circle (.6);
		\begin{scope}
			\clip (0,0) circle (2.5);
			\draw[darkblue] (0,2.5) to[out=180,in=135] (-.5,1.3) to[out=-45,in=180] (0,1) to[out=0,in=45] (.5,1.3) to[out=45,in=0] (0,2.5);
			\draw[darkblue] (.02,2.5) -- (.02,1.6);
			\draw[darkblue] (.02,.4) -- (.02,-.4);
			\draw[darkblue] (.02,-2.5) -- (.02,-1.6);
			\draw[darkred] (-.02,2.5) -- (-.02,1.6);
			\draw[darkred] (-.02,.4) -- (-.02,-.4);
			\draw[darkred] (-.02,-2.5) -- (-.02,-1.6);
			\draw[darkred] (0,-2.5) to[out=180,in=-135] (-.5,-1.3) to[out=45,in=180] (0,-1) to[out=0,in=-45] (.5,-1.3) to[out=-45,in=0] (0,-2.5);
			\node at (-1.5,0) {$\natural\mathrm{f}$};
			\draw[dotted] (0,.6) -- (-1,.8) node[left] {$\mathrm{pf}_+$};
			\draw[dotted] (0,-.6) -- (-1,-.8) node[left] {$\mathrm{pf}_-$};
		\end{scope}
		\begin{scope}[shift={(.25,-.5)}] 
			\node[darkred, right] at (3.5,-.275) {$\operatorname{WF}_{\calctwo}(e^{-ic^2 t} v_-)$};
			\filldraw[fill=darkred] (3,-.525) rectangle (3.5,-.025);
			\node[darkblue, right] at (3.5,.5) {$\operatorname{WF}_{\calctwo}(e^{+ic^2 t} v_+)$};
			\filldraw[fill=darkblue] (3,.75) rectangle (3.5,.25);
		\end{scope} 
	\end{tikzpicture}
	\caption{The projection of the $\calctwo$-wavefront set of $v$ onto frequency space at $c=\infty$. The portion in $\natural\mathrm{f}$ is contained entirely over the closure of $\{t=0\}$. The $\calB$ portion is over the whole $\mathrm{cl}\{t=0\}$, while the rest is only over $\partial \bbM$. The portion of the wavefront set in $\mathrm{pf}_\pm$ is only over $\partial \bbM$. The $\mathrm{par,I,res}$-wavefront set looks similar, except the $\calB$ portion stays disjoint from the other portion also at fiber infinity. }
\end{figure}
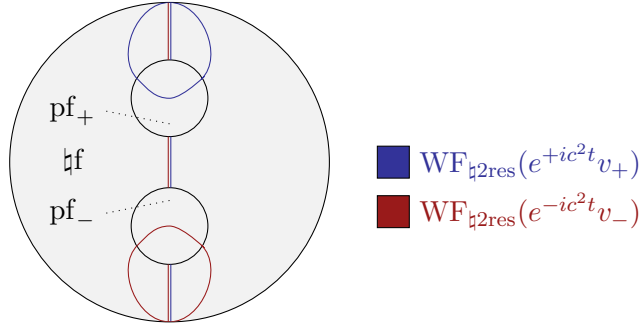

\subsubsection{Discussion of initial data}
\label{subsub:init_data}
Let us now discuss the choice of $\phi_\pm$ in \cref{eq:misc_407}. We have already given one argument for the correct choice,  \cref{eq:advanced/retarded_fdecomp}. Now we give another.

Presumably, we should choose these such that, defining $v = e^{-ic^2 t} v_- + e^{ic^2 t} v_+$,
\begin{equation} 
Pv \approx \delta(t) f(x) +  c^{-2} \delta'(t) g(x),
\label{eq:misc_205}
\end{equation} 
where the $\approx$ means neglecting terms that look like they should be suppressed as $c\to\infty$, relative to the included terms, and where $f(x)=f(x,c=\infty)$, $g(x)=g(x,c=\infty)$. 
It will turn out (essentially because what matters in the following argument is the normal operator at $\natural\mathrm{f}$) that, to get the correct relationship between $f,g$ and $\phi_\pm$, it suffices to consider the case of the constant-coefficient (i.e.\ free) Klein--Gordon operator:

\begin{example}
	For the constant-coefficient case, 
	\begin{equation}
	P = - \frac{1}{c^2} \frac{\partial^2}{\partial t^2} - \triangle - c^2, \quad N(P_\pm) = \mp 2 i \frac{\partial}{\partial t} - \triangle , \quad
	P_\pm = -\frac{1}{c^2}\frac{\partial^2}{\partial t^2} \mp 2i \frac{\partial}{\partial t} - \triangle  = -\frac{1}{c^2} \frac{\partial^2}{\partial t^2} + N(P_\pm) .
	\label{eq:misc_369}
	\end{equation}
	We can now compute 
	\begin{align}
	Pv = e^{-ic^2 t} P_- v_- + e^{ic^2 t} P_+ v_+ & = \delta(t)( \phi_-(x) +  \phi_+(x)) - \frac{e^{-ic^2 t}}{c^2} \frac{\partial^2 v_-}{\partial t^2} - \frac{e^{ic^2 t}}{c^2} \frac{\partial^2 v_+}{\partial t^2}, \\
	\begin{split}  
	\frac{\partial^2 v_\pm }{\partial t^2} &= \pm \frac{i}{2 } \frac{\partial}{\partial t} (\delta(t) \phi_\pm  + \triangle v_\pm) =  \pm \frac{i}{2} \delta'(t) \phi_\pm \pm \frac{i}{2} \triangle \frac{\partial v_\pm}{\partial t} 
    \\ & =  \pm \frac{i}{2} \delta'(t) \phi_\pm - \frac{1}{4} \triangle ( \delta(t) \phi_\pm + \triangle v_\pm ),
	\end{split} \\ 
	\frac{e^{\pm i c^2 t}}{c^2} \frac{\partial^2 v_\pm}{\partial t^2} & = \frac{1}{c^2}\Big( \pm \frac{i}{2} \delta'(t) \phi_\pm - \frac{1}{4} \triangle ( \delta(t) \phi_\pm + e^{\pm i c^2 t}\triangle v_\pm )\Big) + \frac{1}{2} \delta(t) \phi_\pm , 
	\end{align}
	where the last term comes from $e^{\pm i c^2 t} \delta'(t) = \delta'(t) \mp i c^2 \delta(t)$. 
	So, 
	\begin{equation}
	Pv  \approx \frac{\delta(t)}{2} (\phi_- + \phi_+)  + \frac{i\delta'(t)}{2c^2}(\phi_--\phi_+) .
	\end{equation}
	So, if we want $Pv \approx \delta(t) f(x) +  c^{-2} \delta'(t) g(x)$, then we should take $\phi_\pm  = f \pm i g $. With this choice, $Pv$ given exactly by
	\begin{equation}
	Pv = \delta(t) f(x) + c^{-2} \delta'(t) g(x) + \frac{1}{2c^2} \Big( \delta(t) \triangle f + \frac{1}{2} \triangle^2 v   \Big).
	\end{equation}
	The final, unwanted term is multiplied by $1/c^2$ (and has no very irregular terms, such as $\delta'(t)$), so we should expect it to be negligible. (Remember that the regularity of $v$ is known from the analysis in \cite{Parabolicsc}, as described above, so the fact that $v$ appears on the right-hand side is not problematic.)
    \label{ex:free}
\end{example}

It may seem surprising that this computation requires keeping track of $c^{-2}\delta'(t)$ and $c^{-2}\partial_t^2 \in P_\pm$, even though these terms look like they should be suppressed as $c\to\infty$. We already explained why $c^{-2} \delta'(t)$ should be considered leading order (in the context of our previous derivation of $\phi_\pm = f\pm i g$), and the reasoning regarding the term $c^{-2}\partial_t^2$ in $P_\pm$ is similar.
All of this is just to say that, in \cref{eq:misc_205}, the $\approx$ should be interpreted as meaning up to terms which are lower order at both of $\mathrm{pf}_-,\mathrm{pf}_+$ (which is implied by being lower order at all of the $\calczero$-phase space's $\natural\mathrm{f}$).

Let us give an example of how this heuristic can be applied to systematically determine which terms in $P$ can be ignored; $P$ may have terms like $\smash{c^{-3} \partial_t \partial_{x_j}}$ in it. Despite this having fewer powers of $1/c$ than $\smash{c^{-4}} \partial_t^2$, and despite the two operators having \emph{the same} order at $\natural\mathrm{f}$ (they are both order zero, two orders lower than 
\begin{equation} 
\square\in \operatorname{Diff}_{\calczero}^{2,0,2}
\end{equation}
there), it is only the latter that enters our analysis (in the form of the coefficient we have been calling $\aleph$). The reason is that $c^{-3} \partial_t \partial_{x_j}$ is also suppressed at $\mathrm{pf}_\pm$, whereas $c^{-4} \partial_t^2$ is not: 
\begin{align}
e^{\mp i c^2 t} (c^{-3} \partial_t \partial_{x_j}) e^{\pm i c^2 t} &= c^{-3} \partial_t \partial_{x_j} \pm i c^{-1} \partial_{x_j} \in \operatorname{Diff}_{\calc}^{2,0,0,-1}, \\
e^{\mp i c^2 t} (c^{-4} \partial_t^2) e^{\pm i c^2 t} &= c^{-4} \partial_t^2 \pm 2i c^{-2} \partial_{t} - 1 \in \operatorname{Diff}_{\calc}^{2,0,0,0},
\end{align}
where the key point is the orders at $\mathrm{pf}$: $-1$ and $0$, respectively. So, $c^{-4} \partial_t^2$ is order 0 at $\mathrm{pf}_\pm$ while $c^{-3} \partial_t \partial_{x_j}$ is order $-1$ there:
\begin{equation}
c^{-4} \partial_t^2\in \Psi_{\calctwo}^{2,0,0;0,0}, \quad c^{-3} \partial_t \partial_{x_j} \in \Psi_{\calctwo}^{2,0,0;-1,-1}.
\end{equation} 
Since $\square$ is order $0$ at $\mathrm{pf}_\pm$, we conclude that $c^{-4} \partial_t^2$ has the same order as $\square$ there, whereas $c^{-3} \partial_t \partial_{x_j}$ is suppressed by one $O(1/c)$.

This is why the coefficient $\aleph$ of $c^{-4} \partial_t^2$ entered $N(P_\pm)$, whereas the coefficient of $c^{-3}\partial_t \partial_{x_j}$ did not. 

\subsubsection{ \texorpdfstring{Main estimates on $u-v$}{Main estimates on u-v}}
\label{subsub:main}

Recalling our discuss in \Cref{subsec:assumptions},  
\begin{equation}
\NPcalctwo = \square + \frac{\aleph}{c^4}\frac{\partial^2}{\partial t^2}  + c^2 +\Big(\frac{i\beta}{c^2}\frac{\partial}{\partial t} + i\bfB\cdot \nabla + W \Big) |_{c=\infty}
\end{equation}
agrees with $P$ modulo terms which are lower order at \emph{all three} of $\natural\mathrm{f},\mathrm{pf}_\pm$. That is, 
\begin{equation}
    P- \NPcalctwo \in \operatorname{Diff}_{\calctwo}^{2,-1,1;-1,-1}. 
\label{eq:P-P1-membership_0}
\end{equation}
(For comparison, recall that $P,P_1\in \operatorname{Diff}_{\calctwo}^{2,-1,2;0,0}$.)
Recall that when $P$ satisfied \cref{eq:def-general-quadratic-convergence},
$O(1/c)$ terms in the coefficients of $P$ are missing, then this can be improved to 
\begin{equation}
    P- \NPcalctwo \in \operatorname{Diff}_{\calctwo}^{2,-1,0;-2,-2},
\label{eq:P-P1-membership_1}
\end{equation}
and then all of the estimates below will be improved by a $c^{-1}$-factor.

Notice that we even get one order of spacetime decay: $P,P_1$ are order $0$ at $\mathrm{bf}$, but $P-P_1$ is order $-1$.

We can now go through a computation similar to that in \Cref{ex:free} with $P\approx P_1$ in place of the free Klein--Gordon operator $P_0$, but the conclusion is actually unchanged, as long as we use the correct normal operator $N(P_\pm)$, which takes into account the terms in $P_1-P_0$.\footnote{The $\aleph c^{-4}\partial_t^2$ and $\beta c^{-2} \partial_t$ terms are not in $N(P_\pm)$, but the important terms they contribute when applied to $e^{\pm ic^2 t}$ are.} The formula for $N(P_\pm)$ can be found in \Cref{prop:normal}.

Here is the precise proposition we need:
\begin{proposition} Let $O,v$ be as above. In particular, $\phi_\pm = f \pm i g$.
	For any $P$ as in \S\ref{sec:true_intro},
	\begin{equation}
	P(Ov) = O(\delta(t) (f(x)+c^{-2} r(x,c)) + c^{-2} \delta'(t) g(x) ) +  R_-(e^{-ic^2 t}  v_-) + R_+(e^{ic^2 t} v_+),
	\label{eq:misc_433}
	\end{equation}
	for some 
	\begin{itemize}
		\item $r\in C^\infty([0,1)_{1/c};\calS(\bbR^d_x))$;
		\item $R_\pm \in \Psi_{\calctwo}^{-\infty,0,2;-1,-1}$ having wavefront set disjoint from $\mathrm{df}$. 
	\end{itemize}
    For $w=u-Ov$, we have
	\begin{multline}
	P w = (1-O)(\delta(t) f(x) + c^{-2} \delta'(t) g(x) ) + c^{-2}\delta(t) \tilde{f}(x) + c^{-4} \delta'(t) \tilde{g}(x) - c^{-2} O(\delta(t) r(x,c) )  \\ -R_-(e^{-ic^2 t}  v_-) - R_+(e^{ic^2 t} v_+)
	\label{eq:misc_43o}
	\end{multline}
	for some $\tilde{f},\tilde{g}\in C^\infty((1,\infty]_{c};\calS(\bbR^d_x))$

    In addition, if $f,g$ and $P$ satisfy \cref{eq:def-general-quadratic-convergence} (for $P$, this means \cref{eq:P-P1-membership_1}), then the conclusions above can be strengthened to
	\begin{itemize}
		\item  $r,\tilde{f},\tilde{g}$ have no $O(1/c^2)$ term in their $c\to\infty$ expansions; 
		\item $R_\pm \in \Psi_{\calctwo}^{-\infty,0,2;-2,-2}$.
	\end{itemize}
	\label{prop:Pv_error}
\end{proposition}
\begin{proof}
	We have $P(Ov) = OP(e^{-ic^2 t} v_- + e^{ic^2 t} v_+ ) +[P,O](e^{-ic^2 t} v_- + e^{ic^2 t} v_+ )$. Since 
	\begin{equation} 
	[P,O] \in \Psi_{\calczero}^{-\infty,-1,1}
	\end{equation} 
	has essential support disjoint from $\mathrm{df}$, \emph{as well as $\mathrm{pf}_-,\mathrm{pf}_+$}, this term can be included in $R_\pm$. So, it suffices to verify that $P(\exp(-ic^2 t) v_- + \exp(ic^2 t) v_+ )$ has the same form as the right-hand side of \cref{eq:misc_433}, except without the $O$ in front, and with the condition on $R_\pm$ weakened so as to only state finite order at $\mathrm{df}$.

	Moreover, $P-P_1$ can be included in $R_\pm$, by \cref{eq:P-P1-membership_0}, \cref{eq:P-P1-membership_1}, so it suffices to consider $P_1$. 
    Setting
    \begin{equation}
        P_{1,\pm} = e^{\mp ic^2t}P_1 e^{\pm ic^2t},
    \end{equation}
    we have
	\begin{equation}
	P_{1,\pm} = -c^{-2}\frac{\partial^2}{\partial t^2} + \frac{\aleph}{c^4} \frac{\partial^2}{\partial t^2} \pm 2i \frac{\aleph}{c^2}\frac{\partial}{\partial t} + \frac{i\beta|_{c=\infty}}{c^2}\frac{\partial}{\partial t} + N(P_\pm) .
	\end{equation}
	So, letting $\beth_\pm = \pm 2 \aleph + \beta|_{c=\infty}$, 
	we compute that $P_1v = e^{-ic^2 t} P_{1,-} v_- + e^{ic^2 t} P_{1,+} v_+$ is given by 
	\begin{multline}
	P_1v = \delta(t)( \phi_-(x) +  \phi_+(x)) 
    +  \frac{1}{c^2} \Big(-1+\frac{\aleph}{c^2} \Big)\Big[ e^{-ic^2 t}\frac{\partial ^2v_-}{\partial t^2} +  e^{ic^2 t} \frac{\partial^2 v_+}{\partial t^2} \Big] \\ 
    + \frac{1}{c^2} \Big( e^{-ic^2 t} i\beth_-\frac{\partial v_-}{\partial t} + e^{ic^2 t}i\beth_+ \frac{\partial v_+}{\partial t} \Big).
	\label{eq:misc_443}
	\end{multline}
	Let 
    \begin{equation} \label{eq:def-Lpm}
        L_\pm= N(P_{\pm}) \pm 2 i \partial_t \in \mathrm{Diff}_{\calctwo}^{2,0,2;0,0}.
    \end{equation}  
    This is a second-order differential operator and it has no time derivatives (cancelling with those in $N(P_\pm)$ is what the $\pm 2i\partial_t$ is doing), so it commutes with multiplication by $\exp(-i c^2 t),\exp(i c^2 t)$. 
    We now use this to work out: 
	\begin{align}
	\frac{\partial v_\pm}{\partial t} &= \pm \frac{i}{2} (N(P_{\pm}) - L_\pm )v_\pm = \pm \frac{i}{2} ( \delta(t) \phi_\pm - L_\pm v_\pm ), \\ 
	\begin{split}  
	\frac{\partial^2 v_\pm }{\partial t^2} &= \pm \frac{i}{2 } \frac{\partial}{\partial t} (\delta(t) \phi_\pm - L_\pm v_\pm) =  \pm \frac{i}{2} \delta'(t) \phi_\pm \mp \frac{i}{2} L_\pm \frac{\partial v_\pm}{\partial t} \mp \frac{i}{2} [\partial_t,L_\pm] v_\pm  
    \\ &=  \pm \frac{i}{2} \delta'(t) \phi_\pm + \frac{1}{4} \Big( \delta(t) L_\pm \phi_\pm - L_\pm^2 v_\pm \Big) \mp \frac{i}{2} (\partial_t L_\pm) v_\pm, 
	\end{split} 
	\end{align}
	where $(\partial_t L_\pm)= [\partial_t,L_\pm]$ is the first-order differential operator that results from differentiating the coefficients of $L_\pm$ in $t$. So, 
	\begin{equation}
	e^{\pm ic^2 t}i\beth_\pm \frac{\partial v_\pm}{\partial t} = \mp \frac{1}{2} \delta(t) \beth_\pm|_{t=0} \phi_\pm \pm \frac{1}{2} e^{\pm ic^2 t} \beth_\pm L_\pm v_\pm 
	\end{equation}
	and 
	\begin{multline}
	\Big( -1 + \frac{\aleph}{c^2}\Big) e^{\pm ic^2 t} \frac{\partial^2 v_\pm}{\partial t^2} = \Big( -1 + \frac{\aleph|_{t=0} }{c^2} \Big)\Big( \pm \frac{i}{2} \delta'(t) \phi_\pm+  \delta(t)  \frac{L_\pm \phi_\pm}{4}\Big) \\ 
    +\frac{1}{2} \Big( c^2 \Big( -1 + \frac{\aleph|_{t=0}}{c^2} \Big) \mp i\frac{\partial_t \aleph|_{t=0}}{c^2} \Big) \delta(t) \phi_\pm  
     + \Big(-1+\frac{\aleph}{c^2} \Big) e^{\pm i c^2 t} \Big( - \frac{1}{4} L^2_\pm v_\pm 
    \mp \frac{i}{2} (\partial_t L_\pm) v_\pm \Big).
	\end{multline}
	So, summing up, we can write 
	\begin{equation} \label{eq:OPv}
	P_1 v = \delta(t) (f(x)+c^{-2} r(x,c)) + c^{-2} \delta'(t) (g(x) + c^{-2} \tilde{g}(x))   +  e^{-ic^2 t} \tilde{R}_{-}  v_- +  e^{ic^2 t} \tilde{R}_{+}v_+ ,
	\end{equation}
	for 
	\begin{align}
	r(x) &= \sum_\pm \Big[ \pm \frac{1}{2} \beth_\pm|_{t=0} +\frac{1}{4} \Big(- 1+\frac{\aleph|_{t=0}}{c^2}\Big) L_\pm  +\frac{\aleph|_{t=0}}{2}\mp \frac{i}{2c^2} \frac{\partial_t \aleph|_{t=0}}{\partial t}   \Big]\phi_\pm, \\
	\tilde{R}_{\pm} &= c^{-2}\Big( \pm\frac{1}{2}\beth_\pm L_\pm + \Big(-1+\frac{\aleph}{c^2} \Big) \Big( - \frac{1}{4} L_\pm^2 \mp \frac{i}{2} (\partial_t L_{\pm}) \Big) \Big),
	\end{align}
    and $\tilde{g}  = -\aleph|_{t=0} \cdot g(x)$. 
	Note that $\tilde{R}_\pm$ has no time derivatives in it.
    Consequently, 
    \begin{equation} 
	\tilde{R}_\pm \in \operatorname{Diff}_{\calctwo}^{4,0,2;-2,-2},
	\end{equation} 
   and $\tilde{R}_\pm$ commutes with multiplication by $\exp(\pm i c^2 t)$. So, we can move the $\tilde{R}_\pm$ through $e^{\pm ic^2 t}$ to get the desired formula. 

   The improvements when $f,g$ have no $O(1/c)$ terms and $P$ satisfies \cref{eq:P-P1-membership_1} follows by observing that all error terms are improved by $c^{-1}$ in the process above.
    
\end{proof}

Next we give the following version of \cite[Prop.\ C.3]{NRL_I} restated in terms of the $\calctwo$-norm.

\begin{lemma} \label{lemma:NRLI-PropC.3-calctwo-version}
Suppose $B,G \in \Psi_{\calctwo}^{0,0,0;0,0}$ \footnote{In our application below, $B$ will be elliptic on the characteristic set $\Sigma$, but this is not needed for this estimate to be true.} with $G$ being elliptic on $\WF'_{\calctwo}(B)$ and let $\mathsf{s}\in C^\infty({}^{\calctwo}\overline{T}^* \bbM)$ be a variable order, monotonic under the Hamiltonian flow near the essential support of $B$, such that $\mathsf{s} |_{\calR_{-\varsigma} } > -1/2$ and $\mathsf{s} |_{\calR_{\varsigma} } < -1/2$.

Suppose $w \in \calS'$ satisfies the ``incoming/outgoing'' condition 
	\begin{equation}
		\operatorname{WF}_{\mathrm{sc}}^{m_0,s_0}(w) \cap (\calR^{\mathrm{KG},-}_{-\varsigma}\cup \calR^{\mathrm{KG},+}_{-\varsigma}) = \varnothing
		\label{eq:above-threshold-condition-advanced/retarted}
	\end{equation}
    for each individual $h$, where $\calR^{\mathrm{KG},\pm}_\varsigma$ is the radial set of the Klein--Gordon equation in the component of the characteristic set of energy $\pm \tau>0$ and over the hemisphere of $\partial \bbM$ with $\varsigma t>0$. 
Then we have
\begin{align}
	\begin{split} 
	\lVert  B w \rVert_{H_{\calctwo}^{m,\mathsf{s},\ell;0,0}}  &\lesssim \lVert GPw \rVert_{H_{\calctwo}^{m-1,\mathsf{s}+1,\ell-1;0,0}} + \lVert  Pw \rVert_{H_{\calctwo}^{-N,\mathsf{s}+1,-N;0,0}} ,
	\end{split}
	\label{eq:est-NRLI-C.3}
	\end{align}
holds (in the usual strong sense) for $h \in (0,h_0)$ for some fixed $h_0>0$.
\end{lemma}
 \begin{proof}
      This follows from \cite[Prop.\ C.3]{NRL_I} by taking (in terms of notations there) $Z_\pm$ there to be $\mathrm{Id}$, $\varsigma_+=\varsigma_-=\varsigma$,
      and variable orders
      $\mathsf{s}_{\pm} \in C^\infty({}^{\calc}\overline{T}^* \bbM)$ 
		are chosen so that they coincide with $\mathsf{s}$ translated by one unit in $\tau_\natural$ with $\mathrm{pf}_\pm$ moved to $\mathrm{pf}$ away from $\tau_{\natural} = \mp 2$ (that is, where the blow down of $\mathrm{pf}_\mp$ lives).
      
      Then \cite[Prop.\ C.3]{NRL_I}, which is stated in terms of the $\calc$-norm, implies \cref{eq:est-NRLI-C.3}, which uses the $\calctwo$-norm. 
      More concretely, this follows from the relationship between the $\calc$-norm and the $\calctwo$-norm explained after \cref{eq:def-calcatwo-Sobolev} (with details in \cite[Section~2.10]{NRL_I}, which is referred there as well).
 \end{proof}

We now turn to the main estimate of this section. 

\begin{proposition}
	Let $m,\ell\in \bbR$, and let $\calB$ be as in \cref{eq:calB-def},
	\begin{itemize}
		\item $B \in \Psi_{\calctwo}^{0,0,0;0,0}$ satisfy $\operatorname{WF}^{\prime}_{\calctwo}(B)\cap \calB= \varnothing$,
		\item $\mathsf{s}\in C^\infty({}^{\calctwo}\overline{T}^* \bbM)$ be a variable order, monotonic under the Hamiltonian flow near the essential support of $B$, such that $\mathsf{s} |_{\calR_{-\varsigma} } > -1/2$ and $\mathsf{s} |_{\calR_{\varsigma} } < -3/2$ on the essential support $\operatorname{WF}^{\prime}_{\calc}(B)$ of $B$, 
		\item $O\in \Psi_{\calctwo}^{-\infty,0,0;0,0}$ equals to the identity microlocally in a neighborhood of $\mathrm{pf}_\pm$.
	\end{itemize}
        \begin{enumerate}[label=(\alph*)]
            \item As long as the essential support of $O$ is sufficiently small (depending on $\mathsf{s}$), then, defining $v$ as above, 
	\begin{equation}
	\lVert B (u- O  v) \rVert_{H_{\calctwo}^{m,\mathsf{s},\ell;0,0} } = O\Big( \frac{1}{c} \Big)
	\label{eq:misc_421}
	\end{equation}
	as $c\to\infty$. So, via \Cref{lem:sc_comp} and the Sobolev embedding theorem, $B(u-Ov) \to 0$ uniformly in compact subsets, at a $1/c$ rate.
    \item Moreover, if $m,\ell$ are sufficiently negative, then the same result holds for any $B\in \Psi_{\calctwo}^{0,0,0;0,0}$, in particular $B=1$. 
        \end{enumerate}
    In addition, if $f,g$ and $P$ satisfy \cref{eq:def-general-quadratic-convergence}, then the right hand side of \cref{eq:misc_421} can be improved to $O(1/c^2)$. 
    \label{prop:advanced/retarded_asymp_main}
\end{proposition}
\begin{proof}
    As above, we will only write the proof of the $O(1/c)$ estimate, with the upgrade to $O(1/c^2)$ straightforward.

	We will prove the result when $B$ is elliptic on the characteristic set $\Sigma$. Since $\calB$ is disjoint from $\Sigma$ (and these are both closed), such $B$ exist. Moreover, for \emph{arbitrary} $B$ satisfying the hypotheses of the lemma, we have an elliptic estimate
	\begin{equation}
	\lVert B (u-Ov) \rVert_{H_{\calctwo}^{m,\mathsf{s},\ell;0,0}} \lesssim \lVert B_0 (u-Ov) \rVert_{H_{\calctwo}^{m,\tilde{\mathsf{s}},\ell;0,0}} + \lVert u-O v \rVert_{H_{\calctwo}^{-N,\tilde{\mathsf{s}},-N;0,0}}
	\end{equation}
	for any $N\in \bbR$, where $B_0$ satisfies the same hypotheses of $B$, is elliptic on $\Sigma$, and is  elliptic on the essential support of $B$, and where $\tilde{\mathsf{s}}\geq \mathsf{s}$ near the essential support of $B$ and satisfies the same hypotheses as $\mathsf{s}$ (except everywhere rather than just on the essential support of $B$). We can then apply the proposition (with $B$ replaced by $B_0, \mathrm{Id}$) to get 
	\begin{equation}
	\lVert B_0 (u-Ov) \rVert_{H_{\calctwo}^{m,\mathsf{s}_1,\ell;0,0}} = O\Big(\frac{1}{c}\Big), \; \lVert u-O v \rVert_{H_{\calctwo}^{-N,\tilde{\mathsf{s}},-N;0,0}} = O\Big(\frac{1}{c}\Big) ,
	\end{equation} 
	at least if $N$ is sufficiently large.

	We will apply \Cref{lemma:NRLI-PropC.3-calctwo-version} with $w=u - Ov$. The fact that $w$ satisfies the hypotheses there follows from the fact that $\calR_{-\varsigma}(h)$ is contained, for each individual $h>0$, over $\partial \bbM$ on the north or south hemisphere, without the equator. So, the fact that $u,v$ are vanishing there guarantees that neither of them has sc-wavefront set there, and since $O$ restricted to each individual $h>0$ is in  $\Psi_{\mathrm{sc}}$, the same applies to $w=u-Ov$.
    So \Cref{lemma:NRLI-PropC.3-calctwo-version} gives 
	\begin{align}
	\begin{split} 
	\lVert  B w \rVert_{H_{\calctwo}^{m,\mathsf{s},\ell;0,0}}  &\lesssim \lVert GPw \rVert_{H_{\calctwo}^{m-1,\mathsf{s}+1,\ell-1;0,0}} + \lVert  Pw \rVert_{H_{\calctwo}^{-N,\mathsf{s}+1,-N;0,0}} ,
	\end{split}
	\label{eq:misc_422}
	\end{align}
    where $G$ is a microlocalizer to a small neighborhood of the essential support of $B$.
	We want to show that both terms on the right-hand side are $O(1/c)$ (if $N$ is sufficiently large). 
	\begin{itemize}
		\item First consider $\lVert  Pw \rVert_{H_{\calctwo}^{-N,\mathsf{s}+1,-N;0,0}}$, which is the easier of the two terms to estimate.  
		
		\Cref{eq:misc_43o} gives 
		\begin{equation}
		\lVert  Pw \rVert_{H_{\calctwo}^{-N,\mathsf{s}+1,-N;0,0}} \leq \mathrm{I} + \mathrm{II} + \mathrm{III} + \mathrm{IV} + \mathrm{V} + \mathrm{VI} , 
		\label{eq:misc_471}
		\end{equation}
		where the $j$th term on the right-hand side is the $H_{\calctwo}^{-N,\mathsf{s}+1,-N;0,0}$-norm of the $j$th term on the right-hand side of \cref{eq:misc_43o}. See below if it is not clear which terms are which.
		
		Beginning with $\mathrm{I}$, the fact that $1-O$ has essential support away from $\mathrm{pf}_-\cup \mathrm{pf}_+$ gives 
		\begin{multline}
		\mathrm{I} = \lVert (1-O) (\delta(t) f(x) + c^{-2} \delta'(t) g(x)) \rVert_{H_{\calctwo}^{-N,\mathsf{s}+1,-N;0,0 }} \\ \lesssim \lVert \delta(t) f(x) + c^{-2} \delta'(t) g(x) \rVert_{H_{\calctwo}^{-N,\mathsf{s}+1,-N;-N,-N }}. 
		\label{eq:misc_440} 
		\end{multline}
		If $N\geq 1$ is sufficiently large, then 
		\begin{equation} 
			\lVert \delta(t) f(x) + c^{-2} \delta'(t) g(x) \rVert_{H_{\calctwo}^{-N,\mathsf{s}+1,-N/2;0,0 }} = O(1),
		\end{equation} 
		in which case the right-hand side of \cref{eq:misc_440} is $O(1/c^{N/2})$. 
		Moving on to $\mathrm{II}$, 
		\begin{equation}
		\mathrm{II} = c^{-2} \lVert \delta(t) \tilde{f}(x) \rVert_{H_{\calctwo}^{-N,\mathsf{s}+1,-N;0,0 }} = O(1/c^2) 
		\end{equation}
		if $N$ is large enough. 
		
		Term $\mathrm{III} = c^{-4} \lVert \delta'(t)\tilde{g}(x) \rVert_{H_{\calctwo}^{-N,\mathsf{s}+1,-N;0,0 }}$ is similar, and in fact $O(1/c^4)$. 
		
		Term $\mathrm{IV}$ is 
		\begin{equation}
		c^{-2} \lVert O (\delta(t) r(x,c)) \rVert_{H_{\calctwo}^{-N,\mathsf{s}+1,-N;0,0 }} \lesssim 	c^{-2} \lVert \delta(t) r(x,c) \rVert_{H_{\calctwo}^{-N,\mathsf{s}+1,-N;0,0 }} = O(1/c^2) 
		\end{equation}
		if $N$ is sufficiently large.

		Finally, consider term $\mathrm{V}$ (the term $\mathrm{VI}$ is completely analogous, since it just involves switching $i$ with $-i$). 
		As long as the essential support of $O$ is sufficiently small,\footnote{Taking $O$ to be essentially supported close to $\mathrm{pf}_\pm$ is necessary because the pull-back of the Schr\"odinger radial set to ${}^{\mathrm{par,I,res}}\overline{T}^* \bbM$ diverges from the Klein--Gordon radial set $\calR$ away from $\mathrm{pf}_-\cup \mathrm{pf}_+$; cf.\ \cite[Fig.\ 11]{NRL_I}.
        } then we can find a variable order $\mathsf{s}_0\in C^\infty({}^{\mathrm{par}}\overline{T}^* \bbM )$ such that $\beta^* \mathsf{s}_0> \mathsf{s}$ on the essential support of $R_-$ (which depended on $O$) and such that 
		\begin{equation} 
			v_- \in H_{\mathrm{par}}^{m_0,\mathsf{s}_0}
		\end{equation} 
		for some $m_0\in \bbR$. Then, for $Q\in \Psi_{\calczero}^{0,0,0}$ with essential support disjoint from $\mathrm{df}$ and with $Q=1$ essentially on some sufficiently large neighborhood of $\mathrm{pf}_-\cup \mathrm{pf}_+$ (where what ``sufficiently large'' means can be taken arbitrarily close to $\mathrm{pf}_-\cup \mathrm{pf}_+$ by shrinking the essential support of $O$ closer to $\mathrm{pf}_-\cup\mathrm{pf}_+$), 
		\begin{equation}
		\mathrm{V} = \lVert R_- (e^{-ic^2 t}v_-) \rVert_{H_{\calctwo}^{-N,\mathsf{s}+1,-N;0,0 }}  \lesssim \lVert  R_- Q (e^{-ic^2 t}v_-) \rVert_{H_{\calctwo}^{-N,\mathsf{s}+1,-N;0,0 }} + \lVert v_- \rVert_{H_{\calczero}^{-N,-N,-N}}.
		\end{equation}
		Letting $\tilde{Q} = e^{ic^2t}Qe^{-ic^2 t} \in \Psi_{\calczero}^{-\infty,0,0}$, we have 
		\begin{multline}
		\qquad\lVert R_-  Q(e^{-ic^2 t}v_-) \rVert_{H_{\calctwo}^{-N,\mathsf{s}+1,-N;0,0 }}  \lesssim 	c^{-1} \lVert   Qe^{-ic^2 t}v_- \rVert_{H_{\calctwo}^{-N,\mathsf{s}+1,-N+3;0,0 }} \\ \lesssim 	c^{-1} \lVert \tilde{Q} v_- \rVert_{H_{\calc}^{-N,\mathsf{s}_{+}+1,-N+3,0 }} ,
		\end{multline} 
		where $\mathsf{s}_{+} \in C^\infty({}^{\calc}\overline{T}^* \bbM )$ 
		is a variable $\calc$-order bounded below by $\mathsf{s}$ translated by one unit in $\tau_\natural$, so that $\mathrm{pf}_-$ is moved to $\mathrm{pf}$.  
		Because $\mathsf{s}<-3/2$ on $\calR_\varsigma$, we can choose $\mathsf{s}_+$ such that $\mathsf{s}_++1$ satisfies the usual threshold hypotheses. So, as long as $N$ is large enough, and as long as $\tilde{Q}$ has support close enough to $\mathrm{pf}$ (which is arranged by taking $O$ to be essentially supported close enough to $\mathrm{pf}_-\cup \mathrm{pf}_+$),
		\begin{equation}
			\lVert \tilde{Q} v_- \rVert_{H_{\calc}^{-N,\mathsf{s}_{+}+1,-N+3,0 }} \lesssim 1.
		\end{equation}
		So, we have. 
		\begin{equation} 
		\mathrm{V} = O(1/c).
		\end{equation} 
		
		To summarize, each of terms I, II, III, IV are $O(1/c^2)$, and V, VI are $O(1/c)$. 
		So, all in all, we conclude that 
		\begin{equation} 
			\lVert  Pw \rVert_{H_{\calctwo}^{-N,\mathsf{s}+1,-N;0,0}} = O(1/c).
		\end{equation} 
		\item  Moving on, 
		\begin{equation}
			\lVert GPw \rVert_{H_{\calctwo}^{m-1,\mathsf{s}+1,\ell-1;0,0}}\leq \mathrm{I} + \mathrm{II} + \mathrm{III} + \mathrm{IV} + \mathrm{V} + \mathrm{VI},
		\end{equation}
		where $\mathrm{I},\dots,\mathrm{VI}$ are as above with the extra factor of $G$ inserted, and with the Sobolev orders $m-1,\mathsf{s}+1,\ell-1$ at $\mathrm{df}$, $\mathrm{bf}$, $\natural\mathrm{f}$ respectively.
		
		We then have 
		\begin{equation}
		\mathrm{I} = \lVert G(1-O) (\delta(t) f(x) + c^{-2} \delta'(t) g(x)) \rVert_{H_{\calctwo}^{m-1,\mathsf{s}+1,\ell-1;0,0 }}. 
		\end{equation}
		The product $G(1-O) \in \Psi_{\calczero}^{0,0,0}$ has essential support disjoint from $\mathrm{pf}_-,\mathrm{pf}_+$ \emph{and} $\calB$, which contains the $\calctwo$-wavefront set of $\delta(t) f(x)$, $\delta'(t) g(x)$. So, 
		\begin{equation}
		G(1-O) (\delta(t) f(x) + c^{-2} \delta'(t) g(x)) \in H_{\calczero}^{\infty,\infty,\infty}, 
		\end{equation}
		which means that every Schwartz seminorm of
		$\mathrm{I}$ is $O(1/c^\infty)$.

		Likewise, $\mathrm{II}$ (and $\mathrm{III}$, $\mathrm{IV}$ are both similar, as we saw above) satisfies 
		\begin{equation}
		\mathrm{II} = c^{-2} \lVert G \delta(t) \tilde{f}(x) \rVert_{H_{\calctwo}^{m-1,\mathsf{s}+1,\ell +1;0,0 }} = O(1/c^2) ,
		\end{equation}
		because $G \delta(t) \tilde{f}(x) \in H_{\calctwo}^{\infty,\infty,\infty;0,0} = L^\infty((1,\infty]_c;\calS(\bbR^{1,d}) )$. 
		
		Finally, we consider the $\mathrm{V}$-term, and $\mathrm{VI}$ can be treated in the same way.
		This satisfies
		\begin{equation}
		\mathrm{V}=\lVert GR_- (e^{-ic^2 t}v_-) \rVert_{H_{\calctwo}^{m-1,\mathsf{s}+1,\ell-1;0,0 }}  = 	c^{-1} \lVert  G (c R_-)( e^{-ic^2 t}v_-) \rVert_{H_{\calctwo}^{m-1,\mathsf{s}+1,\ell-1;0,0 }}.
		\end{equation}
		Using the description of $R_-$ in \Cref{prop:Pv_error}, we have $G (c R_-) \in \Psi_{\calctwo}^{-\infty,0,3;0,0}$, so 
		\begin{equation}
			\lVert  G (c R_-)( e^{-ic^2 t}v_-) \rVert_{H_{\calctwo}^{m-1,\mathsf{s}+1,\ell-1;0,0 }} \lesssim \lVert  e^{-ic^2 t} \tilde{A} v_- \rVert_{H_{\calctwo}^{m-1,\mathsf{s}+1,\ell+2;0,0 }} + \lVert e^{-ic^2 t} v_- \rVert_{H_{\calctwo}^{-N,-N,-N;0,0} } 
		\end{equation}
		for any 
		\begin{equation} 
			\tilde{A}\in \Psi_{\calctwo}^{-\infty,0,0;0,0}
		\end{equation} 
		with essential support disjoint from $\mathrm{df}$ and elliptic on the essential support of $G(cR_-)$. In particular, we can choose $\tilde{A}$ to have essential support disjoint from $\calB$. So, as long as the essential support of $O$ is sufficiently small (so that the essential support of $G$ can be taken in a region in which $\mathsf{s}$ obeys the relevant threshold conditions \emph{on the Schr\"odinger radial sets} in the $\mathrm{par,I,res}$-phase space), $\mathrm{V}$ is $O(1/c)$.  
	\end{itemize}
	This completes the proof of part (a) when $m,\ell$ are arbitrary. When $m,\ell$ are sufficiently negative, then the argument bounding $\lVert P w \rVert_{H_{\calctwo}^{-N,\mathsf{s},-N;0,0}}$ suffices to give part (b).
\end{proof} 

\begin{remark}
The microlocal cutoff $B$ played a central role in this argument, mollifying the singularity at $\calB$ introduced into $u,v,w$ by the $\delta$-type forcing. (Of course, only a finite amount of mollification is needed for any individual Sobolev estimate.)

Looking ahead to the application of \Cref{prop:advanced/retarded_asymp_main} to the Cauchy problem, the reader may be concerned about the necessity of $B$ above ---
when studying the solution of the (homogeneous) Cauchy problem, we will want global estimates, so without a microlocalizer such as $B$. However, because $\calB$ is in the \emph{elliptic} set of $P$, it is trivial to establish control there, in terms of the forcing. The Cauchy problem is \emph{homogeneous} (unlike the advanced/retarded problem with $\delta$-type forcing): its solution $u$ satisfies $Pu=0$. The approximant $v\approx u$ does not satisfy $Pv=0$, but it does satisfy $Pv\approx 0$, so $w=u-Ov$ satisfies $Pw\approx 0$. So, elliptic regularity gives a good bound on $w$ in $P$'s elliptic set.

Where we use
\Cref{prop:advanced/retarded_asymp_main} is to study the Cauchy problem's solution \emph{within the characteristic set}, which is disjoint from $\calB$.

\end{remark}

Combining \Cref{prop:ansatz_insensitivity} and \Cref{prop:advanced/retarded_asymp_main}, and using \Cref{prop:Sobolev_embedding}, we have
\begin{proposition}
	As $c\to\infty$, the difference $u-v$ between the solution $u$ of \cref{eq:advanced/retarded} and the ansatz $v$ converges to $0$ uniformly in compact subsets of spacetime not intersecting $\{t=0\}$. Indeed, if $\chi \in C^\infty(\partial \bbM)$ has support disjoint from $\operatorname{cl}_\bbM\{t=0\}$, then, for any $\varepsilon>0$, 
	\begin{equation} 
	\lVert \chi( u-v) \rVert_{(1+r^2+t^2)^{3/4+\varepsilon}L^\infty(\bbR^{1,d}) } =O(1/c)
	\end{equation} 
	as $c\to\infty$.  Under the same conditions under which \Cref{prop:advanced/retarded_asymp_main} can be improved to $O(1/c^2)$, the same is true here.
    \label{prop:main_advanced/retarded}
\end{proposition}

\subsection{Back to the Cauchy problem}
\label{subsec:Cauchy}

We now turn to the Cauchy problem
\begin{equation}
\begin{cases}
Pu=0, \\ 
u|_{t=0} = \varphi, \\ 
u_t|_{t=0} = c^{2} \psi
\end{cases}
\label{eq:Cauchy}
\end{equation}
for $\varphi,\psi \in \calS(\bbR^d)$ (this being well-posed).

In order to prove \Cref{thm:Cauchy}, we want to approximate $u$ by a linear combination 
\begin{equation} 
v=e^{-ic^2 t} v_- + e^{ic^2 t} v_+
\label{eq:v_usual}
\end{equation} 
of solutions of the Schr\"odinger initial value problem
\begin{equation}
\begin{cases}
N(P_\pm) v_\pm = 0 \\
v_\pm(0,x) = \varphi_\pm 
\end{cases}
\label{eq:Schrodinger_Cauchy}
\end{equation}
for some $\varphi_\pm\in  \calS(\bbR^d)$. 
In analogy with our discussion for the advanced/retarded problem, the most reasonable course of action is to choice $\phi_\pm$ such that $v|_{t=0} = \varphi(x)$, $v_t|_{t=0} \approx c^{2} \psi$.
So, we should take 
\begin{equation} 
\varphi_\pm = (\varphi\mp i \psi)/2.
\label{eq:usual_id}
\end{equation} 
\Cref{eq:v_usual}, \cref{eq:Schrodinger_Cauchy}, and \cref{eq:usual_id} together define $v$. Our goal is to prove $u\approx v$, under the stated assumptions.

\subsubsection{Proof sketch}

As we stated in the introduction to this section, the Cauchy problem can be reduced to the advanced/retarded problem (which we studied in the previous subsection, \S\ref{subsec:advanced/retarded}) by setting 
$u^+ = \Theta(t)u,\; u^-=(1-\Theta(t))u$, where $\Theta(t)=1_{t>0}$ is the Heaviside function. 
Then, \cref{eq:Cauchy} is equivalent to 
\begin{equation} \label{eq:def-u-varsigma}
	P u^\varsigma = \delta(t) f_{\varsigma}(x,c) + \frac{\delta'(t)}{c^2} g_{\varsigma}(x,c)  
	\end{equation}
	for some $f_{\varsigma},g_{\varsigma} \in C^\infty((1,\infty]_c;\calS(\bbR^d_x))$ (which can be worked out explicitly); $f_\varsigma \approx \varsigma \psi$, $g_\varsigma\approx \varsigma\varphi$ to leading order in $1/c$. By construction, $u^\pm$ solve the advanced/retarded problem with $\delta$-type forcing.
    Then, by \Cref{prop:ansatz_insensitivity} and \Cref{prop:advanced/retarded_asymp_main}, we can approximate $u^\varsigma$ in the $c\to\infty$ limit by 
\begin{equation} 
v^\varsigma = e^{-ic^2 t} v_-^\varsigma + e^{ic^2 t} v_+^\varsigma,
\end{equation} 
where $v^\varsigma_\pm$ is the solution of the Schr\"odinger equation's advanced/retarded problem discussed in the previous section.\footnote{This means $N(P_\pm)v^\varsigma_\pm = \delta(t) \phi_{\varsigma,\pm}(x)$ for $\phi_{\varsigma,\pm}=\varsigma \psi \pm i \varsigma\varphi$.}
We write $u^\varsigma\approx v^\varsigma$. Consequently, if we let $v=v^-+v^+$ (mind the difference between $v^\pm$ and $v_\pm$ here), then, since $u=u^++u^-$, we have $u\approx v$.

The last step in the proof is to show that the $v$ just defined is the same as that defined by \cref{eq:v_usual}, \cref{eq:Schrodinger_Cauchy}, \cref{eq:usual_id}. This is a straightforward computation.\footnote{Indeed,  \cref{eq:Schrodinger_Cauchy} means that $v_\pm^\varsigma$ should satisfy $N(P_\pm)v^\varsigma_\pm = \pm 2 \varsigma i \varphi_\pm$. But, as recorded in the previous footnote, $N(P_\pm)v^\varsigma_\pm = \delta(t) \phi_{\varsigma,\pm}$ for $\phi_{\varsigma,\pm} = \varsigma \psi \pm \varsigma i \varphi$. These two are consistent, because   $ \varsigma \psi \pm \varsigma i \varphi= \pm 2 \varsigma i \varphi_\pm$, as plugging in \cref{eq:usual_id} yields.}

One small complication is that, in the previous section, we only controlled $u^\varsigma-v^\varsigma$ microlocally away from the ``bad set'' $\calB\subset {}^{\calctwo}\overline{T}^* \bbM$, at which $u^\varsigma$ acquires wavefront set from $Pu^\varsigma\ni \delta(t),\delta'(t)$. So, the argument above only controls $u-v$ away from $\calB$. 
However, we saw that $\calB$ lies in the $\calctwo$-elliptic region of $P$. Because $u$ (unlike $u^\varsigma$) solves the homogeneous problem $Pu=0$, elliptic regularity controls $u$ there. Actually, we want to prove that the difference $w=u-v$ is small, but this is shown similarly, once $Pv$ is known to be small. This will follow from the fact that $v_\pm$ (unlike $v_\pm^\varsigma$) solves the homogeneous problem $N(P_\pm)v_\pm = 0$.\footnote{Actually, we will stick in a microlocalizer: $w=u-Ov$, for $O$ as in \cref{eq:O-def-1}. This is an inessential point.}

\subsubsection{Details}

As before, we can use a microlocalizer to cut things off away from $\mathrm{pf}_-\cup \mathrm{pf}_+$ without worsening the ansatz:
\begin{proposition}
    Let $v$ be defined as above, and let $O \in \Psi_{\calczero}^{-\infty,0,0}$ be as in \cref{eq:O-def-1}.
	Then, 
	\begin{equation}
	\lVert  (1-O) v \rVert_{H_{\calctwo}^{m,s_0,\ell;0,0} } = O\Big( \frac{1}{c^\infty} \Big)
	\end{equation}
	as $c\to\infty$, 
	for every $m,\ell\in \bbR$ and $s_0<-1/2$. Consequently, via \Cref{lem:sc_comp} and the Sobolev embedding theorem, $ (1-O)v \to 0$ uniformly in compact subsets of $\bbR^{1,d}$ as $c\to\infty$, faster than any polynomial in $1/c$. 
	\label{prop:Cauchy_asymp_aux}
\end{proposition}
\begin{proof}
	The proof is the same as the proof of \Cref{prop:ansatz_insensitivity}, except now $v\in H_{\mathrm{par}}^{\infty,s_0}$ for any $s_0<-1/2$ without needing to cut-off near the equator, so we do not need the multiplier $M_\chi$. 
\end{proof}

Our goal is to bound $w=u-Ov$ for $O$ as in \cref{eq:O-def-1}. Away from $\calB$, this follows immediately from  \Cref{prop:advanced/retarded_asymp_main} (after splitting $w=w^+=w^-$ into the advanced/retarded parts $w^+=\Theta(t) w$ and $w^-=(1-\Theta(t))w$, applying the cited proposition to each). So, only the situation near $\calB$ needs investigation. 

Suppose that $A\in \Psi_{\calctwo}^{0,0,0;0,0}$ has essential support in $P$'s elliptic region (such as near $\calB$). Then, we have an estimate 
\begin{equation} \label{eq:est-w-near-calB}
    \lVert A w \rVert_{H_{\calctwo}^{m,\mathsf{s},\ell;0,0}} \lesssim \lVert G P w \rVert_{H_{\calctwo}^{m-2,\mathsf{s},\ell-2;0,0}} + \lVert w \rVert_{H_{\calctwo}^{-N,-N,-N;0,0}},  \quad G \in \Psi_{\calctwo}^{0,0,0;0,0}
\end{equation}
where $G$ is elliptic on the essential support of $A$. The global result in \Cref{prop:advanced/retarded_asymp_main} (part (b)) shows that the final term is $O(1/c)$ (and is improved to $O(1/c^2)$ when the estimates in that proposition are also so improved).\footnote{That proposition was stated for the advanced/retarded problem, meaning the $w$ there is really $\Theta w$ or $(1-\Theta)w$. But since $w=\Theta w + (1-\Theta) w$, the same global estimate applies to what we are calling $w$ here.} In the estimate \cref{eq:est-w-near-calB}, $\mathsf{s}$ is any variable order, including the possibility of constant order.

So, it suffices to estimate the $H_{\calctwo}^{m-2,s,\ell-2}$-norm of $Pw$; then, we have an upper bound on the right-hand side of \cref{eq:est-w-near-calB}, and therefore on $w$. 
\begin{proposition}
	$Pw = -R_-(e^{-ic^2 t}v_-) - R_+(e^{ic^2 t} v_+)$ where $R_\pm \in \Psi_{\calctwo}^{-\infty,0,2;-1,-1}$ and have essential support disjoint from $\mathrm{df}$.
    In addition, if $\varphi,\psi$ and $P$ satisfy \cref{eq:def-general-quadratic-convergence}, 
    this can be improved to 
    $R_\pm \in \Psi_{\calctwo}^{-\infty,0,2;-2,-2}$.
	\label{prop:Pv_error2}
\end{proposition}
\begin{proof}
	This is the result of summing the two right-hand sides of \Cref{prop:Pv_error} (for $w^\pm$). 
\end{proof}
So, $Pw$ is small, as claimed.

What this yields for $w$ is:
\begin{proposition}
	Let $m,\ell\in \bbR$, $s<-3/2$, and $O$ be as in \cref{eq:O-def-1}.
	Defining $v$ as above, 
	\begin{equation}
	\lVert u- O  v \rVert_{H_{\calctwo}^{m,s,\ell;0,0} } = O\Big( \frac{1}{c} \Big)
	\label{eq:misc_42c}
	\end{equation}
	as $c\to\infty$. So, via the Sobolev embedding lemma \Cref{prop:Sobolev_embedding}, $u-Ov \to 0$ uniformly in compact subsets of spacetime. 

    In addition, if $\varphi,\psi$ and $P$ satisfy \cref{eq:def-general-quadratic-convergence}, the right hand side of \cref{eq:misc_42c} can be improved to $O(1/c^2)$.
	\label{prop:Cauchy_main_asymp}
\end{proposition}
\begin{proof}
We prove only the $O(1/c)$ bound, with $O(1/c^2)$ being similar.

    As we explained above, we can write $w=u-Ov$ as $w=w^++w^-$, and then, for any $B\in \Psi_{\calctwo}^{0,0,0;0,0}$, 
    \begin{equation}
        \lVert w \rVert_{H_{\calctwo}^{m,s,\ell;0,0}}  \leq \lVert Bw^+ \rVert_{H_{\calctwo}^{m,s,\ell;0,0}} +\lVert B w^- \rVert_{H_{\calctwo}^{m,s,\ell;0,0}}  +\lVert (1-B) w \rVert_{H_{\calctwo}^{m,s,\ell;0,0}} .
    \end{equation}
    Choose $B$ to be essentially supported away from $\calB$ and such that $1-B$ is essentially supported away from the $\calctwo$-characteristic set of $P$. \Cref{prop:advanced/retarded_asymp_main} tells us that the first two terms on the right-hand side are $O(1/c)$. The last term is estimated using the elliptic estimate \cref{eq:est-w-near-calB}, in which the final term on the right-hand side is already known to be $O(1/c)$.
    So, as stated above, it suffices to prove that the other term on the right-hand side of \cref{eq:est-w-near-calB} is $O(1/c)$:
    \begin{equation}
        \lVert Pw \rVert_{H_{\calctwo}^{m-2,s,\ell-2;0,0}} = O\Big(\frac{1}{c}\Big).
    \end{equation}
    For this, we use \Cref{prop:Pv_error2}, which gives 
	\begin{align} \label{eq:Pw-est-1}
	\begin{split} 
	\lVert Pw \rVert_{H_{\calctwo}^{m-1,s+1,\ell-1;0,0} } &\lesssim \lVert R_-(e^{-ic^2 t} v_-) \rVert_{H_{\calctwo}^{m-1,s+1,\ell-1;0,0} }  + \lVert R_+(e^{+ic^2 t} v_-) \rVert_{H_{\calctwo}^{m-1,s+1,\ell-1;0,0} } \\
	& \lesssim \frac{1}{c} \big( \lVert (c R_-)(e^{-ic^2 t} v_-) \rVert_{H_{\calctwo}^{m-1,s+1,\ell-1;0,0} }  + \lVert e^{+ic^2 t} (cR_+)v_- \rVert_{H_{\calctwo}^{m-1,s+1,\ell-1;0,0} } \big),
	\end{split} 
	\end{align}
	where $c R_\pm \in \Psi_{\calctwo}^{-\infty,0,3;0,0}$.
	Because we are assuming that $s<-3/2$, 
	\begin{equation} 
		v_\pm \in H_{\mathrm{par}}^{\infty,1+s }
	\end{equation} 
	(this is not true for $v^\varsigma_\pm$, since these have ordinary wavefront set at $\calB$, as we saw in the previous subsection). So, \Cref{lem:par_comp_better} implies that, as long as $O$ is essentially supported sufficiently close to $\mathrm{pf}_-\cup\mathrm{pf}_+$, 
	\begin{equation}
	\lVert (c R_-)(e^{\pm ic^2 t} v_\pm ) \rVert_{H_{\calctwo}^{m-1,s+1,\ell-1;0,0} }  = O(1)
	\end{equation}
	as $c\to\infty$. So, all in all, we get \cref{eq:misc_42c}.
\end{proof}

Combining \Cref{prop:Cauchy_asymp_aux} and \Cref{prop:Cauchy_main_asymp} with the $\natural$-Sobolev embedding theorem (say, \Cref{prop:Sobolev_embedding}), we have \Cref{thm:Cauchy}, which we restate here:
\begin{proposition}
	As $c\to\infty$, the difference $u-v$ between the solution $u$ of the Cauchy problem, \cref{eq:Cauchy-intro}, and the ansatz $v = e^{-ic^2 t} v_- +  e^{ic^2 t} v_+$ defined above converges to $0$ in the following sense: for any $\varepsilon>0$, $p\in [2,\infty]$,
	\begin{equation} 
		\lVert u-v \rVert_{(1+r^2+t^2)^{3/4+\varepsilon}L^p(\bbR^{1,d}) } =O(1/c)
		\label{eq:misc_ab7c}
	\end{equation} 
	as $c\to\infty$. 
    Moreover, the same estimate applies to $L(u-v)$ for $L$ any (constant coefficient) differential operator formed out of $c^{-2}\partial_t,\partial_{x_j}$.

    In addition, if $\varphi,\psi$ and $P$ satisfy \cref{eq:def-general-quadratic-convergence}, the right hand side of \cref{eq:misc_ab7c} can be improved to $O(1/c^2)$.
\end{proposition}

The statement on derivatives follows from the observation that $c^{-2}\partial_t \in \operatorname{Diff}_{\calczero}^{1,0,0}$ and \begin{equation} 
\partial_{x_j}\in \operatorname{Diff}_{\calc}^{1,0,1;0,0},
\end{equation} 
so, applying these to $u-v$, we only worsen the $\mathrm{df},\natural\mathrm{f}$ orders, in terms of which $u-v$ is infinitely regular.

\section*{Acknowledgements}

This work began at the \href{https://www.matrix-inst.org.au/}{MATRIX} workshop ``\href{https://www.matrix-inst.org.au/events/hyperbolic-pdes-and-nonlinear-evolution-problems/}{Hyperbolic PDEs and Nonlinear Evolution Problems}.'' We thank MATRIX for the hospitality and inspiring environment. We are particularly grateful to Timothy Candy, whose talk at MATRIX, on \cite{CaHe}, served as the immediate impetus for this project. A.H. and Q.J. are supported by the Australian Research Council through grant FL220100072. E.S.\ is supported by the National Science Foundation under grant number DMS-2401636.
A.V. gratefully acknowledges support from the National Science Foundation under grant number DMS-2247004 and from a Simons Fellowship.

\appendix

\section{Index of notation}
\label{sec:notation}
\noindent Some standard notation:
\begin{itemize}
	\item $\triangle = -\sum_{j=1}^d \partial_{x_j}^2$ is the positive semi-definite Euclidean Laplacian.
	\item $D_{\bullet} = -i\partial_{\bullet}$, with $\bullet$ being one of the coordinates. An exception is in \S\ref{sec:intro} (in particular  \S\ref{subsec:time_independent}, \S\ref{subsec:time_dependent}), where we use $D_\pm$ to $D$ to denote the ``gauge covariant derivative'' $i \nabla + \bfA(x)$ and $D_\pm$ to denote the half-Klein--Gordon operator \cref{eq:Dpm_def}.
	\item  $r=(x_1^2+\dots+x_d^2)^{1/2}= \lVert x \rVert$ is the usual Euclidean radial coordinate, and $\langle r\rangle = (1+r^2)^{1/2}$ is the ``Japanese bracket'' of $r$.
	\item $\calF$ is the spacetime Fourier transform, defined using the convention 
	\begin{equation}
	\calF u(\tau,\xi) = \int_{\bbR^{1+d}} e^{-it \tau -ix\cdot \xi } u(t,x) \dd t \dd^d x .
	\end{equation}
	\item $H_p$ is the Hamiltonian vector field associated to the symbol $p\in C^\infty(T^* \bbR^{1,d})$. The sign convention is unimportant.
\end{itemize}
Operators:
\begin{itemize}
	\item $P$ the Klein--Gordon operator considered; see \S\ref{subsec:assumptions}. 
	\item $P_0$ is usually the free Klein--Gordon equation, but sometimes includes some other terms. 
	\item $P_\pm = e^{\mp ic^2 t} P e^{\pm i c^2 t}$. 
\end{itemize}
Manifolds-with-corners (mwc):
\begin{itemize}
	\item $\overline{\bbR^d}=\bbR^d\sqcup \infty\bbS^{d-1}$ is the radial compactification of $\bbR^d$. As a special case, $\bbM = \overline{\bbR^{1,d}}$. 
	\item  ${}^\bullet\overline{T}^* \bbM$ is a compactification of $T^* \bbR^{1,d}$ associated to the $\bullet$-calculus, the \emph{$\bullet$-phase space}. We use the label `$\mathrm{df}$' (or variants thereof) to refer to the closure of fiber infinity in whatever phase space we are considering. Also, 
	\begin{equation}
	{}^\bullet T^* \bbM={}^\bullet\overline{T}^* \bbM\backslash \mathrm{df}. 
	\end{equation}
	\item $\mathrm{df},\mathrm{bf},\natural\mathrm{f},\mathrm{pf}$ are the boundary hypersurfaces of the $\calc$-phase space. We also use `$\natural\mathrm{f}$' to refer to the analogous boundary hypersurfaces of the $\calczero$- and $\calctwo$- phase spaces.  
	\item $\mathrm{pf}_\pm$ the analogues of $\mathrm{pf}$ in the $\calctwo$-phase space. 
	\item $\rho_{\mathrm{f}}$ is, for each boundary hypersurface $\mathrm{f}$ of a mwc, a boundary-defining-function of that face.
	\item $\calV_\bullet$ is the space of sections of the dual bundle ${}^\bullet T \bbM$ to ${}^{\bullet}T^* \bbM$.
\end{itemize}
Pseudodifferential calculi: 
\begin{itemize}
	\item $\Psi_{\mathrm{sc}}(\bbR^N)$ the scattering- (sc-) calculus on $\overline{\bbR^N}$. 
	\item $\Psi_{\mathrm{par}}=\Psi_{\mathrm{par}}(\bbR^{1,d})$ is the parabolic scattering calculus on $\bbR^{1+d}$, defined in \cite{Parabolicsc} in order to study the Schr\"odinger equation. 
	\item $\Psi_{\mathrm{par,I,res}}$ is the result of second-microlocalizing the calculus of smooth families of elements of $\Psi_{\mathrm{par}}$, specifically second-microlocalizing at the corner of 
	\begin{equation} 
		{}^{\mathrm{par,I}}\overline{T}^* \bbM = [0,\infty)\times {}^{\mathrm{par}}\overline{T}^* \bbM.
	\end{equation}
	\item $\Psi_{\calczero}$ is the natural calculus, defined in \S\ref{subsec:calczero}.
	\item $\Psi_{\calc}$ is the $\calc$-calculus, defined in \S\ref{subsec:calc}.
	\item $\Psi_{\calctwo}$ is the $\calctwo$-calculus, defined in \S\ref{subsec:calctwo}.
	\item $\Psi_\bullet^{m,s,\dots}$ is the subset of $\Psi_\bullet$ consisting of operators of orders $m,s,\dots$. 
	\item $\sigma_\bullet^{m,s,\dots}$ is the principal symbol map associated to the $\bullet$-calculus and orders $m,s,\dots$.
	\item $N(A)$ is the leading-order part of $A\in \Psi_{\calc}$ at $\mathrm{pf}$. 
\end{itemize}
Function spaces: 
\begin{itemize}
		\item If $X$ is a manifold-with-corners and $I\subseteq \bbC$, then $C^\infty(X;I)$ denotes the set of smooth functions $f:X^\circ \to I$ such that, if $X_0\supseteq X$ is a closed manifold in which $X$ is embedded, then $f$ extends to an element of $C^\infty(X_0)$. 
	\item If $X$ is a manifold-with-boundary, then $\calS(X)$ denotes the space of Schwartz functions on $X$, i.e.\ smooth functions on $X^\circ$ decaying rapidly, together with all derivatives, at $\partial X$. 
	\item $\calA^{\calE,\calF,\dots}(X)$ is the space of polyhomogeneous functions on $X$ with index sets $\calE,\calF,\dots$ at the boundary hypersurfaces of $X$. Replacing $\calE,\calF,\dots$ by a real number $\alpha$ means that only $O(\rho_{\mathrm{f}}^\alpha)$ conormality is required at that face. Also, 
	\begin{equation}
	\calA^{(\calE,\alpha),\calF,\dots}(X) = \calA^{\calE,\calF,\dots}(X) + \calA^{\alpha,\calF,\dots}(X), 
	\end{equation}
	and so on. 
	\item We use $(j,k)$ as an abbreviation for the index set $\{(j+n,k):n\in \bbN \}$.
\end{itemize}
Miscellaneous:
\begin{itemize}
	\item ${}^\bullet\mathsf{H}_p$ is $H_p$ multiplied by just enough boundary-defining-functions to make a b-vector field on the $\bullet$-phase space.
	\item $\Sigma,\Sigma_{\pm},\Sigma_{\mathrm{bad}}$ all denote portions of various characteristic sets. In this paper, it is mainly the $\calctwo$-characteristic set of $P$ which matters. We will usually denote this $\Sigma$.
	\item $\calR^\bullet_\bullet$ denotes the radial sets. The subscript specifies whether we are looking over the future (+) or past ($-$) hemisphere of spacetime infinity, and the superscript specifies whether we are looking at positive/negative energy (i.e.\ particles vs. anti-particles). 
	\item $\calB \subseteq {}^{\calctwo}\overline{T}^* \bbM$ is the set containing the portion of the wavefront set of $f(x)\delta(t)$ for Schwartz $f$.  See \Cref{fig:Bfig}.
	\item $\aleph$ is used for the leading-order relativistic correction to the coefficient of $\partial_t^2$ in $P$.
\end{itemize}

\section{Example: the free Klein--Gordon equation}
\label{sec:free}
In this appendix, we discuss the free Klein--Gordon equation, both as an exactly solvable example of the non-relativistic limit and as an illustration of the frequency-space ideas which go into its proof. 
Our microlocal tools are meant to ``microlocalize'' what can be done using the spacetime Fourier transform. It is  therefore important to understand how the non-relativistic limit works in frequency space.

Consider the $V,\bfA,W=0$ case  of \cref{eq:P_intro}, i.e.\ the free Klein--Gordon equation. 
Then, the properties of the $c\to\infty$ limit are most easily brought to the fore by using either 
\begin{enumerate}[label=(\roman*)] 
	\item the spacetime Fourier transform, or 
	\item the advanced/retarded Green's functions (a.k.a.\ propagators) 
	\begin{equation} 
		G_\pm(t,x;c)\in C^\infty(\bbR^+_c;\calS'(\bbR^{1,d}_{t,x})),
	\end{equation} 
	which are determined by 
	\begin{equation}
		\begin{cases}
			(c^{-2} \partial_t^2 + \triangle +  c^2 ) G_\pm  = \delta(t)\delta^d(x)\\ 
			G_\pm(t,x;c)=0\text{ in }\{\pm t<0\}. 
		\end{cases}
		\label{eq:Greens} 
	\end{equation} 
    These can be used to solve the Cauchy problem (cf.\ \cref{eq:def-u-varsigma}).
\end{enumerate}
The former is discussed in \S\ref{subsub:Fourier}.  Green's functions are discussed briefly in \S\ref{subsec:Greens}.

\subsection{Spacetime Fourier transform}
\label{subsub:Fourier}
Let us begin with (i). Letting 
\begin{equation}
	\calF u(\tau,\xi) = \int_{\bbR^{1+d}}e^{-it \tau-ix\cdot \xi} u(t,x;c)\dd t\dd^d x \in \calS'(\bbR^{1,d}_{\tau,\xi})
\end{equation}
denote the spacetime Fourier transform of $u$, the Klein--Gordon equation becomes $(-\tau^2+c^2\xi^2+ c^4) \calF u=0$. This implies that $\calF u$ is supported on the ``mass shell'' 
\begin{equation}
	\Sigma(c)=\{(\tau,\xi)\in \bbR^{1,d}: \tau^2 = c^2\xi^2 +c^4\}, 
\end{equation} 
which is, for each $c>0$, a hyperboloid in frequency space $\bbR^{1,d}_{\tau,\xi}$. In line with later terminology, let us call this the \emph{characteristic} hyperboloid. 

\subsubsection{The non-relativistic limit at the level of the characteristic set}

The effect of $c$ is to dilate $\Sigma(c)$ and stretch it in one direction. Indeed, if we let 
\begin{equation}
	\tau_\natural = \tau/c^2, \qquad \xi_\natural = \xi/c
\end{equation}
denote ``natural frequencies'' (these being the frequency coordinates dual to the natural units $t_\natural,x_\natural$ in \cref{eq:nat}), then $\Sigma(c)$ is just defined by $\{\tau_\natural^2 = \xi_\natural^2 + 1\}$. So, as $c\to\infty$, $\Sigma(c)$ stretches faster along the $\tau$-axis than along the $\xi$-axes.

We can begin to understand the non-relativistic limit just by examining the  behavior of $\Sigma(c)$ in the $c\to\infty$ limit. The foci of $\Sigma(c)$ move along the $\tau$-axis as $c\to\infty$, so let us reposition: let 
\begin{align}
	\Sigma_{\pm}(c) &= \{(\tau\mp c^2,\xi) : (\tau,\xi) \in \Sigma(c)  \} =\Big\{ (\tau,\xi): \tau \pm c^2 = \pm \sqrt{c^2\xi^2 +c^4}\Big\} \\ 
	\Sigma_{\mathrm{bad}, \pm}(c) &= \{(\tau\mp c^2,\xi) : (\tau,\xi) \in \Sigma(c)  \} =\Big\{ (\tau,\xi): \tau \pm c^2 =\mp  \sqrt{c^2\xi^2 +c^4}\Big\}
\end{align}
denote the two components of the result of translating the characteristic hyperboloid up/down by $c^2$ units, so that $\Sigma_\pm$ passes through the origin. The effect of a translation $\tau\mapsto \tau \mp c^2$ in the frequency-space is to multiply by an oscillatory factor in position-space; if we compare a function whose Fourier transform is supported on $\Sigma(c)$ vs. one whose Fourier transform is supported on $\Sigma_\pm(c)\cup \Sigma_{\mathrm{bad},\pm}$, the former is oscillating with an extra $\exp(\pm i c^2 t)$. The component $\Sigma_{\mathrm{bad},\pm}(c)$ is contained outside of the large ball $\{(\tau,\xi) : \tau^2+\xi^2 <2c^2\}$ and we will ignore it here; see also the comments after \cref{eq:misc_B10}. Instead, focus on $\Sigma_\pm(c)$, which, passing through the origin in frequency space, should contribute some non-oscillatory term to our solution.

\begin{figure}
	\begin{center}
		\includegraphics[scale=.8]{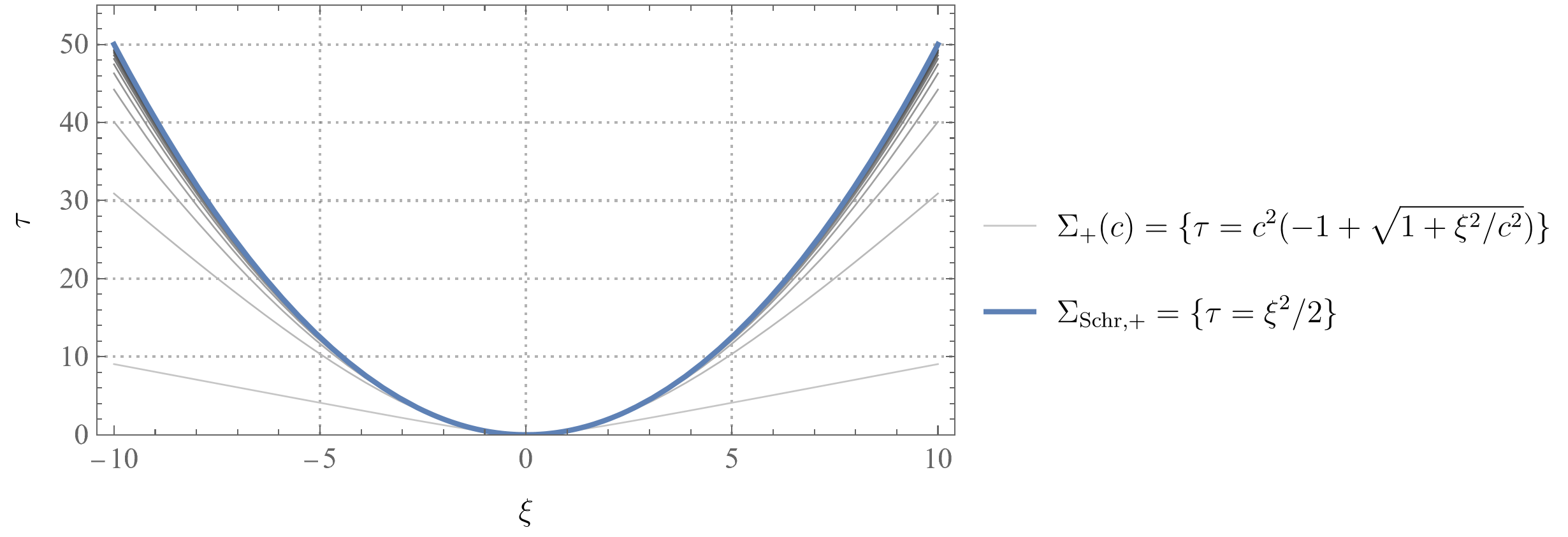}
	\end{center}
	\caption{The frequency sets $\Sigma_{\mathrm{Schr},+}$ (in blue) and $\Sigma_+(c)$ (in gray), for different values of $c$. As $c\to\infty$, the hyperboloid sheet $\Sigma_+(c)$ converges up to the paraboloid $\Sigma_{\mathrm{Schr},+}$. }
	\label{fig:char_convergence}
\end{figure}

Taylor expanding, 
\begin{equation}
	\sqrt{c^2\xi^2+c^4} =c^2 + \frac{ \xi^2}{2 }  +O\Big( \frac{\xi^{4}}{c^2} \Big).
	\label{eq:misc_k42}
\end{equation} 
So, 
\begin{equation}
	\Sigma_\pm(c) = \Big\{ (\tau,\xi) \in \bbR^{1,d}: \tau = \pm \frac{\xi^2}{2} \Big( 1 + O\Big( \frac{\xi^2}{c^2} \Big) \Big)  \Big\}.
\end{equation}
This means that, if we focus on the the portion of $\Sigma_\pm(c)$ with $\lVert \xi \rVert \ll c$, then $\Sigma_\pm(c)$ will be very close to the set
\begin{equation} 
	\label{eq:misc_B10}
    \Sigma_{\mathrm{Schr},\pm}=\Big\{(\tau,\xi)\in \bbR^{1,d} : \tau =\pm \frac{\xi^2}{2} \Big\}.
\end{equation} 
See \Cref{fig:char_convergence}.
As the subscript indicates, $\Sigma_{\mathrm{Schr}}$ is the set of frequencies on which the Fourier transform of a solution of the Schr\"odinger equation is supported, the \emph{characteristic paraboloid}. So, the set of frequencies on which a positive/negative-energy solution of the Klein--Gordon is supported converges, after a translation, to the set of frequencies on which a solution of the positive/negative-energy Schr\"odinger equation is supported. So, as long as we can ignore frequencies with $\lVert \xi \rVert \not \ll c$ (this also including $\Sigma_{\mathrm{bad},\pm}$), which will ultimately be justified based on the regularity of the prescribed initial data, then the non-relativistic limit seems in good order.

\subsubsection{Sketch of rigorous argument}
Let us see how this works in a bit more detail.
The Fourier transform $\calF u$ of a solution $u$ of the Klein--Gordon equation has the form $ u = u_-+u_+$ for $u_\pm \in \calS'(\bbR^{1,d})$ such that $\calF u_\pm$ is supported on the $\{\pm \tau>0\}$ sheet of the mass shell. Then, $u_\pm$ has the form 
\begin{equation}
	u_\pm = \frac{1}{(2\pi)^d}\int_{\bbR^d} e^{\pm i c t \sqrt{\xi^2+ c^2} + i x\cdot \xi } \hat{\phi}_\pm(\xi;c)  \dd^d \xi 
	\label{eq:misc_a13}
\end{equation}
for some $\phi_\pm(x;c) \in \calS'(\bbR^d_x)$. Thus, the restriction of $u_\pm$ to $t=0$ makes sense and is given by $u_\pm(0,x;c) = \phi_\pm(x;c)$.

The two $\phi_\pm$ can be solved for in terms of the prescribed initial data $\varphi,\psi\in \calS(\bbR^d)$. The result is 
\begin{equation} 
	\phi_\pm = 2^{-1} (\varphi \mp ic ( \triangle+ c^2)^{-1/2} \psi ).
	\label{eq:misc_019}
\end{equation} 
It can be checked that $\phi_\pm\in C^\infty([0,\infty)_{1/c^2}; \calS(\bbR^d))$. Note that, if $c\gg 1$, then $\phi_\pm \approx 2^{-1}(\varphi \mp i  \psi)$, in accordance with \Cref{thm:simplest}. 

\begin{figure}[t!]
	\begin{subfigure}{\textwidth}
		\includegraphics[scale=.75]{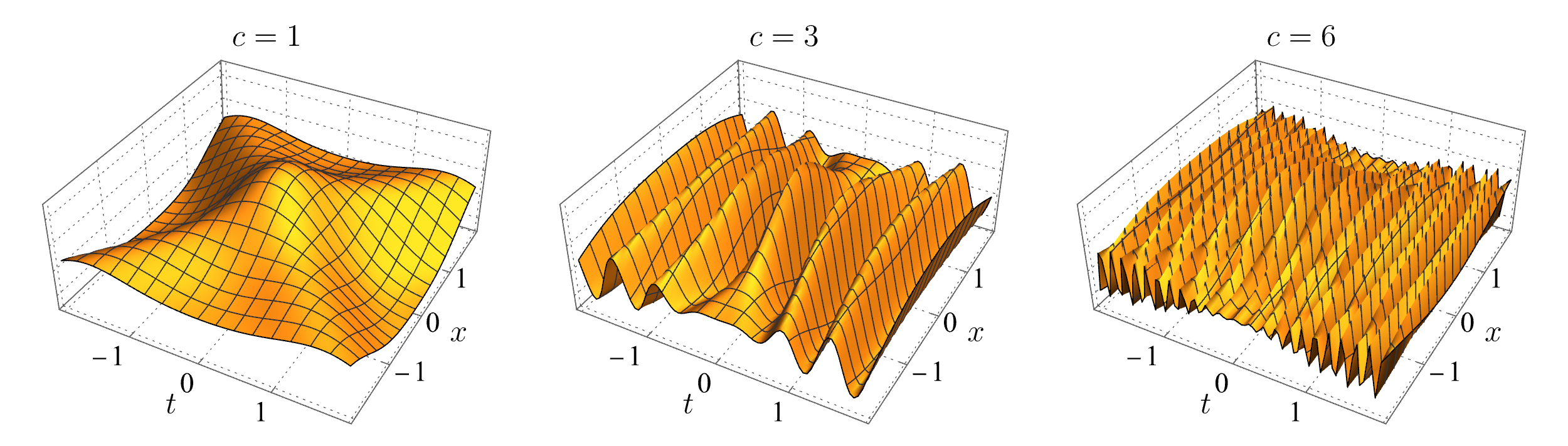}
		\caption{Plots of $\Re[u_+]$. Note the rapid $\exp(+ic^2 t)$ oscillations.}
	\end{subfigure}
	\begin{subfigure}{\textwidth}
		\includegraphics[scale=.75]{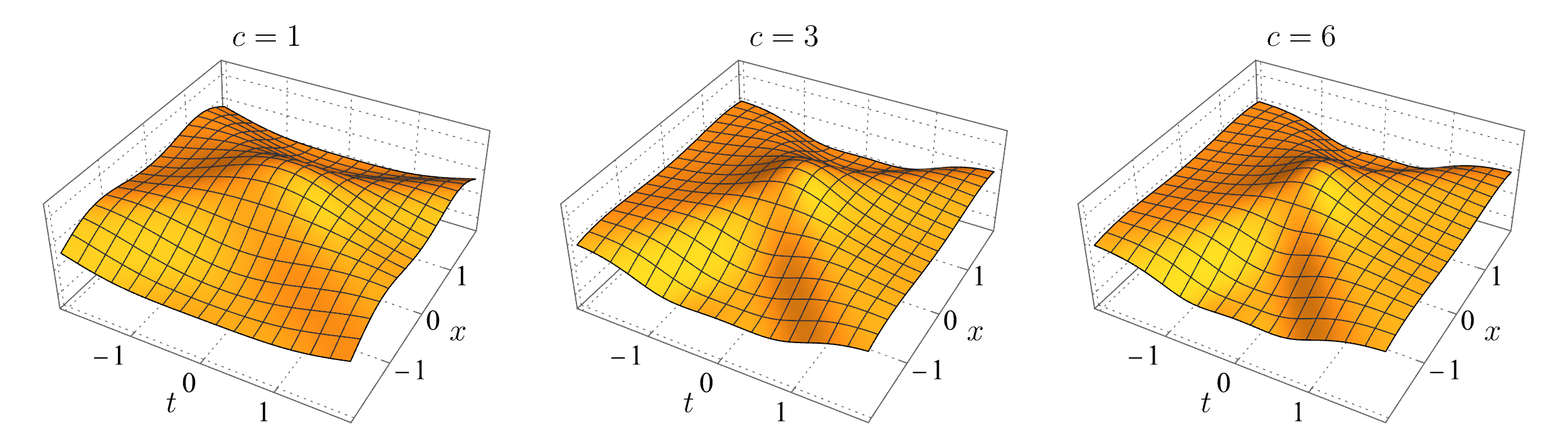}
		\caption{Plots of $\Re[e^{-ic^2 t} u_+]$. Already for $c\geq 3$ it is difficult to discern changes in $e^{-ic^2 t}u_+$ as $c$ varies.}
	\end{subfigure}
	\caption{Numerically calculated plots of the real parts of $u_+$ and $e^{-ic^2 t}u_+$ for $u_+$ given by \cref{eq:misc_a13}, for $\varphi$ Gaussian and $\psi=0$. Different values of $c$ are shown in the three columns.}
	\label{fig:free}
\end{figure}

Divide up the integral in \cref{eq:misc_a13} as  
\begin{equation}
	\int_{\bbR^d} e^{\pm i c t \sqrt{\xi^2+c^2} + i x\cdot \xi } \hat{\phi}_\pm(\xi;c)  \dd^d \xi  = \Big( \int_{\lVert \xi \rVert \leq \Xi(c) } + \int_{\lVert \xi \rVert \geq \Xi(c)} \Big)  e^{\pm i c t \sqrt{\xi^2+ c^2} + i x\cdot \xi } \hat{\phi}_\pm(\xi;c)  \dd^d \xi, 
\end{equation}
into a contribution from small $\xi$ and a contribution from not so small $\xi$. The cutoff $\Xi(c)>0$ will have to be chosen appropriately.
Because $\hat{\phi}_\pm(\xi;c)\in \calS(\bbR^d_\xi)$ is decaying rapidly as $\xi\to\infty$, as long as $\Xi(c)\geq c^\varepsilon$ for some $\varepsilon>0$, 
the contribution coming from large $\xi$ is negligible as $c\to\infty$, and is in fact rapidly decaying:
\begin{equation}
	\int_{\lVert \xi \rVert \geq \Xi(c)} e^{\pm i c t \sqrt{\xi^2+c^2} + i x\cdot \xi } \hat{\phi}_\pm(\xi;c)  \dd^d \xi \approx 0
	\label{eq:misc_021}
\end{equation}
\begin{equation}
	\int_{\bbR^d} e^{\pm i c t \sqrt{\xi^2+ c^2} + i x\cdot \xi } \hat{\phi}_\pm(\xi;c)  \dd^d \xi  \approx  \int_{\lVert \xi \rVert \leq \Xi(c) }   e^{\pm i c t \sqrt{\xi^2+c^2} + i x\cdot \xi } \hat{\phi}_\pm(\xi;c)  \dd^d \xi, 
	\label{eq:misc_02a}
\end{equation}
to all orders in $1/c^2$. Thus, the $c\to\infty$ asymptotics of $u_\pm$ entirely come from the region where $\xi$ is small. 
(This is closely related to the discussion in the previous few subsections. Because $\phi_\pm$ are varying on scales large with respect to natural units, the Fourier transform is concentrated on $\lVert \xi \rVert \ll c$.
If, instead, the dependence of $\phi_\pm$ on $c$ was via 
\begin{equation} 
	\phi_\pm(x;c) = \phi_\pm(cx;1)
\end{equation} 
(cf.\ \cref{eq:IVP_2}), 
then $\hat{\phi}_\pm(\xi;c)$ would be large on scales of order $\xi\sim c$, so \cref{eq:misc_021} would have no reason to hold.)

In order to compute the $c\to\infty$ asymptotics of the right-hand side of \cref{eq:misc_02a}, we need the expansion of $\smash{\hat{\phi}_\pm}(-;c)$ in powers of $1/c^2$.  
Using \cref{eq:misc_019}, this is given by formally expanding $(\triangle+ c^2)^{-1/2}$ in powers of $\triangle/c^2$:
\begin{equation}
	\phi_\pm = \frac{\varphi}{2} \mp \frac{i}{2} \Big(1- \frac{\triangle}{2 c^2} + \frac{3\triangle^2}{8c^4} - \frac{5 \triangle^3}{16 c^6}+\cdots  \Big)  \psi,
	\label{eq:misc_k41}
\end{equation}
where the error from throwing away the terms past $\triangle^k/c^{2k}$ is in $c^{-2k-2} C^\infty([0,\infty)_{1/c^2}; \calS(\bbR^d))$. Taking the Fourier transform of \cref{eq:misc_k41},  
\begin{equation}
	\hat{\phi}_\pm = \frac{\hat{\varphi}}{2} \mp \frac{i}{2} \Big(1- \frac{\xi^2}{2c^2} + \frac{3\xi^4}{8c^4}- \frac{5 \xi^6}{16 c^6}+\cdots  \Big)  \hat{\psi}.
	\label{eq:misc_i3k}
\end{equation}

Notice that \cref{eq:misc_i3k}, \cref{eq:misc_k42} are expansions in powers of $\xi^2/c^2$. Consequently, as long as $\Xi(c) \leq c^{1-\epsilon}$ for some $\epsilon>0$, in the region $\lVert \xi \rVert \leq \Xi(c)$ successive terms are suppressed by some positive power of $1/c$ as $c\to\infty$. 
So,
\begin{equation}
	\int_{\lVert \xi \rVert\leq \Xi(c)}e^{\pm i c t \sqrt{\xi^2+c^2} + i x\cdot \xi } \hat{\phi}_\pm(\xi;c)  \dd^d \xi \approx e^{\pm i  c^2 t} \int_{\bbR^d} e^{\pm i t \xi^2 / 2 +ix\cdot \xi} \hat{\phi}_\pm(\xi;\infty) \dd^d \xi 
\end{equation}
to leading order, as follows from keeping only the leading order terms in \cref{eq:misc_k41}, \cref{eq:misc_k42}. So, 
\begin{equation}
	\lim_{c\to\infty}  e^{\mp i  c^2 t} u_\pm  = \frac{1}{(2\pi)^d}\int_{\bbR^d} e^{\pm i  t \xi^2 / 2 +ix\cdot \xi} \hat{\phi}_\pm(\xi;\infty) \dd^d \xi .
\end{equation}
The right-hand side is just the solution $v_\pm$ of the (free) time-dependent Schr\"odinger equation with initial data $\phi_\pm $.
So, letting 
\begin{equation} 
	v=\exp(-ic^2 t) v_- + \exp(ic^2 t) v_+
\end{equation} 
as in \Cref{thm:simplest}, $u\approx v$. More precisely, the above argument shows that $u-v$ is $O(1/c^{2-\epsilon})$ for every $\epsilon>0$, and the $L^\infty$ estimate \cref{eq:misc_a06} holds. So,  
\Cref{thm:simplest} holds in the free case, up to an $\epsilon$ loss that we have not ruled out yet (but which will be ruled out once subleading terms are computed).

\subsubsection{Error term}
\label{subsub:freeerror}

Now let us discuss the error term $u-v$. Each of $E^\pm(t,x,c)=e^{\mp i c^2 t} c^{-2}(u_\pm - v_\pm)$ has a full asymptotic expansion
\begin{equation}
	e^{\mp i c^2 t}c^{-2}(u_\pm - v_\pm) = E_1^\pm (t,x) + c^{-2} E_2^\pm (t,x) + c^{-4} E_3^\pm (t,x)+\dots 
	\label{eq:misc_671}
\end{equation}
in powers of $1/c^2$, the first few of which we will compute.

\begin{figure}[t]
	\includegraphics[scale=.75]{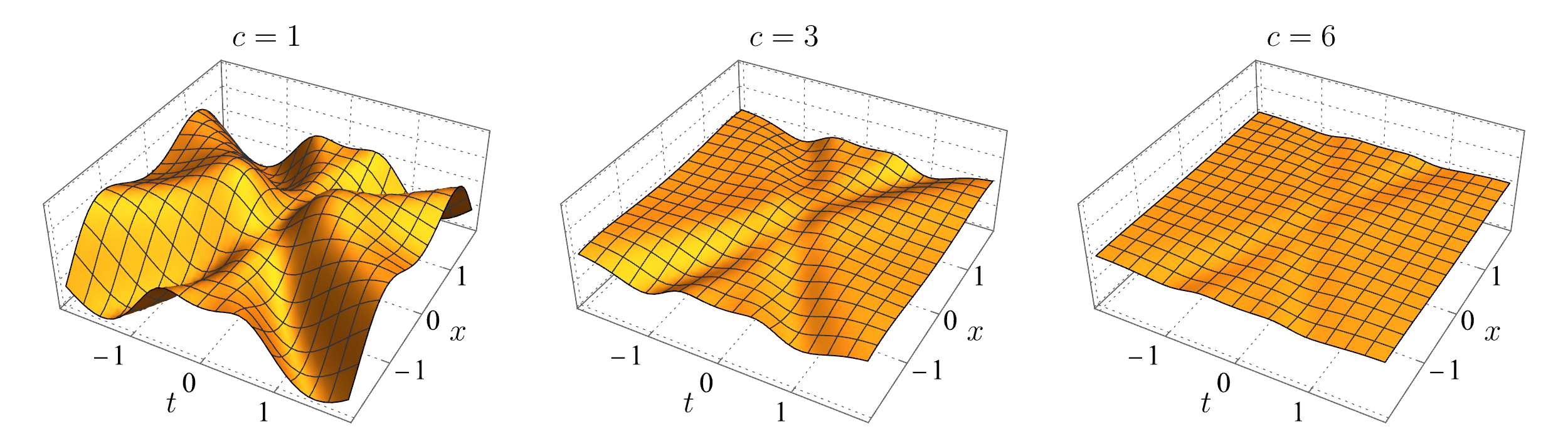}
	\caption{Plots of the error $\Re[e^{-ic^2 t} u_+-v_+]$. The vertical axis has been rescaled compared to \Cref{fig:free}, but these three plots share a single scale. We can see how $\exp(-ic^2 t)u_+ \to v_+$ as $c\to\infty$.}
\end{figure}

The terms of \cref{eq:misc_671} come from the sub-principal terms in \cref{eq:misc_k41} and the sub-sub-principal terms in \cref{eq:misc_k42}, which we have yet to work out. Expanding further, 
\begin{equation}
	c  \sqrt{\xi^2+c^2} =c^2 + \frac{ \xi^2}{2 } - \frac{\xi^4}{8c^2 } + \frac{\xi^6}{16c^4} - \frac{5 \xi^8}{128 c^6}+ O\Big( \frac{\xi^{10}}{c^8} \Big).
\end{equation}

This, with \cref{eq:misc_k41}, is all that is needed to work out the first few $E_\bullet^\pm$'s:
\begin{equation}
	E_1^\pm =  \pm \frac{ i}{2^{d+2} \pi^d} \int_{\bbR^d} e^{\pm i t \xi^2 / 2 +ix\cdot \xi} \xi^2 \Big(   \hat{\psi}(\xi)  - \frac{t \xi^2}{2} \hat{\phi}_\pm(\xi;\infty)  \Big) \dd^d \xi, 
\end{equation}
\begin{equation}
	E_2^\pm = \frac{1}{2^{d+1}\pi^d} \int_{\bbR^d} e^{\pm i t \xi^2 / 2 +ix\cdot \xi} \xi^4 \Big[ t \xi^2\Big( \pm \frac{i}{8} - \frac{t\xi^2}{64} \Big) \hat{\phi}_\pm(\xi;\infty) 	+\frac{1}{8} \Big( \frac{t\xi^2}{2} \mp 3i \Big) \hat{\psi}(\xi) \Big] \dd^d \xi, 
\end{equation}
and 
\begin{multline}
	E_3^\pm = \frac{1}{2^{d+5} \pi^d} \int_{\bbR^d} e^{\pm i t \xi^2 / 2 +ix\cdot \xi} \xi^6 \Big[ t \xi^2 \Big( \mp \frac{5i}{4} + \frac{t\xi^2 }{4} \pm \frac{i t^2 \xi^4}{96} \Big) \hat{\phi}(\xi;\infty) 
	\\ + \Big( \pm 5 i - \frac{5t \xi^2}{4} \mp \frac{i t^2\xi^4}{16}  \Big) \hat{\psi}(\xi)\Big] \dd^d \xi. 
\end{multline}
Note that, as $j$ increases, the well-definedness of $E_j^\pm$ requires that $\hat{\phi}(\xi;\infty)$, $\hat{\psi}(\xi)$ decay more rapidly as $\xi\to
\infty$; producing more terms in the $c\to\infty$ asymptotic expansion requires control on more derivatives of the initial data.

\subsection{The Green's function}
\label{subsec:Greens}

The Green's function $G_\pm$, defined by \cref{eq:Greens} can be solved for using the Fourier transform:
\begin{align}
\begin{split} 
    G_\pm(t,x;c) &= \frac{1}{(2\pi)^{1+d}} \lim_{\varepsilon \to 0^+} \int_{\mathbb{R}^{1+d}} \frac{e^{i t \tau +i x \cdot \xi} e^{- \varepsilon \lVert \xi \rVert}  \dd \tau \dd^d \xi}{ -c^{-2}(\tau\mp i\varepsilon)^2 + \xi^2 + c^2  } \\
    &=  \lim_{\varepsilon \to 0^+} \int_0^\infty \bigg(  \int_{\bbS^{d-1}}  e^{i\Xi x\cdot \theta} \dd \Omega(\theta) \bigg) \bigg( \int_{-\infty}^\infty  \frac{e^{i t \tau}  \dd \tau}{-c^{-2}(\tau\mp i\varepsilon)^2 + \Xi^2 + c^2  }\bigg) \frac{e^{-\varepsilon \Xi} \Xi^{d-1}  \dd \Xi }{(2\pi)^{1+d}}  .
    \end{split} 
    \label{eq:Greens_formula}
\end{align}
The first expression in the parentheses is a standard form of the Bessel $J$-function: for $\Xi>0$, 
\begin{equation}
     \int_{\bbS^{d-1}}  e^{i\Xi x\cdot \theta} \dd \Omega(\theta) = \begin{cases}
       (2\pi)^{\nu+1}  (\Xi |x|)^{-\nu} J_\nu(\Xi |x| ) & (x \neq 0) \\ 
         \operatorname{Area}(\bbS^{d-1}) & (x = 0), \\ 
     \end{cases}
     \label{eq:Bessel_form}
\end{equation}
where $\nu=d/2-1$.\footnote{The singularity of $z^{-\nu} J_\nu(z)$ at $z=0$ is removable, with $z^{-\nu} J_\nu(z)|_{z=0} = 2^{-\nu} /\Gamma(d/2)$; see \cite[\href{http://dlmf.nist.gov/10.2.E2}{Eq. 10.2.2}]{NIST}. Since the surface area of a hypersphere is 
\[
\operatorname{Area}(\bbS^{d-1}) = 2 \pi^{d/2} / \Gamma(d/2),
\]
the formula in the top line of \cref{eq:Bessel_form} agrees with that in the bottom line after removing the singularity in the former.}
The second factor in parentheses in \cref{eq:Greens_formula} can be calculated using the Cauchy integral formula:
\begin{equation}
    \int_{-\infty}^\infty  \frac{e^{i t \tau}  \dd \tau}{-c^{-2}(\tau\mp i\varepsilon)^2 + \Xi^2 + c^2  } = \pm \frac{2 c\pi}{\sqrt{c^2+\Xi^2}} \sin (c t \sqrt{c^2+\Xi^2}) \Theta(\pm t),
\end{equation}
where $\Theta(\cdot)$ is the Heaviside step function. So, 
\begin{equation}
    G_\pm(t,x;c) = \pm \frac{c}{(2\pi)^{d/2} |x|^\nu} \Theta(\pm t) \int_0^\infty \frac{\Xi^{d/2} }{ \sqrt{c^2+\Xi^2}} \sin(ct \sqrt{c^2+\Xi^2}) J_\nu(\Xi |x|)  \dd \Xi  .
\end{equation}

\subsubsection{\texorpdfstring{$d=1$ case}{d=1 case}}
When $d=1$, the Bessel function $J_\nu=J_{-1/2}$ is just $J_{-1/2}(z)= (2/\pi z)^{1/2}\cos z$, so 
\begin{equation}
    G_\pm (t,x;c) = \pm \frac{ c \Theta(\pm t) }{\pi}  \int_{0}^\infty \frac{1}{\sqrt{c^2+\Xi^2}}  \sin( ct \sqrt{c^2+\Xi^2}) \cos(\Xi |x|) \dd \Xi . 
\end{equation}
The integral can be evaluated explicitly (it does not converge when $|t|=|x|$, but the singularity is mild enough that we cannot get a $\delta$-function contribution there):
\begin{equation}
    \int_{0}^\infty \frac{1}{\sqrt{c^2+\Xi^2}}  \sin( ct \sqrt{c^2+\Xi^2}) \cos(\Xi |x|) \dd \Xi=  \frac{\pi}{2}\operatorname{sgn}(t)  \Theta(c^2 t^2-x^2)J_0(c \sqrt{c^2 t^2-x^2}); 
\end{equation}
see [Scharf \S2.3]. So, we can conclude
\begin{equation}
    G_\pm(t,x;c) = \frac{c }{2} \Theta(\pm t)\Theta(c^2 t^2-x^2)J_0(c \sqrt{c^2 t^2-x^2}). 
\end{equation}
When $z=c(c^2 t^2 -x^2)^{1/2} \gg 1$, then it is reasonable to approximate $J_0(z) = (2/\pi z)^{1/2} \cos(z+\pi/4)+O(1/z)$, yielding, for $t>0$,  
\begin{equation}
     G_\pm(t,x;c) \approx \frac{1}{\sqrt{2\pi  t}} \cos \Big(  c \sqrt{c^2 t^2-x^2}  + \frac{\pi}{4}\Big) \approx \frac{1}{\sqrt{2\pi  t}} \cos \Big(   c^2 t - \frac{x^2}{2 t}  + \frac{\pi}{4}\Big).
\end{equation}
The right-hand side is closely related to the Schr\"odinger propagator on the line: $ (2\pi i t)^{-1/2} e^{ix^2/2t}$.

\section{Supplement to numerics}

The numerical solutions of \cref{eq:IVP}, \cref{eq:Schrodinger_initial} were carried out using Mathematica 13.3's \texttt{NDSolve} algorithm \cite{NDSolve}, with artificial Dirichlet boundary conditions enforced at $x=\pm 10$ (overriding initial conditions beyond). 

The initial data is highly concentrated near the origin. So, based on the amount of time it takes a wave to travel from the origin to $x=\pm 10$ and back, we can be confident that the artificial boundary conditions have negligible effect on the solution in the plotted region, for the sampled values $c=1,3,6$ of $c$. Larger $c$ require a more distant artificial boundary.

\Cref{fig:FEMIm} shows the imaginary parts of the functions whose real parts are shown in \Cref{fig:FEM}. The same phenomena are seen in the two figures. 

\begin{figure}[h!]
	\begin{subfigure}{\textwidth}
		\centering
		\includegraphics[width=.99\textwidth]{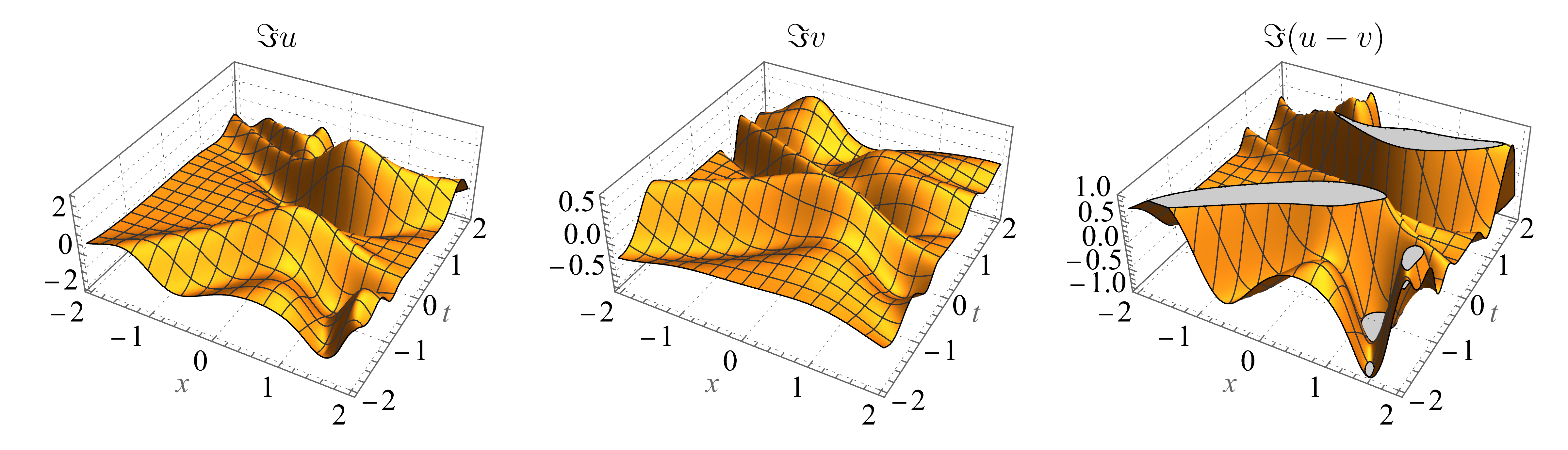}
		\caption{$c=1$. }\vspace{1em}
	\end{subfigure}
	\begin{subfigure}{\textwidth}
		\centering
		\includegraphics[width=.99\textwidth]{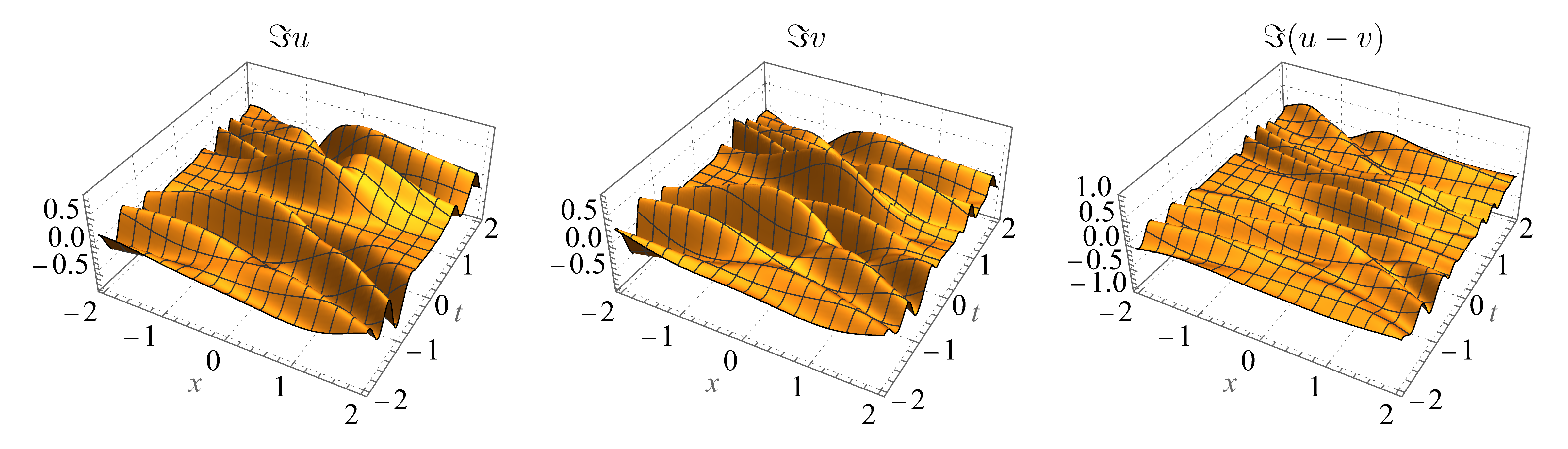}
		\caption{$c=3$.}\vspace{1em}
	\end{subfigure}
	\begin{subfigure}{\textwidth}
		\centering
		\includegraphics[width=.99\textwidth]{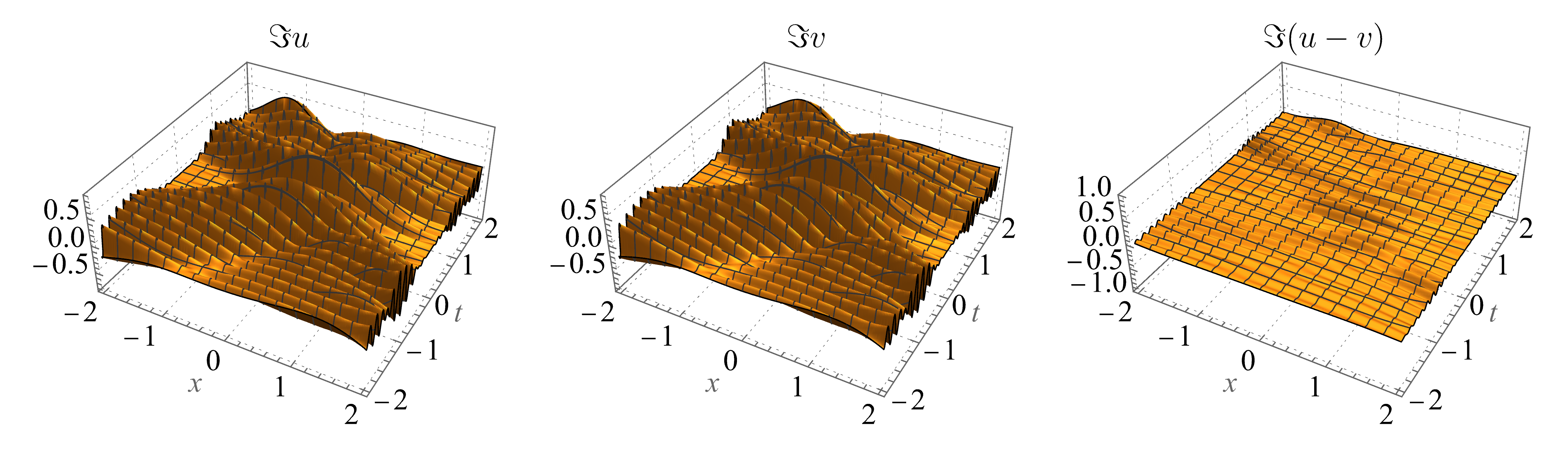}
		\caption{$c=6$.}
	\end{subfigure}
	\caption{The same functions plotted \Cref{fig:FEM}, but now the imaginary parts.} 
	\label{fig:FEMIm}
\end{figure}

\subsection{\texorpdfstring{A remark about \Cref{fig:FEM2}}{A remark about Figure 2}}

\label{sec:fig_disc}
We noted in the caption of \Cref{fig:FEM2} that the solutions of Schr\"odinger's equation with potential \cref{eq:V_ex} (and the prescribed Gaussian initial data) closely follow classical trajectories. This can be understood as a ``semiclassical'' effect, and more precisely a $\natural$-effect. In this brief appendix, we explain how. 

\begin{figure}[h!]
    \centering
    \includegraphics[width=.45\textwidth]{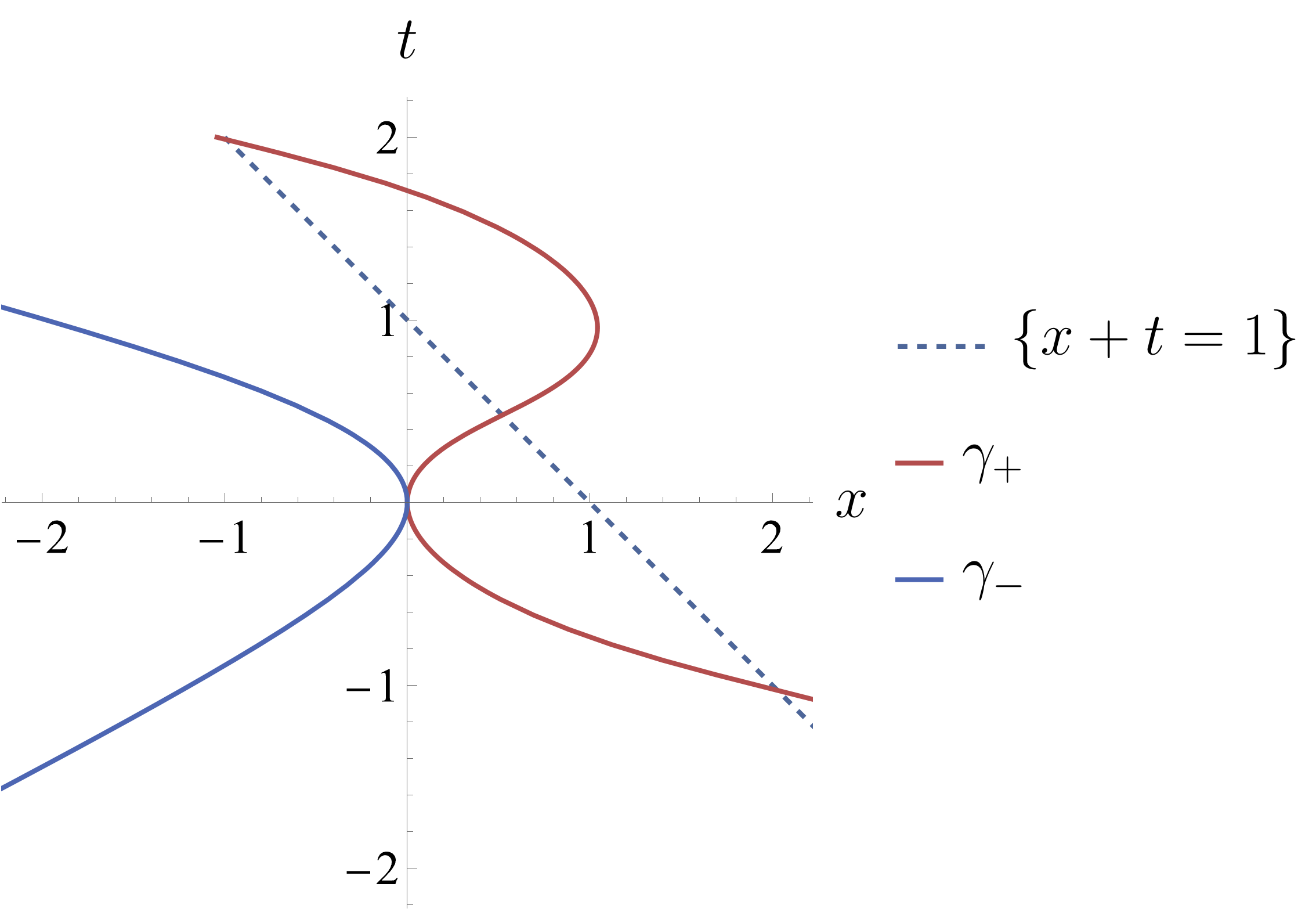}
    \caption{The paths $\gamma_\pm$ followed by a classical particle moving in the force-field $F=\mp \nabla V$ generated by the potential $V$ defined by \cref{eq:V_ex}. The dashed line shows the position of the center of the potential well/barrier $V$ as it moves over time. Compare these paths to the plots in \Cref{fig:FEM2}.}
    \label{fig:paths}
\end{figure}

In \Cref{fig:FEM2}, we were plotting the $L^2$-mass density $|v_\pm|^2$ of the solutions $v_\pm \in C^\infty(\bbR^{1,1})$ of Schr\"odinger's equation 
\begin{equation}
    (\pm i \partial_t - 2^{-1} \triangle \mp  \mathsf{Z} W )v_\pm = 0,\quad W = 1/(1+(x-1+t)^2),\quad \mathsf{Z}>0
    \label{eq:Schr_ap}
\end{equation}
with the prescribed initial data. What we called $V$ there are now called $V=\mathsf{Z} W$, because we are going to want to talk about what happens when $\mathsf{Z}$ varies. We had taken $\mathsf{Z}=8$, but now let $\mathsf{Z}$ be a parameter.

Dividing Schr\"odinger's equation \cref{eq:Schr_ap} through by $\mathsf{Z}$, and letting $h=1/\mathsf{Z}^{1/2}$, the resulting PDE is 
\begin{equation}
     (\pm i h^2  \partial_t - 2^{-1}  (ih\nabla)^2  \mp  W )v_\pm = 0.
\end{equation}
This lies in $\operatorname{Diff}_{\natural}^{2,0,0}$ and is of real principal type at finite $\natural$-frequency, i.e.\ away from fiber infinity. We can therefore apply $\natural$-analysis to understand the $h\to 0$ limit. This corresponds to the $\mathsf{Z}\to \infty$ limit. Thus, $\mathsf{Z}$ is a large parameter, playing a role analogous to that played by the speed-of-light $c$ in the non-relativistic limit of the Klein--Gordon equation. For our current numerical purpose, $8\gg 1$ is large enough (for the specified initial data) to apply to \Cref{fig:FEM2}.

To be more precise, we can apply propagation estimates in $\Psi_{\natural}$. This explains why, as seen in the figure, the $L^2$-mass of $v_\pm$ closely follows classical trajectories. However, a fully rigorous proof requires controlling large $\natural$-frequencies.
Because our initial data is smooth, contributions to $v_\pm$ at large $\natural$-frequency are highly suppressed when $\mathsf{Z}\gg 1$. (A basic estimate suffices for this purpose. For example, conservation of $L^2$-mass for the Schr\"odinger equation can be cited.)

The only reason we needed to stay away from fiber infinity above is that the Schr\"odinger equation is parabolic, not hyperbolic, so the $\natural$-analysis degenerates at fiber infinity. A full analysis of the $\mathsf{Z}\to\infty$ limit of \cref{eq:Schr_ap} would define a parabolic analogue of 
$\operatorname{Diff}_{\natural}$, in which $h^2 \partial_t$ is considered zeroth order at $\natural \mathrm{f}$ and \emph{second} order at fiber infinity, like $h^2 \triangle$.

\section{Some lemmas relating various calculi}
\label{sec:relations}
We prove some inequalities between various Sobolev scales used in this article. 

\begin{lemma}
	For any $m,s\in \bbR$ and $\ell\geq \max\{0,2m\}$, 
	\begin{equation}
		\lVert v \rVert_{H_{\calczero}^{m,s,0 }}\lesssim \lVert v \rVert_{H_{\mathrm{par}}^{\ell,s}} 
	\end{equation}
	holds, in the usual strong sense, uniformly for $h \in (0, 1]$. 
	\label{lem:par_comp}
\end{lemma}
\begin{proof}
	For simplicity, we prove the $s=0$ case, with the general case following by replacing $v$ with $\ang{z}^s v$. We have
	\begin{multline}
		\lVert v \rVert_{H_{\calczero}^{m,0,0}}\lesssim \big\lVert  (-h^4\partial_t^2 + h^2\triangle + 1 )^{m/2}  v \big\rVert_{L^2} \\ = \Big\lVert (-h^4\partial_t^2 + h^2\triangle + 1 )^{m/2}  \frac{1}{(-\partial_t^2 + \triangle^2+1)^{\ell/4}} (-\partial_t^2 + \triangle^2+1)^{\ell/4} v \Big\rVert_{L^2} .
		\label{eq:misc_482}
	\end{multline}
	Letting $w = (\partial_t^2 + \triangle^2+1)^{\ell/4} v$,
	\begin{equation}
		\lVert w \rVert_{L^2} \lesssim \lVert v \rVert_{H_{\mathrm{par}}^{\ell,0}}, 
	\end{equation}
	by definition. Also, as long as $\ell \geq \max\{0,2m\}$,
	\begin{equation}
		\frac{(h^4 \tau^2 + h^2 |\xi|^2 +1 )^{m/2}}{(\tau^2 + |\xi|^4 + 1)^{\ell/4} }  \in L^\infty([0,1]_h\times \bbR^{1,d}_{\tau,\xi} ),
		\label{eq:misc_310}
	\end{equation}
	so $(-h^4\partial_t^2 + h^2\triangle + 1 )^{m/2}(-\partial_t^2 + \triangle^2+1)^{-\ell/4}$ is a uniformly bounded one-parameter family of operators on $L^2$. Then, \cref{eq:misc_482} gives $\lVert v \rVert_{H_{\calczero}^{m,0,0}}\lesssim \lVert w \rVert_{L^2} \lesssim \lVert v \rVert_{H_{\mathrm{par}}^{\ell,0}}$.
\end{proof}

\begin{lemma}  
For any $m,s \in \bbR$ and $\ell \geq \max\{0,2m\}$, and for any $O\in \Psi_{\calczero}^{-\infty,0,0}$ 
microlocally equal to the identity near $\mathrm{pf}$ (i.e.\ near $\{ \taun = 0, \xin = 0, h=0 \}$), we have 
	\begin{equation}
		\lVert (1-O) v \rVert_{H_{\calczero}^{m,s,\ell}} \lesssim \lVert v \rVert_{H_{\mathrm{par}}^{\ell,s}}
		\label{eq:misc_03g}
	\end{equation}
    uniformly for $h \in (0, h_0)$, $h_0$ sufficiently small. 
\label{lem:par_comp2}
\end{lemma}
\begin{proof}
	We proceed as in the proof of \Cref{lem:par_comp}, except, instead of \cref{eq:misc_310}, what we want to show is that, for any $\delta>0$, there exists a $C>0$ such that 
	\begin{equation} \label{eq:compare_zerocalc_to par_calc}
		\frac{(h^4 \tau^2 + h^2 \xi^2 +1 )^{m/2}}{h^\ell(\tau^2 + \xi^4 + 1)^{\ell/4} } \leq C\text{ whenever }h^4 \tau^2 + h^2 \xi^2 >\delta\text{ and }0 < h \leq h_0.
	\end{equation} The two differences between \cref{eq:compare_zerocalc_to par_calc} and the corresponding goal, \cref{eq:misc_310}, in the proof of \Cref{lem:par_comp} are
	\begin{itemize}
		\item the factor of $h^\ell$ in the denominator of \cref{eq:compare_zerocalc_to par_calc}, which is present because we are trying to control the $H_{\calczero}^{m,s,\ell}$-norm instead of the $H_{\calczero}^{m,s,0}$-norm, and  
		\item the restriction to $h^4 \tau^2 + h^2 \xi^2 >\delta$, that is $\taun^2 + \xin^2 >\delta$,  and $0 < h \leq h_0$. This restriction suffices because the contribution to $\lVert (1-O) v \rVert_{H_{\calczero}^{m,s,\ell}}$ from near $\mathrm{pf}$ is controlled by 
		\begin{equation}
			\lVert v \rVert_{H_{\calczero}^{-N,-N,-N}},  
		\end{equation}
		for any $N\in \bbN$, and we already know (e.g.\ by \Cref{lem:par_comp}), that $\lVert v \rVert_{H_{\calczero}^{-N,-N,-N}} \lesssim \lVert v \rVert_{H_{\mathrm{par}}^{\ell,s}}$.
	\end{itemize} 
	Let us now prove \cref{eq:compare_zerocalc_to par_calc}, which can be rewritten 
	\begin{equation}
		(\tau_\natural^2 + \xi_\natural^2 +1 )^{m/2}  \lesssim (\tau_\natural^2 + \xi_\natural^4 + h^4)^{\ell/4}.
		\label{eq:misc_314}
	\end{equation}
    We will handle the $m\leq 0$ and $m>0$ cases separately.
	
	If $m\leq 0$, then the left-hand side is bounded above by $1$. If  $\tau_\natural^2+\xi_\natural^2>\delta$, then 
	\begin{equation}
		\delta_0 \lesssim \tau_\natural^2 + \xi_\natural^4 ,
		\label{eq:misc_z3g}
	\end{equation}
	for some $\delta_0>0$, so 
	\begin{equation}
		\delta_0^{\ell/4} \leq (\tau_\natural^2 + \xi_\natural^4 + h^4)^{\ell/4}, 
	\end{equation}
	using the assumption $\ell\geq 0$. So, \cref{eq:misc_314} holds in this case.
	
	If $m>0$, then \cref{eq:misc_314} is equivalent to 
	\begin{equation}
		\tau_\natural^2+ \xi_\natural^2 +1 \lesssim \tau_\natural^2+\xi_\natural^4 + h^4, 
	\end{equation}
	which is equivalent to the conjunction of 
	\begin{enumerate}[label=(\roman*)]
		\item $\tau_\natural^2 \lesssim \tau_\natural^2+\xi_\natural^4 + h^4$, which holds trivially,   
		\item \:\,$1 \lesssim \tau_\natural^2+\xi_\natural^4 + h^4$, which holds in the relevant region by \cref{eq:misc_z3g}, and \vspace{.3em}
		\item 
        \begin{equation} 
            \xi_\natural^2 \lesssim  \tau_\natural^2+\xi_\natural^4 + h^4;
        \end{equation}
        this last inequality holds in $\{ \xi_{\natural}^2 >\varepsilon\} $, for arbitrary $\varepsilon>0$ (with an $\varepsilon$-dependent constant). On the other hand, if $\varepsilon$ and $h$ are sufficiently small, then $\{ \xi_{\natural}^2 <2\varepsilon\}$, together with $\tau_\natural^2+\xi_\natural^2>\delta$, forces $\tau_\natural^2>\varepsilon$, so the desired inequality holds in this region too.
	\end{enumerate}

\end{proof}

\begin{lemma}
	For any $m,s,\ell \in \bbR$, there exists $h_0 > 0$ such that 
	\begin{equation} 
		\lVert v \rVert_{H_{\calc}^{m,s,\ell,0} } \lesssim \lVert v \rVert_{H_{\mathrm{par}}^{L,s}}
		\label{eq:misc_311}
	\end{equation} 
	holds, in the usual strong sense, whenever $L\geq \max\{0,\ell, 2m\}$ and $h \in (0, h_0]$. 
	
	A variable order analogue also holds: if $\mathsf{s}\in C^\infty({}^{\calc }\overline{T}^* \bbM )$ is a variable $\calc$-order and $\overline{\mathsf{s}}=\mathsf{s}|_{\mathrm{pf}}$, then for any $m,\ell\in \bbR$ and $L\geq \max\{0,\ell,2m\}$, and for any $\varepsilon>0$, there exists an open neighborhood $U$ of $\mathrm{pf}$ such that, whenever $Q\in \Psi_{\calczero}^{-\infty,0,0}$ satisfies $\operatorname{WF}_{\calczero}'(Q)\subseteq U$, then 
	we have 
	\begin{equation} 
	\lVert Q v \rVert_{H_{\calc}^{m,\mathsf{s},\ell,0} } \lesssim \lVert v \rVert_{H_{\mathrm{par}}^{L,\overline{\mathsf{s}+\varepsilon}}},
	\end{equation} 
	where the constant may depend on $Q$ (but not on $v$).
	\label{lem:par_comp_better}
\end{lemma}
\begin{proof} We will only prove the constant order version, as the variable order version can be proved in a similar way.
	Let $O\in \Psi_{\calczero}^{-\infty,0,0}$ have essential support disjoint from $\mathrm{df}$ and be such that $1-O$ has essential support disjoint from $\mathrm{pf}$. Then, 
	\begin{align}
    \begin{split} 
		\lVert v \rVert_{H_{\calc}^{m,s,\ell,0} } 
		&\leq \lVert O v \rVert_{H_{\mathrm{par,I,res}}^{m,s,\ell,0} } + \lVert (1-O)v \rVert_{H_{\calc}^{m,s,\ell,0} }  \\ 
		&\lesssim \lVert  v \rVert_{H_{\mathrm{par,I,res}}^{-N,s,\ell,0} } + \lVert (1-O)v \rVert_{H_{\calc}^{m,s,\ell,0} } + \lVert v \rVert_{H_{\calczero}^{-N,-N,-N} }
        \end{split} 
	\end{align}
	for any $N$. 
	Trivially, 
	\begin{equation} 
		\lVert v \rVert_{H_{\calczero}^{-N,-N,-N}, }\lesssim  \lVert  v \rVert_{H_{\mathrm{par}}^{L,s} }
	\end{equation} 
	if $N$ is sufficiently large, e.g.\ using \Cref{lem:par_comp}. Also, if $N$ is sufficiently large, 
	\begin{align}
		\lVert  v \rVert_{H_{\mathrm{par,I,res}}^{-N,s,\ell,0} } &\lesssim \lVert  v \rVert_{H_{\mathrm{par,I,res}}^{\ell,s,\ell,0} } \lesssim \lVert  v \rVert_{H_{\mathrm{par}}^{\ell,s} } \lesssim \lVert  v \rVert_{H_{\mathrm{par}}^{L,s} }, \\ 
        \intertext{and, using \Cref{lem:par_comp2}}
		\lVert (1-O) v \rVert_{H_{\calc}^{m,s,\ell,0} } &\lesssim \lVert  v \rVert_{H_{\mathrm{par}}^{L,s} } .
	\end{align}
	So, all in all, we get \cref{eq:misc_311}.
\end{proof}

\begin{lemma}
	For $m,s,\ell \in \bbR$, we have 
 \begin{align*}
     \| u \|_{ H_{\calc}^{m,s,\ell,\ell} } &\approx   \| u \|_{ H_{\calczero}^{m,s,\ell} },  \\
    \lVert u \rVert_{H_{\calctwo}^{m,s,\ell;\ell,\ell}}  &\approx   \| u \|_{ H_{\calczero}^{m,s,\ell} }.
 \end{align*}   
	\label{lem:blowdown_equivalence}
\end{lemma}
\begin{proof}
Note that \begin{equation} 
    S_{\calczero}^{m,s,\ell}\subseteq S_{\calc}^{m,s,\ell,\ell}\cap S_{\calctwo}^{m,s,\ell;\ell,\ell}.
\end{equation}
A symbol in $S_{\calczero}^{m,s,\ell}$ which is elliptic is also elliptic in each of $S_{\calc}^{m,s,\ell,\ell}, S_{\calctwo}^{m,s,\ell;\ell,\ell}$. So, the lemma follows immediately from the definition of the $\calc$- and $\calczero$-Sobolev norms.
\end{proof}

\begin{lemma}
	For any $m\geq 0$, $\lVert u \rVert_{H_{\mathrm{sc}}^{m,s}} \lesssim \lVert u \rVert_{H_{\calc}^{m,s,2m,0} }$.
	\label{lem:sc_comp}
\end{lemma}
\begin{proof}
	Again, we only discuss the $s=0$ case, the general case being similar. Then, $H_{\mathrm{sc}}^{m,s}=H_{\mathrm{sc}}^{m,0} = H^m$ is the ordinary $L^2$-based Sobolev space on $\bbR^{1,d}$, and so 
	\begin{equation} 
		\lVert u \rVert_{H_{\mathrm{sc}}^{m,0}} \lesssim \lVert (-\partial_t^2+\triangle+1)^{m/2} u \rVert_{L^2} = \Big\lVert \frac{h^{2m}(-\partial_t^2+\triangle+1)^{m/2}}{(-h^4 \partial_t^2 + h^2 \triangle+1)^{m/2}}  h^{-2m} (-h^4 \partial_t^2 + h^2 \triangle+1)^{m/2}  u \Big\rVert_{L^2}.
		\label{eq:misc_507}
	\end{equation} 
	We have 
	\begin{equation} 
		\frac{h^{2m} (\tau^2 + \xi^2 + 1)^{m/2} }{(h^4 \tau^2 + h^2 \xi^2 + 1)^{m/2} }  \leq 1
	\end{equation}  
	for $h\leq 1$, 
	so the right-hand side of \cref{eq:misc_507} is $\lesssim \lVert h^{-2m} (-h^4 \partial_t^2 + h^2 \triangle+1)^{m/2}  u \rVert_{L^2} \lesssim \lVert u \rVert_{H_{\calczero}^{m,0,2m}}$. 
	Applying this with $(1-O)u$ in place of $u$, where $O\in \Psi_{\calczero}^{-\infty,0,0}$ has essential support disjoint from $\mathrm{df}$ and where $1-O$ has essential support disjoint from $\mathrm{pf}$, yields 
	\begin{equation}
	\lVert (1-O) u \rVert_{H_{\mathrm{sc}}^{m,s}} \lesssim \lVert (1-O) u \rVert_{H_{\calczero}^{m,s,2m}} \lesssim \lVert u \rVert_{H_{\calc}^{m,s,2m,0}}.
	\end{equation} 
	
	On the other hand, 
	\begin{equation}
	\lVert O u \rVert_{H_{\mathrm{sc}}^{m,s}}\lesssim \lVert O u \rVert_{H_{\mathrm{par}}^{2m,s}} \lesssim \lVert Ou \rVert_{H_{\mathrm{par,I,res}}^{2m,s,2m,0} } \lesssim \lVert u \rVert_{H_{\calc}^{m,s,2m,0} }. 
	\end{equation}
	So, altogether, $\lVert u \rVert_{H_{\mathrm{sc}}^{m,s}} \leq \lVert (1-O) u \rVert_{H_{\mathrm{sc}}^{m,s}}+ \lVert O u \rVert_{H_{\mathrm{sc}}^{m,s}} \lesssim \lVert u \rVert_{H_{\calc}^{m,s,2m,0} }.$
\end{proof}

Next, a Sobolev embedding theorem:
\begin{proposition} Let $s\in \bbR$, $m > (d+1)/2$. 
	\begin{enumerate}[label=(\alph*)]
		\item If $\ell\geq (d+2)/2$, then $\| u \|_{(1+r^2+t^2)^{-s/2}L^{\infty}} \lesssim \| u \|_{ H^{m,s,\ell}_{\calczero} }$.
		\item If $\ell > d+1$, then $\| u \|_{(1+r^2+t^2)^{-s/2}L^{\infty}} \lesssim \| u \|_{ H^{m,s,\ell,0}_{\calc} }$. 
	\end{enumerate}
 \label{prop:Sobolev_embedding}
\end{proposition}

\begin{proof}
	It suffices to consider the $s=0$ case.
	\begin{enumerate}[label=(\alph*)]
		\item  The proof proceeds like the proof of the usual Sobolev embedding theorem. We write the usual Fourier transform of $u$, denoted $\hat u$, in the coordinates $\taun, \xin$:
\begin{equation}
    \hat u(\taun, \xin) = \int e^{-i(h^{-2} \taun t + h^{-1} \xin \cdot x)} u(t, x) \, dt \, dx. 
\end{equation}
It follows that we have the inversion formula 
\begin{equation}
    u(t, x) = (2\pi)^{-n} h^{-(d+2)} \int \hat u(\taun, \xin) \, d\taun \, d\xin.
\end{equation}
We then estimate 
		\begin{align}
			\begin{split} 
				\| u \|_{L^\infty} &\lesssim   h^{-(d+2)}\| \hat{u} \|_{L^1(\bbR^{1+d}_{\tau_{\calcshort},\xi_{\calcshort}})} 
				\\ 
				&\leq h^{-(d+2)}\| (1+\tau_{\calcshort}^2+|\xi_{\calcshort}|^2)^{m/2} \hat{u} \|_{L^2(\bbR^{1+d}_{\tau_{\calcshort},\xi_{\calcshort}})} 
				\| (1+\tau_{\calcshort}^2+|\xi_{\calcshort}|^2)^{-m/2} \|_{L^2(\bbR^{1+d}_{\tau_{\calcshort},\xi_{\calcshort}})}\\
				&\lesssim h^{-(d+2)}\| (1+\tau_{\calcshort}^2+|\xi_{\calcshort}|^2)^{m/2} \hat{u} \|_{L^2(\bbR^{1+d}_{\tau_{\calcshort},\xi_{\calcshort}})}  
				\\ & \lesssim h^{-(d+2)/2}\|u\|_{ H^{m,0,0}_{\calczero} }  \lesssim \lVert u \rVert_{H_{\calczero}^{m,0,(d+2)/2}},
			\end{split} 
			\label{eq:Sobolev_comp}
		\end{align}
		where we used (i)
		\begin{equation} 
			\| (1+\tau_{\calcshort}^2+|\xi_{\calcshort}|^2)^{-m/2} \|_{L^2(\bbR^{1+d}_{\tau_{\calcshort},\xi_{\calcshort}})}<\infty,
		\end{equation} 
		which is true if $m>(d+1)/2$, and also (ii)
		\begin{align}
			\begin{split} 
			\| (1+\tau_{\calcshort}^2+|\xi_{\calcshort}|^2)^{m/2} \hat{u} \|_{L^2(\bbR^{1+d}_{\tau_{\calcshort},\xi_{\calcshort}})} &= h^{d+2} \| \operatorname{Op}((1+\tau_{\calcshort}^2+|\xi_{\calcshort}|^2)^{m/2}) u \|_{L^2(\bbR^{1+d}_{t_\natural,x_\natural})} \\ 
			&= h^{(d+2)/2} \| \operatorname{Op}((1+\tau_{\calcshort}^2+|\xi_{\calcshort}|^2)^{m/2}) u \|_{L^2(\bbR^{1+d}_{t,x}) = H_{\calczero}^{0,0,0}} \\ 
			&\lesssim h^{(d+2)/2} \lVert u \rVert_{H_{\calczero}^{m,0,0}}.
			\end{split} 
		\end{align}

        \item Follows from combining the usual Sobolev embedding theorem $\lVert u \rVert_{L^\infty} \lesssim \lVert u \rVert_{H_{\mathrm{sc}}^{m,0}}$ and \Cref{lem:sc_comp}.
	\end{enumerate}
\end{proof}

\section{Asymptotics for the natural Cauchy problem} \label{sec:naturalCauchy}

In this appendix, we briefly address the asymptotics for the \emph{natural} (advanced/retarded) inhomogeneous problem and the \emph{natural} Cauchy problem, which we introduced in \S\ref{subsec:natural}, using some analogues of our $\natural$-calculus. 
We still consider operator $P$ take the form in in \cref{eq:P_def} with metric $g$ as in \cref{eq:metric_form}, but only regularity only on the natural scale is needed.
More concretely, fixing $x_0 \in \bbR^d$, setting
\begin{equation}
   t_{\natural} = t/h^2, \; x_{\natural} = (x-x_0)/h,
\end{equation}
then we require coefficients of $P$ to satisfy bounded geometry type estimates on the $\natural$-scale, which means that $\partial_x,\partial_t$ are replaced by $\partial_{x_{\natural}} = h\partial_x, \partial_{t_\natural} = h^2\partial_t$ and $x,t$ replaced by $x_{\natural},t_{\natural}$.
Specifically, denote the class of functions $a \in C^\infty([0,1)_h \times \bbR^{1+d}_{t,x})$ satisfying (for $h>0$, but uniformly as $h \to 0$)
\begin{equation} \label{eq:def-nat-infty}
    |(h\partial_x)^{\alpha}(h^2\partial_t)^\beta a(t,x,h)| \leq C_{\alpha,\beta}
\end{equation}
by $S_{\natural,\infty}^0$.
Then we denote the class of differential operators spanned by at most $m$-fold products of $h^2\partial_t,h\partial_x$ with coefficients in $S_{\natural,\infty}^0$ by $\Diffnatinf^{m}$ and set $\Diffnatinf^{m,\ell} = h^{-\ell} \Diffnatinf^m$. 
Let $P_0 = \Box - c^2$ with $\Box$ being the d'Alembertian of the Minkowski metric, then our assumption on our operator $P$ in this appendix is 
\begin{equation} \label{eq:assumption-P-appendix-E}
    P - P_0 \in \Diffnatinf^{1,1} = h^{-1}\Diffnatinf^{1}. 
\end{equation}

Ignoring potential further first order perturbations, in terms of the metric $g$, letting $\Box_g \in \Diffnatinf^{2,2}$ be the d'Alembertian of $g$, \cref{eq:assumption-P-appendix-E} means
\begin{equation} 
\Box_g - \Box \in \Diffnatinf^{1,1} = h^{-1}\Diffnatinf^{1},    
\end{equation}
which is, intuitively, $O(h)$ compared with either $\Box_g-c^2$ or $\Box - c^2$. In particular, this is satisfied by metrics of the form \cref{eq:metric_form} while the $C^\infty$-coefficients are replaced by $S_{\natural,\infty}^0$ ones.

In the present setting, the inhomogeneous problem reads:
\begin{equation}
	\begin{cases}
		Pu=f \\ 
		\mathrm{supp} \, u \subset \{ t \geq 0 \}; \\ 
	\end{cases}
	\label{eq:natural-inhomogeneous}
\end{equation}
with $f \in C^\infty([0,1)_h; C_c^\infty([0,T]_{t_\natural} \times \bbR_{x_\natural}^d) )$ for a fixed $T>0$; and the Cauchy problem reads
\begin{equation}
	\begin{cases}
		Pu=0 \\ 
		u(0,x_{\natural}) = \varphi(x_{\natural},h) \\ 
		\partial_t u|_{t=0} = h^{-2} \psi(x_{\natural},h)
	\end{cases}
	\label{eq:natural-Cauchy}
\end{equation}
for $\varphi,\psi \in C^\infty([0,1)_h; C_c^\infty([0,T]_{t_\natural} \times \bbR_{x_\natural}^d) )$. This probes the non-relativistic limit on the natural scale (a small length scale), as emphasized in \S\ref{subsec:natural}. 
\begin{remark}
    If one keeps track of the precise amounts of regularity and decay needed to carry out
the discussion below, in particular noticing that estimates we are deriving are uniform in terms of the spatial support of functions, conditions on $f,\varphi,\psi$ can be relaxed to requiring only finite amount of regularity and decay. 
\end{remark}

Our goal is to show that $u\approx u_0$, where $u_0$ is the solution of the natural Cauchy problem for the \emph{free} Klein--Gordon equation, i.e.\ replace $P$ by the free Klein--Gordon operator  $P_0 = \square + h^{-2}$ in \cref{eq:natural-Cauchy}, freezing the forcing $f(-,h)$ or the profiles 
$\varphi(-,h),\psi(-,h)$ at $h=0$ (or equivalently $c= h^{-1} = \infty$). 
That is, for the forcing problem, $u_0$ satisfies
\begin{equation} \label{eq:free-KG-inhomogeneous}
	\begin{cases}
		P_0u_0=f(-,0) \\ 
		\mathrm{supp} \, u_0 \subset \{ t \geq 0 \};
	\end{cases}
\end{equation}
while for the Cauchy problem, $u_0$ satisfies
\begin{equation} \label{eq:free-KG-Cauchy}
	\begin{cases}
		P_0u_0=0 \\ 
		u_0(0,x) = \varphi(x_{\natural},0) \\ 
		\partial_t u_0|_{t=0} = h^{-2} \psi(x_{\natural},0).
	\end{cases}
\end{equation}
Once we have proven that $u\approx u_0$ holds (in an appropriate sense), it 
will serve as yet another validation of the heuristic that the free Klein--Gordon equation governs the non-relativistic on the natural scale.
Precisely, in terms of $c=h^{-1}$, our approximation for the forced problem is:
\begin{theorem}
Let $u,u_0,P,P_0$ be as in \cref{eq:natural-Cauchy}\cref{eq:free-KG-Cauchy} and set $w=u-u_0$, then for any fixed $T > 0$, we have
\begin{equation}
\| w \|_{L^\infty} = O(c^{ -1 }), \quad t \in \Big[0, \frac{T}{c^2} \Big],
\end{equation}
as $c\to\infty$. 
Here the implicit constant on the right-hand side depends linearly on norms of the forcing $f$.
		\label{thm:natural-inhomogeneous}
\end{theorem}

For the Cauchy problem, our approximation is:
\begin{theorem}
Let $u,u_0,P,P_0$ be as in \cref{eq:natural-inhomogeneous}\cref{eq:free-KG-inhomogeneous} and set $w=u-u_0$, then for any fixed $T > 0$, we have
\begin{equation}
\| w \|_{L^\infty} = O(c^{ -1 }), \quad t \in \Big[0, \frac{T}{c^2} \Big],
\end{equation}
as $c\to\infty$. 
Here the implicit constant on the right-hand side depends linearly on norms of the initial data $\varphi,\psi$.
\label{thm:naturalCauchy}
\end{theorem}

As discussed in \S\ref{subsec:natural}, $u_0$ is just given by rescaling the solution to the $h=1$ problem: 
\begin{equation}
	u_0(t,x;c) = u_0(t_\natural,x_\natural;1).
\end{equation}
Thus, $u_0$ is dilated down and simultaneously squished by an extra dilation in the time direction as $h \to 0$. Similar behavior is to be expected for $u$. 
Since the difference of their initial datum and operators happen at most at $O(h)$ level, it is natural to expect a local $L^\infty$ estimate for $u-u_0$, with right-hand side that is (at least) $o(1)$ as $h \to 0$.

We consider the forced problem \cref{eq:natural-inhomogeneous} first and show that on the natural scale one has lossless,
i.e.\ elliptic in terms of $h$-asymptotics, estimates.

In fact, due to the foliation structure that the $\natural$-algebra has, one can
even have longer time estimates with a loss of $h^{2-\delta}$ relative
to this if we proceed up to $O(h^\delta)$ times, which is an
$h^{2-\delta}$ gain relative to the natural scale. Concretely, this
holds for $\delta\geq 1$ without additional conditions even in fully
geometric settings, and for
$\delta\geq 0$ in the setting of the main body of this paper, where we require the regularity of coefficients of the operator on the laboratory scale, i.e. at the $(t,x)$-level. However, for simplicity, and as this appendix is illustrative, we proceed only with the above ``elliptic
in $h$'' estimate.

To see this, recall how energy estimates work in geometric settings, for instance
in the present setting but for fixed $h>0$. 
The following is essentially the standard proof of energy estimates,
see e.g.\ \cite[Section~2.8]{Taylor:Partial-I}. Here we phrase it as done in
\cite[Sections~3 and 4]{Vasy:AdS} and \cite[Section~3.3]{vasy2013asympthyp}. Consider $V=-i Z$,
$Z=\chi(\wavetime)W$, and let $W$ be given by $W=G(d\wavetime,.)$
($G=g^{-1}$ the dual metric), considered as a first order
differential operator on $M=\bbR^{1+d}$; where $\wavetime$ is a future oriented timelike
function (i.e.\ $d\wavetime$ is such), such as $t$ or $t-\ep(x-x_0)^2$
near $(0,x_0)$. 

As usual in energy estimates, we want to consider the `commutator'
\begin{equation}\label{eq:sigma-dep-comm}
-i(V^*\Box_g-\Box_g^* V),
\end{equation}
where the adjoint should be with respect to a positive definite inner
product (which the metric density provides on functions, and then $\Box_g^*=\Box_g$).

We first recall the computation of
$-i(V^*\Box_g-\Box_g V)$ with adjoints
taken using the Lorentzian density (so $\Box_g$ is formally
self-adjoint), see \cite[Section~3]{Vasy:AdS} for it written down in
this form. Using the standard summation convention, and the Lorentzian
inner product also on the fibers of the cotangent bundle, we have
\begin{equation}\begin{split}\label{eq:B_ij-formula}
&-i(V^*\Box_g-\Box_g V)=d^*Cd,\ C_i^j=g_{i\ell}B^{\ell j},\\
&B^{ij}=-J^{-1}\pa_k(JZ^kG^{ij})
+G^{ik}(\pa_k Z^j)+G^{jk}(\pa_k Z^i),
\end{split}\end{equation}
where $C_i^j$ are the matrix entries of $C$ relative to the basis $\{dz_\ell\}$
of the fibers of the cotangent bundle (rather than a rescaled cotangent
bundle), $z_\ell=x_\ell$ for $1\leq \ell\leq d$, $z_0=t$.
Expanding $B$ using $Z=\chi W$, and separating the terms
with $\chi$ derivatives, gives
\begin{equation}\begin{split}\label{eq:B_ij-exp-gen}
B^{ij}&=
G^{ik}(\pa_k Z^j)+G^{jk}(\pa_k Z^i)-J^{-1}\pa_k(JZ^kG^{ij})\\
&=(\pa_k \chi) \big(G^{ik}W^j+G^{jk}W^i-G^{ij}W^k\big)\\
&\qquad\qquad\qquad+\chi
\big(G^{ik}(\pa_k W^j)+G^{jk}(\pa_k W^i)-J^{-1}\pa_k(JW^kG^{ij})\big).
\end{split}\end{equation}
Multiplying the first term on the right hand side by
$\alpha_i\,\overline{\alpha_j}$ (and summing over $i,j$; here
$\alpha\in\Cx^n\simeq\Cx T^*_q(M)$, $q\in M$)
gives
\begin{equation}\begin{split}\label{eq:stress-energy}
E_{W,d\chi}(\alpha)& \overset{\mathrm{def}}{=}
(\pa_k \chi) (G^{ik}W^j+G^{jk}W^i-G^{ij}W^k)\alpha_i\,\overline{\alpha_j}\\
&=( \alpha,d\chi)_G \,\overline{\alpha(W)}
+\alpha(W)\,( d\chi,\alpha)_G-d\chi(W) (\alpha,\alpha)_G=\chi'(\wavetime)\,E_{W,d\wavetime},
\end{split}
\end{equation}
where 
\begin{equation}
E_{W,d\wavetime}=( \alpha,d\wavetime)_G \,\overline{\alpha(W)}
+\alpha(W)\,( d\wavetime,\alpha)_G-d\wavetime(W) (\alpha,\alpha)_G.
\end{equation}
Now, $E_{W,d\wavetime}$ is twice the sesquilinear stress-energy tensor
associated to $\alpha$, $W$ and $d\wavetime$ and is well-known to be positive definite in
$\alpha$, i.e.\ $E_{W,d\chi}(\alpha)\geq 0$, with vanishing if and only if $\alpha=0$,
when $W$ and $d\wavetime$ are both forward time-like for smooth Lorentz metrics,
see e.g.\ \cite[Section~2.7]{Taylor:Partial-I} or
\cite[Lemma~24.1.2]{Ho4}.

Since below we will consider the Klein-Gordon operator $\Box_g-h^{-2}$
(with $\Box_g$ itself of size $h^{-2}$, so in the present fixed $h$ setting
$\Box_g-1$ being the analogue), let
us remark the helpful contribution of the mass term (which is not
needed, except on the longer time scales that we do not address here,
but is convenient to use) by considering $\Box_g - 1$, i.e., the fixed $h=1$ setting:
$$
i(V^*1-1 V)=i(V^*-V)=J^{-1}\pa_k(JZ_k)=(\pa_k\chi)
W_k+\chi J^{-1}\pa_k(JW_k),
$$
and for $W$ and $d\wavetime$ both forward time-like, the coefficient
$d\wavetime(W)$ of $\chi'$
is positive, matching the sign of $E_{W,d\wavetime}$.

Correspondingly, for $P$\footnote{One can think of $P$ here as our $P$ restricted to $h=1$.} with $P-\Box_g+1\in\Diff^1$,
\begin{equation}\begin{split}\label{eq:perturbed-wave-comm}
&-i(V^* P-P^* V)\\
&=-i(V^*
(\Box_g-1)-(\Box_g-1) V) -i V^*
(P-\Box_g+1)+i (P-\Box_g+1)^* V\\
&=d_0^*
\tilde C_0 \,d_0+\hat E \chi d+d^*\chi\hat E^*,
\end{split}\end{equation}
where $d_0=(d,1)$, with $1$ taking care of the additional mass term, $\hat E\in\CI(M;TM)$,
and
$$
\tilde C_0=\chi'\tilde A_0+\chi \tilde R_0.
$$
Here the second order terms arising from the $P-\Box_g+1$ terms are
simply incorporated into $\chi \tilde R_0$ as they do not
involve derivatives of $\chi$, i.e.\ we are not taking advantage of
any potential drop in order, which in any case is generally not present except
under leading order symmetry assumptions on $P-\Box_g+1$.
A standard argument, making $\chi'$ large relative to
$\chi$, completes the proof of the energy estimate; here we have the
advantage of the mass term giving the positivity required for right
away estimating the $H^1$-norm (hence we do not need to use a
Poincar\'e inequality). 
Thus, we obtain (taking $0$ initial data, or
equivalently working with distributions supported in $t\geq 0$) the
spacetime estimate, in a slab of $\wavetime$ level sets (or a tweaked
version, such as if $\wavetime=t-\ep(x-x_0)^2$ is used above, and the
region is $t\geq 0$, $\wavetime$ less than a small positive constant),
$$
\|u\|^2_{H^1}\leq C'|\langle Pu, Vu\rangle|+\gamma\|u\|^2_{H^1}\leq \frac{1}{2}C \|Pu\|^2_{L^2}+\frac{1}{2}\|u\|_{H^1}^2 +\gamma\|u\|^2_{H^1},
$$
with $\gamma>0$ term arising from the undifferentiated $\chi$ term in
the commutator and $\gamma$ is small if $\chi'$ is large relative to
$\chi$, in which case the last two terms can be absorbed in the left
hand side. 
Here all norms and pairings are over $[0,T_F] \times \bbR^d_{\natural}$ for some $T_F>0$ so that $\supp \chi \cap [0,\infty) \subset [0,T_F]$.
And the same choice regions of integration for $(t_\natural,x_\natural)$ applies below, which corresponds to $[0,h^2T_F] \times \bbR^d_{x}$ in terms our original $(t,x)$.

We now describe the argument for the estimate in full detail in
the $\natural$-setting; the main point for now is that
\eqref{eq:perturbed-wave-comm} leads to such an estimate, so the key
is to establish an analogue of \eqref{eq:perturbed-wave-comm} in the
$\natural$-setting. We already remark also that if one has a family of
Lorentzian metrics and corresponding operators smoothly (or indeed just continuously) depending on
a parameter, then the energy estimate, with the same $\wavetime$
remains valid for the perturbations of a metric and of the operator $P$.

Now, before proceeding with the detailed argument in the
$\natural$-setting we comment that the geometric nature of the energy
estimate means that it could be applied by rescaling the variables to
$x_\natural=(x-x_0)/h$, $t_\natural=t/h^2$. In the new variables
$h^2P$ is a standard differential operator $\tilde P_\natural$, or rather a family
depending on $h$, uniformly down to $h=0$. Correspondingly, the above
argument, using $\wavetime_\natural$ as the argument of $\chi$, for
instance $\wavetime_\natural=t_\natural=t/h^2$, provides uniform estimates
$$
\|u_\natural\|^2_{H^1}\leq C \|\tilde P_\natural u_\natural\|^2_{L^2},
$$
which means on the original scale that
$$
\|u\|_{H^1_\natural}\leq Ch^2\|Pu\|_{L^2}
$$
as claimed earlier, where the $H^1_\natural$-norm in this appendix is the sum of the $L^2$-norm of $h^2\partial_tu,\, h\partial_xu$ over $[0,h^2T_F]_t \times \R_x^d$.
Note that this {\em only} uses the
$\natural$-regularity of the coefficients of $P$.

We work this out in a more global manner by changing to a local basis of the $\natural$-cotangent bundle and using the
$\natural$-differential $\natdiff=(h^2 \pa_t,h\pa_x)$ and the local basis
$\{h^{-2}\,dt,h^{-1}\,dx_1,\ldots h^{-1}\,dx_d\}$ of the fibers of the b-cotangent
bundle, $\hat\pa_j=\delta_j\pa_j$, $\delta_j=h$ for $1\leq j\leq d$, $\delta_0=h^2$,
for the local basis of the fibers of the b-tangent bundle, write
$\hat G^{ij}$, $\hat g_{ij}$ for the corresponding metric entries, $\hat
Z^i$ for the vector field components. For clarity, we already mentioned that will take the cutoff $\chi$
to be of the form $\chi(\wavetime_\natural)$. This yields
\begin{equation}\begin{split}\label{eq:b-B_ij-formula}
&-i(V^*\Box_g-\Box_g V)=\natdiff^*\,\hat C\,\natdiff,\ \hat C_i^j=\hat g_{i\ell}\hat B^{\ell j}\\
&\hat B^{ij}=-J^{-1}\delta_k^{-1}\delta_i^{-1}\delta_j^{-1}\hat\pa_k(J\delta_k\hat Z^k\delta_i\delta_j\hat G^{ij})
+\hat G^{ik}(\delta_j^{-1}\hat\pa_k \delta_j\hat Z^j)+\hat G^{jk}(\delta_i^{-1}\hat\pa_k \delta_i\hat Z^i).
\end{split}\end{equation}
We conclude that
\begin{equation}\begin{split}\label{eq:natural-B}
\hat B^{ij}=&(\hat\pa_k \chi) \big(\hat G^{ik}\hat W^j+\hat G^{jk}\hat
W^i-\hat G^{ij}\hat W^k\big)\\
&\qquad+\chi
\big(\hat G^{ik}\delta_j^{-1}(\hat \pa_k \delta_j\hat W^j)+\hat G^{jk}\delta_i^{-1}(\hat\pa_k
\delta_i \hat W^i)\\
&\qquad\qquad\qquad-J^{-1}\delta_k^{-1}\delta_i^{-1}\delta_j^{-1}\hat\pa_k(J\delta_k\delta_i\delta_j\hat
W^k\hat G^{ij})\big)\\
=&(\hat\pa_k \chi) \big(\hat G^{ik}\hat W^j+\hat G^{jk}\hat
W^i-\hat G^{ij}\hat W^k\big)\\
&\qquad+\chi
\big(\hat G^{ik}(\hat \pa_k \hat W^j)+\hat G^{jk}(\hat\pa_k \hat
W^i)-J^{-1}\hat\pa_k(J \hat W^k\hat G^{ij})\big),
\end{split}\end{equation}
and so
\begin{equation*}
\hat C=\chi' A+\chi R,
\end{equation*}
with $A$ positive definite in the following sense (keeping in mind
that $G=h^{-2}G'$, $G'$ unweighted, i.e.\ a smooth non-degenerate
$\natural$-dual metric): fixing any positive definite
$\natural$-inner
product of order $0$ (again, this means smooth and non-degenerate)
$\tilde g$ on $\Tnat M$,
\begin{equation}\label{eq:pos-def-natural-Lorentz}
  \langle\hat C
\alpha,\alpha\rangle_g\geq ch^{-2}\|\alpha\|^2_{\tilde g},\qquad
\alpha\in {}^{\Cx}\Tnat^*M,
\end{equation}
while $R$ is in the analogous sense
$O(h^{-2})$. Here we wrote
$$
d(\chi(\wavetime_\natural))=\chi'(\wavetime_\natural)
d(\wavetime_\natural),
$$
including the latter factor in $A$. Notice that the
inner product on the left hand side of \eqref{eq:pos-def-natural-Lorentz} is with respect to $g$, and is
thus not positive definite. It is thus useful to rewrite
$$
\natdiff^*_g\,\hat C\,\natdiff=\natdiff^*_{\tilde g}\,\tilde C\,\natdiff,\
\tilde C=\chi' \tilde A+\chi \tilde R,
$$
where the subscript denotes the inner product with respect to which the
adjoint is taken. Also notice that $\hat\pa_k=\delta_k\pa_k$ is the relevant
derivative entering in the tensors above, so it suffices for the
operator coefficients to have $\natural$-regularity, but then
$\hat\pa_k \chi$ still needs to be large relative to $\chi$, which
means, as we already indicated, that the argument of $\chi$ needs to be on the $\natural$-scale as
well, hence our consideration of
$\chi(\wavetime_\natural)$. Note that here {\em any}
timelike replacement for $\wavetime_\natural$ works -- everything is on the $\natural$-scale,
so the structure of the time foliation defining the $\natural$-differential
operators is not entering in this choice {\em
  if phrased in the $\natural$-language}.

As above, it is helpful (though not necessary) to take advantage of
the mass term.
For $P$ with
$P-\Box_g+h^{-2}\in h^{-2}\Diffnatinf^1$ an argument as above in the
standard setting first yields
\begin{equation}\begin{split}\label{eq:perturbed-wave-comm-natural}
&-i(V^* P-P^* V)\\
&=-i(V^*
(\Box_g-h^{-2})-(\Box_g-h^{-2}) V) -i V^*
(P-\Box_g+h^{-2})+i (P-\Box_g+h^{-2})^* V\\
&=\natdiff_0^*
\tilde C_0 \,\natdiff_0+\hat E \chi \natdiff+\natdiff^*\chi\hat E^*,
\end{split}\end{equation}
where $d_0=(d,1)$, with $1$ taking care of the additional mass term,
$\hat E\in h^{-2}\CI(M\times[0,1)_h;\Tnat M)$,
and
$$
\tilde C_0=\chi'\tilde A_0+\chi \tilde R_0,
$$
with $\tilde A,\tilde R$ of similar properties as above. From this,
by an argument we recall below, we conclude
\begin{equation} \label{eq:app-E-est}
h^{-2}\|u\|^2_{H^1_\natural}\leq C'|\langle Pu, Vu\rangle|+\gamma
h^{-2}\|u\|^2_{H^1_\natural}\leq
\frac{h^2}{2}C\|Pu\|^2_{L^2}+\frac{1}{2h^2}\|u\|_{H^1_\natural}^2 +\gamma h^{-2}\|u\|^2_{H^1_\natural},
\end{equation}
and thus (with $\gamma>0$ small)
$$
\|u\|_{H^1_\natural}\leq Ch^2\|Pu\|_{L^2};
$$
note that while in the differential sense this is the usual real
principal type numerology, in the $\natural$-sense it is the elliptic
numerology. Keep in mind here that the estimate is on the
$\natural$-scale in $t$, i.e.\ for finite $t_\natural$, as noted above.

We finally recall the more detailed structure of the argument for the energy
estimate when one wants to make assumptions for $t_\natural$ in $[T_0,T_0']$ (or
instead turn this into initial conditions by letting $T_0=T_0'$, or
under a support assumption simply dropping this term) and
conclusions for $t_\natural$ in $[T_0',T_1]$. Let
$\chi_0(s)=e^{-1/s}$ for $s>0$, $\chi_0(s)=0$ for $s\leq 0$, $\chi_1\in\CI(\RR)$
identically 1 on $[1,\infty)$, vanishing on $(-\infty,0]$,
Thus,
$s^2\chi_0'(s)=\chi_0(s)$ for $s\in\RR$.
Now let $T_1'>T_1$, and
consider
$$
\chi(s)=\chi_0(-\digamma^{-1}(s-T'_1))\chi_1((s-T_0)/(T_0'-T_0)).
$$ 
Then (the argument of $\chi$ is $s=\wavetime_\natural=t_\natural$ below)
\begin{itemize}
\item
$\supp\chi\subset [T_0,T'_1]$,
\item
$s\in [T'_0,T'_1]\Rightarrow \chi'=-\digamma^{-1}\chi_0'(-\digamma^{-1}(s-T'_1)),$
\end{itemize}
so
$$
s\in [T'_0,T'_1]\Rightarrow \chi=-\digamma^{-1} (s-T'_1)^2\chi',
$$
so for $\digamma>0$ sufficiently large, this is bounded by a small
multiple of $\chi'$, namely
\begin{equation}\label{eq:chip-gamma-est}
s\in [T'_0,T'_1]\Rightarrow \chi\leq -\gamma\chi',
\ \gamma=(T'_1-T'_0)^2\digamma^{-1}.
\end{equation}
In particular, for sufficiently large
$\digamma$,
$$
-(\chi' \tilde A_0+\chi \tilde R_0)\geq -\chi_0'\chi_1 \tilde A/2
$$
on $[T'_0,T'_1]$. Thus, without assuming the forward support property and
thus adding a control region,
$$
\langle {}^\natural d^*_0 \tilde C {}^\natural d_0 u,u\rangle\geq -\frac{1}{2}\langle \chi_0'\chi_1
\tilde A {}^\natural d_0 u,{}^\natural d_0 u\rangle-C'h^{-2}\|{}^\natural d_0 u\|^2_{L^2(\wavetime_\natural^{-1}([T_0,T'_0]))}.
$$
So for some $c_0>0$,  by \eqref{eq:pos-def-natural-Lorentz} (made more
positive by the mass term),
\begin{equation}\begin{split}\label{eq:Box-V-est}
c_0h^{-2}&\|(-\chi_0')^{1/2}\chi_1^{1/2}{}^\natural d_0 u\|^2\leq-\frac{1}{2}\langle \chi_0'\chi_1
\tilde A {}^\natural d_0 u,{}^\natural d_0 u\rangle\\
&\leq C'h^{-2}\|{}^\natural d_0
u\|^2_{L^2(\wavetime_\natural^{-1}([T_0,T'_0]))}+C'\|\chi^{1/2}P
u\|\|\chi^{1/2}{}^\natural d_0 u\|+C' h^{-2}\|\chi^{1/2} u\|\|\chi^{1/2}{}^\natural d_0 u\|\\
&\leq C'h^{-2}\|{}^\natural d_0
u\|^2_{L^2(\wavetime_\natural^{-1}([T_0,T'_0]))}+C'h^2\|\chi^{1/2}P
u\|^2+2C'\gamma h^{-2}\|(-\chi_0')^{1/2}\chi_1 {}^\natural d_0 u\|^2,
\end{split}\end{equation}
where we used $\|\chi^{1/2}u\|\leq \|\chi^{1/2}{}^\natural d_0
u\|$ (which holds in view of the last component of ${}^\natural d_0$).
Thus, choosing first $\digamma>0$ sufficiently large (thus $\gamma>0$
is sufficiently small), the last term on the right hand
side can be absorbed into the left hand side. This gives the desired result, except that
$\|P u\|_{L^2(\wavetime_\natural^{-1}([T_0,T_1]))}$ is replaced by
$\|P u\|_{L^2(\wavetime_\natural^{-1}([T_0,T'_1]))}$
(or $\|\chi^{1/2}P
u\|_{L^2(\wavetime_\natural^{-1}([T_0,T'_1]))}$). This however is easily
remedied by replacing $\chi$ by
$$
\chi(s)=H(T_1-s)\chi_0(-\digamma^{-1}(s-T'_1))\chi_1((s-T_0)/(T_0'-T_0)),
$$
where $H$ is the Heaviside step function (the characteristic function
of $[0,\infty)$) so $\supp\chi\subset [T_0,T_1]$. Now $\chi$ is not
smooth, but either approximating $H$ by smooth bump functions and
taking a limit, or indeed directly performing the calculation,
integrating on the domain with boundary $\wavetime_\natural\leq T_1$, the contribution
of the derivative of $H$ to $\chi'$ is a delta distribution at
$\wavetime_\natural=T_1$, corresponding to a boundary term on the domain, which has the
same sign as the derivative of $\chi_0$. Thus, with $\cS_{T_1}$ the
hypersurface $\{\wavetime_\natural=T_1\}$, \eqref{eq:Box-V-est}
holds in the form
\begin{equation*}\begin{split}
c_0h^{-2}&\|{}^\natural d_0 u\|^2_{L^2(\wavetime_\natural^{-1}([T'_0,T_1]))}+c_0h^{-2}\|{}^\natural d_0 u\|^2_{L^2(\cS_{T_1})}\\
&\leq C'h^{-2}\|{}^\natural d_0
u\|^2_{L^2(\wavetime_\natural^{-1}([T_0,T'_0]))}+C'h^2\|\chi^{1/2}P
u\|^2+2C'\gamma h^{-2}\|(-\chi_0')^{1/2}\chi_1 H {}^\natural d_0 u\|^2.
\end{split}\end{equation*}
Now one can simply drop the second term from the left hand side and
proceed as above.

One can easily add additional
derivatives to get
$$
\|u\|_{H^s_\natural}\leq Ch^2\|Pu\|_{H^{s-1}_\natural}.
$$
Recall that this works by commuting a spanning set of $\natural$-vector fields
through $P$. (An alternative is to systematically ``commute'' through
${}^\natural\nabla$ with respect to a Riemannian $\natural$-metric; here ``commute'' is
really intertwining since one goes to tensor bundles of one rank higher.) Denoting this collection by $X_j$, $j=0,1,\ldots,d$, in
order to estimate $X_ju$ we
write
$$
PX_ju=X_jPu+[P,X_j]u
$$
and expand the commutator as $\sum Q_{jk}X_k+Q_{-1k}$, where in our case
(where we even have the mass term to simplify things) it is
convenient to allow $X_{-1}=1$ in this, where
$Q_{jk}\in h^{-2}\Diffnatinf^{1}$, and the $h$ power is due to the
corresponding order of $P$. We then work with a system of equations for
$(u,X_0u,\ldots,X_du)$ which has first order terms $Q_{jk}$. Since one
proceeds inductively, i.e.\ one assumes that for $u$ one already has
the estimates that one is trying to prove for the $X_j u$, any
commutator term that is already sufficiently controlled given the
estimates for $P$ can be put on the right hand side alternatively;
here this would be $0$th order terms in $Q_{jk}$ and all of
$Q_{-1k}$. In any case, the first order terms $Q_{jk}$ do not make any difference
in the energy estimate argument above, so one initially obtains
$$
\|X_j u\|_{H^1_\natural}\leq Ch^2\|Pu\|_{H^1_\natural},
$$
and then one proceeds inductively for the full conclusion.

This in particular says that if $Pu=f$, $P_0u_0=f$ are
forward solutions, with $P,P_0$ operators with similar properties and
$P-P_0=h^{-\mu}\Diffnatinf^1$ then $w=u-u_0$ satisfies $Pw=-(P-P_0)u_0$,
i.e.\ $Pw$ is $O(h^{-\mu})$ worse in terms of $h$-decay than $u,u_0$, and
thus $w$ is $O(h^{2-\mu})$ times the size of $u,u_0$, hence the
solutions are close as long as $\mu<2$. This in particular is
satisfied with $P_0$ being the free Klein-Gordon operator and $\mu=1$
by our assumptions and concludes the proof of \Cref{thm:natural-inhomogeneous}.

This can be immediately translated to the Cauchy problem by further inserting a cut-off $H(t_\natural)$, where $H$ is still the Heaviside step function.
Now we have $Pu=0$ in process above but the derivative of this new factor will give contributions with a $\delta(t_\natural)$-factor, which will end up to be the energy associated to the initial data on the right hand side of \cref{eq:app-E-est} and gives \Cref{thm:naturalCauchy}.


\printbibliography

\end{document}